%% file: main.tex
\documentclass{amsart}
\usepackage[normalem]{ulem}
\usepackage{amsmath,amsthm,amsfonts,amscd,amssymb, tikz-cd} 
\usepackage{verbatim}      	
\usepackage{listings} 
\usepackage{mathtools}
\usepackage{epsfig}         	
\usepackage{url}		
\usepackage{epic,eepic,latexsym}
\usepackage{graphicx}
\usepackage{xcolor}

\usepackage{lcd}
\usepackage{mathrsfs}
\usepackage[centering,text={15.5cm,22cm},
        marginparwidth=20mm,dvips]{geometry}
\usepackage[linktocpage]{hyperref}
\usepackage{lscape}
\usepackage{tabularx}
\usepackage{marginnote}
\usepackage[all]{xy}
\usepackage{enumitem}
\usepackage{stmaryrd}
\usepackage{tipa}
\usepackage{upgreek}

\DeclareMathOperator{\tw}{tw}

\newcommand{\QQ}{\mathbb{Q}}
\newcommand{\cc}{{\bf c}}

\DeclareMathOperator{\Hom}{Hom}

\DeclareMathOperator{\id}{Id}
\DeclareMathOperator{\Id}{Id}

\DeclareMathOperator{\Tr}{Tr}

\DeclareMathOperator{\sign}{sgn}
\DeclareMathOperator{\SL}{SL}

\DeclareMathOperator{\prin}{prin}

\DeclareMathOperator{\In}{in}

\DeclareMathOperator{\Ends}{Ends}

\DeclareMathOperator{\Joints}{Joints}

\DeclareMathOperator{\EE}{\mathbb{E}}

\DeclareMathOperator{\PSL}{PSL}
\DeclareMathOperator{\PGL}{PGL}

\DeclareMathOperator{\Supp}{Supp}

\DeclareMathOperator{\Li}{Li}

\DeclareMathOperator{\Ad}{Ad}
\DeclareMathOperator{\ad}{ad}

\DeclareMathOperator{\up}{up}

\DeclareMathOperator{\midd}{mid}

\DeclareMathOperator{\ord}{ord}
\DeclareMathOperator{\can}{can}

\DeclareMathOperator{\scat}{Scat}

\DeclareMathOperator{\uf}{uf}

\DeclareMathOperator{\Links}{Links}
\DeclareMathOperator{\Sk}{Sk}
\DeclareMathOperator{\SMulti}{SMulti}
\DeclareMathOperator{\wSMulti}{wSMulti}
\DeclareMathOperator{\Bang}{Bang}
\DeclareMathOperator{\Brac}{Brac}
\DeclareMathOperator{\Band}{Band}

\DeclareMathOperator{\gr}{wt}

\newcommand{\E}[1]{\langle #1 \rangle}
\newcommand{\BangE}[1]{\langle #1 \rangle_{\Bang}}
\newcommand{\BracE}[1]{\langle #1 \rangle_{\Brac}}
\newcommand{\BandE}[1]{\langle #1 \rangle_{\Band}}

\def\!{\mskip-\thinmuskip}
\let\?=\overline
\let\:=\colon
\let\llb=\llbracket
\let\rrb=\rrbracket
\let\bb=\mathbb

\let\rar=\rightarrow
\let\f=\mathfrak
\let\s=\mathcal

\let\wh=\widehat
\let\wt=\widetilde

\let\mr=\mathring

\def\risom{\buildrel\sim\over{\smashedlongrightarrow}}
 \def\smashedlongrightarrow{\setbox0=\hbox{$\longrightarrow$}\ht0=1.25pt\box0}

\newcommand {\kk} {\Bbbk}
\newcommand {\kt} {\kk_t}

\newcommand {\sQ} {\mathcal{Q}}

\newcommand {\SSS} {{\bf S}}
\newcommand {\sd} {{\bf s}}
\newcommand {\MM} {{\bf M}}
\newcommand {\jj} {\vec{\jmath}}
\newcommand {\iii} {{\bf i}}
\newcommand {\kkk} {\vec{k}}

\theoremstyle{plain}
\input{macroEnv.tex}

  \newtheorem{ntn}[thm]{Notation}
 
  \newtheorem{dfn}[thm]{Definition}

  \newtheorem{conj}[thm]{Conjecture}

    \newtheorem{asm}[thm]{Assumption}

\newcommand{\Q}{{\mathbb Q}}
\newcommand{\R}{{\mathbb R}}
\newcommand{\Z}{{\mathbb Z}}
\newcommand {\A} {{\bf A}}

\newcommand{\cA}{\mathcal{A}}

\newcommand{\supp}{\operatorname{supp}}

\newcommand{\Mono}{\operatorname{Mono}}

\newcommand{\pr}{\operatorname{pr}} 

\newcommand{\DT}{\operatorname{DT}} 

\newenvironment{myproof}[1][\proofname]{\proof[#1]\mbox{}}{\endproof}

\input{macro.tex}

\title{Bracelets bases are theta bases}
\author{Travis Mandel}
\address{Department of Mathematics \\ University of Oklahoma \\ Norman, OK 73019 \\ USA}
\email{tmandel{\char'100}ou.edu}
\author{Fan Qin}
\address{School of Mathematical Sciences \\ Shanghai Jiao Tong University \\ China}
\email{qin.fan.math@gmail.com}
\dedicatory{Dedicated to Bernhard Keller and Sean Keel on the occasion of their 60th birthdays}
\thanks{When work on this project began, the first author was being supported by the Starter Grant ``Categorified Donaldson-Thomas Theory'' no. 759967 of the European Research Council.}

\begin{document}

        \begin{abstract}	
		The skein algebra of a marked surface, possibly with punctures, admits the basis of (tagged) bracelet elements constructed by Fock-Goncharov and Musiker-Schiffler-Williams.  As a cluster algebra, it also admits the theta basis of Gross-Hacking-Keel-Kontsevich, quantized by Davison-Mandel. We show that these two bases coincide (with a caveat for notched arcs in once-punctured tori).  In unpunctured cases, one may consider the quantum skein algebra. We show that the quantized bases also coincide.  Even for cases with punctures, we define quantum bracelets for the cluster algebras with coefficients, and we prove that these are again theta functions.

			On the corresponding cluster Poisson varieties (parameterizing framed $\PGL_2$-local systems), we prove in general that the canonical coordinates of Fock-Goncharov, quantized by Bonahon-Wong and Allegretti-Kim, coincide with the associated (quantum) theta functions. 
			
			Long-standing conjectures on strong positivity and atomicity follow as corollaries. Of potentially independent interest, we examine the behavior of cluster scattering diagrams under folding.
	\end{abstract}

        \maketitle
 
	\setcounter{tocdepth}{1} \tableofcontents{}
	
	\section{Introduction}
	
	\label{intro}

	\subsection{Overview}
	
	Cluster algebras are algebras with certain combinatorial structures. They were introduced by Fomin and Zelevinsky \cite{FominZelevinsky02} with the desire to understand the dual canonical basis \cite{Lusztig90, Lusztig91,Kas:crystal} and the theory of total positivity \cite{Lusztig96}. In particular, looking for good bases has always been a fundamental and important problem in the theory of cluster algebras. Well-known and important bases for cluster algebras include:
	\begin{itemize}
	\item the generic basis in the spirit of \cite{dupont2011generic}: this generalizes the dual semi-canonical basis to cluster algebras \cite{Lusztig00,GeissLeclercSchroeer10b}, see \cite{plamondon2013generic,qin2019bases} for its construction and existence;
	\item the common triangular basis introduced by \cite{qin2017triangular}: this is a Kazhdan-Lusztig type basis generalizing the dual canonical basis to cluster algebras \cite{qin2020dual}, see also \cite{BerensteinZelevinsky2012} for \textit{acyclic} cases;
	\item the theta basis introduced by \cite{gross2018canonical}: this consists of theta functions arising from the study of mirror symmetry and scattering diagrams.
 
	\end{itemize}
	
	Following \cite{FockGoncharov06a} and \cite{FominShapiroThurston08, fomin2018cluster}, one can construct cluster algebras on the (decorated) Teichmüller spaces of surfaces with marked points. In this case, algebra elements are represented by topological objects: the homotopy classes of unions of curves, which are subject to Kauffman skein relations. In addition to being of interest in Teichm\"uller theory and geometric representation theory, this surface model provides a convenient framework for visualizing the algebraic structures. There are three well-known bases represented by topological objects \cite{musiker2013bases,Thurst}: the bangles basis, the bands basis, and the bracelets basis.  We note that \cite{FockGoncharov06a} has also constructed the bracelets basis as certain coordinates on the moduli space of twisted decorated $\SL_2$-local systems, cf. \S \ref{sec:dec_SL2_moduli}.

	It has been long conjectured \cite{FockGoncharov06a,Thurst} that the bracelets basis satisfies the following properties:
	\begin{itemize}
	\item It is (strongly) positive, i.e., its structure constants under multiplication are non-negative.
	\item It is atomic, i.e., it is the extremal positive basis.
	\end{itemize}
	The strong positivity conjecture was verified by \cite{FockGoncharov06a,Thurst} in the classical case.
	
	In this paper, we prove that the bracelets basis coincides with the theta basis (in most cases) at both the classical and the quantum level. This result enables us to visualize theta functions in terms of topological objects. The above conjectures follow as consequences using the corresponding known properties of the theta bases \cite{gross2018canonical,ManAtomic,davison2019strong}.

	We will not treat other bases in this paper (see \S \ref{sub:other} for further information).

	\subsection{Background and main results}

	\subsubsection*{Background}
	Let us first recall necessary notions and constructions, see \S \ref{SkeinSection}-\S \ref{section:bracelet_skein} for details.
	
	We fix a base ring $\kk\supset \bb{Z}$ and denote $\kk_t=\kk[t^{\pm D}]$ for some sufficiently divisible positive integer $D$.
	
	Let $\SSS$ be a compact oriented surface, possibly with boundary. Let $\MM$ be a finite collection of distinct marked points in $\SSS$. Markings in the interior of $\SSS$ will be called punctures. We will consider the marked surface $\Sigma=(\SSS,\MM)$ with a mild condition that $\Sigma$ is triangulable (see \S \ref{sub:ClSk}).
	
	We consider curves on $\Sigma$ which either end at $\MM$ or are closed (meaning with empty boundary).  The closed curves are called loops, and the non-closed curves arcs.  A multicurve is a union of curves considered up to homotopy fixing the endpoints. The skein algebra $\?{\Sk}(\Sigma)$ associated to $\Sigma$ is the $\kk$-algebra generated by the multicurves, subject to the Kauffman skein relations (cf. \S \ref{SectionSkeinDef}), whose multiplication is provided by the union.  Diagrams which cannot be reduced by the skein relations are said to be simple. The (localized) skein algebra $\Sk(\Sigma)$ is the localization of $\?{\Sk}(\Sigma)$ at the boundary arcs. 
	
	For any ideal triangulation $\Delta$ of $\Sigma$, we have a seed $\sd_\Delta$ with boundary coefficients as in \cite{FockGoncharov06a,FominShapiroThurston08,muller2016skein}. Then $\Sk(\Sigma)$ is contained in the corresponding upper cluster algebra $\s{A}^{\up}(\sd_\Delta)$. A weighted simple multicurve $C=\bigsqcup w_i C_i$ is a weighted sum of non-intersecting non-homotopic components $C_i$, such that $C_i$ are simple arcs or loops, $w_i\in \Z$, and $w_i\in \NN$ whenever $C_i$ is not a boundary arc. By \cite{musiker2013bases,Thurst}, we can construct bracelet elements $\BracE{C}=\prod \BracE{w_i C_i}\in \s{A}^{\up}(\sd_\Delta)$. Here, $\BracE{w_iC_i}\coloneqq [C_i]^{w_i}$ if $C_i$ is an arc, and if $C_i$ is a loop, then $\BracE{w_iC_i}\coloneqq T_{w_i}([C_i])$ where $T_{w_i}$ is the $w_i$-th Chebyshev polynomial of the first kind (cf. \S \ref{sec:bracelet_band}). The bracelet elements form a basis for $\Sk(\Sigma)$.
	
	Following \cite{muller2016skein}, when $\Sigma$ is unpunctured, we can also introduced the quantum skein algebra $\Sk_t(\Sigma)$ by working with $\kk_t$ instead of $\kk$, using links instead of multicurves, and refining the skein relations with powers of $q=t^2$ (cf. Figure \ref{SkeinFig}). One also refines $\sd_\Delta$ into a quantum seed. Then $\Sk_t(\Sigma)$ is contained in the quantum upper cluster algebra $\s{A}^{\up}_t(\sd_\Delta)$. Quantum bracelets $\BracE{C}$ can be defined similarly \cite{Thurst}.
	
	\subsubsection*{Generalization}
	Before stating our results, we need to generalize the previous notions and constructions to $\Sigma$ possibly with punctures.
	
	Following \cite{fomin2018cluster}, we consider (generalized) tagged arcs by tagging the endings of arcs as plain  or notched. Then we extend $\Sk(\Sigma)$ to the extended skein algebra $\Sk^{\Box}(\Sigma)$ by adjoining tagged arcs (\S \ref{S:tagged_sk}). Using techniques involving ramified covering spaces, we prove that tagged arcs are always contained in $\s{A}^{\up}(\sd_\Delta)$ even if they do not correspond to cluster variables (Propositions \ref{prop:compound} and \ref{SkAPropPun}).  Our definition of the generalized tagged arcs in \S \ref{S:tagged_sk} is in terms of lambda lengths between horocycles as in \cite{FT}, but following a suggestion from G. Muller, we later show that $\Sk^{\Box}(\Sigma)$ can equivalently be defined using a nice collection of generators and relations; cf. \S \ref{subsec:local_digon} and \S \ref{subsub:tagged-skein}.
 
   In \S \ref{Sec:tag-brac}, we generalize weighted simple multicurves to weighted tagged simple multicurves by including compatible tagged arcs as components.  We then use this to generalize the definition of bracelets to a definition of tagged bracelets.
	
	We say two seeds are similar of they have the same principal exchange matrices (i.e., they only differ in their frozen parts). For $\sd$ similar to $\sd_{\Delta}$, \cite{fomin2018cluster} shows how $\sd$ can be understood in terms of laminated Teichm\"uller space.  In the case of principal coefficients, \cite{musiker2013bases} gives a construction of bracelets bases extending the constructions described above.  On the other hand, the correction technique in cluster algebras \cite{Qin12,qin2017triangular} gives a simple way to relate bases for similar seeds (see \S \ref{sec:similarity}), and we thus characterize the bracelets for any $\sd$ similar to $\sd_{\Delta}$ (\S \ref{sec:class-brac-coeff}).
 
    For surfaces $\Sigma$ with punctures, it is unclear how to quantize $\Sk(\Sigma)$ directly (cf. Remark \ref{rem:no-q}),\footnote{A quantization of skein algebras is proposed in the ongoing work \cite{Shen2023skein}. Special cases were treated in \cite{roger2014skein}.} but quantization of the upper cluster algebra becomes possible if we replace $\sd_{\Delta}$ by a similar seed satisfying the Injectivity Assumption (e.g., we may take $\sd\coloneqq\sd_{\Delta}^{\prin}$).  The correction techniques mentioned above apply to theta functions in the quantum setting as well.

	So suppose that we have fixed a quantum seed $\sd$ similar to $\sd_{\Delta}$.  It remains to construct the quantum bracelet $\BracE{C}$ in the quantum upper cluster algebra $\s{A}^{\up}_t(\sd)$ for any weighted tagged simple multicurve $C$ (\S \ref{sec:q_tag_bracelet}). For any loop $L$, $\BracE{L}$ is constructed from a quantum trace \cite{Allegretti2015duality,BW} (cf. \S \ref{sec:qbracelet_loop}), or using a cut-and-paste technique (\S \ref{sec:cut}). For a tagged arc $\gamma$, $\BracE{\gamma}$ can usually be constructed as a quantum cluster variable. But for $\gamma$ a notched arc on a once-punctured closed surface, $\BracE{\gamma}$ is defined via the Donaldson-Thomas transformation (when the genus is at least $2$). A more direct and natural definition for this case would be desirable, cf. Remark \ref{rem:quantization_tagged_arc}.
	
	\subsubsection*{Main results}

	We will show that (quantum) tagged bracelets are (quantum) theta functions, with a caveat for notched arcs in the once-punctured torus. 
	\begin{thm}[{Theorem \ref{thm:punctured_bracelet_theta}}]
		Let $C=\bigsqcup_{1\leq i\leq s} w_i C_i$ denote any weighted tagged simple multicurve.
		\begin{enumerate}[label=(\roman*), ref=(\roman*)]
			\item If no component $C_i$ is a notched arc in a once-punctured torus, then $\BracE{C}$ coincides with the quantum theta function $\vartheta_{g(C)}$ in $\s{A}^{\up}_t(\sd)$, where $g(C)$ denotes the $g$-vector (leading degree) of $\BracE{C}$.
		\item In the coefficient-free classical setting, $\BracE{C}=4^k\vartheta_{g(C)}$ where $k$ is the sum of the weights $w_i$ for all notched arcs $C_i$ in once-punctured closed torus components.
		\end{enumerate}
	\end{thm}
	
    The extra factors in the once-punctured torus case can be interpreted as follows.  The theta functions we consider are those associated to a scattering diagram $\f{D}=\f{D}^{\sd_{\Delta}^{\prin}}$ as in \cite{gross2018canonical} (setting frozen variables equal to $1$ after computing the theta functions).  One can define a new scattering diagram $\f{D}'\coloneqq \f{D}\sqcup \{(H,f_H)\}$ consisting of a single additional wall with support a single hyperplane $H$ and central scattering function $f_H$, cf. Theorem \ref{thm:tag_sk_up_cl_alg}.  The bracelets for the once-punctured torus coincide with the theta functions associated to this $\f{D}'$.  We speculate that $\f{D}'$ may be the stability scattering diagram \cite{bridgeland2017scattering} for the associated quiver (with a non-degenerate potential); cf.  \cite{qin2019bases}\cite{mou2019scattering}.  Upcoming work will investigate this by examining the behavior of stability scattering diagrams under folding \cite{chen2023comparison}.
 
     We obtain the following important consequences concerning skein algebras and the atomicity of their bases (see \S \ref{sec:positivity_atlas}).

	\begin{thm}[{Theorem \ref{thm:sk_unpunct_basis}}]
		When $\Sigma$ is unpunctured, the (untagged) quantum bracelets form the atomic basis for the quantum skein algebra $\Sk_t(\Sigma)$ with respect to the cluster atlas.
	\end{thm}

        Theorem \ref{thm:punctured_bracelet_theta} shows that our quantum tagged bracelets are theta functions for punctured surfaces as well (except with the once-punctured torus).  The following summarizes some of our main results for the classical setting.
 
	\begin{thm}[{See Theorem \ref{thm:tag_sk_up_cl_alg} for details}]
		The following statements are true.
		\begin{enumerate}[label=(\roman*)]
			\item The bracelets form the atomic basis for the skein algebra $\Sk(\Sigma)$ with respect to the ideal triangulation atlas.
			
			\item The tagged bracelets form the atomic basis for the extended skein algebra $\Sk^{\Box}(\Sigma)$ with respect to the scattering atlas\footnote{The scattering atlas is the same as the extended cluster atlas in \cite[Examples after Conj. 5.7]{FG1}.  We will also refer to this as the tagged triangulation atlas. 
 If no component of $\Sigma$ is a once-punctured closed surface, then this also agrees with the cluster atlas.  See \S \ref{PosSub} for more on the cluster atlas and \S \ref{sec:positivity_atlas} for other atlases.}.
			
			\item When no component of $\Sigma$ is a once-punctured torus, $\Sk^{\Box}(\Sigma)$ coincides with the middle cluster algebra $\s{A}^{\midd}(\sd_\Delta)$. Moreover, if every component of $\Sigma$ has non-empty boundary, then the full Fock-Goncharov conjecture holds (cf. Theorem \ref{fFG-surfaces}).	
		\end{enumerate}
	\end{thm}

	 Since the theta basis has positive structure constants (i.e. strong positivity), the above bases are strongly positive as well.  We thus prove D. Thurston's Conjecture \cite[Conj. 4.20]{Thurst} on the strong positivity of the quantum bracelets basis for unpunctured surfaces; cf. Theorem \ref{Thurst-conj}---our proof applies to all (not necessarily triangulable) unpunctured surfaces with non-empty boundary.

    In Corollaries \ref{cor:tag_sk_up_cl_alg} and \ref{cor:q_cluster_skein_equal}, we show that the non-localized skein algebras $\?{\Sk}_t(\Sigma)$ and $\?{\Sk}^{\Box}(\Sigma)$ have bases consisting of the bracelets associated to weighted tagged simple multicurves whose weights are all non-negative.

    \subsubsection{Canonical coordinates for moduli of framed $\PGL_2$-local systems} Given a seed $\sd_{\Delta}$ associated to an ideal triangulation of a marked surface as above, Fock and Goncharov \cite{FockGoncharov06a} interpret the corresponding cluster Poisson variety as a moduli space $\s{X}_{\Sigma}$ of framed $\PGL_2$-local systems.  They then construct canonical coordinates in terms of features like traces and eigenvalues of the monodromies of these local systems.  Allegretti and Kim \cite{Allegretti2015duality} develop a quantum analog of these canonical coordinates using Bonahon-Wong's \cite{BW} quantum trace map.  We review these constructions in \S \ref{sec:cluster_poisson}.
 
 Universal positivity for these quantum canonical bases has been proved in \cite{cho2020laurent}, and strong positivity is known in the classical setting by \cite[Thm. 12.2.5]{FockGoncharov06a}.  Up to this point, atomicity in the classical and quantum setting (\cite[Conj. 12.2]{FockGoncharov06a} and \cite[Conj. 1.5]{musiker2013bases}) and strong positivity in the quantum setting \cite[Conj. 12.4.5]{FockGoncharov06a} have been open.  These conjectures follow as immediate corollaries of the following (and of the analogous known results for theta bases):
 \begin{thm}[Theorem \ref{thm:ThetaX}]
The (quantum) canonical coordinates on $\s{X}_{\Sigma}$, as constructed in \cite{FockGoncharov06a,Allegretti2015duality}, are precisely the (quantum) theta bases.
 \end{thm}

     \subsection{Structure of the paper}

     \S \ref{sec:review}-\S \ref{sec:bracelet_cluster_alg} serve to introduce the relevant background, definitions, and lemmas.  Additional background and lemmas for the cluster Poisson setting are introduced in \S \ref{sec:cluster_poisson}.  The bulk of the proofs for our main results are in \S \ref{S:BracTheta}-\S \ref{sec:general_surface} and \S \ref{sec:X-q-can}-\S \ref{sec:ThetaX}. We outline this in more detail below.
     
     \subsubsection{Background and setup}

    In \S \ref{sec:review}, we review basics of cluster algebras and their quantizations. Positivity properties (for the cluster atlas) are reviewed in \S \ref{PosSub}.
	
    In \S \ref{SkeinSection}, we review skein algebras and cluster algebras on marked surfaces following \cite{FominShapiroThurston08,muller2016skein}.  We also extend to tagged skein algebras as in \cite{fomin2018cluster,MusikerSchifflerWilliams09}, understood here in terms of functions on Teichm\"uller space.  Of possibly independent interest, we prove that if one adjoins tagged arcs to $\Sk(\Sigma)$ to produce $\Sk^{\Box}(\Sigma)$, then $\Sk^{\Box}(\Sigma)$ will contain all generalized tagged arcs as well (Proposition \ref{prop:compound}). In fact, we prove in \S \ref{subsec:local_digon} that generalized tagged arcs satisfy a ``local digon relation,'' and in \S \ref{subsub:tagged-skein} (as an application of our main results) we find that the usual skein algebra relations plus the local digon relation can be used to give a relatively simple alternative construction of $\Sk^{\Box}(\Sigma)$.  We  review laminations, the associated tropical coordinates, and Dehn twists in \S \ref{sec:shear_coord}.

In \S \ref{section:bracelet_skein}, we review known construction of the bracelets and quantum bracelets as in \cite{musiker2013bases,Thurst}. Then we extend the classical bracelets in the punctured cases by allowing contributions from tagged arcs (as suggested in \cite{musiker2013bases}).

In \S \ref{Section_Scat} we switch to discussing scattering diagrams and theta functions as in \cite{gross2018canonical,davison2019strong}.  Proposition \ref{indep} is a new general result saying that theta functions corresponding to points in the \textit{closure} of the cluster complex are linearly independent (and a bit more than this), thus generalizing \cite[Thm. 7.20]{gross2018canonical}.  We consider some special properties for cluster algebras from surfaces---\cite[Thm. 1.2]{yurikusa2020density} says that the cluster complex is dense everywhere or in a half-space for cluster algebras from surfaces (cf. Proposition \ref{gdense}), and we apply this to prove Theorem \ref{fFG-surfaces} regarding the full Fock-Goncharov conjecture for cluster algebras from surfaces.  In \S \ref{sec:positivity_atlas}, we discuss properties of theta functions, both in general and in the case of cluster algebras from surfaces.  Positivity properties for bracelets are discussed in \S \ref{sec:positivity_skein_bracelets}.

In \S \ref{sec:seed_change}, we examine the behavior of theta functions under various manipulations of seeds.  \S \ref{sec:similarity} relates theta functions for similar seeds, i.e., seeds with the same principal (or non-frozen) part.  This is applied in \S \ref{sec:gluing_frozen_vertices} to show how theta functions behave when frozen vertices are glued together.  In \S \ref{sec:unfreezing} we look at conditions implying that a theta function remains a theta function when an index is unfrozen (Lemmas \ref{lem:positive_unfreeze_theta} and \ref{lem:same_support_theta_func}).

In \S \ref{sec:bracelet_cluster_alg}, we define bracelets for cluster algebras from surfaces with arbitrary coefficients (\S \ref{sec:class-brac-coeff}) and their quantizations (\S \ref{sec:q_tag_bracelet}).  For loops, these definitions rely on constructions from \S \ref{sec:cluster_poisson} (the cluster Poisson setting), and for tagged arcs we define the bracelets to be the corresponding cluster variables (except with notcched arcs in once-punctured closed surfaces, which are more complicated).  In \S \ref{sec:cut} we prove a useful lemma about cutting surfaces into simpler pieces (Lemma \ref{lem:triangulation}) and we use this to give another (equivalent) characterization of the quantum bracelet elements associated to loops.

 \subsubsection{Outline of the main proofs}
    The bulk of our proofs lies in \S \ref{S:BracTheta}-\S \ref{sec:cluster_poisson}, plus some details in the appendices.  Here we give a rough outline of these arguments. We begin with \S \ref{S:BracTheta}, where we assume that $\Sigma$ is unpunctured.
    
    \textbf{The Gluing Lemma (\S \ref{Sec:glue}):} Let $\Sigma$ be a marked surface, and let $\Sigma'$ be another marked surface obtained from $\Sigma$ by gluing pairs of boundary arcs together.  Assume $\Sigma$ and $\Sigma'$ are both unpunctured.  Our Gluing Lemma (Lemma \ref{glue}) says that if a bracelet $\BracE{C}\in \Sk_t(\Sigma)$ is a theta function, then the corresponding element $\BracE{C'}\in \Sk_t(\Sigma')$ is also a theta function.

    To prove this, we note that gluing two boundary arcs together corresponds to identifying two frozen indices, followed by unfreezing the glued index (this process is called ``amalgamation'' in \cite[\S 19.1]{goncharov2019quantum}).  As we mentioned above, the operation of gluing frozen indices is shown to act on theta functions in a simple and natural way in \S \ref{sec:gluing_frozen_vertices}.  Then Lemma \ref{lem:positive_unfreeze_theta} shows that the operation of unfreezing also respects theta functions under a certain positivity assumption, and this positivity assumption holds for bracelets as an immediate consequence of known results for classical bracelets in unpunctured surfaces (their universal positivity).

    \textbf{Annular and Non-annular loops (\S \ref{sec:annular_loops} - \S \ref{sec:non-annular-loop}):} In light of the Gluing Lemma, to show that a weighted loop $\BracE{wL}\in \Sk_t(\Sigma)$ is a theta function, we can cut $\Sigma$ into simpler surfaces first.  In fact, this cutting allows us to reduce to the following cases:
    \begin{enumerate}[label=\roman*]
    \item (Annular loops) Lemma \ref{ArcLoop1Lem} shows that we can often cut as in Figure \ref{AnnCut} to reduce to the case where $\Sigma$ is an annulus with one marking on each boundary component and no other markings, and $L$ is the unique simple loop in $\Sigma$.  In these cases, $\sd_{\Delta}$ is the seed associated to the Kronecker quiver (with frozen vertices).  The claim that $\BracE{wL}$ is a theta function in these cases is checked in Lemma \ref{AnnLoop} (also cf. Example \ref{Ann-broken}) using explicit computations.
    \item (Non-annular loops) By Lemma \ref{non-annular-L} and cutting as in Figure \ref{DLoop}, the other cases to consider are where $\Sigma$ has a single boundary component with a single marked point, and $L$ is homotopic to this boundary component.  We show that $\BracE{L}$ is a theta function in \S \ref{subsub:nonannular1} by using the skein relations to multiply by an arc and then applying properties of bracelets (e.g., positivity; cf. Lemma \ref{lem:q-univ-pos}) and of theta functions (e.g., atomicity and bar-invariance).
    
    To deal with $\BracE{wL}$, we first prove some lemmas on the equivariance of the mapping class group action on $\Sk_t(\Sigma)$ and on the theta functions; cf. \S \ref{subsub:mcg}.  Then in \S \ref{subsub:Dehn}, we repeatedly apply Dehn twists to $\Delta$ to obtain a triangulation $\Delta'$ whose corresponding chamber $C_{\Delta'}$ in the scattering diagram is arbitrarily close to $g(wL)$ (using a result of \cite{yurikusa2020density}).  By analyzing the Laurent expansions of $\vartheta_{g(wL)}$ in these chambers, we deduce that the theta functions satisfy the desired Chebyshev recursion.  The equality $\BracE{wL}=\vartheta_{g(wL)}$ follows.
    \end{enumerate}

    \textbf{Unions of compatible bracelets (\S \ref{Disj-union-section}):} In \S \ref{thm:bracelet-theta-no-punct}, we consider an arbitrary bracelet $\BracE{C}$ for $C=\bigsqcup_i w_iC_i$.  Using Dehn twists as before, we construct a chamber $C_{\Delta'}$ of the scattering diagram which is arbitrarily close to all $g(w_iL_i)$.  We then choose a base-point $\sQ\in C_{\Delta'}$ and apply arguments similar to those from the non-annular case to prove that $\vartheta_{g(C)}=\prod_i \vartheta_{g(w_iC_i)}$.  This completes the proof that the bracelets and theta bases agree for unpunctured surfaces.

    \textbf{Punctured surfaces (\S \ref{sec:theta_punctured}):} We extend to allow for punctured surfaces in \S \ref{sec:theta_punctured}.  First, we show in Lemma \ref{lem:loop_puncture_bracelet} that bracelets consisting only of weighted loops are theta functions.  This is essentially by using the Gluing Lemma to reduce to the unpunctured setting, but we first have to prove the classical universal positivity of loops in punctured surfaces so that we can apply Lemma \ref{lem:positive_unfreeze_theta} (unfreezing respects theta functions, assuming positivity).  In Lemma \ref{lem:commute_arc_loop}, we prove in the quantum setting that loops commute with tagged arcs which they don't intersect---here we use new results on quantum DT-transformations (\S \ref{sec:DT}) to deal with doubly-notched arcs.  Proposition \ref{prop:cluster-commute-theta} tells us that a product of $t$-commuting theta functions, one of which is a cluster monomial, will still be a theta function (up to a power of $t$).  It then follows that bracelets and theta functions coincide in most cases.

    \textbf{Once-punctured closed surfaces:} In these cases, products of compatible notched arcs and weighted loops still give (tagged) bracelets, but since the notched arcs in these cases are not cluster variables, we need another argument to show that they are theta functions.  To do this, we first prove new results on quiver folding and scattering diagrams in \S \ref{GenFoldApp}.  We then apply this to finite covering spaces of $\Sigma$ in \S \ref{sec:closed_surface}.  The scattering diagram for a once-punctured closed surface $\Sigma$ is thus related to ``slice'' of the scattering diagram for a closed surface $\wt{\Sigma}$ with multiple punctures, with the former giving the usual theta functions and the latter giving the bracelets.  We find that these two scattering diagrams agree except in the case where $\Sigma$ is a once-punctured torus---in this case the scattering diagram from the covering space gives the one extra wall $(H,f_H)$.

    We put our findings together in \S \ref{sec:skein_atomic_bases} to prove a number of theorems for bases of skein algebras.

    \textbf{Cluster Poisson Varieties (\S \ref{sec:cluster_poisson})}    In \S \ref{sec:mod_PGL2}, we review Fock-Goncharov's construction of canonical coordinates for the moduli space $\s{X}$ of framed $\PGL_2$-local systems in \S \ref{sec:mod_PGL2}---these spaces $\s{X}$ are the cluster Poisson varieties associated to the cluster algebras from surfaces.  Then in \S \ref{sec:dec_SL2_moduli} we review the construction of moduli of twisted decorated $\SL_2$-local systems as in \cite{FockGoncharov06a}---this gives another interpretation of the surface skein algebras we consider.  We review the quantization $\s{X}_q$ of $\s{X}$ in \S \ref{sec:Chek-Fock}, and the quantum canoncial coordinates of \cite{Allegretti2015duality} in \S \ref{sec:X-q-can}.  Finally, in \S \ref{sec:ThetaX}, we prove that the (quantum) canonical coordinates coincide with the (quantum) theta functions.  For unpunctured surfaces, this is essentially because the cluster $\s{A}$ and $\s{X}$-varieties nearly coincide when the injectivity assumption holds, so we reduce to the usual surface skein algebra setting.  Punctured cases then follow by a cutting-and-gluing argument.  

    \subsubsection{The appendices}

    In \S \ref{GenFoldApp}, we review the notions of covering and folding of seeds (including for skew-symmetrizable seeds), and we look at how scattering diagrams behave under covering of seeds (Theorem \ref{thm:folding-D}) and folding of seeds (Corollary \ref{cor:gtr} and Lemma \ref{lem:restriction-strong}).
    
    We apply this to covering spaces of surfaces in \S \ref{sec:closed_surface}.  For once-punctured closed surfaces, this allows us to relate notched arcs to theta functions

    Finally, in \S \ref{sec:DT} we prove some features and useful lemmas regarding the action of the quantum DT-transformation on theta functions and bracelets.

    \subsection{Relation to other works}\label{sub:other}

	The original construction of a quantum bracelets basis (for the case of a closed torus), including the application of the Chebyshev polynomials, was \cite{frohman2000skein}.

	The bangles basis is conjectured to coincide with the generic basis. This has been proved in \cite{geiss2020generic} for unpunctured surfaces. The bands basis is conjectured to provide an analog of the dual canonical basis \cite{Thurst}, i.e. it should coincide with the common triangular basis.

	\cite{bousseau2020strong} proves that the (untagged) quantum bracelet elements are theta functions for the once-punctured torus and the four-punctured sphere, both for the cluster algebra and the associated cluster Poisson algebra, thus imply Thurston's strong positivity conjecture in these cases.  Bousseau's arguments utilize the characterization of the quantum scattering diagram in terms of certain higher-genus log Gromov-Witten invariants \cite{Bou3}.  Bousseau also suggests a mathematical physics interpretation for the agreement between bracelets and theta bases.

	 \cite{QW-ewp,QW-Khovanov} studies positivity of quantum bracelets and other bases using categorification. Building on this, \cite{queffelec2022gl2} proves strong positivity for the quantum bands basis for surfaces without marked points.

	\cite{FeT,CT} extended bracelets bases (for unpunctured cases) to orbifolds, with the help of unfolding. It is possible that our main result (bracelets $=$ theta functions) extends to orbifolds as well.

    Cluster structures and cluster Poisson structures have been defined on moduli of twisted decorated $G$-local systems and framed $G^{\vee}$-local systems, respectively, for much more general Lie groups $G$ with Langlands dual $G^{\vee}$ \cite{FockGoncharov06a,LeIan,GonSh,goncharov2019quantum}---the case considered in this work is $G=\SL_2$, $G^{\vee}=\PGL_2$. It would be very interesting to prove a geometric/representation-theoretic characterization of the theta bases for other groups $G$.  This more general setting is significantly more complicated, but many of our tools do generalize naturally, including the equivariance of the mapping class group action and the relationship between unfolding and covering spaces.  Moreover, upcoming results in \cite{CMMM} will use properties of tropical theta functions to prove a general version of our Gluing Lemma.

	\subsection{Acknowledgements} The second author was inspired to study algebras on surfaces by attending the conference ``Interactions between Representation Theory and Homological Mirror Symmetry'' in Leicester, UK, in 2019.  He would like to thank the organizers, Pierre-Guy Plamondon and Sibylle Schroll, for inviting him. This work began in October 2019 when the second author visited the first at the University of Edinburgh, and they would like to thank the university for its hospitality.  We are grateful to Dylan Allegretti, Peigen Cao, Ben Davison, Emily Gunawan, Min Huang, Greg Muller, and Gregg Musiker for many enlightening conversations.

	\section{Review of cluster algebras}\label{sec:review}
	
	We will work over a fixed base ring $\kk$.  Often one takes $\kk$ to be $\ZZ$, $\QQ$, or $\bb{C}$, but unless otherwise stated, any commutative ring containing $\ZZ$ as a $\bb{Z}$-subalgebra will work equally well.

	\label{Cluster_Section}
	
	\subsection{Seeds}
	
	A (skew-symmetric) cluster algebra is determined by the data of a \textbf{seed} $\sd$,
	by which we mean the following collection of data: 
	\begin{align*}
		\sd=(N,I,E=\{e_{i}|i\in I\},F,\omega)
	\end{align*}
	where $N$ is a lattice of finite rank $r$, $I$ is an index-set
	with $|I|=r$ (whose elements are called vertices), $E$ is a basis for $N$ indexed by $I$, $F$ is a
	subset of $I$, and $\omega$ is a $\bb{Q}$-valued skew-symmetric\footnote{We shall not consider the more general ``skew-symmetrizable'' cases except in Appendix \S \ref{GenFoldApp}.}
	form on $N$ such that $\omega(e_{i},e_{j})\in\bb{Z}$ whenever $i$
	or $j$ is in $I_{\ufv}:=I\setminus F$. The elements $e_{i}\in E$ with $i\in F$ (resp. $i\in I_{\ufv}$)
	are called the \textbf{frozen} vectors (resp. \textbf{unfrozen} vectors). We denote $E_{\uf}:=\{e_{i}|i\in I\setminus F\}$,
	and we let $N_{\uf}\subset N$ denote the span of $E_{\uf}$. As is
	common practice in the cluster literature,\footnote{Beware, the pairing called $B$ in \cite{davison2019strong} is actually the pairing $\omega=B^T$ here.} we shall denote the transpose
	of $\omega$ by $B$. Denote $b_{ij}=B(e_i,e_j)$. The matrix $(b_{ij})_{i,j\in I_\ufv}$ is called the principal $B$-matrix. If $\sd$ is not clear from context, we may
	decorate the data with the subscript $\sd$, e.g., as in $N_{\sd}$
	or $e_{\sd,i}$.
	\begin{rem}
		\label{QS} Consider a seed $\sd$ as above for which $\omega$ is
		$\bb{Z}$-valued on all of $N$. The data of such a seed is equivalent
		to the data of a finite quiver $Q$ without loops ($1$-cycles) or
		oriented $2$-cycles, and with some vertices labelled as frozen. Given
		$\sd$, the vertices $Q_{0}$ of the quiver correspond to the elements
		of $I$, with $F$ corresponding to the frozen vertices. For $i\in I$,
		let $V_{i}$ denote the corresponding vertex. Then the number of arrows
		from $V_{i}$ to $V_{j}$ is $B(e_{i},e_{j})$, where the arrows go
		the reverse direction if $B(e_{i},e_{j})<0$. Similarly, given the
		quiver $Q$, one recovers a seed $\sd$ with lattice $N=\bb{Z}^{Q_{0}}$
		by inverting this correspondence.  
	\end{rem}
	
	We next introduce some examples which shall be built upon throughout
	the paper. 
	\begin{example}
		\label{A2ex} The $A_{2}$-quiver $1\leftarrow2$ corresponds to a
		seed with $N\cong\bb{Z}^{2}$, $I=\{1,2\}$, $E=\{e_{1},e_{2}\}$,
		$F=\emptyset$, and $\omega(e_{1},e_{2})=1$. 
	\end{example}

	\begin{example}
		\label{QuivEx} Consider the quiver $Q$ of Figure \ref{QAn}. For
		the corresponding seed $\sd$, we have $N=\bb{Z}^{4}$, $I=\{1,2,3,4\}$,
		$E=\{e_{i}\}_{i\in I}$ the standard basis, $F=\{3,4\}$, and $B$
		(the transpose of $\omega$) given in the basis $E$ by the matrix
		\begin{align*}
			B=\left(\begin{matrix}0 & -2 & 1 & 1\\
				2 & 0 & -1 & -1\\
				-1 & 1 & 0 & 0\\
				-1 & 1 & 0 & 0
			\end{matrix}\right).
		\end{align*}
	\end{example}
	
	\begin{figure}[htb]
		\global\long\def\svgwidth{150pt}
		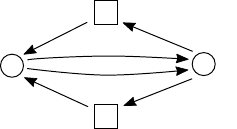 \caption{A quiver with frozen vertices represented as boxes.}\label{QAn}
	\end{figure}
	
	Given a seed $\sd$ as above, let $M:=N^{*}=\Hom(N,\bb{Z})$, and
	let $\{e_{i}^{*}|i\in I\}\subset M$ be the dual basis to $E$. Let
	$N_{\bb{Q}}:=N\otimes\bb{Q}$ and $M_{\bb{Q}}:=M\otimes\bb{Q}$. Denote
	\begin{align}
		\omega_{1}:N_{\bb{Q}}\rar M_{\bb{Q}},\quad\omega_{1}(n)=\omega(n,\cdot).\label{om1}
	\end{align}
	Note that $\omega_{1}|_{N_{\uf}}$ has image in $M$. We denote $M_{\uf}:=\omega_{1}(N_{\uf})$.  It will also be useful to denote $M_F:=\bb{Z}\langle e_i^*|i\in F\rangle$.  We may denote the dual pairing between $N$ and $M$ by $\langle\cdot,\cdot\rangle:N\times M\rar \bb{Z}$.
	
	\subsection{Compatible pairs}
	
	Let $\Lambda$ be a $\bb{Q}$-valued skew-symmetric form on $M$.
	Denote $\Lambda_{1}(m):=\Lambda(m,\cdot)\in N_{\bb{Q}}$ and $\Lambda_{2}(m):=\Lambda(\cdot,m)$.
	As in \cite[(47)]{davison2019strong} (originally due to \cite{BerensteinZelevinsky05} with a different sign; cf. Remark \ref{comp-mat-rmk}), we say that $(\sd,\Lambda)$ is a \textbf{compatible pair} if 
\begin{align}
	\Lambda_{2}(\omega_{1}(e_{i}))=de_{i}\label{Lambda-B}
\end{align}
for each $i\in I\setminus F$ and for some fixed\footnote{In \cite{davison2019strong}, the factor $d$ is always taken to be $1$. But the generalization
	to the setup here is straightforward---e.g.,
	dividing $\Lambda$ by $d$ while substituting $t^{1/D}\mapsto t^{d/D}$
	in \eqref{qtor} preserves the quantum torus algebra, and similar
	substitutions are easily made for all other constructions we consider.  We will mostly avoid the more
	general skew-symmetrizable setups like \cite{BerensteinZelevinsky05} where $d$ is replaced by different numbers
	$d'_{i}$ for different $i$ (cf. \S \ref{sec:skew-symmetrizable}).\label{foot:d}
} multiplier $d\in\bb{Q}_{>0}$  --- for quantum cluster algebras from
surfaces as we shall consider, we will always have $d=4$. 
Note that $\Lambda$ being compatible with $\sd$ implies 
\begin{align}
	\Lambda(\omega_{1}(e_{i}),\omega_{1}(e_{j}))=d\omega(e_i,e_j)=dB(e_{j},e_{i})\label{Lambda-v1v2}
\end{align}
whenever $i$ or $j$ is in $I\setminus F$.

\begin{example}
	\label{A2ex2} Consider again the seed $\sd$ associated to the $A_{2}$-quiver
	$1\leftarrow2$ as in Example \ref{A2ex}. Since there are no frozen
	indices, \eqref{Lambda-B} uniquely determines $\Lambda$ up to the
	factor $d$. Explicitly, we have 
	\begin{align*}
		\omega_{1}(e_{1})=e_{2}^{*}\quad\mbox{and}\quad\omega_{1}(e_{2})=-e_{1}^{*},
	\end{align*}
	so using \eqref{Lambda-v1v2}, we see that $\Lambda$ satisfies 
	\begin{align*}
		\Lambda(e_{1}^{*},e_{2}^{*})=\Lambda(-\omega_1(e_2),\omega_1(e_1)) = -d\omega(e_2,e_1)=d.
	\end{align*}
\end{example}

\begin{rem}
	\label{comp-mat-rmk} It is common to express \eqref{Lambda-B} in
	terms of matrix multiplication, viewing $\omega$ and $B$ as matrices
	in the basis $E$ and $\Lambda$ as a matrix in the dual basis $\{e_{i}^{T}\}_{i\in I}$.  Notice that the vector $\omega_{1}(e_{i})$, viewed as a column vector, agrees with $(e_i^T\omega)^T=\omega^T e_i=Be_{i}$
	in this matrix convention.  Thus, $\Lambda_2(\omega_1(e_i))=\Lambda B e_i$, and so the compatibility condition \eqref{Lambda-B}
	becomes 
	\begin{align}\label{comp-matrix}
		\Lambda B e_i = de_i
	\end{align}
	for all $i\in I\setminus F$.  In other words, the columns of $\Lambda B$ indexed by $I\setminus F$ agree with the corresponding columns of the diagonal matrix $d\Id_{|I|\times |I|}$. 
	
	On the other hand, taking the transpose of \eqref{comp-matrix} and using the skew-symmetry of $\Lambda$ yields $$e_i^T B^T \Lambda = -de_i^T.$$  I.e., the rows of $B^T\Lambda$ indexed by $I\setminus F$ agree with the corresponding rows of the diagonal matrix $-d\Id_{|I|\times |I|}$.  This is the compatibility condition of \cite[Def. 3.1]{BerensteinZelevinsky05} except for the minus sign.  Because of this sign, our quantum parameter $t$ will correspond to $q^{-1/2}$ in \cite{BerensteinZelevinsky05}.
\end{rem}

\begin{example}
	\label{AnComp} Consider the seed $\sd$ of Example \ref{QuivEx}.
	With respect to the dual basis to $E$, a compatible $\Lambda$ is
	given by
	\begin{align*}
		\Lambda=\left(\begin{matrix}0 & 2 & 0 & 0\\
			-2 & 0 & 0 & 0\\
			0 & 0 & 0 & 0\\
			0 & 0 & 0 & 0
		\end{matrix}\right).
	\end{align*}
	Indeed, we have 
	\begin{align*}
		\Lambda B=\left(\begin{matrix}4 & 0 & -2 & -2\\
			0 & 4 & -2 & -2\\
			0 & 0 & 0 & 0\\
			0 & 0 & 0 & 0
		\end{matrix}\right),
	\end{align*}
	so the condition as in Remark \ref{comp-mat-rmk} is satisfied for
	$d=4$.
\end{example}

We will often wish to impose the following assumption:

\begin{asm}[The Injectivity Assumption]\label{inj-assumption}
	One says the Injectivity Assumption holds for $\sd$ if $\omega_{1}|_{N_{\uf}}$
	is injective.
\end{asm}

Notice the assumption is equivalent to the matrix $(B(e_i,e_k))_{i\in I,k\in I_\ufv}$ being full rank, and this is also equivalent to the existence of a compatible
$\Lambda$; see \cite{GekhtmanShapiroVainshtein03}\cite{GekhtmanShapiroVainshtein05}.

\subsection{Principal coefficients}

\label{prinsub}

A common approach for dealing with cases where the Injectivity Assumption
fails is to work with \textbf{principal coefficients}. For this, given
a seed $\sd$ as above, one associates another seed $\sd^{\prin}$
defined as follows: 
\begin{itemize}
	\item $N_{\sd^{\prin}}\coloneqq N\oplus M$. 
	\item $I_{\sd^{\prin}}$ is the disjoint union of two copies of $I$. We
	shall call them $I_{1}$ and $I_{2}$ to distinguish between them. 
	\item $E_{\sd^{\prin}}\coloneqq\{(e_{i},0)\colon i\in I_{1}\}\cup\{(0,e_{i}^{*})\colon i\in I_{2}\}$. 
	\item $F_{\sd^{\prin}}\coloneqq F_{1}\cup I_{2}$, where $F_{1}$ is $F$
	viewed as a subset of $I_{1}$. 
	\item $\omega_{\sd^{\prin}}((n_{1},m_{1}),(n_{2},m_{2}))\coloneqq\omega(n_{1},n_{2})+m_{2}(n_{1})-m_{1}(n_{2})$. 
\end{itemize}
We may also denote $N_{\sd^{\prin}}$ as $N^{\prin}$, $\omega_{\sd^{\prin}}$ as $\omega^{\prin}$, and similarly for the other pieces of data.  We denote the analog of $\omega_{1}$ for $\omega_{\sd^{\prin}}$
by $\omega_{1}^{\prin}\colon N^{\prin}\rar M^{\prin}\coloneqq(N^{\prin})^{*}=M\oplus N$,
identifying $N=N^{**}$. I.e., 
\begin{align}
	\omega_{1}^{\prin}\colon(n,m)\mapsto\omega_{\sd^{\prin}}((n,m),\cdot)=(\omega_{1}(n)-m,n).\label{pi1prin}
\end{align}
The form $\omega_{\sd^{\prin}}$ is unimodular (i.e., has determinant
$1$) on $N_{\sd^{\prin}}$. In particular, the Injectivity Assumption
always holds for $\sd^{\prin}$.

Given a seed $\sd$ and the corresponding principal coefficients seed
$\sd^{\prin}$, define 
\begin{align}
	\rho:M^{\prin}=M\oplus N\rar M,\qquad(m,n)\mapsto m.\label{rhodef}
\end{align}
Note that if $(\sd,\Lambda)$ is a compatible pair, then one can pull
$\Lambda$ back by $\rho$ to give a compatible pair $(\sd^{\prin},\rho^{*}\Lambda)$
with the same value of $d$. Here, 
\begin{align*}
\rho^{*}\Lambda((m_{1},n_{1}),(m_{2},n_{2}))=\Lambda(m_{1},m_{2}).
\end{align*}

In general, even if the Injectivity Assumption fails for $\sd$ and
so there is no compatible $\Lambda$, one can still define a compatible
pair $(\sd^{\prin},\Lambda^{\prin})$ with $d=1$ \ 
\begin{align}
	\Lambda^{\prin}((m_{1},n_{1}),(m_{2},n_{2})):=\omega(n_2,n_1)-m_2(n_1)+m_1(n_2).\label{LambdaPrin}
\end{align}
Indeed, 
\begin{align*}
	\Lambda^{\prin}(\omega_{1}^{\prin}(n_{1},m_{1}),\omega_{1}^{\prin}(n_{2},m_{2}))&=\Lambda^{\prin}((\omega_1(n_1)-m_1,n_1),(\omega_1(n_2)-m_2,n_2))
	\\
	&=\omega(n_2,n_1)-(\omega_1(n_2)-m_2)(n_1)+(\omega_1(n_1)-m_1)(n_2) \\
	&=\omega(n_2,n_1)-\omega_1(n_2)(n_1)+\omega_1(n_1)(n_2)+m_2(n_1)-m_1(n_2) \\
	&=\omega(n_1,n_2)+m_2(n_1)-m_1(n_2) \\
	&=\omega^{\prin}((n_{1},m_{1}),(n_{2},m_{2}))
\end{align*}
as in \eqref{Lambda-v1v2}.  Equivalently, $\Lambda^{\prin}=(B^{\prin})^{-1}$.

\begin{eg}\label{eg:A_2_prin_Lambda}
	Building on Example \ref{A2ex}, we construct $I_{\sd^{\prin}}=I_{1}\sqcup I_{2}=\{1,2\}\sqcup\{1',2'\}$
	from two copies of $I$. We have $N^{\prin}=N\oplus M=\Z e_{1}\oplus \Z e_{2} \oplus\Z e_{1}^{*}\oplus \Z e_{2}^{*} $,
	$M^{\prin}=M\oplus N=\Z e_{1}^{*}\oplus \Z e_{2}^{*}\oplus \Z e_{1}\oplus \Z e_{2}$. In the basis $\{e_{1},e_2,e_1^*,e_{2}^{*}\}$
	of $N^{\prin}$, the form $\omega_{\sd^{\prin}}$ takes the following
	matrix form:
	\begin{eqnarray*}
		\omega_{\sd^{\prin}} & = & \left(\begin{array}{cccc}
			0 & 1&1&0\\
			-1 & 0&0&1\\
			-1 & 0&0&0\\
			0 & -1&0&0
		\end{array}\right).
	\end{eqnarray*}
	
	In the basis $\{e_1^*,e_{2}^{*},e_{1},e_2\}$ of $M^{\prin}$, we have $\omega_{1}^{\prin}(e_{1})=\left(\begin{array}{c}
		0\\
		1\\
		1\\
		0
	\end{array}\right)$, $\omega_{1}^{\prin}(e_{2})=\left(\begin{array}{c}
		-1\\
		0\\
		0\\
		1
	\end{array}\right)$. The matrix for $\Lambda^{\prin}$ is
	\begin{eqnarray*}
		\Lambda^{\prin}  =  (\omega^{\prin})^{-T} =  \left(\begin{array}{cccc}
			0 & 0&1&0\\
			0 & 0&0&1\\
			-1 & 0&0&-1\\
			0 & -1&1&0
		\end{array}\right).
	\end{eqnarray*}
	\end{eg}

\subsection{Seed mutations}\label{sec:seed-mut}

Given a seed $\sd=(N,I,E=\{e_{i}\}_{i\in I},F,\omega)$, the (positive)
\textbf{mutation} of $\sd$ with respect to $j\in I\setminus F$ is
the seed $\mu_{j}(\sd)\coloneqq(N,I,E'=\{e_{i}'\}_{i\in I},F,\omega)$,
where the vectors $e_{i}'$ are defined by 
\begin{align}
	e'_{i}\coloneqq\mu_{j}(e_{i})\coloneqq\left\{ \begin{array}{lr}
		e_{i}+\max(0,\omega(e_{i},e_{j}))e_{j} & \mbox{if \ensuremath{i\neq j}~}\\
		-e_{j} & \mbox{if \ensuremath{i=j}.}
	\end{array}\right.\label{eprime}
\end{align}
We extend this to compatible pairs $(\sd,\Lambda)$ by defining $\mu_{j}(\sd,\Lambda)=(\mu_{j}(\sd),\Lambda)$,
i.e., $\Lambda$ is unchanged by mutation\footnote{The matrix representations
	of $B$ and $\Lambda$ with respect to $E$ and the dual basis do change as the bases change, but the
	coordinate-independent versions used here do not.}. Let $f_i=e_i^*$, $i\in I$, denote the dual basis for $E$ in $M$. As in \cite[(8)]{FG1}, mutation of this dual basis is given by
\begin{align}
	f'_{i}\coloneqq\mu_{j}(f_{i})\coloneqq\left\{ \begin{array}{lr}
		f_{i}& \mbox{if \ensuremath{i\neq j}~}\\
		-f_{j}+\sum_{s\in I}\max(0,-\omega(e_{j},e_{s}))f_{s}& \mbox{if \ensuremath{i=j}.}
	\end{array}\right.\label{fprime}
\end{align}

\begin{rem}
	
	For simplicity, we do not define a signed mutation (see \cite[(2.1)]{qin2019bases}).
	Positive mutations alone suffice for calculating upper cluster algebra
	elements, see \cite[Rmk. 2.5]{gross2013birational}.
	
\end{rem}

Following \cite[(1.23)]{gross2018canonical} and \cite[\S 4.5]{davison2019strong}, we also consider the piecewise-linear ``tropical'' transformations $T_j^{\sd}:M_{\bb{R}}\rar M_{\bb{R}}$ defined by
\begin{align}\label{eq:diagram_mutation}
T_j^\sd(m)=m+\max(0,\langle e_j,m\rangle) \omega_1(e_j).
\end{align}
By design, $T_j^{\sd}(\omega_1(e_{\sd,i})) = \omega_1(e_{\mu_j(\sd),i})$ for all $i\neq j$, while $T_j^{\sd}(\omega_1(e_{\sd,j})) = -\omega_1((e_{\mu_j(\sd),j}))$.  We also note that $T_j^{\sd}(f'_i)=f'_i$ for all $i$.

Consider a sequence $\jj=(j_{1},\ldots,j_{s})$ of elements $j_{i}$
in $I\setminus F$.  For each $k=1,\ldots,s$, let $\jj_k$ be the truncated tuple $(j_1,\ldots,j_{k-1})$.  We recursively define $\sd_{\jj}$ in terms of $\sd_{\jj_{s-1}}$ as follows.  Take $\sd_{()}\coloneqq \sd$, and denote
\begin{align}\label{Tjj}
T_{\jj}\coloneqq T_{j_s}^{\sd_{\jj_s}} \circ \cdots \circ T_{j_1}^{\sd_{\jj_1}}:M_{\bb{R}}\rar M_{\bb{R}}.
\end{align}
We now define $\sd_{\jj}$ to be the seed $(N,I,E_{\jj},F,\omega)$ where $E_{\jj}$ is the dual basis to 
\begin{align*}
\{f_{\sd_{\jj},i}\coloneqq T_{\jj}(f_i)|i\in I\}.
\end{align*}
In particular, $\sd_j\coloneqq\sd_{(j)}$ is equal to $\mu_j(\sd)$ for each $j\in I\setminus F$.

\subsection{Quantum torus algebras}

\label{qtoralg}

Let there be given a compatible pair $(\sd,\Lambda)$. Choose and fix a positive integer $D$ such that $\Lambda$ and $\omega$ are $\frac{1}{D}\Z$-valued. Let $t$ denote a formal parameter and define $\kk_{t}:=\kk[t^{\pm 1/D}]$. We consider the \textbf{quantum torus algebra} 
\begin{align}
	\kk_{t}^{\Lambda}[M]:=\kk_{t}[z^{m}|m\in M]/\langle z^{u}z^{v}=t^{\Lambda(u,v)}z^{u+v}|u,v\in M\rangle.\label{qtor}
\end{align}
We may leave the superscript $\Lambda$ out of the notation. Note
that setting $t=1$ recovers the classical torus algebra 
\begin{align*}
	\kk[M]:=\kk[z^{m}|m\in M]/\langle z^{u}z^{v}=z^{u+v}|u,v\in M\rangle.
\end{align*}
For any submonoid $M'\subset M$, we may define $\kk_{t}[M']\subset\kk_{t}[M]$
by restricting to those exponents which lie in $M'$.

Let $M^{\oplus}$ denote the $\bb{Z}_{\geq0}$-span of $\{\omega_{1}(e_{i})|i\in I\setminus F\}\subset M_{\uf}$,
and let $M^{+}:=M^{\oplus}\setminus\{0\}$. Suppose the injectivity
assumption holds, so $M^{\oplus}$ is contained in a strongly convex\footnote{For $L$ a lattice, a cone $\sigma\subset L_{\bb{R}}:=L\otimes \bb{R}$ is called \textbf{strongly convex} if it  is convex and contains no line through the origin.  By an abuse of terminology, we may also describe $\sigma\cap L$ as being strongly convex.}
cone. Then let 
\[
\kk_{t}\llb M\rrb\coloneqq\kk_{t}[M]\otimes_{\kk_{t}[M^{\oplus}]}\kk_{t}\llb M^{\oplus}\rrb,
\]
i.e., $\kk_{t}\llb M\rrb$ is spanned over $\kk_{t}$ by formal sums
of the form $\sum_{v\in M^{\oplus}}a_{v}z^{m+v}$ for various $m\in M$ and
coefficients $a_{v}\in\kt$.

Given a subset $P$ of $M$ (typically $P=M^{+}$) and any $k\in\bb{Z}_{\geq1}$,
we denote 
\begin{align}\label{eq:k_copy_subsets}
	kP=\{m_{1}+\ldots+m_{k}|m_{1},\ldots,m_{k}\in P\}.
\end{align}

We also consider (Langlands) dual versions of these concepts: 
as in \cite[\S 3.1]{FG1}, even without the Injectivity Assumption, we can define the dual quantum torus algebra 
\begin{align*}
	\kk_{t}^{\omega}[N]:=\kk_{t}[z^{n}|n\in N]/\langle z^{u}z^{v}=t^{\omega(u,v)}z^{u+v}|u,v\in N\rangle.
\end{align*}
As before, we may leave the superscript $\omega$ out of the notation.  Under the Injectivity Assumption, \eqref{Lambda-v1v2} implies there is a natural $\kk_{t}$-algebra embedding (denoted $\omega_1$ by abuse of notation) defined by
\begin{align}\label{omega1}
	\omega_1:\kk_{t}^{\omega}[N_{\uf}]\hookrightarrow \kk_{t}^{\Lambda}[M], \qquad t^{\frac{1}{D}} \mapsto t^{\frac{d}{D}}, \quad z^n\mapsto z^{\omega_1(n)}.
\end{align}

As in \cite{bridgeland2017scattering,davison2019strong}, let $N^{\oplus}$ denote the $\bb{Z}_{\geq0}$-span of $\{e_{i}|i\in I\setminus F\}$,
so $M^{\oplus}=\omega_{1}(N^{\oplus})$, and let $N^{+}:=N^{\oplus}\setminus\{0\}$.
Then define 
\begin{align*}
	\kk_{t}\llb N\rrb:=\kk_{t}[N]\otimes_{\kk_{t}[N^{\oplus}]}\kk_{t}\llb N^{\oplus}\rrb
\end{align*}
so $\kk\llb N\rrb$ is spanned over $\kk_{t}$ by formal sums of the
form $\sum_{v\in N^{\oplus}}a_{v}z^{n+v}$ for $n\in N$ and coefficients
$a_{v}\in\kk_{t}$. Given a subset $P$ of $N$ (typically $P=N^{+}$)
and any $k\in\bb{Z}_{\geq1}$, we denote 
\[
kP=\{n_{1}+\ldots+n_{k}|n_{1},\ldots,n_{k}\in P\}.
\]

The quantum torus algebras $\kk_t[N]$ and $\kk_t[M]$ admit involutive $\kk$-algebra automorphisms, called the \textbf{bar involutions}, induced by $t\mapsto t^{-1}$ and $z^v\mapsto z^v$ for all $v$ in $N$ or $M$, respectively.  These involutions naturally extend to the completions $\kk_t \llb N\rrb$ and $\kk_t\llb M\rrb$.  We say an element $f$ of any of these algebras is \textbf{bar-invariant} if $f$ is invariant under the appropriate bar involution.

\subsection{Cluster mutation}

\label{ClAlg} In the following, we assume the existence of a compatible
$\Lambda$ in order to define quantum cluster $\s{A}$-algebras.  The existence of compatible $\Lambda$ is not required for the quantum $\s{X}$-analog.  These same constructions also apply when $t=1$ without assuming the existence of compatible $\Lambda$, and this yields the classical cluster $\s{A}$- and $\s{X}$-algebras.

Given a seed $\sd$, denote the associated (quantum) torus algebra
\[
\s{A}_{t}^{\sd}:=\kk_{t}[M]
\]
and the (partially compactified) subalgebra
\begin{equation}\label{eq:A-bar}
	\?{\s{A}}_{t}^{\sd}:=\kk_{t}[\sigma_{\sd}]
\end{equation}
where $\sigma_{\sd}$ is the monoid 
\begin{align*}
	\sigma_{\sd}:=\bb{Z}\langle e_{i}^{*}|i\in I\setminus F\rangle+\bb{Z}_{\geq0}\langle e_{i}^{*}|i\in F\rangle\subset M.
\end{align*}
We consider the monomials $$A_{i}:=z^{e_{i}^{*}}\in\kk_{t}[M]$$ for
each $i\in I$. When $\sd$ is not clear from context, we will write
$A_{i}$ as $A_{\sd,i}$.

Dually, let 
\begin{align*}
	\s{X}_{t}^{\sd}:=\kk_{t}[N].
\end{align*}
Given a seed $\sd$, for each $i\in I$, let 
\begin{align}
	X_{i}:=z^{e_{i}}\in\s{X}_{t}^{\sd}.\label{Xi}
\end{align}

Let $\mr{\s{A}}_{t}^{\sd}$ and $\mr{\s{X}}_{t}^{\sd}$ denote the
skew-fields of fractions of $\s{A}_{t}^{\sd}$ and $\s{X}_{t}^{\sd}$,
respectively.\footnote{See \cite[\S 11]{BerensteinZelevinsky05} for a brief review of skew-fields of fractions and Ore localization in this context.} For each $j\in I\setminus F$, we define an isomorphism
$\mu_{j}^{\s{A}}:\mr{\s{A}}_{t}^{\sd}\risom\mr{\s{A}}_{t}^{\sd_{j}}$ such that the corresponding inverse map gives
\begin{align}
	(\mu_{j}^{\s{A}})^{-1}(A_{\sd_j,i}):=\begin{cases}
		A_{\sd,i} & \mbox{if \ensuremath{i\neq j},}\\
		z^{e_{\sd_j,j}^{*}}+z^{e_{\sd_j,j}^{*}+\omega_{1}(e_{\sd,j})} & \mbox{if \ensuremath{i=j}.}
	\end{cases}\label{mujAInverse}
\end{align}
Note that the dual bases in $\sd$ and $\sd_j=\mu_j(\sd)$ are related by \eqref{fprime}. Then the quantum mutation at $A_{\sd_{j},j}$ can be expressed as in \cite[(4.23)]{BerensteinZelevinsky05} via 
\begin{align*}
	(\mu_{j}^{\s{A}})^{-1}(A_{\sd_{j},j})=z^{-e_{j}^*+\sum_{i|\omega(e_{i},e_{j})>0}\omega(e_{i},e_{j})e_{i}^*}+z^{-e_{j}^*-\sum_{i|\omega(e_{i},e_{j})<0}\omega(e_{i},e_{j})e_{i}^*}.
\end{align*}
If we take $t=1$, then we recover the classical mutations given in \cite{FominZelevinsky02} as follows: 
\begin{align*}
	A_{\sd,j}\cdot \left((\mu_{j}^{\s{A}})^{-1}(A_{\sd_{j},j})\right)=\prod_{\omega(e_{i},e_{j})>0}A_{\sd,i}^{\omega(e_{i},e_{j})}+\prod_{\omega(e_{i},e_{j})<0}A_{\sd,i}^{-\omega(e_{i},e_{j})}.
\end{align*}

Given a tuple $\jj=(j_{1},\ldots,j_{s})\in(I\setminus F)^{s}$ and denoting $\sd_{\jj}$ as in \S \ref{sec:seed-mut}, we
define 
\[
\mu_{\jj}^{\s{A}}:\mr{\s{A}}_{t}^{\sd}\risom\mr{\s{A}}_{t}^{\sd_{\jj}}
\]
as the composition 
\begin{align*}
	\mu_{\jj}^{\s{A}}:=\mu_{j_{s}}^{\s{A}}\circ\cdots\circ\mu_{j_{1}}^{\s{A}}
\end{align*}
where $\mu_{j_{k}}^{\s{A}}$ here is viewed as mapping $\mr{\s{A}}_{t}^{\sd_{(j_{1},\ldots,j_{k-1})}}\risom\mr{\s{A}}_{t}^{\sd_{(j_{1},\ldots,j_{k})}}$.
Let 
\begin{align*}
	A_{\sd_{\jj_{1}},i}^{\sd_{\jj_{2}}}:=\mu_{\jj_{2}}^{\s{A}}\circ(\mu_{\jj_{1}}^{\s{A}})^{-1}(A_{\sd_{\jj_{1}},i}).
\end{align*}
In particular, 
\[
A_{\sd_{\jj},i}^{\sd}=(\mu_{\jj}^{\s{A}})^{-1}(A_{\sd_{\jj},i}).
\]
We may denote $A_{\sd_{\jj},i}^{\sd}$ more simply as $A_{\jj,i}$. 

\begin{defn}[Quantum cluster variables]\label{def:q-cluster-var}
	The elements $A_{\jj,i}=A_{\sd_{\jj},i}^{\sd}$ for any mutation sequence $\jj$
	and index $i\in I$ are called the (quantum) cluster variables. 
	
	Note that, for frozen vertices $i\in F$, we always have $A_{\jj,i}=A_i$.  These are called frozen variables or coefficients. The Laurent monomials in the frozen variables are called the frozen factors.
\end{defn}

\begin{rem}[{\cite[Rmk. 2.5]{gross2013birational}}]
	It is straightforward to check that $A_{(j, j),i}=A_i$, for all $i$, even though $e_{\sd_{(j, j)},j}^*\neq e_{\sd,j}^*$ in general.
\end{rem}

Similarly, for each $j\in I\setminus F$, one defines an isomorphism
$\mu_{j}^{\s{X}}:\mr{\s{X}}_{t}^{\sd}\risom\mr{\s{X}}_{t}^{\sd_{j}}$
by specifying that for $n\in N$ with $\omega(n,e_{j})\geq0$, we
have 
\begin{align}
	\mu_{j}^{\s{X}}(z^{n})=\sum_{k=0}^{\omega(n,e_{j})}\binom{\omega(n,e_{j})}{k}_{t}z^{n+ke_{j}}\label{Xmut}
\end{align}
as in \cite[(53)]{davison2019strong}---the coefficients are explained below; cf. \eqref{qbinom}.  Then, for $\omega(n,e_{j})\leq 0$, the inverse map reads (cf. \cite[(2.3)]{qin2020dual}):
\begin{align}
	(\mu_{j}^{\s{X}})^{-1}(z^{n})=\sum_{k=0}^{-\omega(n,e_{j})}\binom{-\omega(n,e_{j})}{k}_{t}z^{n+ke_{j}}.\label{XmutInverse}
\end{align}

Here, $\binom{a}{k}_{t}$ denotes the (bar-invariant) quantum binomial
coefficient defined as follows: for each integer $k$, let $[k]_{t}:=\frac{t^{k}-t^{-k}}{t-t^{-1}}$.
Note that $\lim_{t\rar1}[k]_{t}=k$. Then define $[0]_{t}!\coloneqq1$
and 
\begin{align*}
	[k]_{t}!\coloneqq[k]_{t}[k-1]_{t}\cdots[2]_{t}[1]_{t}
\end{align*}
for $k\in\bb{Z}_{\geq1}$. Now for $a,k\in\bb{Z}_{\geq0}$ with $a\geq k$,
one defines 
\begin{align}
	\binom{a}{k}_{t}\coloneqq\frac{[a]_{t}!}{[k]_{t}![a-k]_{t}!}\in\bb{Z}_{\geq0}[t^{\pm1}].\label{qbinom}
\end{align}
We note that, regardless of the sign of $\omega(n,e_{j})$, the classical
limit of this mutation induced by \eqref{Xmut} is given by 
\begin{align*}
	\lim_{t\rar1}\mu_{j}^{\s{X}}(z^{n})=z^{n}(1+z^{e_{j}})^{\omega(n,e_{j})},
\end{align*}
as in \cite[(2.5)]{gross2013birational}. We also note that $\mu_{j}^{\s{A}}$ could similarly be expressed
in terms of quantum binomial coefficients, and both $\mu_{j}^{\s{A}}$
and $\mu_{j}^{\s{X}}$ could be expressed in terms of conjugation
by a quantum exponential, cf. \eqref{AdPsiBinom} and \eqref{mu-dilog}.

As on the $\s{A}$-side, given a tuple $\jj=(j_{1},\ldots,j_{s})\in(I\setminus F)^{s}$,
we define 
\[
\mu_{\jj}^{\s{X}}:\mr{\s{X}}_{t}^{\sd}\risom\mr{\s{X}}_{t}^{\sd_{\jj}}
\]
as the composition 
\begin{align}
	\mu_{\jj}^{\s{X}}:=\mu_{j_{s}}^{\s{X}}\circ\cdots\circ\mu_{j_{1}}^{\s{X}}\label{mujjX}
\end{align}
where $\mu_{j_{k}}^{\s{X}}$ is viewed as mapping $\mr{\s{X}}_{t}^{\sd_{(j_{1},\ldots,j_{k-1})}}\risom\mr{\s{X}}_{t}^{\sd_{(j_{1},\ldots,j_{k})}}$. As before, define
\begin{align*}
	X_{\sd_{\jj_{1}},i}^{\sd_{\jj_{2}}}:=\mu_{\jj_{2}}^{\s{X}}\circ(\mu_{\jj_{1}}^{\s{X}})^{-1}(X_{\sd_{\jj_{1}},i}).
\end{align*}
and denote $X_{\jj,i}=X_{\sd_{\jj},i}^{\sd}$.

\begin{eg} Let us take $\sd$ as in Example \ref{A2ex} and take $t=1$ for simplicity. For $\sd_1=\mu_1 (\sd)$, we have $e_{\sd_1,1}=-e_1$, $e_{\sd_1,2}=e_2$. Since $\omega(e_{\sd_1,2},e_{1})=-1<0$, we have $(\mu_1^{\s{X}})^{-1}(z^{e_{\sd_1,2}})=z^{e_{\sd_1,2}}(1+z^{e_{1}})$ or, equivalently, $(\mu_1^{\s{X}})^{-1}(X_{\sd_1,2})=X_{\sd,2}(1+X_{\sd,1})$. 
	
	Similarly, for $\sd_2=\mu_2(\sd)$, we have $e_{\sd_2,1}=e_1+e_2$, $e_{\sd_2,2}=-e_2$. Since $\omega(e_{\sd_2,1},e_2)=1>0$, we have $\mu_2^{\s{X}}(z^{e_{\sd_2,1}})=z^{e_{\sd_2,1}}(1+z^{e_2})$ or, equivalently, $z^{e_1+e_2}(1+z^{e_2})^{-1}=(\mu_2^{\s{X}})^{-1}(z^{e_1+e_2})$. Recall that $X_{\sd_2,1}=z^{e_1+e_2}$. The last equation can be rewritten as $X_{\sd,1}X_{\sd,2}(1+X_{\sd,2})^{-1}=(\mu_2^{\s{X}})^{-1}(X_{\sd_2,1})$.
\end{eg}

\begin{eg}\label{Ex:mu1}
	Take $(\sd,\Lambda)$ as in Examples \ref{QuivEx} and \ref{AnComp}.  Then
	\begin{align*}
		(\mu_1^{\s{A}})^{-1}(A'_1) =z^{-e_1^*+e_3^*+e_4^*}+z^{-e_1^*+2e_2^*}.
	\end{align*}
	Denoting $A'_1:=(\mu_1^{\s{A}})^{-1}(A_{\sd_1,1})$, we can rewrite this relation as
	\begin{align}\label{Eq:mu1}
		A_1A'_1=A_3A_4+t^4 A_2^2.
	\end{align}
\end{eg}

\subsection{$g$-vectors}\label{sub:g}

\begin{thm}[\cite{FominZelevinsky07,DerksenWeymanZelevinsky09,Tran09,gross2018canonical}]\label{thm:cluster_expansion}
	The quantum cluster variable $A_{\jj,i}$
	takes the following form:
	\begin{eqnarray}
		A_{\jj,i} & = & \sum_{n\in N^{\oplus}}c_{n}z^{g_{\jj,i}+\omega_1(n)} \label{Aji}
	\end{eqnarray}
	for $c_{n}\in\kk_{t}$, $c_{0}=1$, and some $g_{\jj,i}\in M$ such that $g_{\jj,i}(e_h)\geq0$ for all $h\in F$.
\end{thm}

Note that the bases $E$ and $E^*=\{e_i^*:i\in I\}$ give natural identifications of $N=\bb{Z}^I$ and $M=\bb{Z}^I$. Denote the natural projection $\pr_{I_{\uf}}:\Z^I\rightarrow \Z^{I_{\uf}}$.

\begin{defn}\label{def:g-vec}
	The vector $g_{\jj,i}\in\Z^{I}$ is
	called the \textbf{$i$-th (extended) $g$-vector} of the seed $\sd_{\jj}$
	with respect to the initial seed $\sd$. Its projection $\pr_{I_{\uf}}g_{\jj,i}$
	is called the \textbf{principal $g$-vector}.
	
	The sum $\sum_{n\in N^{\oplus}}c_{n}z^{\omega_1(n)}$ with coefficients as in \eqref{Aji} is called the \textbf{$F$-polynomial} of the quantum cluster variable $A_{\jj,i}$. 
\end{defn}

We note that $g_{\jj,i}=e_{\sd_{\jj},i}^*$ under the identification $M=\bb{Z}^I$.

\begin{rem}
	There is a multivariate polynomial $F_{\jj,i}$ in the variables $X_k:=z^{e_k}\in \kk_t[N_{\uf}]$, $k\in I_{\uf}$, such that the factor $\sum_{n\in N^{\oplus}}c_{n}z^{\omega_1(n)}$ is obtained from it by setting $X_k=A^{\omega_1(e_k)}$ and replacing $t$ by $t^{d}$; i.e., by applying $\omega_1$ as in \eqref{omega1}. The polynomial $F_{\jj,i}$ is called the
	$i$-th $F$-polynomial of the seed $\sd_{\jj}$ with respect to the
	initial seed $\sd$.  
	It is well-known that $\pr_{I_{\uf}}g_{\jj,i}$
	and $F_{\jj,i}$ only depend on $\jj$, $i$, and the
	submatrix $(\omega_{ij})_{i,j\in I_{\uf}}$, see \cite{FominZelevinsky07}.
\end{rem}

\subsection{Cluster algebras}

The \textbf{ordinary quantum cluster algebra} $\s{A}_{t}^{\ord}$
is defined to be the sub $\kk_{t}$-algebra of $\mr{\s{A}}_{t}^{\sd}$
generated by the cluster variables together with $A_i^{-1}$
for $i\in F$. 

In some settings, it is desirable to not include the inverses of frozen
variables. We will write $\?{\s{A}}_{t}^{\ord}$ to denote the sub
$\kk_{t}$-algebra of $\s{A}_{t}^{\ord}$ generated only by the cluster
variables $A_{\jj,i}$ for all $\jj$ and all $i\in I$.

For each $\jj$, let 
\begin{align}
	\?{C_{\jj}}^{+}:=\bb{R}_{\geq0}\langle e_{\sd_{\jj},i}^{*}|i\in I\rangle\subset M_{\bb{R}}\label{Cjplusplus}
\end{align}
(i.e., $\?{C_{\jj}}^{+}$ is generated by the extended $g$-vectors of $\sd_{\jj}$), and let 
\begin{align}
	C_{\jj}^{+}:=\?{C_{\jj}}^{+}+\bb{R}_{\geq0}\langle-e_{\sd_{\jj},i}^{*}|i\in F\rangle\subset M_{\bb{R}}.\label{Cjplus}
\end{align}
Then $(\mu_{\jj}^{\s{A}})^{-1}(z^{m})$ is in $\s{A}_{t}^{\ord}$
(respectively, $\?{\s{A}}_{t}^{\ord}$) for each $m\in C_{\jj}^{+}\cap M$
(respectively, each $m\in \?{C_{\jj}}^{+}\cap M$). The elements of this
form are called (quantum) \textbf{cluster monomials}.

For each sequence $\jj$ in $I\setminus F$ as above, the subring 
\[
\s{A}_{\jj,t}:=(\mu_{\jj}^{\s{A}})^{-1}(\s{A}_{t}^{\sd_{\jj}})
\]
of $\mr{\s{A}}_{t}^{\sd}$ is called an $\s{A}$-\textbf{cluster},
or simply a cluster. The \textbf{upper quantum cluster algebra}
$\s{A}_{t}^{\up}$ is defined to be the intersection in $\mr{\s{A}}_{t}^{\sd}$
of all the clusters, i.e., 
\begin{align*}
	\s{A}_{t}^{\up}:=\bigcap_{\jj}\s{A}_{\jj,t}\subset\mr{\s{A}}_{t}^{\sd}.
\end{align*}
Additionally, we denote
\begin{align*}
	\?{\s{A}}_{\jj,t}:=(\mu_{\jj}^{\s{A}})^{-1}(\?{\s{A}}_{t}^{\sd_{\jj}})
\end{align*}
for $\?{\s{A}}_t^{{\sd}_{\jj}}$ as in \eqref{eq:A-bar} and then define 
\begin{align*}
	\?{\s{A}}_{t}^{\up}:=\bigcap_{\jj}\?{\s{A}}_{\jj,t}\subset\mr{\s{A}}_{t}^{\sd}.
\end{align*}
Note that $\?{\s{A}}_{t}^{\up}\subset\s{A}_{t}^{\up}$.

Similarly, for each $\jj$, the subring 
\[
\s{X}_{\jj,t}:=(\mu_{\jj}^{\s{X}})^{-1}(\s{X}_{t}^{\sd_{\jj}})
\]
of $\mr{\s{X}}_{t}^{\sd}$ is called an $\s{X}$-\textbf{cluster}.
Then we define the quantum cluster Poisson algebra
\begin{align*}
	\s{X}_{t}^{\up}:=\bigcap_{\jj}\s{X}_{\jj,t}\subset\mr{\s{X}}_{t}^{\sd}.
\end{align*}

As noted at the start of this subsection, the above definitions make
sense in the classical limit as well, and mutations may be defined
as in \eqref{mujAInverse} and \eqref{Xmut} with $t=1$. In this way one obtains the corresponding classical cluster algebras $\s{A}^{\ord}$, $\s{A}^{\up}$, $\?{\s{A}}^{\ord}$, and $\?{\s{A}}^{\up}$. Similarly, $\s{X}^{\up}:=\s{X}_{1}^{\up}$. We note that defining
these classical algebras does not require the injectivity
assumption to hold.

Note that Theorem \ref{thm:cluster_expansion} implies the ``Laurent phenomenon,'': $\s{A}_t^{\ord}\subset \s{A}_t^{\up}$ and $\?{\s{A}}_t^{\ord}\subset \?{\s{A}}_t^{\up}$, see \cite{FominZelevinsky02} and \cite{BerensteinZelevinsky05}.

If we want to make clear that the algebra $\s{A}_t^{\ord}$, $\s{A}_t^{\up}$, $\s{X}_t^{\up}$, etc., under consideration was constructed from a certain seed (or mutation equivalence class of seeds) $\sd$, we may write $\sd$ in parentheses as in $\s{A}_t^{\ord}(\sd)$, $\s{A}_t^{\up}(\sd)$, $\s{X}_t^{\up}(\sd)$, etc.

Given a seed $\sd$ and the corresponding principal coefficients seed
$\sd^{\prin}$, let $\s{A}^{\prin,\up}=\s{A}^{\up}(\sd^{\prin})$ denote the upper cluster
algebra with principal coefficients, i.e., the upper cluster algebra
associated to $\sd^{\prin}$. Recall $\rho:M^{\prin}\rar M$ as in \eqref{rhodef}.
Applying $\rho$ to the exponents of monomials induces an
algebra homomorphism 
\[
\rho:\s{A}^{\prin,\up}\rar\s{A}^{\up}.
\]
Recall that if $(\sd,\Lambda)$ is a compatible pair, then so is $(\sd^{\prin},\rho^{*}\Lambda)$.
When quantizing using these compatible pairs, $\rho$ also induces
an algebra homomorphism (cf. \cite[Lem. 4.1]{davison2019strong}) 
\[
\rho:\s{A}_{t}^{\prin,\up}\rar\s{A}_{t}^{\up}.
\]

\subsection{Positivity}

\label{PosSub}

An element of $\kk_{t}=\kk[t^{\pm1/D}]$ is called \textbf{non-negative} if
its coefficients lie in $\bb{Z}_{\geq0}$.

Let there be given an initial seed $\sd$ and compatible $\Lambda$.
Consider the upper cluster algebra $\s{A}_{t}^{\up}\subset\s{A}_{t}^{\sd}=\kk_t[M]$
as before. Given any element $f\in\s{A}_{t}^{\up}\subset\s{A}_{t}^{\sd}$
and any sequence of non-frozen indices $\jj$, applying $\mu_{\jj}^{\s{A}}$
to $f$ yields a Laurent polynomial 
\[
\mu_{\jj}^{\s{A}}(f)=\sum_{m\in M}a_{\jj,m}z^{m}\in\s{A}_{t}^{\sd_{\jj}}=\kk_{t}[M].
\]
With this notation, one says that $f$ is \textbf{universally positive
	with respect to the cluster atlas}, or \textbf{cluster positive}
for short, if $a_{\jj,m}$ is non-negative for each $\jj$ and $m$.
Similarly for the classical setting but with ``non-negative'' taking the usual meaning, i.e., $a_{\jj,m}\geq \bb{Z}_{\geq0}$.

Now let $f\in\s{A}_{t}^{\up}$ (or $\s{A}^{\up}$) be a nonzero cluster
positive element. We say that such $f$ is \textbf{atomic with respect to the cluster atlas}, or \textbf{cluster
	atomic} for short, if it cannot be written as a sum of two other nonzero
cluster positive elements.

Next suppose that $\s{B}=\{f_{m}\}_{m\in M}$ is a $\kk_{t}$-module
basis for a $\kk_{t}$-algebra $\s{A}_{t}$ (respectively, a $\kk$-module
basis for a $\kk$-algebra $\s{A}$). Given $p_{1},\ldots,p_{s},p\in M$,
the \textbf{structure constant} $\alpha(p_{1},\ldots,p_{s};p)\in\kk_{t}$
(respectively, $\kk$) is defined by 
\begin{align*}
	f_{p_{1}}\cdots f_{p_{s}}=\sum_{p\in M}\alpha(p_{1},\ldots,p_{s};p)f_{p}.
\end{align*}
One says that $\s{B}$ is \textbf{strongly positive} (in the quantum
or classical setting) if all of the structure constants are non-negative.

One says that the basis $\s{B}$ is universally positive
with respect to the cluster atlas if each element of $\s{B}$ is universally positive with respect to the cluster atlas. It is further called atomic with respect to the cluster atlas if it consists precisely of the cluster atomic elements of the algebra.

\begin{lem}
	\label{PosImplications} Let $\s{B}$ be a basis for $\s{A}_{t}^{\up}$
	(or $\s{A}^{\up}$) which includes all the cluster monomials. Then
	we have the following implications: 
	\begin{align*}
		& \mbox{\ensuremath{\s{B}} is atomic with respect to the cluster atlas}\\
		\Rightarrow & \mbox{\ensuremath{\s{B}} is strongly positive}\\
		\Rightarrow & \mbox{\ensuremath{\s{B}} is universally positive with respect to the cluster atlas.}
	\end{align*}
\end{lem}

\begin{proof}

	The first implication is straightforward because the product of any collection of universally positive elements remains universally positive.

The second implication is also well-known in cluster theory; see the proof of \cite[Prop. 2.2]{HernandezLeclerc09}.
\end{proof}
We note that one similarly defines cluster positivity, cluster atomicity,
and strong positivity for elements/bases of $\s{X}_{t}^{\up}$
and $\s{X}^{\up}$.

\section{Cluster algebras from surfaces}

\label{SkeinSection}

We now review cluster algebras arising from marked surfaces. See \cite{FominShapiroThurston08,MusikerSchifflerWilliams09,musiker2013bases,fomin2018cluster}
for more on these constructions in the classical setting, and see
\cite{Thurst,muller2016skein} for the quantum setting.

\subsection{Definition of the skein algebras}

\label{SectionSkeinDef}

Let $\SSS$ be a compact oriented surface, possibly with boundary---we will say $\SSS$ is ``closed'' to mean $\partial \SSS=\emptyset$.
Let $\MM$ be a finite collection of distinct marked points in $\SSS$. Such a pair $\Sigma=(\SSS,\MM)$ is called a \textbf{marked surface}. Markings in the interior of $\SSS$ are called \textbf{punctures}. Starting in \S \ref{sub:ClSk}, we will assume that $\Sigma$ is triangulable, as defined there.

A \textbf{multicurve} is an immersion $\phi:C\rar\SSS$ of a compact
unoriented $1$-manifold $C$ such that the boundary of $C$ maps
to $\MM$, but no interior points of $C$ map to $\partial\SSS$ or
$\MM$. One calls two multicurves \textbf{homotopic} if there is a
homotopy between them whose fibers are all multicurves. A \textbf{curve}
is a connected multicurve. A \textbf{loop} is a curve whose domain is homeomorphic to a circle,
and an \textbf{arc} is a curve whose domain is homeomorphic to a closed interval. A \textbf{boundary
	arc} is an arc in $\Sigma$ which is homotopic to the closure of a
component of $\partial\SSS\setminus(\MM\cap\partial\SSS)$.  An arc which is not a boundary arc is called an \textbf{interior arc}.

For $p\in\SSS$, a \textbf{strand} of a multicurve $\phi:C\rar\SSS$
near $p$ is a connected component of $\phi^{-1}(D_{\epsilon}(p))$,
where $D_{\epsilon}(p)$ is an arbitrarily small closed disk around $p$,
or a closed half disk if $p\in\partial\SSS$. For any given
arc, we choose two strands which contain the endpoints and
do not intersect in $\SSS\setminus\MM$, and we call these strands the \textbf{ends}
of the arc.

A multicurve is called \textbf{transverse} if all self-intersections
are transverse (i.e., all strands have different tangent directions
at intersection points) and all interior crossings are between only
two strands --- by an \textbf{interior crossing} we mean an intersection
between strands at a point $p$ in $\SSS\setminus\MM$.  A transverse
multicurve is called \textbf{simple} if it has no interior crossings
and no contractible components.

By a \textbf{contractible arc}, we mean an arc which has both endpoints at the same marked point $p\in \MM$ and is contractible in $\SSS\setminus (\MM\setminus p)$.  Similarly, a \textbf{contractible loop} is a loop which has no interior crossings and is contractible in $\SSS\setminus \MM$.

A \textbf{link} is a transverse multicurve together with a choice
of ordering of the strands at each interior crossing,\footnote{\cite{muller2016skein} also chooses equivalence relations and orderings of
	strands at marked points. This is useful for some computations and
	gives a geometric interpretation for multiplication by $t^{\pm1}$
	and for the $t^{k}$-factor in the definition of the superposition
	product. However, it is not necessary for defining the skein algebra,
	cf. \cite[Rmk. 3.1]{muller2016skein}. So for simplicity we view all intersections
	at endpoints as being at the same height. \label{boundary_ordering_footnote}} i.e., a choice of which strand is ``over'' and which is ``under.''
As in \cite{muller2016skein}, we require homotopies between links to remain
in the set of transverse multicurves.\footnote{This transversality condition ensures that intersection points are
	not created nor removed under homotopy, and so the over/under labellings
	of strands are also preserved. Alternatively, one could consider framed links in $\SSS\times [0,1]$ up to ambient isotopy.  These descriptions differ in that the
	former does not allow for certain framed Reidemeister moves, but these moves
	follow later as a consequence of the skein relations, cf. \cite[Rmk. 2.3 and \S 3.2]{muller2016skein}.\label{foot:framed-Reid}} Recall $\kk_{t}:=\kk[t^{\pm 1/D}]$ for some fixed positive integer $D$.  Let $\kk_{t}^{\Links}(\Sigma)$ denote the free $\kk_{t}$-module with basis given by the homotopy equivalence classes $[C]$ of links $C$ in $\Sigma$.

Assume for now that $\Sigma$ has no punctures.  Then the skein module $\?{\Sk}_{t}(\Sigma)$
is defined to be $\kk_{t}^{\Links}(\Sigma)/R$, where $R$ denotes
the module of relations in $\kk_{t}^{\Links}(\Sigma)$ generated by the following
``skein'' relations --- here and throughout, we use the notation 
	\begin{align*}
		q=t^{-2}
	\end{align*}
	so $t$ can be viewed as a fixed choice of square root of $q^{-1}$:

	\begin{itemize}
		\item Contractible arcs are equivalent to $0$; 
		\begin{figure}[htb]
			\begin{minipage}[b]{0.15\linewidth}
				\centering    \def\svgwidth{60pt}
				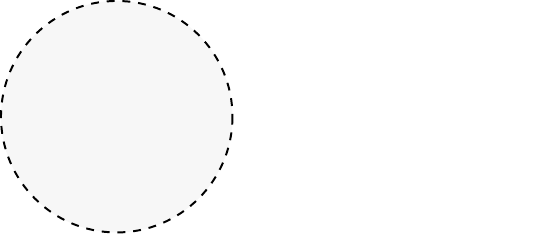
			\end{minipage}
			\hspace{0.4cm}
			\begin{minipage}[b]{0.15\linewidth}
				\centering\def\svgwidth{60pt}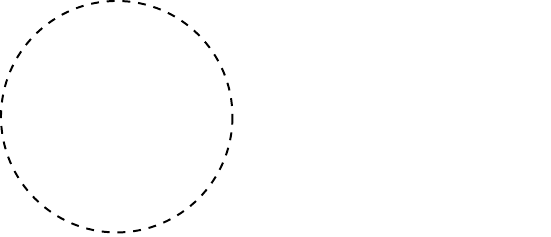 
			\end{minipage}
		\end{figure}
		\item A contractible loop is equivalent to $-(q^{2}+q^{-2})\cdot[\emptyset]$, 
		where $[\emptyset]$ denotes the empty link; 
		\begin{figure}[htb]
			\def\svgwidth{140pt}
			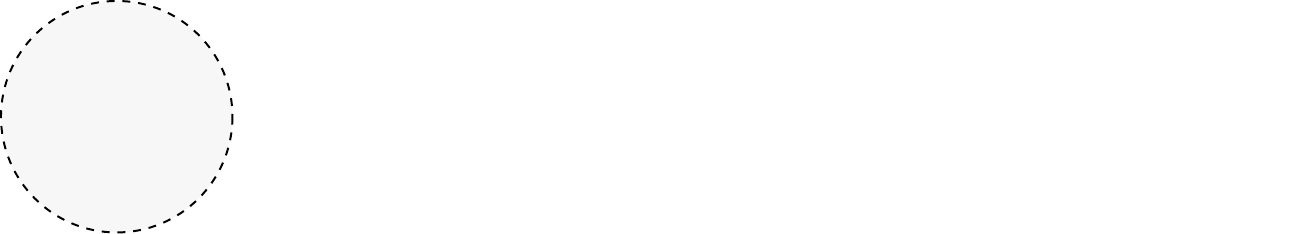\end{figure}
		\item The Kauffman skein relation, i.e., the following equivalence, understood
		to apply to links locally inside some disk while the link is preserved
		outside the disk: 
		\begin{figure}[htb]
			\global\long\def\svgwidth{150pt}
			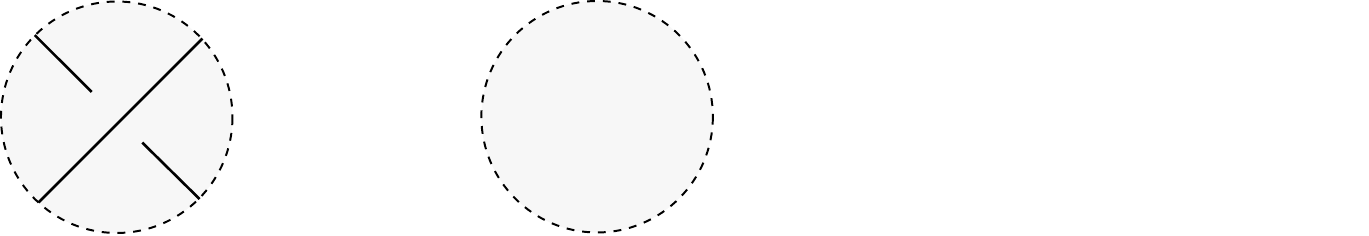 \caption{The Kauffman skein relation.\label{SkeinFig}}
		\end{figure}
	\end{itemize}
	One makes $\?{\Sk}_{t}(\Sigma)$ into an associative algebra with unit
	$[\emptyset]$ using the \textbf{superposition product}: if $X$
	and $Y$ are links such that $X\cup Y$ has transverse crossings,
	then $[X]\cdot[Y]$ is, for a certain $k\in\bb{Z}$ (defined below),
	equal to $t^{k}$ times the class of the link $[X\cup Y]$ in which
	strands of $X$ always cross over strands of $Y$ at each interior crossing --- if $X$ and
	$Y$ are not transverse to each other, one first replaces them with
	homotopic links which are transverse. The integer $k$ is defined as follows: for each arc $i$, let $\partial_{1}(i)$
	and $\partial_{2}(i)$ denote the two ends of $i$ (for arbitrary
	numbering). Then given two arcs $i,j$, define\footnote{Here, ``$\partial_{a}(i)$ is clockwise of $\partial_{b}(j)$'' means that we can obtain $\partial_{a}(i)$ from  $\partial_{b}(j)$ by a clockwise rotation of $\partial_{b}(j)$ inside the surface with their common endpoint on the boundary fixed. Our definition of $\Lambda$ differs from \cite[\S 6.2]{muller2016skein} by a sign, see Remark \ref{comp-mat-rmk}.}
	\begin{align}
		\Lambda(i,j):=\sum_{a,b\in\{1,2\}}\begin{cases}
			0 & \mbox{if \ensuremath{\partial_{a}(i)} and \ensuremath{\partial_{b}(j)} have different endpoints,}\\
			-1 & \mbox{if \ensuremath{\partial_{a}(i)} is clockwise of \ensuremath{\partial_{b}(j)},}\\
			1 & \mbox{if \ensuremath{\partial_{a}(i)} is counterclockwise of \ensuremath{\partial_{b}(j)}.}
		\end{cases}\label{Lambdaij}
	\end{align}
	Then the exponent $k$ is given by $k=\sum_{i,j}\Lambda(i,j)$,
	where the sum is over all pairs of arcs $i\in X$ and $j\in Y$. See Example \ref{AnnEx}.
	
	This algebra $\?{\Sk}_{t}(\Sigma)$ with the superposition product is
	called the \textbf{(Kauffman) skein algebra} of $\Sigma$. We may
	also call this the quantum skein algebra, and we may refer to $\?{\Sk}_{1}(\Sigma)$
	as the classical skein algebra. For brevity, we will often work only
	in the quantum setup in this section since the classical analog following
	easily by setting $t=1$. In the classical setting, the Kauffman
	skein relation implies that swapping the top and bottom strands of
	a link preserves the corresponding element of the skein algebra, so
	one may work with multicurves rather than links.
	
	In cases with punctures, we still define the \textbf{classical skein
		algebra} $\?{\Sk}(\Sigma)$ in the same way (with $t=1$) except with one more set of relations: 
	\begin{itemize}
		\item A simple loop around a single puncture is equivalent to $2$. 
		\begin{figure}[htb]
			\def\svgwidth{60pt}
			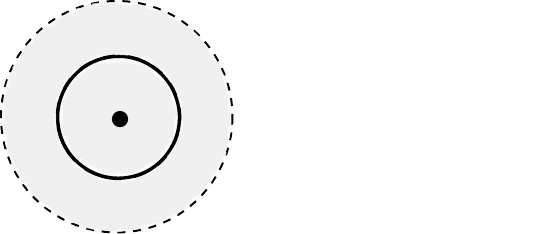\end{figure}
	\end{itemize}
	For triangulable $\Sigma$ (cf. \S \ref{sub:ClSk} for the definition of ``triangulable'') this relation is essentially forced on us by the the previous relations if we want $\?{\Sk}(\Sigma)$ to be an integral domain.  Indeed, if we multiply a loop $L$ around a puncture by an arc $\alpha$
	which intersects $L$ once (or, if $\Sigma$ is a once-punctured surface without boundary, multiply $L$ by a non-contractible arc $\alpha$ which intersects $L$ twice), then the result is $2\alpha$; hence $(L-2)\alpha=0$.
	\begin{rem}\label{rem:no-q}
		We note that if one attempts to define $\?{\Sk}_{t}(\Sigma)$ as above
		for a punctured surface, multiplying a loop $L$ around a puncture
		$p$ by an arc $\alpha$ intersecting $L$ once yields $(q+q^{-1})\alpha$,
		suggesting $L=q+q^{-1}$. But then doing the same for an arc
		$\alpha$ which hits $L$ twice yields different results (e.g.,
		one computes $L=q^{2}+q^{-2}$ if both ends of $\alpha$ are
		at $p$ and both strands of $\alpha$ pass under $L$), so the
		quantum relations are incompatible in the punctured cases without
		imposing conditions on $q$. Thus, without specializing $t$, $\?{\Sk}_{t}(\Sigma)$ is only defined
		for unpunctured surfaces. Introducing enough extra coefficients (e.g.,
		principal coefficients like in \cite{musiker2013matrix,fomin2018cluster}) would make
		quantization possible in general---we consider this in \S \ref{sec:similarity}. See \cite{roger2014skein} for another approach in which one adjoins variables associated to the punctures, and see \cite[\S 8]{Le-qtrace} for an approach in which, for each puncture, one cuts out a small disk containing the puncture in its boundary.
	\end{rem}
	
	The \textbf{localized skein algebra} $\Sk_{t}(\Sigma)$ is
	the Ore localization of $\?{\Sk}_{t}(\Sigma)$ at the set of boundary
	arcs of $\Sigma$, cf \cite[\S 5.1]{muller2016skein}. Similarly, $\Sk(\Sigma)$ is the localization
	of $\?{\Sk}(\Sigma)$ at the set of boundary arcs of $\Sigma$.
	
	\textit{For convenience, we shall often continue to include $t$ in our notation
		as in the quantized setting, even when quantization is not possible,
		with the understanding that one should set $t=1$ if necessary.}
	
	Given an element of $\?{\Sk}_t(\Sigma)$ or $\Sk_t(\Sigma)$ represented by a link $L$, let $\?{L}$ denote the element obtained by reversing the orderings of all crossings of $L$; i.e., at each crossing, we change which strand is on top.
	
	\begin{lem}[\cite{muller2016skein}]\label{lem:bar_involution_Sk}
		There is an involutive $\kk$-algebra automorphism of $\?{\Sk}_t(\Sigma)$ mapping $t\mapsto t^{-1}$ and $L\mapsto \?{L}$ for all links $L$. Furthermore, this extends to an involutive $\kk$-algebra automorphism of $\Sk_t(\Sigma)$.
	\end{lem}
	\begin{proof}
		The statement for $\?{\Sk}_t(\Sigma)$ is \cite[Prop. 3.11]{muller2016skein}.  The extension to $\Sk_t(\Sigma)$ is via  $\?{(xy^{-1})}=\?{y}^{-1}\?{x}$; cf. the discussion following \cite[proof of Prop. 5.2]{muller2016skein}.
	\end{proof}
	
	The involution of Lemma \ref{lem:bar_involution_Sk} is called the \textbf{bar involution} for the skein algebra.  Elements which are invariant under the bar involution are said to be \textbf{bar-invariant}.

	\subsection{Ideal triangulations and cluster structure}
	
	\label{sub:ClSk}

	An \textbf{(ideal) triangulation} $\Delta$ of $\Sigma$ is a maximal
	collection of pairwise non-homotopic simple arcs in $\Sigma$ which
	do not intersect each other in the interior of $\SSS\setminus\MM$.
	A triangle with only two distinct sides is called \textbf{self-folded},
	cf. Figure \ref{self-fold}. We shall refer to the boundary of a self-folded
	triangle as a \textbf{noose} --- i.e., a noose is an arc which bounds
	a once-punctured monogon.
	
	\begin{figure}[htb]
		{ \includegraphics{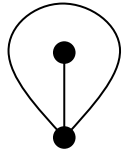} } \caption{A self-folded triangle. The outer arc is called a noose. Note that
			the marked point in the center must be a puncture.}
		\label{self-fold} 
	\end{figure}

	A marked surface $\Sigma=(\SSS,\MM)$ is \textbf{triangulable} if all the following hold (see \cite[Section 2]{FominShapiroThurston08}):
	\begin{itemize}
		\item each connected component of $\SSS$ contains at least one marked point; 
		\item each boundary component of $\SSS$ contains at least one marked point; 
		\item no connected component is a disk with $\leq 2$ marked points or a sphere with $\leq 3$ marked points.\footnote{We exclude the $3$-punctured sphere for technical reasons despite it admitting a nice triangulation. Nevertheless, the structure for the associated cluster algebra is easily understood: it is associated with the quiver consisting of three disconnected unfrozen vertices.}
	\end{itemize}
	 We will assume from now on that $\Sigma$ is triangulable.  
	
	Given a triangulation $\Delta$ of $\Sigma$, one associates a seed
	\begin{align}
		\sd_{\Delta}=(N,I,E=\{e_{i}|i\in I\},F,\omega)\label{sdDelta}
	\end{align}
	as follows. First, we take the index-set $I$ to be the set of arcs
	in $\Delta$, with $F$ the subset consisting of boundary arcs. Then
	$N:=\bb{Z}^{I}$, and $E=\{e_{i}\}_{i\in I}$ is the obvious basis.
	
	One defines $\omega=B^T$ as follows. First, for each arc $i\in\Delta$,
	if $i$ is the arc inside a self-folded triangle, let $i'$ be the
	corresponding arc which bounds the self-folded triangle. For arcs
	$i\in\Delta$ which are not inside self-folded triangles, let $i'=i$.
	Then for each pair of arcs $i,j\in\Delta$, we define $B(e_{i},e_{j})$
	in terms of the following sum over all triangles $T$ in $\Delta$
	which are not self-folded:\footnote{Here, we view a triangle $T$ as an oriented circle via homotopy, with the orientation induced by that of $\SSS$, and we consider the relative positions of the arcs $i'$ and $j'$ on the circle.} 
	\begin{align*} 
		B(e_{i},e_{j}):=\sum_{T\subset\Delta}\begin{cases}
			0 & \mbox{if \ensuremath{i'} and \ensuremath{j'} are not both in \ensuremath{T};}\\
			-1 & \mbox{if \ensuremath{j'} is the arc immediately counterclockwise of \ensuremath{i'} in \ensuremath{T};}\\
			1 & \mbox{if \ensuremath{j'} is the arc immediately clockwise of \ensuremath{i'} in \ensuremath{T}.}
		\end{cases}
	\end{align*}
	In particular, $B(e_{i},e_{j})\in\{-2,-1,0,1,2\}$ for each
	pair of arcs $i,j\in\Delta$. Following \cite{FominShapiroThurston08,muller2016skein}, $B$ is called
	the \textbf{signed adjacency matrix} or skew-adjacency matrix. See Example \ref{AnnEx}.
	
	In cases without punctures, the \textbf{orientation matrix} $\Lambda$
	is defined to be the matrix associated to the pairing defined on $M=N^*$
	by 
	\begin{align}
		\Lambda(e_{i}^{*},e_{j}^{*})=\Lambda(i,j)\label{Lambda-ei-ej}
	\end{align}
	for $\Lambda$ as in \eqref{Lambdaij}.
	\begin{lem}[\cite{muller2016skein}, Prop. 7.8]
		\label{lem:sdLambda-comp} For $\Sigma$ a marked surface without
		punctures, $(\sd_{\Delta},\Lambda)$ is a compatible pair satisfying \eqref{Lambda-B}
		for $d=4$. 
	\end{lem}
	\begin{proof} 
Any internal arc $x_j$ is contained in exactly two triangles, whose remaining arcs are denoted by $x_{k_i}$, $1\leq i\leq 4$, see Figure \ref{fig:adjacentArcs}. Note that the two endpoints of $x_j$ might be the same, and some remaining arcs might be repeated same.  Denote $x_{k_{i+4}}=x_{k_i}$ for simplicity. Let $v_i$ denote to the vertex of the quadrilateral containing $x_j$ which corresponds to the intersection $x_{k_i}\cap x_{k_{i+1}}$ in Figure \ref{fig:adjacentArcs}. Note that $x_{k_{i}}$ and $x_{k_{i+1}}$ can not be the same arc, because otherwise $v_i$ would become a puncture.
		
Let $e_j$ denote the basis vector in $N$ corresponding to $x_j$. For any arc $\gamma\in \Delta$, let $e_{\gamma}^*$ be the corresponding basis vector in $M$. By our definition of the matrix $B$, we get $\omega_1(e_j)=\sum_{1\leq i\leq 4}(-1)^{i-1}e_{x_{k_i}}^*$. We have to show that $\Lambda(e_\gamma^{*},\omega_1(e_j))=0$ for any arc $\gamma\in \Delta$ non-homotopic to $x_j$, and $\Lambda(e_{x_j}^*,\omega_1(e_j))=4$.
		
		Let us first consider $\Lambda(e_\gamma^{*},\omega_1(e_j))$ for $\gamma$ not homotopic to $x_j$. In view of \eqref{Lambdaij}, in order to calculate $\Lambda(\gamma,x_{k_i})$, we compute the contribution from the strands of the pair $(\gamma,x_{k_i})$ at each vertex of the quadrilateral, where we can assume the vertices are distinct in the computation because we work with strands.  If $\gamma=x_{k_i}$ we may perturb $x_{k_i}$ to a homotopic arc bending into the quadrilateral when computing $\Lambda$.  Now at any vertex $v_i$, $1\leq i\leq 4$, the contributions to $\Lambda(\gamma,x_{k_i})$ and $\Lambda(\gamma,x_{k_{i+1}})$ are the same, while the contributions to $\Lambda(\gamma,x_{k_{i+2}})$ and $\Lambda(\gamma,x_{k_{i+3}})$ are zero. We deduce that $\Lambda(e^*_{\gamma},\omega_1(e_j))=\sum_{1\leq i\leq 4}(-1)^{i-1}\Lambda(\gamma,x_{k_i})=0$.
		
		Similarly, at the vertex $v_i$, $i=1,3$, the contribution to $\Lambda(x_j,x_{k_i})$ is $1$ and the contribution to $\Lambda(x_j,x_{k_{i+1}})$ is $-1$, while the contributions to $(\gamma,x_{k_{i+2}})$ and $(\gamma,x_{k_{i+3}})$ are zero. We deduce that $\Lambda(e^*_{x_j},\omega_1(e_j))=\sum_{1\leq i\leq 4}(-1)^{i-1}\Lambda(x_j,x_{k_i})=4$.
	\end{proof}
	
\begin{figure}[htb]
	\label{fig:adjacentArcs}\global\long\def\svgwidth{100pt}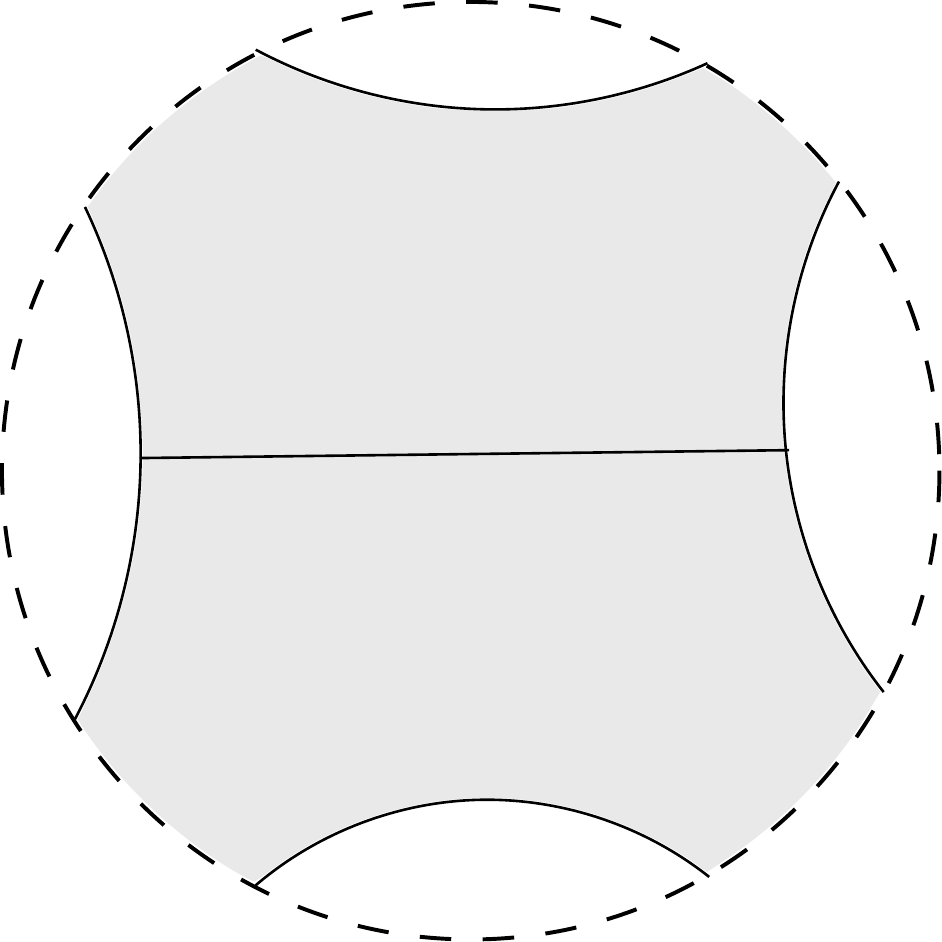
	\caption{The arcs adjacent to an internal arc $x_j$.}
\end{figure}

	Given a triangulation $\Delta$ and an interior arc $i\in\Delta$
	which is not inside a self-folded triangle, note that there is a unique
	quadrilateral $R$ in $\Delta$ with $i$ as a diagonal. The \textbf{flip}
	(or mutation, or Ptolemy transform) of $\Delta$ at $i$, denoted
	$\mu_{i}(\Delta)$, is the triangulation obtained by replacing $i$
	with the other diagonal $i'$ of $R$ while keeping the rest of $\Delta$
	the same.  See Figure \ref{Ptolemy}.
	
	\begin{figure}[htb]
		\global\long\def\svgwidth{220pt}
		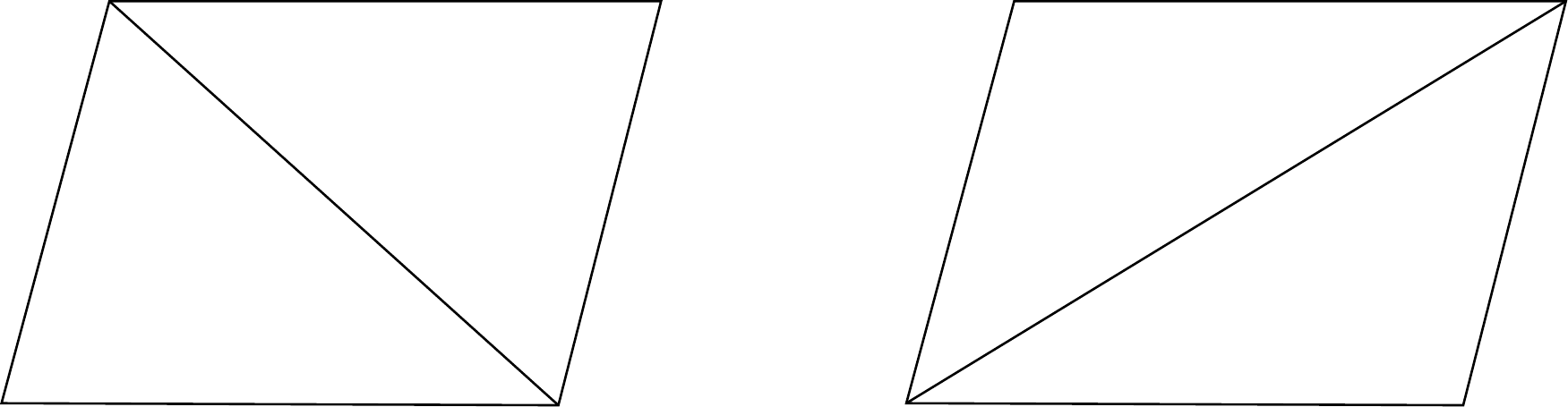 \caption{Two triangulations of a quadrilateral related by the flip at the arc
			$\gamma$. \label{Ptolemy}}
	\end{figure}
	
	\begin{example}\label{AnnEx} 
		Let $\Sigma$ be the annulus with one marking on each boundary component
		as in Figure \ref{Annulus}. The annulus on the left side of the figure
		has the triangulation $\Delta$ consisting of two boundary arcs $b_{1},b_{2}$ and two additional
		arcs $\gamma_{1},\gamma_{2}$. One checks that if we label $\gamma_{1},\gamma_{2},b_{1},b_{2}$
		as $1,2,3,4$, respectively, then the corresponding compatible pair
		$(\sd_{\Delta},\Lambda)$ is that of Examples \ref{QuivEx} and \ref{AnComp}.

		In Proposition \ref{SkAProp}, we will see that $\Sk_t(\Sigma)$ can be identified with the cluster algebra associated to this $(\sd_{\Delta},\Lambda)$. Using the Kauffman skein relation, we see $[\gamma_1][\gamma_1']=q^{-2} [\gamma_2]^2 + [b_1][b_2]$. This equality recovers the the mutation $\mu_1$ on $\sd_{\Delta}$; compare to \eqref{Eq:mu1} using $[\gamma_1]=A_1$, $[\gamma'_1]=A'_1$, $[\gamma_2]=A_2$, $[b_1]=A_3$, $[b_2]=A_4$, and $q=t^{-2}$. 
	\end{example}
	
	\begin{figure}[htb]
		\begin{tabular}{cc}
			\global\long\def\svgwidth{140pt}
			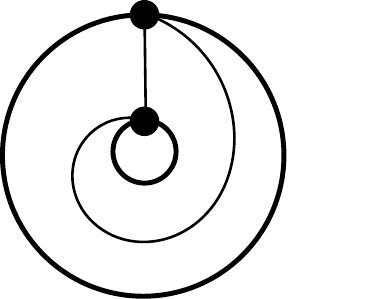  & 
			\global\long\def\svgwidth{160pt}
			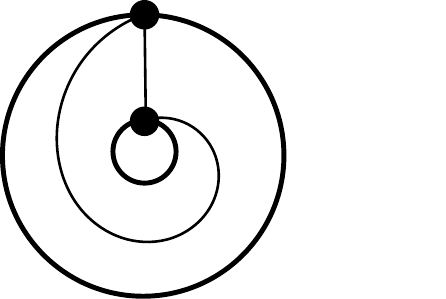 \tabularnewline
		\end{tabular}\caption{Two triangulations of an annulus with one marking on each boundary component.
			The triangulations are related by the flip at $\gamma_1$.
			\label{Annulus}}
	\end{figure}
	
	Given a triangulation $\Delta$, let $[\Delta]$ be the element of
	$\?{\Sk}_{t}(\Sigma)$ corresponding
	to the link consisting precisely of the union of the arcs in $\Delta$. Let $[\Delta_{\ufv}]$
	be the element which is similarly defined but using only the non-boundary
	arcs of $\Delta$.

	For any two simple multicurves $a,b$, let $\cc(a,b)$ be the \textbf{crossing number} of
	$a$ and $b$, i.e., the smallest possible number of intersection
	points in $\SSS\setminus\MM$ between all possible pairs of transverse multicurves homotopic to $a$ and $b$ respectively, not counting intersections at marked points. As in \cite[\S 4.2]{muller2016skein}, we extend this definition to arbitrary $a,b\in\?{\Sk}_{t}(\Sigma)$ by specifying that for $a=\sum_i \alpha_i [C_i]$, $b=\sum_j \beta_j [D_j]$ denoting the decompositions of $a$ and $b$ into linear combinations of distinct simple multicurves $[C_i]$, $[D_j]$,\footnote{The decompositions are deduced from skein relations. They are unique by Lemma \ref{BangLem}.}  the crossing number $\cc(a,b)$ is defined as the maximum of $\cc(C_i,D_j)$ over all pairs $([C_i], [D_j])$ appearing.

	\begin{lem}[\cite{muller2016skein}, Lem. 4.11]
		Let $a$ be a simple arc in $\Sigma$ and let $b\in\?{\Sk}_{t}(\Sigma)$ with $\cc(a,b)\geq1$. Then $\cc(a,ab)\leq\cc(a,b)-1$. 
	\end{lem}
	
	\begin{proof}
		In the unpunctured case, this is \cite[Lem. 4.11]{muller2016skein}. Essentially
		the same argument works in the presence of punctures. 
	\end{proof}
	\begin{cor}
		\label{Deltak} For any triangulation $\Delta$ of $\Sigma$ and any
		$b\in\?{\Sk}_{t}(\Sigma)$, the product
		$[\Delta_{\ufv}]^{k}\cdot b$ for $k\gg0$ is a linear combination of products
		of arcs of $\Delta$. Similarly for any $b\in \Sk_t(\Sigma)$ with $[\Delta]$ in place of $[\Delta_{\uf}]$.
	\end{cor}

Because it will be useful later, we recall that by \cite[Lem. 2.13]{FominShapiroThurston08}, there always exists a triangulation without self-folded triangles.

The following summarizes some known results on the relationship between (quantum) skein algebras and (quantum) cluster algebras in the unpunctured setting.

\begin{prop}[Skein algebras and cluster algebras---unpunctured setting]
	\label{SkAProp} Let $\Sigma$ be a triangulable marked surface without punctures. Let $\Delta$
	be an ideal triangulation of $\Sigma$ and let $\sd_{\Delta}$ be
	the associated seed as in \eqref{sdDelta}.   Let $\Lambda$
	be the corresponding skew-symmetric pairing on $M$ as in \eqref{Lambda-ei-ej},
	so $(\sd_{\Delta},\Lambda)$ is a compatible pair (Lemma \ref{lem:sdLambda-comp}).
	Then we have inclusions 
	\begin{align}
		\s{A}_{t}^{\ord}\subset\Sk_{t}(\Sigma)\subset\s{A}_{t}^{\up}\label{Sk-in-A-un}
	\end{align}
	and 
	\begin{align}
		\?{\s{A}}_{t}^{\ord}\subset\?{\Sk}_{t}(\Sigma)\subset\?{\s{A}}_{t}^{\up}\label{comp-Sk-in-A-un}
	\end{align}
	along with the corresponding classical analogs of these inclusions.
	Furthermore, these inclusions identify simple arcs bijectively with cluster
	variables, boundary arcs being identified with the frozen cluster
	variables. Triangulations correspond bijectively with clusters, and
	mutation of seeds/clusters corresponds to flips of triangulations.
	If each component of $\Sigma$ contains at least two markings (or even without this condition if $t=1$), then the inclusions in \eqref{Sk-in-A-un} are actually isomorphisms.
\end{prop}

\begin{proof}
	In the classical setting, the relationship between flips and mutations
	is in \cite{FominShapiroThurston08} and \cite{FockGoncharov06a}, and similarly
	for the relationship between arcs and cluster variables, cf.
	\cite[Thm. 6.1]{fomin2018cluster}. The quantum analog of these relationships is
	due to \cite{muller2016skein}.

	The inclusions \eqref{Sk-in-A-un} are \cite[Thm. 7.15]{muller2016skein}.  The statement on the inclusions in \eqref{Sk-in-A-un} actually being
	equalties is \cite[Thm. 9.8]{muller2016skein} in the quantum setting
	with at least two markings on each component, and this is extended
	to the once-marked classical setting by \cite[Thm. 1]{CLS}.
	
	To prove \eqref{comp-Sk-in-A-un}, let $\Delta$ be a triangulation,
	and let $b\in\?{\Sk}_{t}(\Sigma)$. By Corollary \ref{Deltak} and the
	correspondence between arcs and cluster variables, $b$ can be multiplied
	by a cluster monomial from the cluster $\?{\s{A}}_{t}^{\sd_{\Delta}}$, with
	no frozen variable factors, to get a linear combination of cluster
	monomials from this cluster. It follows that $b\in\?{\s{A}}_{t}^{\sd_{\Delta}}$.
	Since $b$ and $\Delta$ were arbitrary, the inclusion $\?{\Sk}_{t}(\Sigma)\subset\?{\s{A}}_{t}^{\up}$
	follows.
\end{proof}

We note that the equality $\Sk_t(\Sigma)=\s{A}_t^{\up}$ also holds for all unpunctured surfaces (including once-marked quantum cases) as a consequence of our later results (Theorem \ref{thm:sk_unpunct_basis}), and similarly, $\?{\Sk}_t(\Sigma)=\?{\s{A}}_t^{\up}$ (Corollary  \ref{cor:q_cluster_skein_equal}).

\begin{lem}\label{lem:bar}
	The bar involutions of the skein algebras are compatible with the bar involutions of the corresponding cluster algebras.
\end{lem}
\begin{proof}
	It suffices to check this for $t$ and for the cluster variables/arcs, and the claim is clear for these elements.
\end{proof}

Extending Proposition \ref{SkAProp} to the punctured setting (with $t=1$) requires introducing tagged arcs.

\subsection{The tagged skein algebra}\label{S:tagged_sk}

In the following, we will extend arcs to tagged arcs following \cite{FominShapiroThurston08,fomin2018cluster}. We use this to extend skein algebras to tagged skein algebras as outlined in \cite[\S 8]{musiker2013bases}, and we prove some basic properties. Some results here, like Proposition \ref{prop:compound}, are not strictly needed for the rest of the paper, but they make the construction more elegant and may be of independent interest.

In the approach here, the tagged skein algebra will be understood in terms of functions on decorated Teichm\"uller space.  Alternatively, one may define the tagged skein algebra similarly to how we handled the skein algebra---using ``generalized tagged multicurves'' as formal generators and then modding out by certain relations (those from the untagged setting in \S \ref{SectionSkeinDef} plus the local digon relation from Figure \ref{fig:localDigon}).  The fact that these approaches are equivalent is Corollary \ref{cor:tag-skein}.

\subsubsection{Tagged arcs and tagged triangulations}

As we saw in \S \ref{Cluster_Section}, it is always possible
to mutate a seed with respect to any non-frozen index. On the other
hand, if $i$ is an interior arc of $\Sigma$ which is inside a self-folded triangle (not the noose), then there is no way to flip $i$. Since
flips should correspond to mutations, this motivated \cite{FominShapiroThurston08} to
define \textbf{tagged triangulations}. Tagged triangulations consist
of tagged arcs. A \textbf{tagged arc} is a simple arc $\alpha$ together
with a labelling of each end of $\alpha$ as ``plain'' or ``notched,''
subject to the following rules:

\begin{enumerate}
	\item $\alpha$ does not cut out a once-punctured monogon (i.e., $\alpha$
	cannot be a noose); 
	\item Endpoints of $\alpha$ which lie in $\partial \SSS$ are tagged plain; 
	\item If both ends of $\alpha$ lie at the same marked point, then they
	are tagged in the same way. 
\end{enumerate}
More generally, a \textbf{generalized tagged arc} is an arc $\alpha$ with each end labelled ``plain'' or ``notched'' subject only to the condition that endpoints lying in $\partial \SSS$ are tagged plain.  Generalized tagged arcs are not required to be simple, and they may violate Conditions (1) and/or (3) above.

Given an untagged arc $\alpha$ which is not a noose, we can associate
a tagged arc $\iota(\alpha)$ by simply tagging both ends as plain.
If $\alpha$ is a noose with ends at $m$ and surrounding the puncture
$p$, cutting out a once-punctured monogon $\Sigma_{\alpha}$, then
the associated tagged arc $\iota(\alpha)$ is represented by a path
inside $\Sigma_{\alpha}$ between $m$ and $p$ which is tagged plain
at $m$ and notched at $p$, cf. Figure \ref{fig:iota_noose}.

\begin{figure}[htb]
	\global\long\def\svgwidth{220pt}
	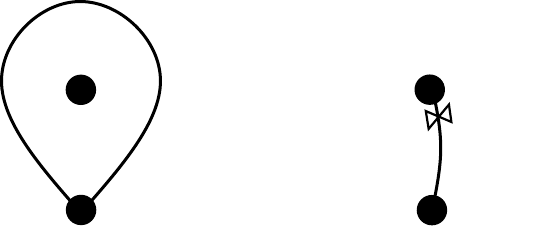 \caption{Left: A noose $\alpha$.  Right: The corresponding tagged arc $\iota(\alpha)$ with notch at $p$ and plain tagging at $m$.  Here and throughout, notches are illustrated with an hourglass/bow tie shape. \label{fig:iota_noose}}
\end{figure}

On the other hand, given a tagged arc $\alpha$, we can forget the
tags to obtain an untagged arc $\alpha^{\circ}$. Following \cite{FominShapiroThurston08}, two tagged arcs
$\alpha_{1},\alpha_{2}$ are called \textbf{compatible} if 
\begin{itemize}
	\item $\alpha_{1}^{\circ}$ and $\alpha_{2}^{\circ}$ do not intersect in $\SSS\setminus\MM$; 
	\item If $\alpha_{1}^{\circ}=\alpha_{2}^{\circ}$ up to isotopy,
	then at least one end of $\alpha_{1}$ is tagged in the same way as
	the corresponding end of $\alpha_{2}$; 
	\item If $\alpha_{1}^{\circ}\neq\alpha_{2}^{\circ}$ up to isotopy
	but $\alpha_{1}$ and $\alpha_{2}$ have a common endpoint $o$, then
	the ends of $\alpha_{1}$ and $\alpha_{2}$ at $o$ are tagged in
	the same way. 
\end{itemize}
Two tagged arcs $\alpha_{1},\alpha_{2}$ are said to be the same if $\alpha_{1}^{\circ}$ is isotopic to $\alpha_{2}^{\circ}$ such that the ends are tagged in the same way. A \textbf{tagged triangulation} is a maximal collection of distinct pairwise-compatible tagged arcs.

Given a tagged triangulation $\Delta$ and \textit{any} non-boundary
arc $i\in\Delta$, the \textbf{flip} of $\Delta$ at $i$ is the unique
triangulation $\mu_{i}(\Delta)=(\Delta\setminus\{i\})\cup\{i'\}$
which can be obtained by replacing $i$ with a different tagged arc
$i'$. The signed adjacency matrix $B$ of $\Delta$ is defined \cite[Def. 9.6]{FominShapiroThurston08} to be that of the ideal triangulation $\Delta^{\circ}$ obtained as follows:
\begin{itemize}
	\item If all arcs of $\Delta$ have the same tagging at some $p\in \MM$, we simply forget this tagging to obtain the end of an ideal arc.
	\item If two arcs of $\Delta$ have different taggings at some $p\in \MM$, then the arc with the notch at $p$ is replaced with the corresponding noose as in Figure \ref{fig:iota_noose}.
\end{itemize}
With these signed adjacency matrices, flips again correspond to mutations as in \eqref{eprime} for $\omega=B^T$; cf. \cite[Lem. 9.7]{FominShapiroThurston08}.

We note that for unpunctured surfaces, tagged arcs and  untagged simple arcs are
equivalent (since notches are only allowed at punctures). We therefore
might not specify in the unpunctured setting whether our arcs are
tagged or not. 

\subsubsection{The tagged skein algebra as functions on decorated Teichm\"uller space}\label{subsub:tagged-skein}

In \cite{fomin2018cluster}, tagged arcs are interpreted geometrically as corresponding to functions on a decorated Teichm\"uller space $\s{T}(\Sigma)$.  Briefly, points in $\s{T}(\Sigma)$ correspond to the following data:
\begin{enumerate}
	\item a complete finite-area hyperbolic structure on $\SSS\setminus \MM$ with constant curvature $-1$, with geodesic boundary on $\partial \SSS\setminus \MM$ and cusps at $\MM$, considered up to diffeomorphisms of $\SSS$ which fix $\MM$ and are homotopic to the identity, plus
	\item a choice of horocycle around each cusp (i.e., a loop perpendicular to all the geodesics through which it passes).
\end{enumerate}
Given a horocycle $h$ of length $l$ around a cusp $p$, the conjugate horocycle $h'$ around $p$ is the one of length $l^{-1}$.  Now, given a tagged arc $\gamma$, let $l_{\gamma}:\s{T}(\Sigma)\rar \bb{R}$ be the function which, for a given hyperbolic structure and choice of horocycles, gives the signed distance between two horocycles (negative if the horocycles intersect) along a geodesic representative of $\gamma$---here, the two horocycles are those associated to the endpoints of $\gamma$, and if an end of $\gamma$ is notched, the corresponding horocycle is replaced with its conjugate.  The associated \textbf{lambda length} is \begin{align}\label{lambda-length}
	\lambda_{\gamma}:=e^{l_{\gamma}/2}:\s{T}(\Sigma)\rar \bb{R}_{>0}.
\end{align}

Let $\Delta$ be an ideal triangulation of $\Sigma$ not containing a self-folded triangle (such triangulations always exist by \cite[Lem. 2.13]{FominShapiroThurston08}).  The corresponding lambda lengths yield a homeomorphism $\s{T}(\Sigma)\rar \bb{R}_{>0}^{|\Delta|}$ \cite[Thm. 7.4]{fomin2018cluster} (building on ideas of Penner \cite{Pen87,Pen04,Pen12}). 
Furthermore, the lambda lengths for all other tagged arcs are Laurent polynomials of the lambda lengths associated to arcs in $\Delta$ (in the case of notched arcs in once-punctured closed surfaces this seems to be new but follows from our Example \ref{ex:doubly-notched-arcloop}).  In fact, \cite[Thm. 8.6]{fomin2018cluster} shows that this algebra of tagged arc lambda lengths forms an ordinary cluster algebra with tagged arcs corresponding bijectively to cluster variables, tagged triangulations corresponding bijectively to clusters, and flips corresponding bijectively with mutations (neglecting notched arcs on once-punctured closed surfaces, which we shall deal with shortly).

We can also understand loops\footnote{In \cite{FockGoncharov06a}, loops are associated to functions on a certain moduli space $\s{A}_{\Sigma}$ of decorated twisted $\SL_2$-local systems; cf. \S \ref{sec:dec_SL2_moduli} for a review of this perspective and \eqref{eq:HolA} for the functions associated to loops.  \cite[Thm. 1.7(b)]{FockGoncharov06a} identifies $\s{T}(\Sigma)$ with the positive real locus of $\s{A}_{\Sigma}$.  The interpretation of loops as functions on $\s{T}(\Sigma)$ given here follows.  The functions associated to arcs are related to lambda lengths in \cite[\S 11.3]{FockGoncharov06a}.} as functions $\s{T}(\Sigma)\rar \bb{R}_{>0}$.  The space of hyperbolic metric as in (1) above is equivalent to
\begin{align}\label{eq:Teich_hom_sp}
	\Teich(\Sigma):=\Hom'(\pi_1(\Sigma),\PSL_2(\R))/\PSL_2(\R)
\end{align}
where $\Hom'(\pi_1(\Sigma),\PSL_2(\R))$ denotes the set of discrete and faithful representations of the the fundamental group $\pi_1(\Sigma):=\pi_1(\SSS\setminus \MM)$ into $\PSL_2(\R)$ such that a loop surrounding a puncture becomes parabolic, and the action of $\PSL_2(\R)$ is by conjugation.  So given a non-contractible loop $L$ in $\Sigma$, a point $x\in \s{T}(\Sigma)$ associates to $L$ an element of $\PSL_2(\R)$ up to conjugation.  This can be represented by an element  of $\SL_2(\R)$ (up to $\SL_2$-conjugation) which we choose to have positive trace if $L$ is not contractible in $\SSS\setminus \MM$ and negative trace otherwise.  This trace is the value of the function associated to $L$ at the point $x$.

The (partially compactified) \textbf{tagged skein algebra} $\?{\Sk}^{\Box}(\Sigma)$ is now defined as the algebra generated by the loops and tagged arcs, viewed as functions $\s{T}(\Sigma)\rar \bb{R}_{>0}$.  The localized version $\Sk^{\Box}(\Sigma)$ (also called the tagged skein algebra) is generated by the tagged arcs, loops, and reciprocals of boundary arcs (viewed again as functions $\s{T}(\Sigma)\rar \bb{R}_{>0}$).

Note that the definition in \eqref{lambda-length} above makes sense for generalized tagged arcs as well (we define $\lambda_{\gamma}=0$ for all contractible generalized tagged arcs $\gamma$).  It seems more natural to include not just loops and tagged arcs (and inverses of boundary arcs) in the tagged arc skein algebras, but also all these generalized tagged arcs.  In fact, this would not add anything new:

\begin{prop}\label{prop:compound}
	$\?{\Sk}^{\Box}(\Sigma)$ contains all generalized tagged arcs.
\end{prop}
This will be proved in \S \ref{subsub:digon}.  We note that many of the ideas used in the proof of Proposition \ref{prop:compound} are sketched in \S \cite[\S 8.4]{musiker2013bases}, though the main focus of loc. cit. is on the more complicated principal coefficients setup.

\subsubsection{Covering spaces and the skein relation}\label{Sk-relations}

Before we consider the skein relations for $\Sk^{\Box}(\Sigma)$, we introduce some additional techniques.

\begin{lem}\label{lem:tag-change}
	Simultaneously changing all tags at a puncture $p$ induces involutive automorphisms $\iii_p$ of $\?{\Sk}^{\Box}(\Sigma)$ and $\Sk^{\Box}(\Sigma)$.
\end{lem}
In terms of $\s{T}(\Sigma)$, this automorphism $\iii_p$ corresponds to the self-homeomorphism which replaces horocycles at $p$ with their conjugates, and $\iii_p$ is the induced map on $\?{\Sk}^{\Box}(\Sigma)$ or $\Sk^{\Box}(\Sigma)$.

\begin{rem}\label{rmk:tag-switch}
	By \cite[Prop. 7.10]{FominShapiroThurston08}, $\Sk^{\Box}(\Sigma)$ for $\Sigma$ a once-punctured closed surface admits two distinct but isomorphic cluster structures: one with cluster variables corresponding to plain arcs and one with cluster variables corresponding to notched arcs.  These cluster structures are swapped by the automorphism $\iii_p$ .
\end{rem}

\begin{rem}\label{rem:FG-Z2Z}
	If $\Sigma$ has $k$ punctures, then the automorphisms described in Remark \ref{rmk:tag-switch} yield an action of $(\bb{Z}/2\bb{Z})^k$ on $\Sk^{\Box}(\Sigma)$.  This is interpreted in \cite[\S 12.6]{FockGoncharov06a} as an action on the moduli space $\s{A}_{\SL_2,\Sigma}$ of decorated twisted $\SL_2$-local systems on $\Sigma$.  Loc. cit. uses this action to construct the tagged bracelets basis from a different perspective; cf. \S \ref{sec:dec_SL2_moduli}.  Loc. cit. also sketches an argument implying the positivity of this tagged bracelets basis with respect to the full cluster atlas (not just ideal triangulations), but we will avoid assuming this result.
\end{rem}

Let $\SSS':=\SSS\setminus \MM$, and let $\pi:\wt{\SSS}'\rar \SSS'$ be a finite covering space.  Given a point in $\s{T}(\Sigma)$ (the choice of point does not matter), we pull back the associated hyperbolic structure to get a hyperbolic structure on $\wt{\SSS}'$, and then we compactify $\wt{\SSS}'$ with cuspidal points to get a covering space $\wt{\SSS}\rar \SSS$ which is possibly ramified at the punctures.  Let $\wt{\Sigma}=(\wt{\SSS},\pi^{-1}(\MM))$.  The inclusion $\pi:\pi_1(\wt{\SSS}')\hookrightarrow \pi_1(\SSS)$ induces a pullback $\pi^*:\Teich(\Sigma)\rar \Teich(\wt{\Sigma})$, which is obviously injective when we describe the Teichm\"uller spaces as in \eqref{eq:Teich_hom_sp}.  Furthermore, a choice of horocycle $H$ for some $p\in \MM$ determines a horocycle for each point in $\pi^{-1}(p)$; namely, the corresponding component of $\pi^{-1}(H)$. We thus obtain an inclusion $\pi^*:\s{T}(\Sigma) \hookrightarrow \s{T}(\wt{\Sigma})$.

Now let $\Sk^{\Box}_{\pi}(\wt{\Sigma})$ and $\Sk^{\Box}_{\pi}(\Sigma)$ be the algebras generated by loops and generalized tagged in $\wt{\Sigma}$ and $\Sigma$, respectively (viewed as functions on $\s{T}(\wt{\Sigma})$ and $\s{T}(\Sigma)$, respectively), which are always tagged plain at puntures where $\pi$ is ramified.  Similarly define $\?{\Sk}^{\Box}_{\pi}(\wt{\Sigma})$ and $\?{\Sk}^{\Box}_{\pi}(\Sigma)$.  Then the map $\pi^*$ induces an action on functions which gives algebra homomorphisms $\pi_*:\?{\Sk}_{\pi}^{\Box}(\wt{\Sigma})\rar \?{\Sk}^{\Box}_{\pi}(\Sigma)$ and 
\begin{align}\label{eq:pi-star}
	\pi_*:\Sk^{\Box}_{\pi}(\wt{\Sigma})\rar \Sk^{\Box}_{\pi}(\Sigma).
\end{align} 
Here, the reason we have to restrict to arcs which are tagged plain at ramified points is that if $\pi$ is ramified with order $k$ at $p$, and $H$ is a horocycle centered at $p$, then the length of the corresponding component of $\pi^{-1}(H)$ is $k$ times the length of $H$---in particular, conjugacy of horocycles around ramified punctures is not preserved under pullback.

Note that for an element $[C]\in \?{\Sk}^{\Box}_{\pi}(\wt{\Sigma})$ represented by a \textbf{generalized tagged multicurve} $C$---i.e., a multicurve where arcs are generalized tagged arcs---it is clear from the construction that $\pi_*[C]=[\pi(C)]$.

The upshot is that relations in $\Sigma$ can be understood by examining relations in a covering space $\wt{\Sigma}$.  We apply this now to recover some skein relations for $\?{\Sk}^{\Box}(\Sigma)$.

\begin{lem}\label{lem:tag-skein}
When viewed as functions on $\s{T}(\Sigma)$, loops and arcs satisfy all the relations from \S \ref{SectionSkeinDef}.
\end{lem}
We note that \cite[\S 6]{musiker2013matrix} shows Lemma \ref{lem:tag-skein} up to some sign issues.
\begin{proof}
	The ($q=1$ case of the) skein relation from Figure \ref{SkeinFig} is the only difficult relation to check.  For this, the case of two distinct arcs crossing at a single point is \cite[Prop. 7.6]{fomin2018cluster} (originally due to \cite[Prop. 2.6(a)]{Pen87}). An argument showing this skein relation for distinct arcs and/or loops crossing a single point is given in \cite[Proof of Thm. 12.2]{FockGoncharov06a}.  For distinct arcs and/or loops crossing an arbitrary number of times, one may always pass to a covering space $\wt{\Sigma}$ where the curves in question cross at only one point, and then the skein relation applies here, hence it holds for $\Sigma$ as well.
	
	It remains to consider the case of a self-intersecting curve $C$ representing $[C]\in \?{\Sk}^{\Box}(\Sigma)$.  Let $[C']$ be the element of $\?{\Sk}^{\Box}(\Sigma)$ obtained by resolving all self-crossings of $C$ via the skein relation---our goal is to show that $[C]=[C']$.  Let $\Delta$ be an ideal triangulation for $\Sigma$ without self-folded triangles. Since $\?{\Sk}^{\Box}(\Sigma)$ is an integral domain, it suffices to show that $[\Delta]^k\cdot [C]=[\Delta]^k \cdot [C']$ for some positive integer $k$.  Consider the representative $\sum_i a_i [C_i]$ (for multicurves $C_i$) of $[\Delta]^k\cdot [C]$ obtained by resolving all crossings of $C$ with $k$ slightly perturbed copies of $\Delta$ via the skein relation (but not resolving self-crossings of $C$), then setting all contractible arcs to $0$, all contractible loops to $-2$, and all loops around punctures equal to $2$.  A representative $\sum_j a_j [C'_j]$ for $[\Delta]^k \cdot [C']$ is obtained from $\sum_i a_i [C_i]$ by resolving all self-crossings of the multicurves $C_i$, so the resulting multicurves $C'_j$ will have no crossings.  Since the skein relations respect homotopy of multicurves\footnote{Recall that we require fibers of homotopies to be multicurves, hence immersions, so the Reidemeister moves of type RII, RIIb, RIII are allowed, but RI is not; cf. \cite[Fig. 2]{Thurst}.  Loc. cit. uses the term ``regular isotopy'' for this.} \cite[Prop. 3.3]{Thurst}, to show the equivalence between $\sum_i a_i [C_i]$ and $\sum_j a_j [C'_j]$, it suffices to show that all crossings of the multicurves $C_i$ can alternatively be removed via homotopy.

	To see this, let us pass to an $s$-sheeted covering space $\wt{\Sigma}$ with a lift $\wt{C}$ of $C$ that no longer self-intersects.  Let $\wt{\Delta}$ be the lifted triangulation for $\wt{\Sigma}$.  As in Corollary \ref{Deltak}, $[\wt{\Delta}_{\uf}]^{\wt{k}}\cdot [\wt{C}]$ for sufficiently large $\wt{k}\in \bb{Z}_{\geq 0}$ is a linear combination of products of arcs in $\wt{\Delta}$.  That is, suppose we resolve all crossings of $\wt{C}$ with $\wt{k}$ slightly perturbed copies of $\wt{\Delta}$, then set contractible arcs, contractible loops, and loops around punctures equal to $0$, $-2$, and $2$, respectively, to obtain a representative $\sum_i a_i [\wt{C}_i]$ of $[\wt{\Delta}_{\uf}]^{\wt{k}}\cdot [\wt{C}]$.  Then each multicurve $\wt{C}_i$ consists of disjoint unions of arcs from $\wt{\Delta}$, up to   homotopy.  By construction, for $k=s\cdot \wt{k}$, the expression $\pi_*(\sum_i a_i [\wt{C}_i])$ is equal to  $\sum_i a_i[C_i]=[\Delta]^k \cdot [C]$ of the previous paragraph and satisfies $C_i=\pi(\wt{C}_i)$.  Since the multicurves $\wt{C}_i$ are   homotopic to disjoint unions of lifts of arcs from $\Delta$, the multicurves $C_i$ must be   homotopic to disjoint unions of arcs from $\Delta$.  In particular, the self-crossings of the multicurves $C_i$ can be removed via   homotopy, as desired.
\end{proof}

\begin{lem}\label{lem:tag-skein2}
	The skein relation of Figure \ref{SkeinFig} (with $q=1$) applies to crossings involving generalized tagged arcs with notches as well.
\end{lem}
\begin{proof}
	By associativity, it suffices to consider tagged multicurves $C$ with only one or two components.  Let $\MM_C\subset \MM$ consist of the marked points which do not appear as ends of differently marked strands of $C$. I.e., we can apply covering space arguments using \eqref{eq:pi-star} for coverings which are possibly ramified at points of $\MM_C$, but not at points of $\MM\setminus \MM_C$.  Let us refer to such covering spaces as $\MM_C$-ramified coverings.
 
    We may also assume that there are no non-essential crossings (crossings which can be removed via homotopy, even including homotopies that move us outside the space of immersions) because (as previously noted) the skein relations respect the framed Reidemeister moves, plus the relation that a single twist is equivalent to $-1$.  In particular, in cases where $C$ contains a contractible component, the claim becomes trivial.

    Now suppose all ends of $C$ are at distinct marked points.  Then we can apply  the tag changing automorphisms of Lemma \ref{lem:tag-change} to reduce to the plain-tagged setting, and then the claim follows from Lemma \ref{lem:tag-skein}.

    Suppose $C$ consists of a single component.  We have reduced to the case where this component is a non-contractible generalized tagged arc with both ends at the same marked point.  By taking an $\MM_C$-ramified covering, choosing a lift of $C$ in which the crossing under consideration still exists, and using \eqref{eq:pi-star}, we reduce to the case where all ends are at distinct marked points as before, and the claim follows.

    So now suppose $C$ consists of two components, $C_1$ and $C_2$.  Since we have now shown that the skein relation is satisfied for cases with a single component, we may assume that both components of $C$ are simple.  We will say that a component $C_i$ of $C$ is contractible in $\SSS\setminus \MM_C$ if $C_i$ cuts out a disk in $\SSS$ whose interior does not contain a point of $\MM_C$ (this allows for arcs with ends in $\MM_C$ to be viewed as contractible in $\SSS\setminus \MM_C$).  If neither component of $C$ is a generalized tagged arc which is contractible in $\SSS\setminus \MM_C$, then we can use an $\MM_C$-ramified covering space again to reduce to the case from above where all ends are at distinct marked points.

    Next suppose that just one component, say $C_1$, is not a generalized tagged arc which is contractible in $\SSS\setminus \MM_C$.  By passing to an $\MM_C$-ramified covering, we may assume that $C_1$ is not an arc with both ends at the same marked point.  Now if $C_2$ \textit{is} a generalized tagged arc which is contractible in $\SSS\setminus \MM_C$, then the only point which could possibly be in $\MM\setminus \MM_C$ is the endpoint of $C_2$, so being contractible in $\SSS\setminus \MM_C$ implies $C_2$ is already contractible as an arc in $\Sigma$, hence is trivial.

    It remains to consider the case where $C_1$ and $C_2$ are both generalized tagged arcs which are contractible in $\SSS\setminus \MM_C$ but not in $\Sigma$.  In this case, $C_1$ and $C_2$ will each have differently tagged ends, and each cuts out a once-punctured monogon with the other component's endpoint as its puncture.  See Figure \ref{fig:c1c2}.
    \begin{figure}[htb]	\global\long\def\svgwidth{100pt}%
	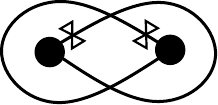
	\caption{\label{fig:c1c2}}
\end{figure}
It will follow from Example \ref{ex:different-tags} below that each component here is equivalent to $(2+L)$ where $L$ is the loop around the two punctures, so $C=(2+L)^2$.  The fact that we get the same thing after applying the skein relation to the crossings follows using the computation of \eqref{eq:alphabeta} in \S \ref{subsec:local_digon}.
\end{proof}

\subsubsection{The digon relation and its applications}\label{subsub:digon}

If no component of $\Sigma$ is a once-punctured closed surface, then the inclusion of the tagged arcs with notches into $\mr{\Sk}(\Sigma)$ (the field of fractions of $\Sk^{\Box}(\Sigma)$) can be realized combinatorially by applying the \textbf{digon relations} from \cite[Def. 8.5]{fomin2018cluster}:
\begin{align}\label{eq:digon}
	\gamma\gamma'=\alpha+\beta
\end{align}
for $\gamma,\gamma',\alpha,\beta$ as in Figure \ref{fig:Digon0}. Here, the endpoints of $\alpha$ and $\beta$ might coincide (cf. Example \ref{ex:noose}), and by Lemma \ref{lem:tag-change}, the relation is similarly imposed for the modifications obtained by simultaneously changing all the tags at any puncture.  By loc. cit., this relation is indeed satisfied by the associated lambda lengths.
\begin{figure}[htb]	\global\long\def\svgwidth{140pt}%
	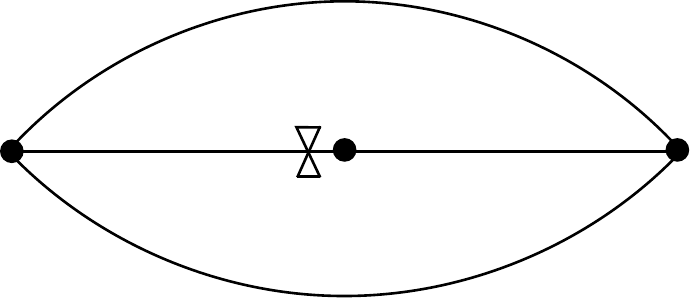
	\caption{The digon relation: $\gamma\gamma'=\alpha+\beta$.
		\label{fig:Digon0}
	}
\end{figure}

We note that formulas for Laurent expansions of tagged arcs in terms of the the arcs of an arbitrary tagged triangulation (even in the principal coefficient setting) have been worked out in \cite{MusikerSchifflerWilliams09} and extended to loops in \cite{musiker2013bases}.  Also cf. \cite{wilson2020surface} for a simplified version of this construction.

\begin{eg}[\cite{fomin2018cluster}, Lem. 8.2]\label{ex:noose}
	Consider a plain-tagged generalized arc $\alpha$ cutting out a once-punctured monogon as in Figure \ref{fig:iota_noose}.  Let $\gamma$ be the plain arc from the vertex $m$ to the puncture $p$ in this monogon, and let $\gamma'=\iota(\alpha)$, i.e., $\gamma'$ is isotopic to $\gamma$ but has a notch at $p$.  Then $\gamma\gamma'=\alpha$.  Indeed, this is the special case of the digon relation where the opposite marked points coincide, making $\alpha$ a noose and $\beta$ a contractible arc (hence equivalent to $0$).  Applying $\iii_m$ yields a similar relation for the case where both ends of $\alpha$ are notched.
\end{eg}

\begin{eg}\label{ex:different-tags}
	Let $\?{\gamma}$ be a non-contractible generalized tagged arc with both ends at the same puncture $p$ such that one end of $\?{\gamma}$ is plain and the other is notched.  Let $\?{\gamma}'$ be the same arc but with both ends tagged plain.   Let $L_1,L_2$ be the two loops forming the boundary of a tubular neighborhood of $\?{\gamma}$ in $\SSS$.  See Figure \ref{ggp}.

        \begin{figure}[htb]
		\centering
		\def\svgwidth{300pt}
		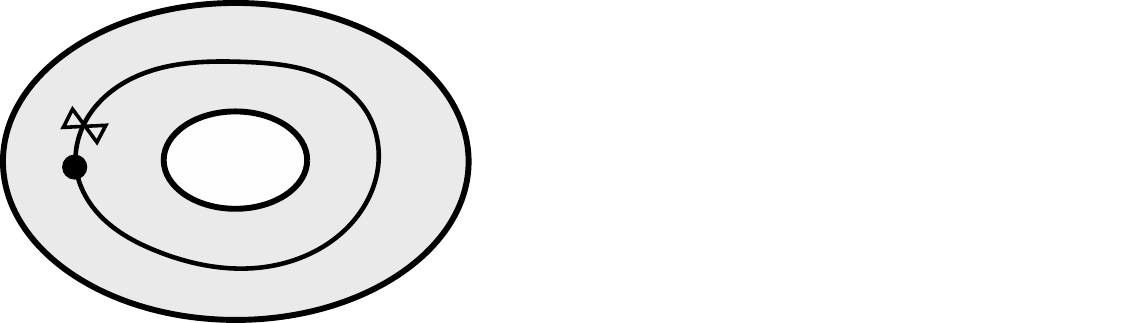
	\caption{A tubular neighborhood of $\?{\gamma}$ and $\?{\gamma}'$ in $\Sigma$, bound by loops $L_1$ and $L_2$. \label{ggp}}
	\end{figure}

  We can find a triple-cover $\pi:\wt{\Sigma}\rar \Sigma$ with lifts $\gamma$ and $\gamma'$ of $\?{\gamma}$, $\?{\gamma}'$, respectively, contained in a once-punctured digon; cf. Figure \ref{abggp}.  The bounding arcs of this punctured digon are labeled $\alpha$ and $\beta$.  We apply the digon relation in $\wt{\Sigma}$ to find $\gamma\gamma'=\alpha+\beta$ as in Figure \ref{fig:Digon0}.

          \begin{figure}[htb]
		\centering
		\def\svgwidth{180pt}
		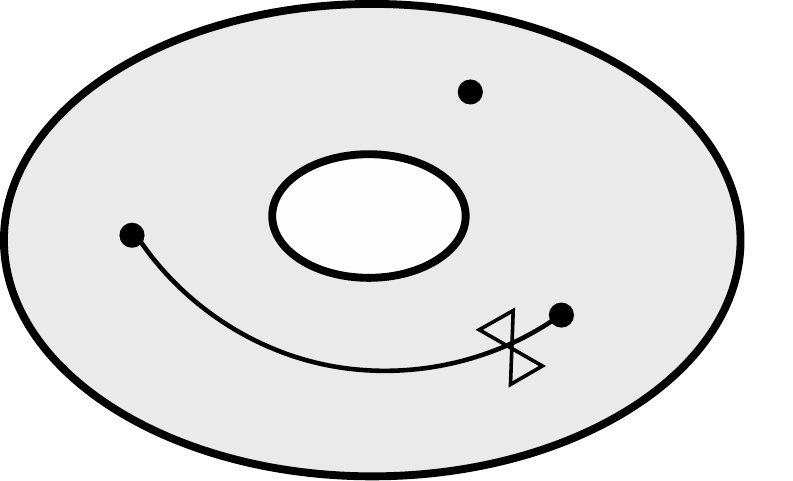
	\caption{A triple cover of the tubular neighborhood from Figure \ref{ggp}, along with lifts $\gamma$ and $\gamma'$ of $\?{\gamma}$ and $\?{\gamma}'$, respectively.  By the digon relation, we have $\gamma\gamma'=\alpha+\beta$. \label{abggp}}
	\end{figure}
  
  We now apply $\pi_*$ to get $\?{\gamma}*\?{\gamma}'=\?{\alpha}+\?{\beta}$, where $\?{\alpha}:=\pi_*(\alpha)$ and $\?{\beta}:=\pi_*(\beta)$, cf. Figure \ref{abggp}.
            \begin{figure}[htb]
		\centering
		\def\svgwidth{300pt}
		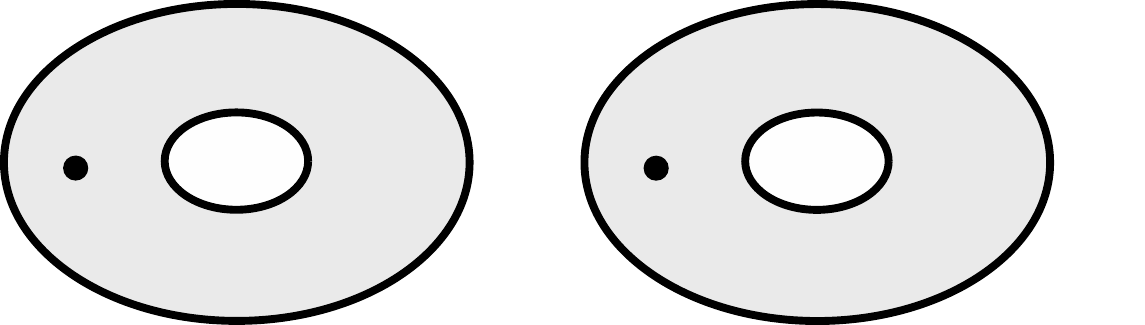
	\caption{$\?{\alpha}\coloneqq\pi_*(\alpha)$ and $\?{\beta}\coloneqq\pi_*(\beta)$.\label{ab}}
	\end{figure}
 By applying the skein relation and recalling that contractible arcs are equivalent to $0$, we see that $\?{\alpha}$ is equivalent to $\?{\gamma}' * L_1$, and  and $\?{\beta}$ is equivalent to $\?{\gamma}' * L_2$.  Dividing both sides by $\?{\gamma}'$ yields
	\begin{align}\label{eq:gamma-bar}
		\?{\gamma}=L_1+L_2.
	\end{align}
Note that we may have $L_1=L_2$ (this happens in the once-punctured torus).  In the case where $\?{\gamma}$ cuts out a once punctured monogon, one of the loops, say $L_2$, surrounds a single puncture and is therefore equal to $2$, yielding $\?{\gamma}=2+L_1$.
\end{eg}

\begin{proof}[Proof of Proposition \ref{prop:compound}]
	We saw in \S \ref{Sk-relations} that the skein relation of Figure \ref{SkeinFig} applies with tagged multicurves, so we can repeatedly apply this to reduce to the case of generalized tagged arcs which are simple (have no crossings).  Consider a simple generalized tagged arc $\alpha$ which violates Condition (1) but not Condition (3) from the start of \S \ref{S:tagged_sk}.  I.e., $\alpha$ cuts out a once-punctured monogon and both ends of $\alpha$ have the same tag.  Such $\alpha$ is contained in $\?{\Sk}^{\Box}(\Sigma)$ by Example \ref{ex:noose}.  Now consider the case of a simple generalized tagged arc $\?{\gamma}$ which violates Condition (3) (and possibly also Condition (1)); i.e., both ends of $\?{\gamma}$ lie at the same puncture, but they are tagged differently.  We saw in Example \ref{ex:different-tags} that such $\?{\gamma}$ can be expressed as a sum of two loops.  The claim follows.
\end{proof}

\begin{eg}\label{ex:doubly-notched-arcloop}
	Consider the same setup as in Example \ref{ex:different-tags}, but replace the curve $\?{\gamma}$ having incompatible tags with a curve $\?{\gamma}^{\diamond}$ with both ends notched.  We find lifts $\gamma^{\diamond},\gamma'$ in a triple cover again and apply the digon relation to compute $\gamma^{\diamond}*\gamma'=\alpha'+\beta'$, where this time the tagged arcs $\alpha'$ and $\beta'$ are notched at one end and plain at the other. Applying $\pi_*$ and then using the skein relations, we find that $\pi_*(\alpha)$ and $\pi_*(\beta)$ equal $L_1*\?{\gamma}$ and $L_2*\?{\gamma}$ for $\?{\gamma}$, $L_1$ and $L_2$ as in Example \ref{ex:different-tags}.  Applying \eqref{eq:gamma-bar}, we find
	\begin{align}\label{eqn:gamma-gamma-tag}
		\?{\gamma}^{\diamond}*\?{\gamma}'=(L_1+L_2)^2.
	\end{align}
	This will be applied in \S \ref{sub:bracelets-1p} to understand the doubly-notched arcs on once-punctured closed surfaces.
\end{eg}

In particular, in \eqref{eqn:gamma-gamma-tag} we find an expression for the doubly-notched arcs $\?{\gamma}^\diamond$ as Laurent polynomials of plain arcs and loops.  Moreover, we observe that they are Laurent polynomials in the arcs of any tagged triangulation $\Delta'$.  Indeed, Lemmas \ref{lem:tag-skein} and \ref{lem:tag-skein2} allow us to apply skein relations as in the proof of Lemma \ref{Deltak} to show that $\?{\gamma}^{\diamond}$ is a Laurent polynomial in the arcs of $\Delta'$ up to re-taggings, and then Examples \ref{ex:different-tags} and \ref{ex:doubly-notched-arcloop} imply that these re-taggings of arcs of $\Delta'$ can themselves be expressed as Laurent polynomials in the arcs of $\Delta'$.  Putting all this together, we have the following analog of Proposition \ref{SkAProp}:

\begin{prop}[Skein algebras and cluster algebras---punctured setting]
	\label{SkAPropPun} Let $\Sigma$ be a triangulable surface. Let $\Delta$
	be a ideal triangulation of $\Sigma$ and let $\sd_{\Delta}$ be
	the associated seed as in \eqref{sdDelta}. 
	Then 
	\begin{align}\label{eq:SkBoxA}
		{\s{A}}^{\ord}\subset {\Sk}^{\Box}(\Sigma) \subset \s{A}^{\up}
	\end{align}
	and
	\begin{align}\label{eq:SkBoxAbar}
		\?{\s{A}}^{\ord}\subset \?{\Sk}^{\Box}(\Sigma) \subset \?{\s{A}}^{\up}.
	\end{align}
	Cluster variables are in bijection with the tagged arcs, minus the notched arcs on once-punctured closed components.  Here, frozen cluster variables correspond to the boundary arcs.  Similarly, clusters are in bijection with tagged triangulations which do not include notched arcs in once-punctured closed components.  Flips of these tagged triangulations correspond to mutations.
	\label{iib}
\end{prop}

In \eqref{eq:SkBoxAbar}, to see that $\?{\Sk}^{\Box}(\Sigma)\subset \?{\s{A}}^{\up}$, we can apply the proof of \eqref{comp-Sk-in-A-un} from Proposition \ref{SkAProp} to see that at least $\?{\Sk}(\Sigma)\subset \?{\s{A}}^{\up}$. In addition, we have seen that arcs with notches belong to $\s{A}^{\up}$, and the containment $\?{\Sk}^{\Box}(\Sigma)\subset \?{\s{A}}^{\up}$ follows by applying \eqref{eq:digon} to see that arcs with notches are in fact contained in $\?{\s{A}}^{\up}$. 

We note that if $\Sigma$ has at least two boundary
markings in each component, then we have $\s{A}^{\ord}=\s{A}^{\up}$ by \cite[Thms 4.1, 10.6]{MuLaca}. So in this case, \eqref{eq:SkBoxA} and \eqref{eq:SkBoxAbar}
consist of isomorphisms.

\begin{rem}\label{rmk:once-torus}
	In the case of a once-punctured
	torus, \cite[\S 5.3]{zhou2020cluster} has shown that $\Sk^{\Box}(\Sigma)$ is not
	equal to $\s{A}^{\up}$ (actually, Zhou considers the ``canonical algebra'' generated by the theta functions,
	but our Theorem \ref{thm:tag_sk_up_cl_alg} and Equation \eqref{eq:torus_tag_arc} imply that this agrees with $\Sk^{\Box}(\Sigma)$ up to $\Q$-scalars).  Rather, loc. cit. shows that $\Sk^{\Box}(\Sigma)$ equals to the intersection
	of $\s{A}^{\up}$ with the upper cluster algebra associated to the
	notched arcs. It is natural to ask if the analogous statements hold
	for all cases with once-punctured closed components.
\end{rem}

\subsubsection{The local digon relation}\label{subsec:local_digon}
Greg Muller has shown us a ``local'' version of the digon relation, cf. Figure \ref{fig:localDigon}.  To our knowledge, this relation has not previously appeared in the literature, so we provide a proof below.

\begin{figure}[htb]
\global\long\def\svgwidth{220pt}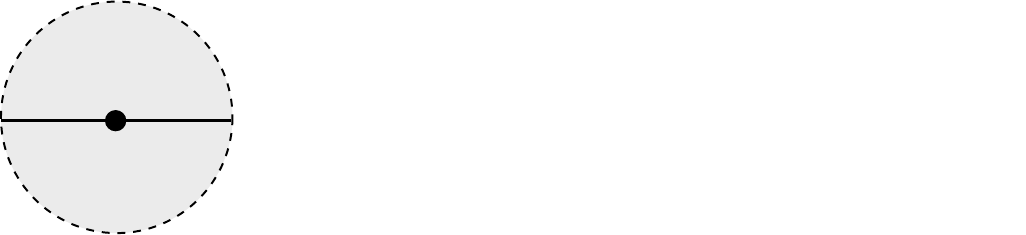
	\caption{The local digon relation.  The relation holds even if there are additional strands ending at the puncture (inserted into each term without interacting with the already-pictured strands).\label{fig:localDigon}}
\end{figure}

\begin{prop}\label{prop:loc-digon}
    Generalized tagged arcs satisfy the local digon relation of Figure \ref{fig:localDigon}.
\end{prop}
\begin{proof}
    Suppose the two strands from the left-hand side of Figure \ref{fig:localDigon} are parts of generalized tagged arcs $\alpha$ and $\beta$.  Let $p$ be the puncture in the figure, and let $q,r$ be the markings corresponding to the the other endpoints of $\alpha$ and $\beta$, respectively.  By applying the skein relations, we may assume that $\alpha$ and $\beta$ have no crossings (with each other or with themselves).

    If $p,q,r$ are distinct, then the relation follows from the usual digon relation \eqref{eq:digon} from \cite{fomin2018cluster}.  Similarly, if $q=r\neq p$, and if the taggings of $\alpha$ and $\beta$ at $q=r$ are the same, then the relation holds by Example \ref{ex:noose}.

    Now suppose $\alpha=\beta$ (by which we mean that the two strands are parts of the same component, not distinct homotopic components), so $p=q=r$.  Then we are in the situation of Example \ref{ex:different-tags}, and the relation holds by \eqref{eq:gamma-bar}.

    Now suppose $p=q\neq r$.  If $\alpha$ is contractible, then the left-hand side of Figure \ref{fig:localDigon} is $0$, and the right-hand side is shown to be $0$ using the skein relations.  If $\alpha$ is not contractible, then we may again pass to a covering space $\pi:\wt{\Sigma}\rar \Sigma$ where the ends of $\alpha$ are distinct and then apply the usual digon relation of \eqref{eq:digon}, followed by $\pi_*$, to obtain the right-hand side of Figure \ref{fig:localDigon}.  The case $p=r\neq q$ is similar.

    Next, suppose again that $p\neq q=r$, but this time suppose that the taggings of $\alpha$ and $\beta$ at $q=r$ are distinct.  By applying the tag-changing automorphism $\iii_p$, we may assume both ends of $\alpha$ are plain and both ends of $\beta$ are notched.  Let $L$ be the loop that goes around both punctures.  The local digon relation in this case claims \begin{align}\label{eq:alphabeta}\alpha \beta = 2+L.\end{align}  Let $\gamma$ be the noose around $q$ based at $p$ which surrounds $\alpha$, and let $\delta$ be the arc obtained from $\alpha$ by notching the end at $q$.  Then $\alpha\delta=\gamma$ by Example \ref{ex:noose}, so 
    \begin{align*}
        \alpha \beta = \frac{\beta \gamma}{\delta}
    \end{align*}
    By the $p=q\neq r$ case of the local digon relation that we checked above, along with the skein relation, we compute $\beta \gamma = \delta(2+L)$, so $\frac{\beta \gamma}{\delta}=2+L$, as desired.

    Finally, suppose $p=q=r$ but $\alpha\neq \beta$.  If $\alpha$ and $\beta$ are both contractible, then the skein relations imply both sides are $0$.  Otherwise, we may reduce to one of the above cases by taking a covering space.
\end{proof}

\subsection{Laminations and coordinates}
\label{sec:shear_coord}
We next collect useful notions and results about laminations and corresponding coordinates for any marked surface $\Sigma=(\SSS,\MM)$ with punctures, following the conventions in \cite{fomin2018cluster,yurikusa2020density}.

\begin{defn}[Lamination \cite{fomin2018cluster,allegretti2016geometry,yurikusa2020density}] \label{def:lamination}
	By \textbf{laminates}, we mean any of the following crossingless connected curves in $\SSS\backslash\MM$,
	considered up to homotopy\footnote{Here, homotopies are in $\SSS \setminus \MM$ and are subject to the requirement that boundaries of curve components must always map to $\partial \SSS \setminus (\partial \SSS \cap \MM)$.},
	\begin{itemize}
		\item closed curves contained in $\SSS\backslash(\partial\SSS\cup\MM)$;
		\item curves whose ends belong to $\partial\SSS\backslash\MM$ or spiral
		into punctures (with clockwise or counterclockwise direction);
	\end{itemize}
	while the following curves are not allowed:
	\begin{itemize}
		\item a null-homotopic closed curve
		\item a curve whose ends belong to $\partial\SSS\backslash\MM$ which is isotopic to an interval inside $\partial\SSS\backslash\MM$
		\item a curve with two ends spiraling into the same puncture in the same
		direction without enclosing anything else.
	\end{itemize}
	A laminate is said to be \textbf{special} if it is a loop around a single puncture or an arc retractable to
	an interval on $\partial S$ containing exactly one marked point.
	It is said to be \textbf{bounded} if it has no spiraling end.  A laminate which is not a closed (boundary-less) curve may be called an \textbf{arc-laminate}.
	
	An \textbf{integer unbounded lamination} $l=\sum w_{i}[l_{i}]$ is a finite
	$\NN$-weighted (formal) sum of the homotopy classes of non-special pairwise non-intersecting laminates $l_{i}$. The collection of such laminations is denoted
	by $\mathcal{\mathcal{\cX}}_{L}(\Sigma,\Z)$.  For use in \S \ref{sec:cluster_poisson}, we define $\wt{\mathcal{\cX}}_{L}(\Sigma,\Z)$ similarly, but allowing the laminates $l_i$ to be retractable to intervals on $\partial S$ containing exactly one marked point and allowing these special laminates to have coefficients in $\bb{Z}$ rather than just $\bb{N}$.

	An \textbf{integer bounded lamination} $l=\sum w_{i}[l_{i}]$ is a finite $\Z$-weighted
	(formal) sum of the homotopy classes of pairwise non-intersecting bounded laminates $l_{i}$,
	such that the weight of non-special laminates are non-negative. The
	collection of such laminations is denoted by $\cA_{L}=\mathcal{A}_{L}(\Sigma,\Z)$. 
\end{defn}

It is natural to generalize the above definitions to define the bounded
or unbounded laminations with rational coefficients. The corresponding
sets are denoted $\cA_L(\Sigma,\Q)$ and $\cX_{L}(\Sigma,\Q)$. 

Choose an (ideal) triangulation $\Delta$. For any laminate $l$ and
any arc $\gamma\in\Delta$, we deform $l$ so that it intersects
with $\gamma$ transversely at a minimal number of points. Then let
$\cc(\gamma,l)$ denote this minimal intersecting number between
$\gamma$ and $l$.

Denote $N_\Q=N\otimes \Q$ for $N=N_{\sd_{\Delta}}=\bb{Z}^{\Delta}$ the lattice for the seed $\sd_{\Delta}$ associated to $\Delta$.
\begin{prop}[\cite{FG4} \S 3.2; \cite{allegretti2016geometry} Prop. 5.1.3]\label{prop:intersection_coordinate}
	For any given ideal triangulation $\Delta$, we have a bijection
	\begin{align}\label{pil}
		\pi:\cA_{L}(\Sigma,\Q) \simeq  N_{\Q}
	\end{align}
	such that any bounded lamination $L=\sum_i w_{i}[l_{i}]$ is sent to
	$\pi(L)=\sum_{\gamma\in \Delta}(\sum_i \frac{1}{2}w_{i}\cc(\gamma,l_{i}))e_{\gamma}$, where $\{e_{\gamma}\}_{\gamma\in \Delta}$ is the basis in the seed $\sd_\Delta$.
\end{prop}

By \cite[Prop. 5.1.4, Thm. 5.1.5]{allegretti2016geometry}, this identification of $\s{A}_L(\Sigma,\bb{Q})$ with $N_{\bb{Q}}$ respects mutations, so the identification is independent of seed if we view $N$ as the set of tropical points of $\s{A}$.

\begin{defn}[{Shear coordinates \cite[Def. 12.2]{fomin2018cluster}\cite[\S 2.2]{yurikusa2020density}}]\label{def:shear}
	Let $l$ be a laminate and $\gamma$ an interior arc
	of the (ideal) triangulation $\Delta$.

	If $\gamma$ is not contained inside a self-folded triangle, the shear
	coordinate $b_{\gamma}^{\Delta}(l)$ of $l$ with respect to the triangulation
	$T$ and its arc $\gamma$ is defined to be the sum of contributions
	from the (minimal\footnote{If $\gamma$ connects to punctures, we can remove small neighborhoods
		of the punctures so that $C$ and $l$ have only finitely many intersections
		\cite[\S 5.1.2]{allegretti2016geometry}. Then we can deform
		$l$ outside the neighborhoods so that the intersection number becomes
		the minimal. Finally, we count the contributions from Figures \ref{fig:shear_coordinate}.}) intersections of $l$ with $\gamma$ as shown in Figure \ref{fig:shear_coordinate}. 
	
	If $\gamma$ is contained in a self-folded triangle, denote the corresponding
	noose by $\gamma'$ and the puncture surrounded by $\gamma'$ by $p$.
	Let $l^{(p)}$ denote the laminate obtained from $l$ by reversing
	the directions of its spirals at $p$ if they exist. Define the shear
	coordinate $b_{\gamma}^{\Delta}(l):=b_{\gamma'}^{\Delta}(l^{(p)})$.
\end{defn}

\begin{figure}[htb]
    \global\long\def\svgwidth{240pt}
    	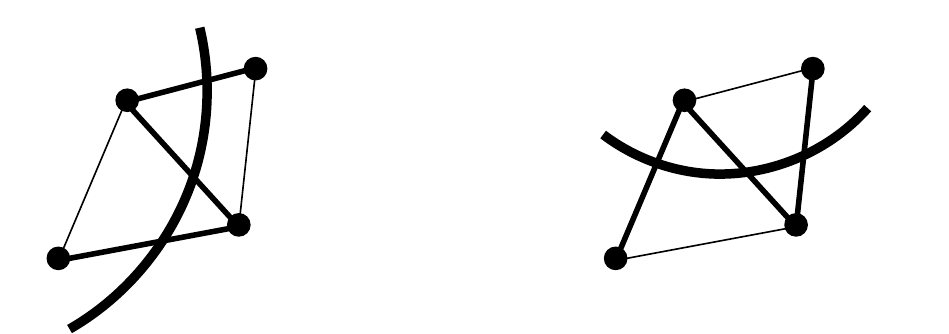
    		\caption{\label{fig:shear_coordinate}  Shear coordinate contributions.  We get a contribution of $+1$ to $b_{\gamma}^{\Delta}(l)$ when $\gamma$ and the adjacent arcs crossed by $l$ form an S-shape, and a contribution of $-1$ when they form a Z-shape.}
\end{figure}

One can further define the shear coordinates for tagged triangulation.
If a tagged triangulation is given by $\iota(\Delta)$ where $\Delta$
is an ideal triangulation, define $b_{\gamma}^{\iota(\Delta)}(l)=b_{\gamma}^{\Delta}(l)$.
For a general tagged triangulation $\Delta'$, there exists a unique
ideal triangulation $\Delta$ such that $\Delta'$ is obtained from
$\iota(\Delta)$ by simultaneously changing all tags at some (possibly empty) collection of punctures
$p_{1},\ldots,p_{r}$, see \cite[Rmk. 5.13]{fomin2018cluster} or \cite[Rmk. 3.11]{MusikerSchifflerWilliams09}.
Then one defines $b_{\gamma}^{\Delta'}(l)=b_{\gamma^{(p_{1}\cdots p_{r})}}^{\iota(\Delta)}(l^{(p_{1}\cdots p_{r})})$,
where $\gamma^{(p_{1}\cdots p_{r})}$ is obtained from $\gamma$ by
changing the tags at the punctures $p_{1},\ldots,p_{r}$ and $l^{(p_{1}\cdots p_{r})}$
is obtained from $l$ by reversing the direction of its spirals
at the punctures $p_{1},\ldots,p_{r}$.
\begin{prop}[{\cite[Theorems 12.3 and 13.6]{fomin2018cluster}}]\label{prop:shear_coord}
	For any given tagged triangulation $\Delta$, we have a bijection
	\begin{align*}
		b^{\Delta}:\cX_{L}(\Sigma,\Z) \simeq  \bigoplus_{\gamma\in \Delta_{\ufv}}\Z e_{\gamma}^* 
	\end{align*}
	such that $L=\sum w_{i}[l]_{i}$ is sent to $b^{\Delta}(L)=\sum_{\gamma\in \Delta_{\ufv}}(\sum_i w_{i}b_{\gamma}^{\Delta}(l_i))e_{\gamma}^*$.
\end{prop}

Let $\Delta'$ denote any tagged triangulation obtained from $\Delta$ by a flip. By \cite[Thm. 4.3]{reading2014universal}, for any laminate $l$, $-b^{\Delta}(l)\in \oplus_{\gamma\in \Delta} \Z e^*_\gamma$ agrees with $-b^{\Delta'}(l)\in \oplus_{\gamma'\in \Delta'} \Z e^*_{\gamma'}$ under the natural identifications $\oplus_{\gamma\in \Delta_{\uf}} \Z e^*_\gamma=N_{\sd_{\Delta},\uf}^*\cong N_{\sd_{\Delta'},\uf}^*=\oplus_{\gamma'\in \Delta'_{\uf}} \Z e^*_{\gamma'}$, where the identification $N_{\sd_{\Delta},\uf}^*\cong N_{\sd_{\Delta'},\uf}^*$ is via the piecewise-linear tropical mutation map as in \eqref{fprime}.  Thus, we may view $-b:\s{X}_L(\Sigma,\bb{Z})\rar N_{\uf}^*$ as being well-defined independently of the choice of triangulation $\Delta$.

\begin{dfn}\label{def:el-lam}
	Let there be given a tagged arc $\gamma$. Following \cite[\S 2.3]{yurikusa2020density},
	we associate to $\gamma$ an ``\textbf{elementary laminate}''  $e(\gamma)$
	as below, see Figure \ref{fig:elementary_laminates}:
	\begin{itemize}
		\item $e(\gamma)$ is a laminate contained in a small neighborhood of $\gamma$
		\item if $\gamma$ has a boundary endpoint $x\in\MM\cap\partial S$, then
		the corresponding endpoint of $e(\gamma)$ is a boundary point located
		near $x$ in the clockwise direction (clockwise when viewed as the boundary of a disk removed from a compactification of $\SSS$---i.e., the positive direction with respect to the orientation induced by the counterclockwise orientation on $\SSS$)
		\item if $\gamma$ has an endpoint at a puncture $p$ which is tagged plain
		(resp. notched), then the corresponding end of $e(\gamma)$ spirals
		into $p$ clockwise (resp. counterclockwise)
	\end{itemize}
	For convenience, if $L$ is a closed loop, we define the corresponding laminate to be $e(L):=L$.
\end{dfn}

\begin{figure}[htb]
    \global\long\def\svgwidth{350pt}
    	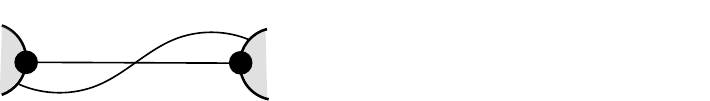
	\caption{Elementary laminates associated to tagged arcs.}
	\label{fig:elementary_laminates}
\end{figure}

Recall from \S \ref{sub:g} the natural projection $\pr_{I_{\uf}}:\bb{Z}^I\rar \bb{Z}^{I_{\uf}}$ modding out by frozen vectors.

\begin{prop}[{\cite[Prop. 5.2]{reading2014universal}}]\label{prop:shear_g}
	For any interior tagged arc $\gamma$ corresponding to a cluster variable and any tagged triangulation $\Delta$,
	we have 
	\begin{align}\label{eq:shear_g}
		-b^{\Delta}(e(\gamma)) =  \pr_{I_{\ufv}}g^{\sd_{\Delta}}([\gamma]),
	\end{align}
	where $\sd_{\Delta}$ is the seed associated to $\Delta$, and $g^{\sd_{\Delta}}([\gamma])$
	is the $g$-vector associated to the cluster variable $[\gamma]$
	with respect to the initial seed $\sd_{\Delta}$.
\end{prop}

We next extend the definition of shear coordinates to account for boundary arcs as well as internal arcs.  For internal arcs $\gamma$, $b_{\gamma}^{\Delta}(l)$ is defined as before.  If $\gamma$ is a boundary arc, then each endpoint of $l$ on $\gamma$ contributes $\pm \frac{1}{2}$ to $b_{\gamma}^{\Delta}(l)$, where the sign is positive (respectively, negative) if the triangle containing $\gamma$, together with the corresponding end of $l$, looks like a triangle from the $+1$ (respectively, $-1$) quadrilateral of Figure \ref{fig:shear_coordinate}; i.e., the contribution is as in Figure \ref{fig:shear2}.  Then $b_{\gamma}^{\Delta}(l)$ is the sum of these contributions from all endpoints of $l$ on $\gamma$.

\begin{figure}[htb]
\global\long\def\svgwidth{250pt}
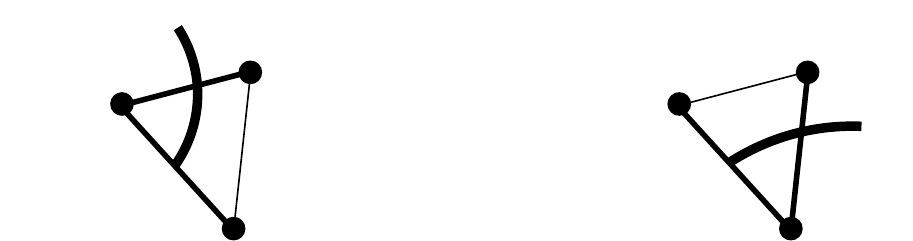
	\caption{The contribution of an end of the laminate $l$ to the shear coordinate $b_{\gamma}^{\Delta}(l)$ for a boundary edge $\gamma$. \label{fig:shear2}}
\end{figure}

We now define
\begin{align*}
	\wt{b}^{\Delta}:\s{A}_L(\Sigma,\bb{Z})\rar \bigoplus_{\gamma\in \Delta} \frac{1}{2}\bb{Z}e_{\gamma}^* = \frac{1}{2} M_{\sd_{\Delta}}
\end{align*}
via 
\begin{align}\label{eq:extended_shear_coord}
	\wt{b}^{\Delta}(\sum_i w_{i}[l]_{i})=\sum_{\gamma\in \Delta}(\sum_i w_{i}b_{\gamma}^{\Delta}(l_i))e_{\gamma}^*.
\end{align}
We may also consider the extension to rational coefficients, $\wt{b}^{\Delta}:\s{A}_L(\Sigma,\bb{Q})\rar \bigoplus_{\gamma\in \Delta} \bb{Q}e_{\gamma}^* = M_{\sd_{\Delta},\bb{Q}}$.

Recall that the bilinear form $\omega$ associated to $\sd_{\Delta}$ gives a linear map $\omega_1$ from $N_{\sd_{\Delta}}=\oplus_{\gamma\in \Delta} \Z e_\gamma$ to $M_{\sd_{\Delta}}=\oplus_{\gamma\in \Delta} \Z e^*_\gamma$. 
\begin{lem}\label{lem:intersection_shear_loop}
	For any ideal triangulation $\Delta$ without self-folded triangles and any $l\in \s{A}_L(\Sigma,\bb{Z})$, we have
	\begin{align}\label{eq:intersection_shear_loop}
		\omega_1(\pi(l))=\wt{b}^{\Delta}(l) 
	\end{align}
\end{lem}
\begin{proof}
	Any internal arc $\gamma$ is a diagonal of some quadrilateral as in Figure \ref{fig:shear_coordinate}. The intersection points at the edges of the quadrilateral contribute to the coordinate on the diagonal $\gamma$ via $\omega_1$. We see that the contributions always sum to twice of the shear coordinate on $\gamma$, as desired.
	
	Similarly, any boundary arc $\gamma$ is contained in exactly one triangle $T$, and the contributions of each intersection with $T\setminus \{\gamma\}$ to $\omega_1(\pi(l))$ and $b^{\Delta}(l)$ are again easily seen coincide.
\end{proof}

It is straightforward to check that, if $\ell$ only consists of loops, then $\wt{b}^{\Delta}(l)$ has vanishing frozen coordinates, i.e, we can write
\begin{align}\label{eq:loop_g}
\omega_1(\pi(l))=b^{\Delta}(l).
\end{align}

\subsection{Dehn twists}
\label{sec:Dehn_twists}
Given an oriented closed laminate $l_{c}$, we denote
by $\tw_{l_{c}}$ the Dehn twist of $\SSS$ along $l_{c}$; see Figure
\ref{fig:Dehn_twist}.

\begin{figure}[htb]
\global\long\def\svgwidth{260pt}
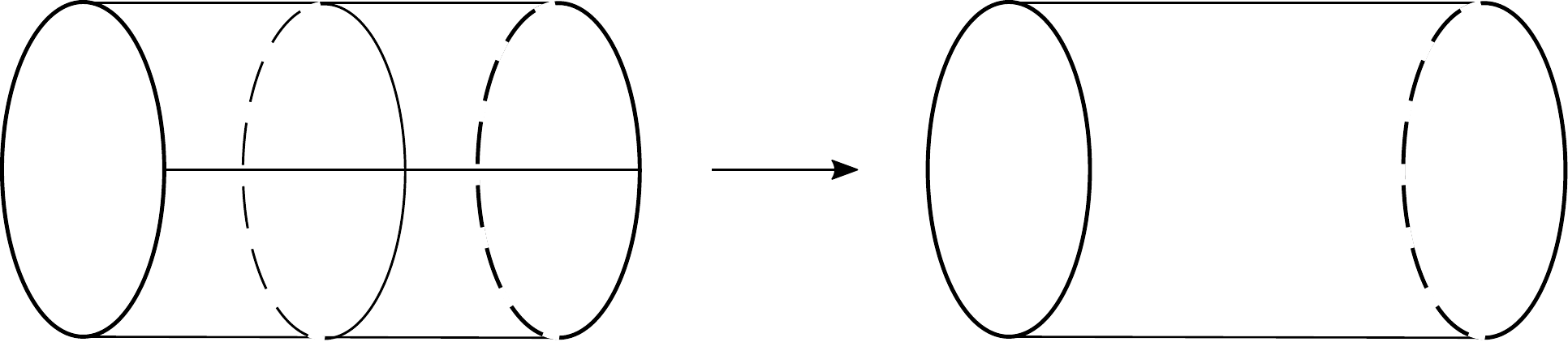
	\caption{Dehn twist. \label{fig:Dehn_twist}}
\end{figure}

\begin{thm}[{\cite[Thm. 2.10]{yurikusa2020density}}]\label{Thm:Yurikusa_Dehn}
	For any given interior arc $\gamma\in\Delta$ (for $\Delta$ either a tagged or ideal triangulation), any non-special laminate $l$, and any closed laminate $l_c$, there exists
	$m'\in\NN$ such that, for any $m\geq m'$, we have
	
	\begin{align}\label{eq:tw}
		b_{\gamma}^{\Delta}(\tw_{l_{c}}^{m}(l))=b_{\gamma}^{\Delta}(\tw_{l_{c}}^{m'}(l))+(m-m')\#(l\cap l_{c})b_{\gamma}^{\Delta}(l_{c}).
	\end{align}
\end{thm}

	\begin{rem}\label{rmk:tw}
		Note that Theorem \ref{Thm:Yurikusa_Dehn} is easily extended to arbitrary integer unbounded laminations $l\in \s{X}_L(\Sigma,\bb{Z})$ by applying the result to each component separately. 
	\end{rem}

For any given tagged triangulation $\Delta=(\gamma)_{\gamma\in\Delta}$,
let $e(\Delta)$ denote the collection $(e(\gamma))_{\gamma\in\Delta_{\uf}}$. The corresponding shear coordinates with respect to any given
tagged triangulation $\Delta'$ span a cone 
\begin{align*}
	C(b^{\Delta'}(e(\Delta)))  :=  \sum_{\gamma\in\Delta_{\uf}}\R_{\geq0}b^{\Delta'}(e(\gamma))\subset \bigoplus_{\gamma\in \Delta_{\uf}} \bb{R}e_{\gamma}^*= (N_{\uf,\bb{R}})^*.
\end{align*}

Let $L=\sum_{i=1}^{r}n_{i}[l_{i}]+\sum_{j=1}^{s}w_{j}[c_{j}]$ denote
an integer unbounded lamination where $l_{i}$ are pairwise non-homotopic
closed laminates and $c_{j}$ non-closed laminates. Let $\Delta^{L}$
denote a tagged triangulation such that $e(\Delta^{L})$ contains
all $c_{j}$ appearing. Denote $N_{i}=\sum_{\gamma\in\Delta^{L}_{\uf}}\#(l_{i}\cap\gamma)$.
Define $\tw_{L}=\prod\tw_{l_{i}}^{\frac{N_{1}\cdots N_{r}}{N_{i}}n_{i}}.$

\begin{prop}[{\cite[Proposition 2.14]{yurikusa2020density}}]\label{prop:closure_laminate_cone}
	We have
	\begin{align*}
		b^{\Delta}(L)  \in  \overline{\cup_{m\geq0}C(b^{\Delta}(e(\tw_L^{m}(\Delta^{L})))}.
	\end{align*}
\end{prop}

\section{The tagged bracelets basis}
\label{section:bracelet_skein}

In \S \ref{sec:weighted_curve}-\S\ref{sec:bracelet_band}, we review the known construction of the bracelets in classical skein algebras $\Sk(\Sigma)$ associated to (weighted) simple multicurves as in \cite{musiker2013bases}.  We simultaneously review the construction of quantum bracelets bases for quantum skein algebras $\Sk_t(\Sigma)$ of unpunctured surfaces, as defined in \cite{Thurst}. Then in \S \ref{Sec:tag-brac} we extend the classical bracelets in the punctured cases by allowing contributions from tagged arcs (as suggested in \cite[\S 8]{musiker2013bases}).

\subsection{Weighted simple multicurves and bangles}

\label{sec:weighted_curve}

Recall from \S \ref{SectionSkeinDef} that a multicurve in $\Sigma$ is an immersion $\phi:C\rar\SSS$ of a closed
unoriented $1$-manifold $C$ such that the boundary of $C$ maps
to $\MM$, but no interior points of $C$ map to $\partial\SSS$ or
$\MM$.  Also recall that a curve is a connected multicurve, and that a multicurve is called simple if it has no interior crossings and no contractible components. For convenience, we will say that a set of simple multicurves is \textbf{non-intersecting} if they are pairwise non-intersecting in $\SSS\setminus \MM$ (even if they intersect in $\MM$). Let $\SMulti(\Sigma)$ denote the set of simple multicurves
in $\Sigma$ considered up to homotopy, no connected component of which
is a \textbf{peripheral} loop, i.e., a loop around a single puncture. Note that, for any simple multicurve $C$, there exists finitely many pairwise non-homotopic and non-intersecting simple curves $C_{1},\ldots,C_{k}$ such that $C$ is homotopic to the union over $i=1,\ldots,k$ of $w_i$ copies of $C_i$ for some integers $w_i\geq 1$. We denote
\begin{align*}
	C=\bigcup_{i=1}^{k} w_{i}C_{i}.
\end{align*}
The curves $C_1,\ldots,C_k$ are called the components of $C$. Note that they may share common endpoints. The empty set $\emptyset$ is also viewed as a simple multicurve.

We call $C$ \textbf{internal} if its components do not include boundary arcs. Let $\SMulti^\circ(\Sigma)\subset \SMulti(\Sigma)$ denote the subset of elements which are internal.

An element of $\SMulti(\Sigma)$ can thus be equivalently defined as a set of pairwise non-homotopic, non-intersecting, non-peripheral simple curves $C_1,\ldots,C_k$, together with associated weights $w_1,\ldots,w_k\in \bb{Z}_{\geq 1}$. Similarly, a \textbf{weighted simple multicurve} is defined to be a collection of pairwise non-intersecting, non-homotopic, non-peripheral simple curves $C_1,\ldots,C_k$ as above with integer weights $w_1,\ldots,w_k$, but now we allow any $w_i\in \bb{Z}$ whenever $C_i$ is a boundary arc (for other curves we still require $w_i\geq 0$; in any case, we identify a weight-$0$ curve with $\emptyset$).  If $C=\bigcup_{i=1}^{k} w_{i}C_{i}$ and $C'=\bigcup_{i=1}^{k} w_{i}'C_{i}$ have pairwise non-intersecting components, their union is defined to be the weighted simple multicurve 
\[
C\cup C'=\bigcup_{i=1}^{k} (w_{i}+w_i')C_{i}.
\]
In particular, $C\cup \emptyset = C$.  The set of weighted simple multicurves on $\Sigma$ is denoted $\wSMulti(\Sigma)$. Note that the set of internal weighted simple multicurves coincides with $\SMulti^{\circ}(\Sigma)$.

Recall that any given link
$L$ represents an element of $\Sk_{t}(\Sigma)$ (where for the classical cases we take $t=1$), which is denoted by $[L]$, or by a slight abuse of notation, simply by $L$.  In particular, since simple multicurves have no interior crossings, they are canonically identified with links, and so any simple multicurve $C=\bigcup_{i=1}^k w_iC_i$ represents an element $[C]\in \Sk_{t}(\Sigma)$.  Furthermore,
\begin{align}\label{eq:bangle}
	[C]= t^{-\alpha} \prod_{i=1}^k [C_i]^{w_i}
\end{align}
where $\alpha$ is chosen so that $[C]$ is bar-invariant --- explicitly,
\begin{align*}
	[C]= t^{-\alpha} [C_1]^{w_1}\cdots [C_k]^{w_k}
\end{align*}
for 
\begin{align}\label{factor-alpha}
	\alpha=\sum w_i w_j \Lambda(C_i,C_j)
\end{align}
where the sum is over all pairs of arc components $C_i,C_j$ of $C$ with $i<j$, and $\Lambda$ is as in \eqref{Lambdaij}. Using \eqref{eq:bangle}, we extend the definition of $[C]\in \Sk_t(\Sigma)$ to any \textit{weighted} simple multicurve $C=\bigcup w_i C_i$ by allowing negative $w_i$ for boundary arcs.

\begin{defn}[Bangles]
	For any weighted simple multicurve $C=\bigcup_{i=1}^{k} w_{i}C_{i}$, the skein algebra element $[C]\in \Sk_{t}(\Sigma)$ is called a bangle element represented by $C$, or simply a \textbf{bangle}. Denote $\BangE{C}=[C]$.
	We call $\BangE{C}$ internal if $C$ is internal.
	
	Let $\Bang_{t}(\Sigma)$, $\?{\Bang}_{t}(\Sigma)$ and $\Bang_{t}^\circ(\Sigma)$   denote the set of bangles, the set of bangles with non-negative weights and the set of internal bangles, respectively.
\end{defn}

When we leave the $t$-subscript off the above notation, it is
understood that $t=1$.
\begin{lem}[\cite{Thurst,muller2016skein}]
	\label{BangLem} 
	Recall that $t$ is understood to equal $1$ if $\Sigma$ has punctures.  With this understanding, the set $\?{\Bang}_{t}(\Sigma)$ forms
	a $\kk_{t}$-module basis for $\?{\Sk}_{t}(\Sigma)$. Similarly, the set $\Bang_{t}(\Sigma)$ forms a $\kk_{t}$-module basis for $\Sk_{t}(\Sigma)$.
\end{lem}

\begin{proof}
	The claims for the classical skein algebra $\?{\Sk}(\Sigma)$ are \cite[Prop. 4.10]{Thurst}. For
	the quantum statements for unpunctured surfaces, \cite[Lem. 4.1]{muller2016skein}
	says that $\?{\Bang}_{t}(\Sigma)$ is a $\kk_{t}$-module basis for
	$\?{\Sk}_{t}(\Sigma)$, and then \cite[Prop. 5.3]{muller2016skein} says that
	$\Bang_{t}(\Sigma)$ is a $\kk_{t}$-module basis for $\Sk_{t}(\Sigma)$. Finally,
	the claim for $\Sk(\Sigma)$ follow from that for $\?{\Sk}(\Sigma)$
	via the same argument used to prove \cite[Prop. 5.3]{muller2016skein}. 
\end{proof}

\subsection{Bracelets}\label{sec:bracelet_band}

\begin{defn}[\cite{musiker2013bases}]
	Let $L$ be a non-contractible non-peripheral simple loop in $\Sigma$. The bracelet curve $\Brac_w L$ is the homotopy class of the closed loop obtained by concatenating $L$ with itself exactly $w$ times {as in the right-hand part of Figure \ref{fig:Annulus-5}}. Note that $\Brac_w L$ has $w-1$ self-crossings.
	
	Let $\BracE{w L}$ denote the element $[\Brac_w L]$ in the classical skein algebra $\Sk(\Sigma)$. It is called a bracelet element or a \textbf{bracelet}. \end{defn}

\begin{figure}[htb]
	\begin{tabular}{ccc}
		\global\long\def\svgwidth{100pt}%
		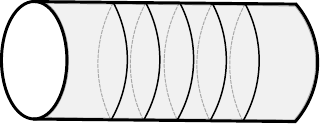 \quad  & \quad
		\global\long\def\svgwidth{100pt}%
		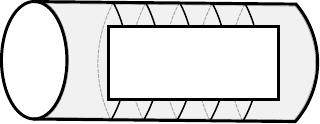 \quad  & \quad
		\global\long\def\svgwidth{100pt}%
		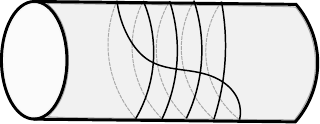 \tabularnewline
	\end{tabular}\caption{A weight-$5$ loop viewed as a bangle (left), a band (middle), and a bracelet (right).
		\label{fig:Annulus-5}}
\end{figure}

Let $\mathfrak{S}_w$ denote the $w$-elements permutation group and $|\mathfrak{S}_w|=w!$ its cardinality. \cite{Thurst} observed that one can associate the trivial $w$-elements permutation to $wL$ and a certain non-trivial permutation to $\Brac(wL)$, with additional multicurves $L_{\sigma}$ being associated to the other permutations $\sigma$ by replacing $w$ parallel strands in $wL$ with strands which cross as indicated by $\sigma$. \cite{Thurst} further considered the formal average of all multicurves $L_{\sigma}$ corresponding to any $w$-elements permutation:
\begin{align*}
	\Band_w L=\sum_{\sigma\in \f{S}_w} \frac{1}{|\mathfrak{S}_w|} L_{\sigma},
\end{align*}
see Figure \ref{fig:Annulus-5}. It represents the following band element.

\begin{defn}[\cite{Thurst}]
	Define $\BandE{wL}$ to be the element of the classical skein algebra $\Sk(\Sigma)$ representing $\Band_w L$.  That is,
	\[
	\BandE{wL}=\sum_{\sigma\in \f{S}_w} \frac{1}{|\mathfrak{S}_w|} [L_{\sigma}].
	\]
	It is called a band element or a \textbf{band}.
\end{defn}

To extend the bracelets and bands to the quantum setting, we first need to define the Chebyshev polynomials.   The \textbf{Chebyshev polynomials of the first kind} are the polynomials
$T_{k}(z)$, $k\in \bb{Z}_{\geq 0}$, determined by: 
\begin{align}\label{Cheb1}
	T_{0}(z) & =2\nonumber \\
	T_{1}(z) & =z\\
	T_{k+1}(z) & =zT_{k}(z)-T_{k-1}(z).\nonumber 
\end{align}
These satisfy 
\begin{align*}
	T_{k}(z)T_{\ell}(z)=T_{k+\ell}(z)+T_{|k-\ell|}(z)
\end{align*}
and can be characterized via 
\begin{align}\label{Tkex}
	T_{k}(z+z^{-1})=z^k+z^{-k}
\end{align}
(this is checked via a straightforward induction).
We note that this last property implies that if $A\in\SL_{2}$, then
$$\Tr(A^{k})=T_{k}(\Tr(A)).$$

\begin{lem}\label{Lem:Tk}
	Let $R$ be any ring.  Let $\{f_k\}_{k\in \bb{Z}_{\geq 0}}$ be a sequence of elements in $R$ with $f_0=1$.  Suppose that for each $s\in \bb{Z}_{\geq 0}$, we have $f_1^s=\sum_{k=0}^s \nu_{s,k}f_k$ where $\nu_{s,k}$ is the coefficient of $z^{k}$ in the Laurent expansion of $(z+z^{-1})^s\in \bb{Z}[z^{\pm 1}]$.  Then $f_{s}=T_s(f_1)$ for all $s\geq 1$.
\end{lem}
\begin{proof}
	We prove this by induction.  The case $s=1$ is trivial.  Now suppose the claim holds for $s-1$.  By the inductive hypothesis,
	\begin{align*}
		f_1^{s} =\sum_{k=0}^{s} \nu_{s,k}f_k= \nu_{s,0}+\left(\sum_{k=1}^{s-1} \nu_{s,k}T_k(f_1)\right)+f_s.
	\end{align*}
	Solving for $f_s$, we find $f_s=h(f_1)$ where $h$ is the polynomial $h(x)=x^s-\nu_{s,0}-\sum_{k=1}^{s-1} \nu_{s,k}T_k(x)\in \bb{Z}[x]$.  By \eqref{Tkex}, the claim holds if $R=\bb{Z}[z^{\pm 1}]$ and $f_k=z^k+z^{-k}$ for $k\geq 1$, so this polynomial $h(x)$ must indeed be $T_s(x)$.
\end{proof}

The \textbf{Chebyshev polynomials of the second kind} are the polynomials
$U_{k}(z)$ defined via the same recursion as in the definition of
$T_{k}(z)$, except that the initial value $T_{0}(z)=2$ is replaced
with $U_{0}(z)=1$.  For $k\in\bb{Z}_{\geq 0}$, $U_{k}$
satisfies 
\begin{align*}
	U_{k}(z+z^{-1})=z^{k}+z^{(k-2)}+\ldots+z^{-(k-2)}+z^{-k}.
\end{align*}

\begin{lem}[{\cite{musiker2013bases,Thurst}}]\label{lem:chebyshev_loops}
	In the classical skein algebra $\Sk(\Sigma)$, for $k\in \bb{Z}_{\geq 1}$, we have
	\begin{align*}
		\BracE{k L}&=T_k([L])\\
		\BandE{k L}&=U_k([L]).
	\end{align*}
\end{lem}

In view of Lemma \ref{lem:chebyshev_loops} and following \cite{musiker2013bases,Thurst}, for any weighted simple multicurve $C=\bigcup_{i=1}^k w_i C_i$, we define the bracelet element $\langle C\rangle_{\Brac}$ in the classical skein algebra $\Sk(\Sigma)$ and, for unpunctured $\Sigma$, in the quantum skein algebra $\Sk_t(\Sigma)$, via
\begin{align}\label{Eq:BracDef}
	\langle C\rangle_{\Brac}:=t^{-\alpha}\prod_{i=1}^k \langle w_iC_i\rangle_{\Brac},
\end{align}
where $\alpha$ is chosen so that $\langle C\rangle_{\Brac}$ is bar-invariant (i.e., as in \eqref{factor-alpha}), and the factors are given as follows:
\begin{itemize}
	\item for $C_i$ a simple non-peripheral loop,
	\begin{align}\label{wLBrac}
		\langle w_iC_i\rangle_{\Brac}:=T_{w_i}([C_i]);
	\end{align}
	\item for $C_i$ a simple arc, $\langle w_i C_i\rangle_{\Brac}:=[C_i]^{w_i}$. 
\end{itemize}

Denote the subsets of bracelet elements in $\Sk_{t}(\Sigma)$ by
\begin{align*}
	\Brac_t(\Sigma)&:=\left\{ \langle C\rangle_{\Brac}\in\Sk_{t}(\Sigma)| C\in\wSMulti(\Sigma)\right\}\\
	\?{\Brac}_t(\Sigma)&:=\left\{ \langle C\rangle_{\Brac}\in\Sk_{t}(\Sigma)| C\in\SMulti(\Sigma)\right\}\\
	\Brac^\circ_t(\Sigma)&:=\left\{ \langle C\rangle_{\Brac}\in\Sk_{t}(\Sigma)| C\in\SMulti^\circ(\Sigma)\right\}
\end{align*}

One defines the band element $\langle C\rangle_{\Band}$ and the sets $\Band_t(\Sigma)$, $\?{\Band}_t(\Sigma)$, and $\Band_t^{\circ}(\Sigma)$ in
the same way as $\Brac_{t}(\Sigma)$, $\?{\Brac}_{t}(\Sigma)$, and $\Brac^{\circ}_{t}(\Sigma)$, respectively, but using $U_{w}$ in place
of $T_{w}$.  Note that, for any given $C\in \wSMulti$, $\BangE{C}$, $\BracE{C}$, $\BandE{C}$ are all bar-invariant. In particular, they use the same factor $t^\alpha$.

We have the following analog of Lemma \ref{BangLem}. 
\begin{lem}[\cite{Thurst,muller2016skein}]
	\label{Lem:bracelet_basis} 
	Recall that $t$ is understood to equal $1$ if $\Sigma$ has punctures.  With this understanding, each of the sets
	$\?{\Brac}_{t}(\Sigma)$ and $\?{\Band}_{t}(\Sigma)$ forms
	a $\kk_{t}$-module basis for $\?{\Sk}_{t}(\Sigma)$. Similarly, each
	of the sets $\Brac_{t}(\Sigma)$ and $\Band_{t}(\Sigma)$ forms a $\kk_{t}$-module basis for $\Sk_{t}(\Sigma)$. 
\end{lem}
\begin{proof}
	The claims for the bracelets and bands follow easily from
	the bangles case (Lemma \ref{BangLem}) and the definitions of $T_{w}$ and $U_{w}$.
\end{proof}

\subsection{Tagged bracelets}\label{Sec:tag-brac}

Let there be given a triangulable marked surface $\Sigma$.
As we saw in Lemmas \ref{BangLem} and \ref{Lem:bracelet_basis}, we have various bases for the classical skein algebra $\?{\Sk}(\Sigma)$ and $\Sk(\Sigma)$.
In the tagged setting $\Sk^{\Box}(\Sigma)$, we modify the constructions from \S \ref{sec:weighted_curve}-\S \ref{sec:bracelet_band} to include tagged arcs. The resulting set of ``tagged bracelets'' is essentially the set denoted $\s{B}$ in \cite[\S 8]{musiker2013bases} except that we are working with boundary coefficients but not principal coefficients (principal coefficients will be considered later).

First, we define a \textbf{tagged simple multicurve} $C$ to be a union of non-intersecting tagged arcs and simple non-peripheral loops, considered up to homotopy, such that the tagged arcs appearing in $C$ are compatible (as in the definition of tagged triangulations in \S \ref{S:tagged_sk}). As before, we can denote
\[
C=\bigcup_i w_i C_i,
\]
such that exactly $w_i$-many tagged arcs or simple loops in $C$ are isotopic to $C_i$. The components $C_i$ of $C$ are said to be \textbf{compatible}.

As before, $C$ is called \textbf{internal} if its components include no boundary arcs. By allowing $w_i\in \Z$ for boundary arcs $C_i$ in $C$, we obtain the notion of a \textbf{weighted tagged simple multicurve}. Let $\SMultiTag$ denote the set of tagged simple multicurves and $\wSMultiTag$ the set of weighted tagged simple multicurves.

We naturally extend the construction of previous bases to tagged cases.  The bangle element $\langle C\rangle_{\Bang}$ represented by a weighted tagged simple multicurve $C\in\wSMultiTag$ is defined as the corresponding element $[C]$ in $\Sk^{\Box}(\Sigma)$, which is defined as
\[
\langle C \rangle_{\Bang}=[C]=\prod_i [C_i]^{w_i}
\]
where $C_i$ are the components of $C$ with weights $w_i$.

Similarly, we define the bracelet element $\langle C\rangle_{\Brac}$ such that
\[
\langle C \rangle_{\Brac}:= \prod_i \langle w_i C_i\rangle_{\Brac}
\]
where the factors are defined as follows:
\begin{itemize}
	\item for $L$ a simple loop with weight $w>0$, $\langle wL \rangle _{\Brac}:=T_{w}([L])$;
	\item for $\gamma$ a tagged arc with weight $w$, $\langle w\gamma \rangle _{\Brac}:=[\gamma]^w$.
\end{itemize}
By replacing $T_w(\ )$ by $U_w(\ )$, we obtain the definition of the band element $\langle C\rangle_{\Band}$ represented by $C$.

Now 
\[
\Brac^{\Box}(\Sigma)
\]
is defined to be the set of all tagged bracelets, and 
\[
\?{\Brac}^{\Box}(\Sigma)
\]
is defined to be the subset of tagged bracelets for which all weights are non-negative (analogously to
in the definition of $\?{\Brac}(\Sigma)$).

A proof that $\Brac^{\Box}(\Sigma)$ forms a basis for $\Sk^{\Box}(\Sigma)$
does not seem to exist in the literature yet, although a sketch of
possible approaches to proving this appeared in \cite[\S 8]{musiker2013bases}.
We shall prove this in Theorem \ref{thm:tag_sk_up_cl_alg} by identifying the tagged bracelets with the theta bases.

\begin{lem}
	We have $\Brac(\Sigma)\subset\Brac^{\Box}(\Sigma)$ and $\?{\Brac}(\Sigma)\subset\?{\Brac}^{\Box}(\Sigma)$.
\end{lem}
\begin{proof}
	The containments are obvious for bracelets which do not include any nooses (arcs bounding once-punctured monogons).  One easily extends to cases with nooses using Example \ref{ex:noose}.
\end{proof}

The following will be useful in the proof of Corollary \ref{cor:tag_sk_up_cl_alg}.
\begin{lem}\label{lem:Brac-spans}
$\?{\Brac}^{\Box}(\Sigma)$ spans $\?{\Sk}^{\Box}(\Sigma)$.
\end{lem}
\begin{proof}
Let $[C]$ be a product of loops and tagged arcs in $\?{\Sk}^{\Box}(\Sigma)$.  Thanks to Lemmas \ref{lem:tag-skein} and \ref{lem:tag-skein2}, any crossings between these loops and tagged arcs can be resolved using the skein relation (Figure \ref{SkeinFig} with $q=1$), so we may assume there are no crossings.

Now suppose some pair of tagged arcs $\gamma_1,\gamma_2$ appearing in $C$ are not compatible.  Since we have assumed there are no crossings, it must be the case that the taggings of $\gamma_1$ and $\gamma_2$ disagree at some common endpoint.  We can thus apply the local digon relation (Figure \ref{fig:localDigon}) to write $\gamma_1*\gamma_2$ as a sum of two other arcs $\alpha$ and $\beta$ (possibly using the trick of first working on a covering space and then projecting).  This may result in new crossings, and we again resolve these with the skein relations.  The total number of arc-factors in each term is now strictly less than before.  For any term which still has incompatible arcs, these steps (digon relation followed by skein relations) can be repeated.  Since the total number of arcs decreases with each iteration, this process must terminate.

We thus express $[C]$ as a linear combination of elements $[C_i]$ with each $C_i$ a tagged simple multicurve.  Since bangles can be expressed as linear combinations of (tagged) bracelets, we see that $[C]$ can be expressed as a linear combination of tagged bracelets.    Since no step in this process can produce arcs with negative weights, the tagged bracelets appearing here all lie in $\?{\Brac}^{\Box}(\Sigma)$, as desired.
\end{proof}

\section{Scattering diagrams and theta functions}\label{Section_Scat}
In this section we review the construction of the theta bases for cluster algebras, due to \cite{gross2018canonical} in the classical setting and \cite{mandel2021scattering,davison2019strong} in the quantum setting.

\subsection{The quantum torus Lie algebra}\label{Sec:q-Lie}

Suppose we have either a compatible pair $(\sd,\Lambda)$, or just a seed $\sd$ for which the Injectivity Assumption possibly fails.  In this latter case, we replace $\sd$ with $\sd^{\prin}$ as in \S \ref{prinsub} so that Assumption \ref{inj-assumption} holds, and then we pick a $\Lambda$ compatible with this new seed.

Recall our notation regarding quantum torus algebras in \S \ref{qtoralg}.  Let $\tilde{\f{g}}^t$ be the Lie algebra over $\kk_t=\kk[t^{\pm 1/D}]$  given by 
\begin{align*}
	\tilde{\f{g}}^t\coloneqq \bigoplus_{m\in M^+} \kk_t\cdot z^m
\end{align*}
with bracket $$[z^a,z^b]=(t^{\Lambda(a,b)}-t^{-\Lambda(a,b)}) z^{a+b}.$$  That is, $\tilde{\f{g}}^t$ is a Lie subalgebra over $\kk_t$ of the quantum torus algebra $\kk_t[M]$ with its commutator bracket --- specifically, it is the Lie subalgebra spanned over $\kk_t$ by $z^m$ with $m\in M^+$.  Then the \textbf{quantum torus Lie algebra} $\f{g}^t$ is defined to be the Lie subalgebra of $\tilde{\f{g}}^t\otimes_{\kk_t} \kk(t)$ generated over $\kk_t$ by the elements of the form 
\begin{align}\label{hatz}
	\hat{z}^m\coloneqq \frac{z^m}{t^{d|m|}-t^{-d|m|}}
\end{align} 
for $m\in M^+$, $d$ as in \eqref{Lambda-B}. Here, $|m|$ denotes the index of $m$ in $M^+$, i.e., the largest integer $k$ for which $\frac{1}{k}m \in M^+$.  We call an element $m$ \textbf{primitive} if $|m|=1$. The following Lemma implies that these elements $\hat{z}^m$ in fact generate $\f{g}^t$ as a $\kk_t$-module (not just as a $\kk_t$ Lie subalgebra).

\begin{lem}\label{lem:str_const_laurent}
	For any $a,b\in M^+$, we have $\frac{(t^{d|a+b|}-t^{-d|a+b|})(t^{\Lambda(a,b)}-t^{-\Lambda(a,b)})}{(t^{d|a|}-t^{-d|a|})(t^{d|b|}-t^{-d|b|})}\in \kk_t$.
\end{lem}
\begin{proof}
	Notice that $\frac{1}{d}\Lambda$ is $\bb{Z}$-valued on $M^+$ by \eqref{Lambda-v1v2}. Let $\alpha=|a|,\beta=|b|,\gamma=|a+b|$.  So  $a=\alpha a_0$, $b=\beta b_0$ for some primitive elements $a_0,b_0\in M^+$.  Let $\delta$ denote the greatest common divisor of $\alpha,\beta$. Then $\delta$ divides $\gamma$. Denote $\lambda=\frac{1}{d}\Lambda(a_0,b_0)$. The fraction in question can be rewritten as
	\begin{align*}
		\frac{(t^{d\gamma}-t^{-d\gamma})(t^{d\alpha \beta \lambda}-t^{-d\alpha \beta \lambda})}{(t^{d\alpha}-t^{-d\alpha})(t^{d\beta}-t^{-d\beta})}
	\end{align*}
	Denote $q=t^{2d}$. It suffices to show that the following fraction belongs to $\kk[q]$:
	\begin{align*}
		\frac{(q^{\gamma}-1)(q^{\alpha \beta \lambda}-1)}{(q^{\alpha}-1)(q^{\beta}-1)}.
	\end{align*}
	The roots of the denominator are of the following types:
	\begin{enumerate}
		\item an $\alpha$-th root of unity which is not a $\beta$-th root of unity, or a $\beta$-th root of unity which is not a $\alpha$-th root of unity.  Such roots have multiplicity $1$;
		\item a $\delta$-th root of unity.  Such roots have multiplicity $2$.
	\end{enumerate}
	The roots of the first kind are also roots of the factor $(q^{\alpha \beta \lambda}-1)$, and those of the second kind are roots of both factors in the nominator. It follows that the fraction indeed belongs to $\kk[q]$.
\end{proof}

The Lie algebra $\f{g}^t$ is equipped with the obvious $M^+$-grading, i.e., $\f{g}^t=\bigoplus_{m\in M^+} \f{g}_m^t$ where
\begin{align*}
	\f{g}^t_m:=\kk_t\cdot \hat{z}^m.
\end{align*}
Here, being $M^+$-graded means that for all $a,b\in M^+$, we have $[\f{g}^t_a,\f{g}^t_b]\subset \f{g}^t_{a+b}$.

Note that $\f{g}^t$ and $\Lambda$ satisfy the following compatibility condition:
\begin{align}\label{skewCondition}
	\mbox{if~}\Lambda(a,b)=0, \mbox{~then~} [\f{g}^t_{a},\f{g}^t_{b}]=0.    
\end{align}

The adjoint action of the commutator algebra of $\kk_t[M]$ induces an action of the Lie subalgebra $\tilde{\f{g}}^t$ on $\kk_t[M]$, and this induces an action of the Lie algebra $\f{g}^t$ on $\kk_t[M]$ by $M$-graded $\kt$-algebra derivations:
\begin{align*}
	\ad_{\hat{z}^a}(z^b) = \hat{z}^a z^b-z^b \hat{z}^a=\frac{t^{\Lambda(a,b)}-t^{-\Lambda(a,b)}}{t^{d|a|}-t^{-d|a|}} z^{a+b} \in \kk_t[M].
\end{align*}

For $k\in \bb{Z}_{\geq 1}$, let $(\f{g}^t)^{\geq k}\subset \f{g}^t$ denote the Lie algebra ideal spanned by the summands $\f{g}^t_m$ with $m\in kM^+$.  We define nilpotent Lie algebras \begin{align}\label{eq:gk}
	\f{g}^t_k\coloneqq \f{g}^t/(\f{g}^t)^{\geq (k+1)}    
\end{align}
and the completion
\begin{align*}
	\hat{\f{g}}^t\coloneqq \varprojlim_k \f{g}^t_k.
\end{align*}

By the Baker--Campbell--Hausdorff formula, we can apply $\exp$ to $\f{g}^t_k$ and $\hat{\f{g}}^t$ to obtain pro-unipotent algebraic groups $G^t_k$ and $\hat{G}^t=\varprojlim_k G_k^t$, respectively.  Note that $\hat{G}^t$ comes with a projection to $G_k^t$ for each $k\in \bb{Z}_{\geq 1}$.  The (adjoint) action of $\f{g}^t$ on $\kk_t[M]$ by $M$-graded $\kt$-derivations induces an action of $\hat{\f{g}}^t$ on $\kk_t\llb M\rrb$ by topologically $M$-graded $\kt$-derivations, hence an (adjoint) action of $\hat{G}^t$ on $\kk_t\llb M\rrb$ by $\kt$-algebra automorphisms: for $g=\exp(b)\in \hat{G}^t$ and $a\in \kk_t\llb M\rrb$, the action is $\Ad_g(a)=gag^{-1}=\exp([b,\cdot])(a)$.

Everything above makes sense in the classical $t=1$ limit as well, thus yielding Lie algebras $\f{g}$, $\f{g}_k$, and $\hat{\f{g}}$, as well as pro-unipotent algebraic groups $G_k$ and $\hat{G}$.  As a $\kk$-module, $\f{g}$ is equal to $\kk[M^+]$.  The bracket on $\f{g}$ is given by
\begin{align}\label{classical-bracket}
	[z^a,z^b]=\Lambda(a,b)z^{a+b},
\end{align}
i.e., $[z^{\omega_1(n)},z^{\omega_1(n')}]=\omega(n,n')z^{\omega_1(n+n')}$.  It is straightforward to check that $t\mapsto 1$ and $\hat{z}^m\mapsto \frac{z^m}{d|m|}$ determines a Lie algebra homomorphism $\lim_{t\rar 1}:\f{g}^t \rar \f{g}$, cf. \cite[\S 2.2.4]{davison2019strong}.  This induces homomorphisms, also denoted $\lim_{t\rar 1}$, on each of the other related pairs of Lie algebras or algebraic groups, e.g., $\hat{\f{g}}^t\rar\hat{\f{g}}$, $\hat{G}^t\rar \hat{G}$, etc. We also have $\lim_{t\rar 1}: \kk_t[M]\rar \kk[M]$ via $t\mapsto 1$, $z^m\mapsto z^m$, and similarly for $\kk_t\llb M\rrb\rar \kk\llb M\rrb$.  This determines actions of $\hat{\f{g}}$ and $\hat{G}$ on $\kk\llb M\rrb$ which intertwine the actions of $\hat{\f{g}}_t$ and $\hat{G}_t$ on $\kk_t\llb M\rrb$.  In particular, $\hat{\f{g}}$ acts on $\kk\llb M\rrb$ via 
\begin{align}\label{eq:classical_Lie_mod}
\ad_{z^{\omega_1(n)}}(z^b)=\Lambda(\omega_1(n),b)z^{\omega_1(n)+b}=-d\langle n,b\rangle z^{\omega_1(n)+b}.
\end{align}

Given $m\in M^+$, let
\begin{align}\label{g-parallel}
	(\f{g}_m^t)^{\parallel}\coloneqq \prod_{k\in \bb{Z}_{>0}} \f{g}_{km}^t\subset \hat{\f{g}}^t,
\end{align}
and let $(G_m^t)^{\parallel} = \exp(\f{g}_m^t)^{\parallel}$.  Note that $(G_m^t)^{\parallel}$ is Abelian.  We similarly define the classical limits $\f{g}_m^{\parallel}$ and $G_m^{\parallel}$.

\subsection{Quantum dilogarithms}\label{qdilog}
Given $v\in M^+$, we shall consider the quantum dilogarithm\footnote{When defining $\f{g}^t$, it is important that we take $\kk$ containing $\bb{Q}$ so that $-\Li(-z^v;t^d)$ lies in $\hat{\f{g}}^t$.  So if $\kk$ does not contain $\bb{Q}$, then in the construction of $\hat{\f{g}}^t$, one should replace $\kk_t$ with $\QQ_t$.  By \eqref{AdPsi}, the relevant action on $\kk_t\llb M\rrb$ will still be well-defined.}
\begin{align}\label{Li}
	-\Li(-z^v;t^d)\coloneqq \sum_{k=1}^{\infty} \frac{(-1)^{k-1}}{k(t^{dk}-t^{-dk})} z^{kv}\in (\f{g}_v^t)^{\parallel} \subset \hat{\f{g}}^t,
\end{align}
as well as the quantum exponential
\begin{align*}
	\Psi_{t^d}(z^v)\coloneqq \exp(-\Li(-z^v;t^d))\in (G_v^t)^{\parallel} \subset \hat{G}^t.
\end{align*}
It is well-known (cf. \cite[Lem. 8]{Kir}) that
\begin{align}\label{Psit}
	\Psi_{t^d}(z^v)=\prod_{k=1}^{\infty} \frac{1}{1+t^{d(2k-1)}z^{v}}.
\end{align}
We note that $(\Psi_{t^d}(z^v))^{\pm 1}$ can also be expressed as plethystic exponentials $\mathbb{E}(\mp z^v)$, see \cite[\S 2.4.2]{davison2019strong}.

Let $p\in M$, and let $\xi\coloneqq \sign \Lambda(v,p)$.  It follows from \eqref{Psit} that the Adjoint action of $(\Psi_{t^d}(z^v))^{\xi}$ on $z^p\in \kk_t[M]$ is given by
\begin{align}\label{AdPsi}
	\Ad_{\Psi_{t^d}(z^v)^{\xi}}(z^p)= z^p \prod_{k=1}^{|\Lambda(v,p)|/d} (1+t^{\xi d(2k-1)}z^{v}).
\end{align}

Using the notation of \eqref{qbinom}, one may show (cf. \cite[\S 2.4.1]{davison2019strong}) that \eqref{AdPsi} can be rewritten as
\begin{align}\label{AdPsiBinom}
	\Ad_{\Psi_{t^d}(z^v)^{\xi}}(z^p) = \sum_{k=0}^{|\Lambda(v,p)|/d} \binom{|\Lambda(v,p)|/d}{k}_{t^d} z^{kv+p}.
\end{align}
Recall that $\lim_{t\rar 1}$ sends $\hat{z}^m\in \f{g}^t$ to $\frac{z^m}{d|m|}\in\f{g}$. Applying $\lim_{t\rar 1}$ yields analogs of all the above in the classical limit.  Let
\begin{align*}
	-\frac{1}{d}\Li(-z^v):=\lim_{t\rar 1}(-\Li(-z^v;t^d)) = \sum_{k=1}^{\infty} \frac{(-1)^{k-1}}{dk^2} z^{kv} \in \f{g}_v^{\parallel} \subset \hat{\f{g}}
\end{align*}
and
\begin{align}\label{eq:classical_dilogarithm}
	\Psi(z^v)^{1/d}:=\lim_{t\rar 1} \Psi_{t^d}(z^v)=\exp\left(-\frac{1}{d}\Li(-z^v)\right)\in G_v^{\parallel}\subset \hat{G}.
\end{align}
In this classical limit, it is clear from either \eqref{AdPsi} or \eqref{AdPsiBinom} that
\begin{align}\label{AdPsi1}
	\Ad_{\Psi(z^v)^{1/d}}(z^p)=z^p(1+z^v)^{\Lambda(v,p)/d}.
\end{align}

We note that the quantum mutations $(\mu_j^{\s{A}})^{-1}$ as in \eqref{mujAInverse} may be expressed as
\begin{align}\label{mu-dilog}
	(\mu_j^{\s{A}})^{-1}(z^p) = \Ad_{\Psi_{t^d}(z^{\omega_1(e_j)})}^{-1}(z^p).
\end{align}
Similarly, in the classical limit we have
\begin{align}\label{eq:mu-class}
	(\mu_j^{\s{A}})^{-1}(z^p)|_{t=1}=\Ad^{-1}_{\Psi(z^{\omega_1(e_j)})^{1/d}}(z^p)=z^p(1+z^{\omega_1(e_j)})^{-\langle e_j,p\rangle}.
\end{align}
To prove \eqref{mu-dilog}, it suffices to check that the two descriptions agree on the cluster variables $A_{\sd_j,i}=z^{e_{\sd_j,i}^*}$.  Note that $\Lambda(\omega_1(e_j),e^*_{\sd_j,_i})=d\delta_{ij}$, so we indeed have $\Ad_{\Psi_{t^d}(z^{\omega_1(e_j)})}^{-1}(z^{e_{\sd_j,i}^*})=z^{e_{\sd_j,i}^*}$ for $i\neq j$, and by applying \eqref{AdPsiBinom}, we see \begin{align*}
	\Ad_{\Psi_{t^d}(z^{\omega_1(e_j)})}^{-1}(z^{e_{\sd_j,j}^*})=z^{e_{\sd_j,j}^*}+z^{e_{\sd_j,j}^*+\omega_1(e_j)}
\end{align*}
as in \eqref{mujAInverse}.

\subsection{Scattering diagrams}

We next review scattering diagrams, which are a key tool in defining the theta functions.  We will work in the quantum setting, but the classical setting is essentially the same, e.g., replacing each $\f{g}^t$ with $\f{g}$. Let $r$ denote the rank of $M$.

Given $v\in M$, let 
\begin{align}\label{eq:Lambda_orthogonal}
	v^{\Lambda\perp}:=\{m\in M_{\bb{R}}|\Lambda(v,m)=0\}.
\end{align}
Recall the notation $\Lambda_1:M_{\bb{R}}\rar N_{\bb{R}}$, $\Lambda_1(v):=\Lambda(v,\cdot)$.  Let $$\ker(\Lambda):=\ker(\Lambda_1)\subset M_{\bb{R}}.$$
A \textbf{wall} in $M_{\bb{R}}$ over $\f{g}^t$ is a pair $(\f{d},f_{\f{d}})$, where
\begin{itemize}
	\item $f_{\f{d}}\in (G_{v_{\f{d}}}^t)^{\parallel}$ for some primitive $v_{\f{d}}\in M^+\setminus \ker(\Lambda)$, and
	\item $\f{d}$ is a closed, convex, $(r-1)$-dimensional, rational-polyhedral cone in $v_{\f{d}}^{\Lambda\perp}\subset M_{\bb{R}}$.
\end{itemize} 
By a closed, convex, rational-polyhedral cone in $M_{\bb{R}}$, we mean an intersection of finitely many sets of the form $\{m\in M_{\bb{R}}|\langle n,m\rangle \geq 0\}$ for various $n \in N$.  We suppress $v_{\f{d}}$ in the notation for the wall since this primitive vector is uniquely determined by $f_{\f{d}}$ as long as $f_{\f{d}}\neq 1$. We call $-v_{\f{d}}$ the \textbf{direction} of the wall.  A wall is called \textbf{incoming} if $v_{\f{d}} \in \f{d}$, and it is called \textbf{outgoing} otherwise.

A \textbf{scattering diagram} $\f{D}$ in $M_{\bb{R}}$ over $\f{g}^t$ 
is a set of walls in $M_{\bb{R}}$ over $\f{g}^t$ such that for each $k >0$, there are only finitely many $(\f{d},f_{\f{d}})\in \f{D}$ with $f_{\f{d}}$ not projecting to $1$ in $G^t_k$.  Given a scattering diagram $\f{D}$ over $\f{g}_t$, for each $k\in \bb{Z}_{>0}$, let $\f{D}_k\subset \f{D}$ denote the finite scattering diagram over $\f{g}$ consisting of the $(\f{d},f_{\f{d}})\in \f{D}$ for which $f_{\f{d}}$ is nontrivial in the projection to $G_k$.  We may also view $\f{D}_k$ as a scattering diagram over $\f{g}_k$ by taking the projections of the elements $f_{\f{d}}$.

\subsubsection{Path-ordered products and consistency}
\label{pop_sec}
For convenience, we will often denote a wall $(\f{d},f_{\f{d}})$ by just $\f{d}$.  Denote $$\Supp(\f{D})\coloneqq  \bigcup_{\f{d}\in \f{D}} \f{d}\subset M_{\bb{R}},$$ and \begin{align}\label{joints}
	\Joints(\f{D})\coloneqq  \left(\bigcup_{\f{d}\in \f{D}} \partial \f{d} \right)\cup \left( \bigcup_{\substack{\f{d}_1,\f{d}_2\in \f{D}\\
			\dim \f{d}_1\cap \f{d}_2 = r-2}} \f{d}_1\cap \f{d}_2 \right)\subset M_{\bb{R}}. 
\end{align}

Consider a smooth immersion $\gamma\colon [0,1]\rar M_{\bb{R}}\setminus \Joints(\f{D})$ with endpoints not in $\Supp(\f{D})$ which is transverse to each wall of $\f{D}$ it crosses.  Let $(\f{d}_i,f_{\f{d}_i})$, $i=1,\ldots, s$, denote the walls of $\f{D}_k$ crossed by $\gamma$, and say they are crossed at times $0<\tau_1\leq \ldots \leq \tau_s<1$, respectively. Define 
\begin{align}\label{WallCross}
	\theta_{\f{d}_i}\coloneqq f_{\f{d}_i}^{\sign \Lambda(-\gamma'(\tau_i),v_{\f{d}_i})} \in G_k^t.
\end{align}
Notice that $\sign \Lambda(-\gamma'(\tau_i),v_{\f{d}_i})=\sign \langle -\gamma'(\tau_i), n_{\f{d}_i}\rangle$ for $v_{\f{d}_i}=\omega_1(n_{\f{d}_i})$ by \eqref{Lambda-B}. 
Let 
\begin{align}\label{WallCross2}
	\theta_{\gamma,\f{D}}^k\coloneqq \theta_{\f{d}_s} \cdots \theta_{\f{d}_1}\in G_k^t,
\end{align} and define the path-ordered product:
\begin{align}\label{pathprod}
	\theta_{\gamma,\f{D}}\coloneqq  \varprojlim_k \theta_{\gamma,\f{D}}^k \in \hat{G}^t.
\end{align}
We note that if $\tau_i=\tau_{i+1}$, then the fact that each $G_v^{\parallel}$ is abelian implies that the ambiguity in the ordering of the walls does not affect $\theta_{\gamma,\f{D}}$.

Two scattering diagrams $\f{D}$ and $\f{D}'$ are \textbf{equivalent} if $\theta_{\gamma,\f{D}} = \theta_{\gamma,\f{D}'}$ for each smooth immersion $\gamma$ as above (assuming transversality with the walls of both $\f{D}$ and $\f{D}'$).  A scattering diagram $\f{D}$ is \textbf{consistent} if each $\theta_{\gamma,\f{D}}$ depends only on the endpoints of $\gamma$, or equivalently, if $\theta_{\gamma,\f{D}}=\id$ whenever $\gamma$ is a closed path.

The following fundamental result on scattering diagrams is due to \cite{KS} in the rank $2$ classical setting, \cite{GS11} in the arbitrary rank classical setting, and \cite{WCS} in the general setting.  See \cite[Thm. 2.13]{davison2019strong} for the version used here or \cite[Thm. 1.12]{gross2018canonical} for the classical analog.

\begin{thm}\label{KSGS}
	Let $\f{D}_{\In}$ be a finite scattering diagram in $M_{\bb{R}}$ over $\f{g}^t$ (or $\f{g}$) whose only walls are incoming.  Then there is a unique-up-to-equivalence scattering diagram $\f{D}$, also denoted $\scat(\f{D}_{\In})$, such that $\f{D}$ is consistent, $\f{D} \supset \f{D}_{\In}$, and $\f{D}\setminus \f{D}_{\In}$ consists only of outgoing walls.
\end{thm}

Given a compatible pair $(\sd,\Lambda)$, let 
\begin{align}\label{DIn}
	\f{D}_{\In}^{\s{A}_t} = \{(e_i^{\perp},\Psi_{t^d}(z^{\omega_1(e_i)})) |i \in I\setminus F\}
\end{align}
for $d$ as in \eqref{Lambda-B}.  We will work with the associated consistent scattering diagram as in Theorem \ref{KSGS}: \begin{equation}\label{DAt}
	\f{D}^{\s{A}_t}:=\scat(\f{D}_{\In}^{\s{A}_t}).
\end{equation}

\begin{example}\label{A2DatEx}
	Recall the compatible pair $(\sd,\Lambda)$ associated to the $A_2$-quiver $1\leftarrow 2$ as in Examples \ref{A2ex} and \ref{A2ex2}.  That is, in terms of the basis $E$ and the dual basis $E^*:=\{e_1^*,e_2^*\}$, we have
	\begin{align*}
		\omega =\left(\begin{matrix}
			0 & 1 \\
			-1 & 0 
		\end{matrix}\right) \quad \mbox{and} \quad \Lambda= \left(\begin{matrix}
			0 & d \\
			-d & 0 
		\end{matrix}\right).
	\end{align*}
	The associated scattering diagram $\f{D}^{\s{A}_t}$ is illustrated in Figure \ref{a2q} using the basis $E^*$.  Note that one outgoing wall (the diagonal wall) has been added to $\f{D}^{\s{A}_t}_{\In}$ to make the scattering diagram consistent.
\end{example}
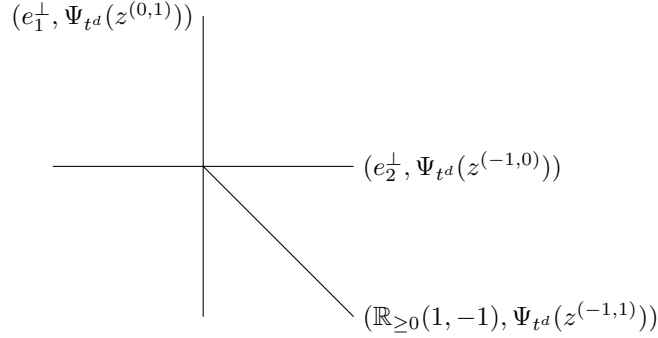
\begin{figure}[htb]
	\centering
	\begin{tikzpicture}
		\draw
		(-2,0) -- (2,0) node[right] {$(e_2^{\perp},\Psi_{t^d}(z^{(-1,0)}))$}
		(0,-2) -- (0,2) node[left] {$(e_1^{\perp},\Psi_{t^d}(z^{(0,1)}))$}
		(0,0) -- (2,-2) node[right] {$(\bb{R}_{\geq 0}(1,-1),\Psi_{t^d}(z^{(-1,1)}))$};
	\end{tikzpicture}
	\caption{The scattering diagram $\f{D}^{\s{A}_t}$ associated to the $A_2$-quiver.
		\label{a2q}}
\end{figure}

\begin{example}\label{DatEx}
	Recall the compatible pair $(\sd,\Lambda)$ from Examples \ref{QuivEx}  and \ref{AnComp}, i.e., the compatible pair associated to an annulus as in Example \ref{AnnEx}.  That is, we have
	\begin{align*}
		B =\left(\begin{matrix}
			0 & -2 & 1 & 1 \\
			2 & 0 & -1 & -1 \\
			-1 & 1 & 0 & 0 \\
			-1 & 1 & 0 & 0
		\end{matrix}\right) \quad \mbox{and} \quad \Lambda= \left(\begin{matrix}
			0 & 2 & 0 & 0 \\
			-2 & 0 & 0 & 0 \\
			0 & 0 & 0 & 0 \\
			0 & 0 & 0 & 0
		\end{matrix}\right).
	\end{align*}
	Also, $I\setminus F = \{1,2\}$ and $d=4$.  In terms of the dual basis to $E$, we have
	\begin{align}\label{v1v2}
		v_1:=\omega_1(e_1)=(0,2,-1,-1) \quad \mbox{and} \quad v_2:=\omega_1(e_2)=(-2,0,1,1).
	\end{align}
	
	The associated scattering diagram $\f{D}^{\s{A}_t}$, intersected with $M_{\uf,\bb{R}}=\bb{R}\langle v_1,v_2\rangle$, is illustrated in Figure \ref{Kr2q}.  Note that infinitely many outgoing wall have been added to $\f{D}^{\s{A}_t}_{\In}$ to make the scattering diagram consistent.  Specifically, there are new outgoing walls of the form
	\begin{align*}
		((e_1^{\perp}\cap e_2^{\perp})+\bb{R}_{\geq 0}u,\Psi_{t^4}(z^{-u}))
	\end{align*}
	for each $u$ of the form $(k,-(k+1))$ or $(k+1,-k)$ with $k\in \bb{Z}_{\geq 1}$.  There is also one additional outgoing wall:
	\begin{align*}
		((e_1^{\perp}\cap e_2^{\perp})+\bb{R}_{\geq 0}(1,-1),\EE(-(t^4+t^{-4})z^{(-1,1)}))
	\end{align*}
	where $\EE(-(t^4+t^{-4})z^{(-1,1)})$ is a ``plethystic exponential''  (see \cite[\S 2.4.2]{davison2019strong} for some background on plethystic exponentials, but beware that since we have $d=4$, our $t^4$ should be treated as $t^1$ of loc. cit.).  For example, as in \cite[Eq. 30]{davison2019strong}, the plethystic exponential here may be given in terms of quantum exponentials by
	\begin{align*}
		\EE(-(t^4+t^{-4})z^{(-1,1)}) = \prod_{r=0}^{\infty} \Psi_{(t^4)^{2^r}}((t^4)^{2^r}z^{2^r(-1,1)})\Psi_{(t^{-4})^{2^r}}((t^{-4})^{2^r}z^{2^r(-1,1)}).
	\end{align*}
	See \cite[Ex. 7.10]{davison2019strong} for more on this and higher Kronecker quivers.
\end{example}
\begin{figure}[htb]
	\centering
	\begin{tikzpicture}
		\draw
		(-2,0) -- (2,0) node[right] {$(e_2^{\perp},\Psi_{t^4}(z^{v_2}))$}
		(0,-2) -- (0,2) node[left] {$(e_1^{\perp},\Psi_{t^4}(z^{v_1}))$}
		(0,0) -- (2,-2) 
		(0,0) -- (1,-2) 
		(0,0) -- (1.333,-2)
		(0,0) -- (1.5,-2)
		(0,0) -- (1.6,-2)
		(0,0) -- (1.667,-2)
		(0,0) -- (1.714,-2)
		(0,0) -- (1.75,-2)
		(0,0) -- (1.778,-2)
		(0,0) -- (1.8,-2)
		(0,0) -- (1.818,-2)
		(0,0) -- (1.833,-2)
		(0,0) -- (1.846,-2)
		(0,0) -- (1.857,-2)
		(0,0) -- (1.867,-2)
		(0,0) -- (1.875,-2)
		(0,0) -- (1.882,-2)
		(0,0) -- (1.889,-2)
		(0,0) -- (1.895,-2)
		(0,0) -- (1.9,-2)
		(0,0) -- (1.9048,-2)
		(0,0) -- (1.9091,-2)
		(0,0) -- (1.913,-2)
		(0,0) -- (2,-1) 
		(0,0) -- (2,-1.333)
		(0,0) -- (2,-1.5)
		(0,0) -- (2,-1.6)
		(0,0) -- (2,-1.667)
		(0,0) -- (2,-1.714)
		(0,0) -- (2,-1.75)
		(0,0) -- (2,-1.778)
		(0,0) -- (2,-1.8)
		(0,0) -- (2,-1.818)
		(0,0) -- (2,-1.833)
		(0,0) -- (2,-1.846)
		(0,0) -- (2,-1.857)
		(0,0) -- (2,-1.867)
		(0,0) -- (2,-1.875)
		(0,0) -- (2,-1.882)
		(0,0) -- (2,-1.889)
		(0,0) -- (2,-1.895)
		(0,0) -- (2,-1.9)
		(0,0) -- (2,-1.9048)
		(0,0) -- (2,-1.9091)
		(0,0) -- (2,-1.913);
	\end{tikzpicture}
	\caption{The scattering diagram $\f{D}^{\s{A}_t}$ intersected with $M_{\uf,\bb{R}}$, oriented so that $v_1$ points up and $v_2$ points to the left.
		\label{Kr2q}}
\end{figure}
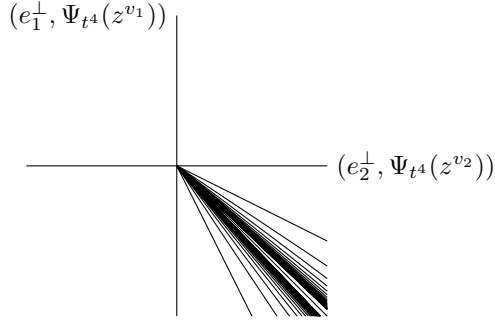

\subsection{Broken lines and theta functions}\label{sec:broken_lines}

Fix a consistent scattering diagram $\f{D}$ in $M_{\bb{R}}$ over $\f{g}^t$. Let $p\in M$, and let $\sQ$ be a generic\footnote{By ``generic'' we mean that $\sQ$ is not contained in $n^{\perp}$ for any $n\in N$.\label{foot-gen}} point in $M_{\bb{R}}$.  A \textbf{broken line} $\gamma$ with respect to $\f{D}$ and with ends $(p,\sQ)$ is the data of a continuous map $\gamma\colon (-\infty,0]\rar
M_{\bb{R}}\setminus \Joints(\f{D})$, values $-\infty < \tau_0 \leq \tau_1 \leq \ldots \leq \tau_{\ell} = 0$, and for each $i=0,\ldots,\ell$, an associated monomial $c_iz^{v_i} \in  \kk_t[M]$, with $c_i\in \kk[t^{\pm 1}]$ and $v_i\in M$, such that:

\begin{enumerate}[label=(\roman*), noitemsep]
	\item $\gamma(0)=\sQ$.
	\item For $i=1\ldots, \ell$, $\gamma'(\tau)=-v_i$ for all $\tau\in (\tau_{i-1},\tau_{i})$.  Similarly, $\gamma'(\tau)=-v_0$ for all $\tau\in (-\infty,\tau_0)$. 
	\item $c_0=1$ and $v_0=p$.
	\item For $i=0,\ldots,\ell-1$, $\gamma(\tau_i)\in \Supp(\f{D})$.  Let 
	\begin{align}\label{fi}
		f_i\coloneqq \prod_{\substack{(\f{d},f_{\f{d}})\in \f{D} \\ \f{d}\ni \gamma(\tau_i)}} f_{\f{d}}^{\sign\Lambda(v_{\f{d}},v_i)} \in \hat{G}^t.
	\end{align}
	I.e., $f_i$ is the $\epsilon\rar 0$ limit of the element $\theta_{\gamma|_{(\tau_i-\epsilon,\tau_i+\epsilon)},\f{D}}$ defined in \eqref{pathprod} (using a smoothing of $\gamma$). Then $c_{i+1}z^{v_{i+1}}$ is a homogeneous (with respect to the topological $M$-grading) term of $\Ad_{f_i}(c_iz^{v_i})$ other than the leading term $c_iz^{v_i}$ (i.e., $\gamma$ bends non-trivially at the time $\tau_i$).
\end{enumerate}

For each $p\in M$ and each generic $\sQ\in M_{\bb{R}}$, one defines an element $\vartheta_{p,\sQ}\in \kk_t\llb M\rrb$ via:
\begin{align}\label{vartheta-dfn}
	\vartheta_{p,\sQ}\coloneqq \sum_{\Ends(\gamma)=(p,\sQ)} c_{\gamma}z^{v_{\gamma}} \in \kk_t\llb M\rrb.
\end{align}
Here, the sum is over all broken lines $\gamma$ with ends $(p,\sQ)$, and $c_{\gamma}z^{v_{\gamma}}$ denotes the monomial attached to the final straight segment of $\gamma$, called the (final) Laurent monomial of $\gamma$. We call $v_{\gamma}$ the final exponent of $\gamma$.  In particular, we always have $$\vartheta_{0,\sQ}=1.$$

Broken lines and theta functions in the $t=1$ setting are defined in the analogous way.

\begin{example}\label{A2-Broken}
	Consider $\f{D}^{\s{A}_t}$ as in Example \ref{A2DatEx}.  Figure \ref{a2broken} shows the three broken lines contributing to $\vartheta_{p,\sQ}$ for $p=-v_1$ and $\sQ$ as pictured.  One computes that
	$\vartheta_{-v_1,\sQ}= z^{v_2}+z^{v_2-v_1}+z^{-v_1}$; the three terms here correspond to the three broken lines, going from left to right.
\end{example}

\begin{figure}[htb]
	\centering
	\begin{tikzpicture}
		\draw
		(-2,0) -- (2,0) node[right] {$\Psi_{t^d}(z^{v_2})$}
		(0,-2) -- (0,2) node[left] {$\Psi_{t^d}(z^{v_1})$}
		(0,0) -- (2,-2) node[right] {$\Psi_{t^d}(z^{v_1+v_2})$};
		\draw (1.06,1.5) node 
		{$\bullet$ $\sQ$};
		\draw[line width=.4mm] (-1.5,-2) -- (-1.5,0) -- (0,1.5);
		\draw[line width=.4mm] (0,1.5) -- (.85,1.5);
		\draw[line width=.4mm] (-.75,-2) -- (-.75,0) -- (.85,1.5);
		\draw[line width=.4mm] (.85,-2) -- (.85,1.5);
	\end{tikzpicture}
	\caption{The three broken lines with ends $(-v_1,\sQ)$ for the scattering diagram $\f{D}^{\s{A}_t}$ from Figure \ref{a2q}.
		\label{a2broken}}
\end{figure}
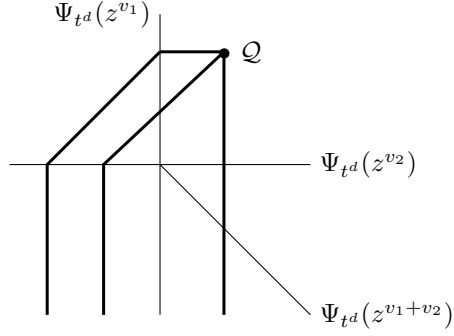

The following tells us that the theta functions are well-defined and form a topological basis in the sense of \cite[\S 2.2.2]{davison2019strong}.
\begin{prop}[\cite{davison2019strong}, Prop. 3.1]\label{TopBasis}
	For fixed generic $\sQ \in M_{\bb{R}}$ and any $p\in M$, \eqref{vartheta-dfn} gives a well-defined element $\vartheta_{p,\sQ} \in \kk_t\llb M\rrb$ of the form 
	\begin{align}\label{pointed}
		\vartheta_{p,\sQ}=z^p\left(1+\sum_{v\in M^+} a_vz^v\right).
	\end{align}
	Furthermore, the theta functions $\Theta_{\sQ}\coloneqq \{\vartheta_{p,\sQ}\colon p\in M\}$ form a topological $\kt$-module basis
	for $\s{A}_{\sQ}\coloneqq \kk_t\llb M \rrb$.  Hence, $\Theta_{\sQ}$ also forms a topological basis for the subalgebra $\s{A}^{\Theta}_{\sQ}\subset \s{A}_{\sQ}$ generated by $\Theta_{\sQ}$.
\end{prop}

The theta functions satisfy the following compatibility condition, due to \cite{carl2022tropical} in the classical limit and \cite[Thm. 2.14]{mandel2021scattering} in the quantum setup.  It says that $\vartheta_{p,\sQ}$ for different values of $\sQ$ are related by the action of path-ordered products.

\begin{lem}\label{CPS}
	Let $\f{D}$ be a consistent scattering diagram in $M_{\bb{R}}$ over $\f{g}^t$ or $\f{g}$.  Fix two generic points $\sQ_1,\sQ_2\in M_{\bb{R}}$.  Let $\gamma$ be a smooth path in $M_{\bb{R}}\setminus \Joints(\f{D})$ from $\sQ_1$ to $\sQ_2$, transverse to each wall of $\f{D}$ it crosses.  Then for any $p\in M$,
	\begin{align*}
		\vartheta_{p,\sQ_2}=\Ad_{\theta_{\gamma,\f{D}}} (\vartheta_{p,\sQ_1}).
	\end{align*}
\end{lem}

In particular, if $\sQ_1,\sQ_2$ lie in the same chamber of $\f{D}$, then $\vartheta_{p,\sQ_1}=\vartheta_{p,\sQ_2}$.

\begin{example}\label{Ann-broken}
	Consider $\f{D}^{\s{A}_t}$ as in Example \ref{DatEx}. Up to any finite order $\ell$, there are two chambers of $\f{D}^{\s{A}_t}_{\ell}$ which contain the limiting wall, i.e., the diagonal wall whose direction is $-\omega_1(e_1+e_2)=2e_1^*-2e_2^*$. Let $p$ denote the corresponding primitive vector $\frac{1}{2}(-v_1-v_2)=e_1^*-e_2^*.$
	
	Figure \ref{Kr2q-broken} shows a broken line contributing $z^{-kp}$ to $\vartheta_{kp,\sQ}$ for $\sQ$ near the limiting wall and $k\in \bb{Z}_{\geq 1}$.  Up to any finite order $\ell$, for $\sQ$ sufficiently close to the limiting wall, the only two broken lines contributing to $\vartheta_{kp,\sQ}$ will be the straight broken line, which contributes $z^{kp}$, and a broken line which wraps around the joint as in the figure, contributing $z^{-kp}$ --- indeed, the broken line contributing $z^{-kp}$ always bends as much as possible towards the joint (or not at all on walls where the only allowed bends are away from the joint), so we see that any broken line with initial monomial $z^{kp}$ which crosses some walls of $\f{D}^{\s{A}_t}_{\ell}$ but has different bends will not be able to wrap all the way around the origin to end at $\sQ$. Thus, up to any fixed finite order $\ell$ and for $\sQ$ sufficiently close to the limiting wall, we have
	\begin{align}\label{ThetaKr}
		\vartheta_{kp,\sQ} \equiv z^{kp}+z^{-kp} \quad \mbox{(in the module $z^{kp}\cdot \kk[M/(\ell M^+)]$.)}
	\end{align}
	
	Consider the Chebyshev polynomials of the first kind $T_k$ as in \eqref{Cheb1}, and the characterization of $T_k$ in \eqref{Tkex}.  It follows from \eqref{ThetaKr} that 
	\begin{align*}
		\vartheta_{kp,\sQ} \equiv T_k(\vartheta_{p,\sQ}) \quad \mbox{(in the module-projection $z^{kp}\kk[M/(\ell M^+)]$)}
	\end{align*} 
	for $\sQ$ sufficiently close to the limiting wall. By Proposition \ref{TopBasis}, we have an expansion $$\vartheta_{kp,\sQ}-T_k(\vartheta_{p,\sQ})=\sum_{m\in kp+\ell M^+} c_m\vartheta_{m,\sQ},$$ and by Lemma \ref{CPS}, this holds for all $\sQ$.  Since $\ell$ is arbitrary, we must in fact have
	\begin{align}\label{ChebAnn} 
		\vartheta_{kp,\sQ} = T_k(\vartheta_{p,\sQ}).
	\end{align} 
\end{example}

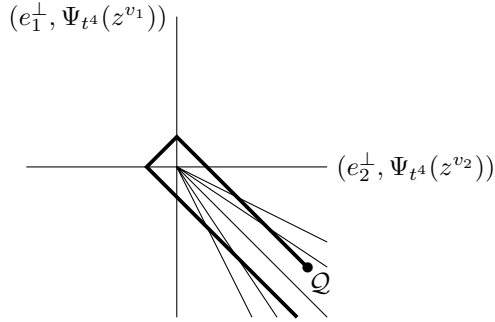
\begin{figure}[htb]
	\centering
	\begin{tikzpicture}
		\draw
		(-2,0) -- (2,0) node[right] {$(e_2^{\perp},\Psi_{t^4}(z^{v_2}))$}
		(0,-2) -- (0,2) node[left] {$(e_1^{\perp},\Psi_{t^4}(z^{v_1}))$}
		(0,0) -- (2,-2) 
		(0,0) -- (1,-2) 
		(0,0) -- (1.333,-2)
		(0,0) -- (2,-1) 
		(0,0) -- (2,-1.333);
		\draw (1.74,-1.35) node 
		{$\bullet$};
		\draw (1.9,-1.55) node 
		{$\sQ$};
		\draw[line width=.5mm] (1.6,-2) -- (-.4,0) -- (0,.4) -- (1.75,-1.35);
	\end{tikzpicture}
	\caption{A broken line (bold) with ends $(kp,\sQ)$ contributing a final Laurent monomial $z^{-kp}$.
		\label{Kr2q-broken}}
\end{figure}

\subsection{Structure constants}\label{StrC}

For any generic $\sQ_1,\sQ_2\in M_{\bb{R}}$, Lemma \ref{CPS} gives us a canonical isomorphism $\s{A}^{\Theta}_{\sQ_1}\risom \s{A}^{\Theta}_{\sQ_2}$ taking $\vartheta_{p,\sQ_1}$ to $\vartheta_{p,\sQ_2}$ for each $p$ (see Proposition \ref{TopBasis} for the definition of $\s{A}^{\Theta}_{\sQ}$).  It is therefore reasonable to define an algebra $\s{A}_t^{\can}:=\s{A}^{\Theta}_{\sQ}$ (the \textbf{canonical cluster algebra}) for an arbitrary fixed choice of generic $\sQ$, denoting the corresponding theta functions $\vartheta_{p,\sQ}$ simply as $\vartheta_p$.

Equivalently, consider the $\kk_t$-module $\hat{\s{A}}^{\can}_t=\hat{\bigoplus}_{p\in M} \kk_t \cdot \vartheta_{p}$, where the hat indicates that we take a completion, allowing all formal sums $\sum_{p\in M} a_p\vartheta_p$ such that, for each $k\in \bb{Z}_{>0}$, only finitely many $a_p$ with $p\in M\setminus kM^+$ are nonzero.  As modules, $\hat{\s{A}}_t^{\can}\cong \s{A}_{\sQ}$ for each generic $\sQ$ via identifying $\vartheta_p$ with $\vartheta_{p,\sQ}$.  We use this identification to give $\hat{\s{A}}_t^{\can}$ the structure of an associative algebra, and then $\s{A}_t^{\can}$ is the sub $\kk_t$-algebra generated by the theta functions $\vartheta_p$, $p\in M$.

Alternatively, the algebra structure on $\hat{\s{A}}_t^{\can}$ can be defined by specifying the structure constants.  As in \S \ref{PosSub}, given $p_1,\ldots,p_s,p\in M$, the structure constant $\alpha(p_1,\ldots,p_s;p)$ is the element of $\kk_t$ such that
\begin{align*}
	\vartheta_{p_1}\cdots \vartheta_{p_s} = \sum_{p\in (\sum_{i=1}^s p_i)+M^{\oplus}} \alpha(p_1,\ldots,p_s;p) \vartheta_{p}.
\end{align*}
The following proposition --- \cite[Prop. 2.15]{mandel2021scattering} in the quantum setting or \cite[Prop 6.4(3)]{gross2018canonical} for the classical version --- expresses the structure constants in terms of broken lines:
\begin{prop}\label{StructureConstants}
	For any $p_1,\ldots,p_s,p\in M$, the corresponding structure constant is given by 
	\begin{align}\label{alpha}
		\alpha(p_1,\ldots,p_s;p)z^p= \sum_{\substack{\gamma_1,\ldots,\gamma_s \\
				\Ends(\gamma_i)=(p_i,\sQ), \medspace i=1,\ldots,s \\
				v_{\gamma_1}+\ldots+v_{\gamma_s} = p \\
		}}
		(c_{\gamma_1}z^{v_{\gamma_1}}) \cdots (c_{\gamma_s}z^{v_{\gamma_s}}) \in \kt\cdot z^p,
	\end{align}
	where $v_{\gamma_i}$ denotes the exponent attached to the last straight segment of $\gamma_i$, and $\sQ$ is any generic point of $M_{\bb{R}}$ which is \textbf{sufficiently close} to $p$ --- by sufficiently close, we mean that $\sQ$ and $p$ share a chamber of the finite scattering diagram $\f{D}_{\ell}$ for some sufficiently large $\ell\in \bb{Z}_{\geq 1}$ (an explicit condition for what is considered ``sufficiently large'' is given in Lemma \ref{lem:fQ}).
\end{prop}

    \begin{cor}\label{cor:alpha-sum}
    In the notation of Proposition \ref{StructureConstants}, if $p=p_1+\ldots+p_s$, then $\alpha(p_1,\ldots,p_s;p)=t^k$ for $k=\sum_{i<j} \Lambda(p_i,p_j)$.
    \end{cor}
    \begin{proof}
        For any $\sQ$, the only $s$-tuple of broken lines contributing to \eqref{alpha} in this case is the one for which each broken line is straight---indeed, any bends would result in $\sum v_{\gamma_i}$ lying in $p+M^+$.  Thus, \eqref{alpha} becomes $\alpha(p_1,\ldots,p_s;p)z^p=z^{p_1}\cdots z^{p_s}$, the right-hand side of which is equal to $t^kz^p$ for $k$ as in the statement of the proposition.
    \end{proof}

Proposition \ref{StructureConstants} is actually an immediate consequence of the following lemma (whose proof is essentially contained in the proof of \cite[Prop. 6.4(3)]{gross2018canonical}):
\begin{lem}\label{lem:fQ}
	Let $f=\sum_{p\in M} \alpha_{f,p}\vartheta_p \in \s{A}_t^{\can}$.  For each generic $\sQ\in M_{\bb{R}}$, let $\iota_{\sQ}$ denote the identification $\s{A}_t^{\can}\risom \s{A}^{\Theta}_{\sQ}\subset\kk_t\llb M\rrb$.  Let $a_{f,\sQ,m}$ denote the coefficients of the Laurent expansion of $\iota_{\sQ}(f)$, i.e.,
	\begin{align*}
		\iota_{\sQ}(f)=\sum_{m\in M} a_{f,\sQ,m} z^m.
	\end{align*}
	Then, for any $p$, we have $\alpha_{f,p}=a_{f,\sQ,p}$ for all generic $\sQ$ sufficiently close to $p$ --- by sufficiently close, we mean that $\sQ$ and $p$ share a chamber of $\f{D}_{\ell}$ for some $\ell$ which is large enough to ensure that $p\notin p'+(\ell+1) M^+$ whenever $\alpha_{f,p'}\neq 0$.
\end{lem}

\begin{proof}
	Let
	\begin{equation*}
		M_{f}:=\{p\in M|\alpha_{f,p}\neq 0\}.
	\end{equation*}
	For any $p\in M$, the convergence of $f$ with respect to the topological structure on $\kk_t\llb M\rrb$ ensures that
	\begin{align}\label{eq:large_l}
		p\notin M_f+(\ell+1)M^+
	\end{align} 
	for sufficiently large $\ell$ (cf. \cite[\S 2.2.2]{davison2019strong}). Indeed, let $\supp_M \iota_{\s{Q}}(f)$ denote the set of Laurent degrees appearing in $\iota_{\sQ}(f)$. Let $\max (\supp_M \iota_{\s{Q}}(f)) $ and $\max M_f$ denote the subsets of $\prec_\sd$-maximal elements (for $\prec_\sd$ as in Definition \ref{def:dom}) in $\supp_M \iota_{\s{Q}}(f)$ and $M_f$ respectively. Note that $\max (\supp_M \iota_{\s{Q}}(f))$ and $\max M_f$ are finite by the convergence assumption. Moreover, they coincide, see \cite[Def.--Lem. 4.1.1]{qin2019bases}. By \cite[Lem. 3.1.2]{qin2017triangular}, we can always choose sufficiently large $\ell$ such that 	\begin{align}\label{eq:large_l_from_max_deg}
		p\notin \max M_f+(\ell+1)M^+=\max (\supp_M \iota_{\s{Q}}(f))+(\ell+1)M^+.
	\end{align} 
 Note that \eqref{eq:large_l} and \eqref{eq:large_l_from_max_deg} are just re-phrasings of the notion of ``sufficiently large'' from the statement of the lemma.

Fix $p\in M$ and sufficiently large $\ell$ as above.  Now for any generic $\sQ$ and any $p'\in M_f$,  \eqref{eq:large_l} ensures that any broken line with ends $(p',\sQ)$ and final exponent $p$ cannot bend on any walls in $\f{D}\setminus \f{D}_{\ell}$.  It follows that the coefficient of $z^p$ in $\vartheta_{p',\sQ}$ must be the contribution from the broken lines defined with respect to $\f{D}_{\ell}$. In addition, if we choose $\sQ$ sharing a chamber of $\f{D}_{\ell}$ with $p$, then the only broken line with respect to $\f{D}_{\ell}$ ending at $\sQ$ and with the final exponent $p$ is the straight broken line with ends $(p,\sQ)$. Therefore, among all $\vartheta_{p',\sQ}$ appearing in the theta basis expansion of $f$, only $\vartheta_{p,\sQ}$ has a nonzero $z^p$-coefficient, and the $z^p$-coefficient of $\vartheta_{p,\sQ}$ is $1$.  The equality $\alpha_{f,p}=a_{f,\sQ,p}$ follows.
\end{proof}

The following result will be used in the proof of Lemma \ref{ChebyLem}.
\begin{lem}\label{lem:const_coeff}
Suppose that $\iota_{\sQ}(\vartheta_p)=z^p$ whenever $\sQ$ is sufficiently close to $p$. Then, for any $p'\neq 0$, the coefficient $\alpha (p,p';p)$ vanishes.
\end{lem}
\begin{proof}
By the assumption, for $\sQ$ sufficiently close to $p$, the Laurent expansion $\iota_{\sQ}(\vartheta_{p} \vartheta_{p'})$ equals $z^p\iota_{\sQ}(\vartheta_{p'})$.  By the broken line construction, the Laurent expansion $\iota_{\sQ}(\vartheta_{p'})$ never has a constant term (for any generic $\sQ$ and $p'\neq 0$), so the product $z^p\iota_{\sQ}(\vartheta_p)$ has no $z^p$-term.  The claim now follows from Lemma \ref{lem:fQ}.
\end{proof}

The following important property of broken lines is an immediate consequence of \cite[Thm. 1.13]{gross2018canonical} in the classical setting and \cite[Thm. 2.15]{davison2019strong} in the quantum setting.
\begin{lem}[Positivity of broken lines]\label{BLpos}
Up to equivalence of the scattering diagram, all broken lines in $M_{\bb{R}}$ with respect to $\f{D}^{\s{A}}$ or $\f{D}^{\s{A}_t}$ have positive integer coefficients (i.e., coefficients in $\bb{Z}_{>0}$ or $\bb{Z}_{>0}[t^{\pm 1}]$, respectively) for every attached monomial.
\end{lem}

\subsection{The cluster complex}

Recall that a compatible pair $(\sd,\Lambda)$ determines a scattering diagram $\f{D}^{\s{A}_t}:=\scat(\f{D}_{\In}^{\s{A}_t})$ where $\f{D}_{\In}^{\s{A}_t}$ is given as in \eqref{DIn} by
\begin{align}\label{DAin}
\f{D}_{\In}^{\s{A}_t} = \{(e_i^{\perp},\Psi_{t^d}(z^{\omega_1(e_i)})) |i \in I\setminus F\}
\end{align}
for $d$ as in \eqref{Lambda-B}.  Also recall the cone 
\begin{align}\label{CS+}
C_{\jj}^+:=\{m \in M_{\bb{R}}|\langle e_{\sd_{\jj},i},m\rangle \geq 0 \mbox{ for all } i\in I\setminus F\} \subset M_{\bb{R}}
\end{align}
as in \eqref{Cjplus}. We may also denote $C_{\jj}^+$ by $C_{\sd_{\jj}}^+$.

The following proposition is a simplified version of \cite[Prop. 4.9]{davison2019strong}.  The classical version is essentially \cite[Constr. 1.30, Thm. 4.4]{gross2018canonical}.
\begin{prop}[Chamber structure of $\f{D}^{\s{A}_t}$]\label{Chambers}
Two cones $C_{\jj}^+$ and $C_{\kkk}^+$ as above coincide if and only if they correspond to the same cluster of $\s{A}_t^{\up}$.  The distinct cones $C_{\jj}^+=:\s{C}_{\sd_{\jj}}$ form the chambers (i.e., the top-dimensional cones) of a fan $\s{C}$ in $M_{\bb{R}}$, called the \textbf{cluster complex}, such that
\begin{itemize}
	\item $\s{C}$ is a sub cone-complex of the cone-complex induced by the supports of the walls of $\f{D}^{\s{A}_t}$.
	\item Each chamber of $\s{C}$ has exactly $\#(I\setminus F)$ facets.  For two seeds $\sd_{\jj}$ and $\sd_{\kkk}$ mutation equivalent to $\sd$, the corresponding chambers $\s{C}_{\sd_{\jj}}$ and $\s{C}_{\sd_{\kkk}}$ in $\s{C}$ intersect along a facet $\f{d}$ if and 
	only if the clusters $\s{A}_t^{\jj}$ and $\s{A}_t^{\kkk}$ are related by a mutation $\mu^{\s{A}}_j$ for some $j\in I\setminus F$, 
	i.e., if $\mu_{\sd_{\jj},j}^{\s{A}}(\s{A}_t^{\jj} )=\s{A}_t^{\kkk}$.  Furthermore, this mutation $\mu_{\sd_{\jj},j}^{\s{A}}$ agrees with the wall-crossing automorphism associated to crossing from $\s{C}_{\sd_{\jj}}$ to $\s{C}_{\sd_{\kkk}}$. 
 \end{itemize}
\end{prop}

The following is essentially \cite[Cor. 4.13]{davison2019strong}, or \cite[Prop. 7.1]{gross2018canonical} in the classical setting.
\begin{cor}\label{UpTheta}
Consider $p\in M$ and generic $\sQ,\sQ_{\jj}$ in $\s{C}_{\sd}$ and $\s{C}_{\sd_{\jj}}$, respectively.  Suppose $\vartheta_{p,\sQ}\in \s{A}_t^{\sd}=\kk_t[M]$, i.e., $\vartheta_{p,\sQ}$ is a Laurent polynomial, not just a formal Laurent series.  Then $\vartheta_{p,\sQ_{\jj}}$ is also a Laurent polynomial.  Thus, $\vartheta_{p}$ determines an element of $\s{A}_t^{\up}$.  Furthermore, if $p\in \s{C}_{\sd_{\jj}}\cap M$, then $\vartheta_{p,\sQ_{\jj}}=z^p$, and so $\vartheta_{p}\in \s{A}_t^{\up}$ is a quantum cluster monomial.  Conversely, all quantum cluster monomials are quantum theta functions.  In particular, $\s{A}_t^{\ord}\subset \s{A}_t^{\can} \cap \s{A}_t^{\up}$.
\end{cor}

The theta functions $\vartheta_p$ which lie in $\s{A}_t^{\up}$ as in Corollary \ref{UpTheta} generate and form a $\kk_t$-module basis for what one calls the \textbf{middle cluster algebra} $\s{A}_t^{\midd}$.  The $g$-vectors\footnote{We refer to $p\in M$ as the \textbf{$g$-vector} of $\vartheta_p$.  If $\vartheta_p$ is a cluster variable $A_{\jj,i}$, then, under the identification $M=\bb{Z}^I$, $p$ agrees with the $i$-th extended $g$-vector of $\sd_{\jj}$ with respect to $\sd$ as in Definition \ref{def:g-vec}.} of these theta function are the integer points $$\Theta^{\midd}=\Theta^{\midd}_{\bb{R}}\cap M$$ of some convex cone $\Theta^{\midd}_{\bb{R}}$, cf. \cite[Thm. 0.3(2-4)]{gross2018canonical} and \cite[Thm. 1.2(2)]{davison2019strong}.  One says that the \textbf{full Fock-Goncharov conjecture} holds (for $\s{A}_t$) if 
\begin{align}\label{fFG-eq}
\s{A}_t^{\midd}=\s{A}_t^{\up}=\s{A}_t^{\can}
\end{align} (so in particular, $\Theta^{\midd}=M$).  Similarly in the classical setting if the Injectivity Assumption holds.

Similarly, one says that the full Fock-Goncharov conjecture holds for $\s{A}_t^{\prin}$ if it holds for the (quantum) cluster algebra associated to $\sd^{\prin}$.  A similar definition applying to $\s{X}_t$ will be given in \S \ref{XupTheta}. The following is part of \cite[Prop. 0.14]{gross2018canonical} in the classical setting, and then the quantum analog follows from \cite[Thm. 1.2(5)]{davison2019strong}.

\begin{prop}\label{FGc}
If the cluster complex $\s{C}$ is \textbf{big}, meaning that it is not contained in a closed half-space, then the full Fock-Goncharov conjecture holds for $\s{A}_t^{\prin}$ and $\s{X}_t$.  If the Injectivity Assumption holds in addition to $\s{C}$ not being contained in a half-space, then the full Fock-Goncharov conjecture holds for $\s{A}_t$ as well.
\end{prop}

Note that the cluster complex $\s{C}$ can be defined even without knowledge of scattering diagrams, and in fact without making the Injectivity Assumption.  It is simply the fan whose maximal cones are the cones $C_{\jj}^+$ as in \eqref{CS+} for tuples $\jj\in I\setminus F$.

\subsection{Theta bases for $\s{X}_t^{\up}$}\label{XupTheta}

Theta functions in $\s{X}^{\up}$ and $\s{X}_t^{\up}$ can be similarly constructed using a scattering diagram $\f{D}^{\s{X}_t}$ and broken lines in $N_{\bb{R}}$ with exponents in $N$.  See \cite[\S 4.3]{davison2019strong} for this construction. Very briefly, the idea is to define $\f{g}^t$ using $\kk_t[N]$ in place of $\kk_t[M]$, then construct the scattering diagram $\f{D}^{\s{X}_t}$ in $N_{\bb{R}}$ over $\f{g}^t$ as $\scat(\f{D}^{\s{X}_t}_{\In})$ for 
\begin{align}\label{DXin}
\f{D}^{\s{X}_t}_{\In}=\{(e_i^{\omega\perp},\Psi_t(z^{e_i}))|i\in I\setminus F\}
\end{align}
(if $e_i\in \ker(\omega)$ for some $i\in I\setminus F$ then the wall for this $i$ is left out). Rather than reviewing this in more detail, we will use \cite[Lem 4.6]{davison2019strong} to recover these theta functions from the theta functions for $\s{A}_t^{\prin}$ --- this is also the approach used for the classical setting in \cite{gross2018canonical}.

Consider the map 
\begin{align}\label{xi}
\xi:N &\rar M\oplus N \\
n &\mapsto (\omega_1(n),n).\nonumber
\end{align}
Recall that $M^{\prin}=M\oplus N$. Notice that $\xi$ is the restriction of $\omega_1^{\prin}$ to $N$ (identified with $N\oplus 0$) in \eqref{pi1prin}, so $(M^{\prin})^\oplus \subset \xi(N)$. If we construct $\s{A}_t^{\prin,\up}$ using $\Lambda^{\prin}$ from \eqref{LambdaPrin} as our compatible form, then the map $\xi$ induces an injection $\xi: \kk_t\llb N \rrb \hookrightarrow \kk_t\llb M^{\prin}\rrb$.  A choice of seed determines inclusions of $\s{X}_t^{\up}$ and $\s{A}_t^{\prin,\up}$ into $\kk_t\llb N\rrb$ and $\kk_t\llb M^{\prin}\rrb$, respectively, and $\xi$ restricts to a map $\xi:\s{X}_t^{\up}\hookrightarrow \s{A}_t^{\prin,\up}$ (cf. \cite[Lem. 4.1]{davison2019strong}) with image the degree $0$ part of $\s{A}_t^{\prin,\up}$ for the grading 
determined by 
\begin{align*}
\text{degree}(z^{(m,n)}):=m-\omega_1(n).
\end{align*}
\begin{defn}[\cite{davison2019strong}, Lem. 4.6]\label{X-Aprin}
Under the above identification $\xi$ of $\kk_t\llb N\rrb$ with a subalgebra of $\kk_t\llb M^{\prin}\rrb$ for compatible form $\Lambda^{\prin}$ as in \eqref{LambdaPrin}, the element $\vartheta_{p,\sQ}\in \kk_t\llb N\rrb\supset \s{X}_t^{\up}$ for $p\in N$ and generic $\sQ\in N_{\bb{R}}$ is defined so that\footnote{Technically, $\xi(\sQ)$ is not generic in the sense of Footnote \ref{foot-gen}, but one may use a more general notion of ``generic'' as in \cite[Footnote 5]{davison2019strong}.}
\begin{align}\label{eq:construct_X_theta}
	\xi(\vartheta_{p,\sQ})=\vartheta_{\xi(p),\xi(\sQ)}\in \s{A}_t^{\prin,\up}
\end{align}
\end{defn}  
 Note that Lemma \ref{BLpos} applies to ensure positivity of all broken lines contributing to the $\s{X}$-type theta functions $\vartheta_{p,\sQ}$, $p\in N$.

For $p\in N$, one can consider not only the local coordinate expansions $\vartheta_{\xi(p),\xi(\sQ)}\in \kk_t\llb \xi (N)\rrb$, but also $\vartheta_{\xi(p),\sQ'}\in \kk_t\llb \xi (N)\rrb$ for any generic $\sQ'\in M^{\prin}_{\bb{R}}=M_{\bb{R}}\oplus N_{\bb{R}}$.  Let $\s{C}^{\prin}$ be the cluster complex in $M_{\bb{R}}^{\prin}$.  Then for $\sQ'$ a generic point in the cone of $\s{C}^{\prin}$ associated to a seed $\sd^{\prin}_{\jj}$, $\vartheta_{\xi(p),\sQ'}$ is the Laurent (series) expansion of $\vartheta_{\xi(p)}$ in $\xi(\s{X}^{\sd_{\jj}}_t)$ (or in the appropriate completion) up to an  automorphism $\psi_{\jj}^{\prin}$ defined by a certain linear map acting on the exponents (cf. \cite[\S 4.7.2]{davison2019strong} for more details). 
Correspondingly (applying $\xi^{-1}$), we obtain the Laurent (series) expansions of $\vartheta_p$ in the cluster $\s{X}_t^{\sd_{\jj}}=\kk_t[N]$ or its completion $\kk_t\llb N\rrb$.

Let $\Theta_{\s{X}}^{\midd}$ denote the set of $p\in M$ such that $\vartheta_{\xi(p),\sQ}\in \xi(\kk_t[N])$ (as opposed to $\xi(\kk_t\llb N\rrb)$ for some (equivalently, all) generic $\sQ\in \s{C}^{\prin}$, cf. \cite[\S 4.6]{davison2019strong}. The theta functions $\{\vartheta_p|p\in \Theta_{\s{X}}^{\midd}\}$ generate and form a $\kk_t$-basis for a subalgebra $\s{X}_t^{\midd}\subset \s{X}_t^{\can}\cap \s{X}_t^{\up}$. One says that the \textbf{full Fock-Goncharov Conjecture} holds for $\s{X}_t$ if $\s{X}_t^{\midd}=\s{X}_t^{\up}=\s{X}_t^{\can}$.

\subsection{Failure of the Injectivity Assumption}\label{FailInj}

So far we have assumed that $\sd$ satisfies Assumption \ref{inj-assumption}, but in the classical setting this is sometimes not the case. When the Injectivity Assumption fails, one can still construct theta functions using the procedure from \cite[Constr. 7.11]{gross2018canonical}, which we shall now recall in an equivalent form.

First, apply the constructions of scattering diagrams, broken lines, and theta functions to the compatible pair $(\sd^{\prin},\Lambda^{\prin})$ defined in \S \ref{prinsub}. Recall the map $$\rho:M^{\prin} \rar M, \qquad (m,n)\mapsto m,$$ and the induced maps on algebras as in \S \ref{ClAlg}. One checks from \eqref{LambdaPrin} or from the compatibility condition that 
\begin{align}\label{LambdaN}
\Lambda^{\prin}((0,n),(m',n'))=0
\end{align}
for all $(m',n')\in \omega_1^{\prin}((N^{\oplus},0))$, hence whenever $(m',n')$ is the direction of a scattering wall. It follows that all scattering walls of $\f{D}^{\s{A}^{\prin}}$ are closed under addition by elements of $(0,N_{\bb{R}})$, and all wall-crossing automorphisms act trivially on $z^{(0,n)}$ for each $n\in N$.  From this one can show that for each $m\in M$, $n_1,n_2\in N$, and each pair of generic points $\sQ,\sQ'\in M^{\prin}_{\bb{R}}$ with $\sQ-\sQ'\in (0,N_{\bb{R}})$, there is a bijection between broken lines with ends $((m,n_1),\sQ)$ and broken lines with ends $((m,n_2),\sQ')$, and the final monomials of corresponding broken lines agree after applying $\rho$.

Consider the cone $\Theta^{\prin,\midd}$ of $(m,n)\in M^{\prin}$ for which $\vartheta_{(m,n),\sQ}$ is a finite Laurent polynomial for all generic $\s{Q}$ in\footnote{We will often write that a point is in $\s{C}^{\prin}$ or $\s{C}$ to mean that it is in the support of $\s{C}^{\prin}$ or $\s{C}$, respectively.} the cluster complex $\s{C}^{\prin}$ of $(\sd^{\prin},\Lambda^{\prin})$.  One sees from the above discussion that $\Theta^{\prin,\midd}$ is closed under addition by $(0,N)$.  Define
\begin{align*}
\Theta^{\midd}:=\rho(\Theta^{\prin,\midd}).
\end{align*}
Now for any $m\in \Theta^{\midd}$ and any generic $\sQ\in M_{\bb{R}}$, let $(m,n)\in \rho^{-1}(m)$ and let $\wt{\sQ}\in \rho^{-1}(\sQ)$, and suppose $\vartheta_{(m,n),\wt{\sQ}}$ is a Laurent polynomial (as opposed to a formal Laurent series). By construction, this is always the case for $\sQ\in \rho(\s{C}^{\prin})$. Then we can define
\begin{align}\label{theta-rho}
\vartheta_{m,\sQ}:=\rho(\vartheta_{(m,n),\wt{\sQ}}),
\end{align}
and this definition is independent of the choices of $n\in N$ and $\wt{\sQ}\in \rho^{-1}(\sQ)$. Let 
\begin{align}\label{eq:Xi}
\wt{\Xi}&:=\{\text{generic~} \wt{\sQ}\in M^{\prin}_{\bb{R}}|\vartheta_{(m,n),\wt{\sQ}}~ \text{is a Laurent polynomial for all } (m,n)\in \Theta^{\prin,\midd}\}, \nonumber\\
\Xi&:=\rho(\wt{\Xi}).
\end{align}
That is, $\Xi$ is the set of points $\sQ\in M_{\bb{R}}$ such that $\vartheta_{m,\sQ}$ is well-defined for all $m\in \Theta^{\midd}$.  Note that $\?{\Xi}\supset \?{\s{C}}$ by construction.

The elements $\vartheta_{m,\sQ}$ for different $\sQ\in \Xi$ are related by a sort of path-ordered product.  For each $\sQ\in \Xi$, let $\kk(M)_{\sQ}$ denote the subring of $\kk(M)$ generated over $\kk$ by $\{\vartheta_{p,\sQ}|p\in \Theta^{\midd}\}$.  Given $f\in \kk(M)_{\sQ}$ a rational function of elements $\vartheta_{m,\sQ}$, we let $\wt{f}\in \kk(M^{\prin})$ denote the corresponding rational function of the lifts $\vartheta_{(m,0),\wt{\sQ}}$ for $\wt{\sQ}\in \rho^{-1}(\sQ)$.  Now for any $\sQ,\sQ'\in \Xi$, we define 
\begin{align*}
\theta_{\sQ,\sQ'}:\kk(M)_{\sQ}&\rar \kk(M)_{\sQ'}\\
f&\mapsto \rho(\Ad_{\theta_{\gamma,\f{D}^{\s{A}^{\prin}}}}(\wt{f})).
\end{align*}
where, $\gamma$ is a path from a lift of $\sQ$ to a lift of $\sQ'$.
\begin{lem}\label{lem:thetaQQ}
The maps $\theta_{\sQ,\sQ'}$ are well-defined isomorphisms satisfying $\theta_{\sQ,\sQ'}(\vartheta_{m,\sQ})=\vartheta_{m,\sQ'}$ for all $m\in \Theta^{\midd}$, $\sQ,\sQ'\in \Xi$.
\end{lem}
\begin{proof}
The independence of the choice of lifts $\wt{\sQ},\wt{\sQ}'$ follows from \eqref{LambdaN} and the surrounding discussion.  The independence from the choice of $\gamma$ holds because $\f{D}^{\s{A},\prin}$ is consistent.  We note that there may be multiple ways to express $f$ as a rational function of theta functions, hence multiple choices of lift $\wt{f}$, but it follows from \eqref{LambdaN} that these different lifts all result in the same $\theta_{\sQ,\sQ'}(f)$.  The conditions for being a homomrphism are now easily checked, and then invertibility follows from noting that $\theta_{\sQ',\sQ}=\theta_{\sQ,\sQ'}^{-1}$.  Finally, we know from Lemma \ref{CPS} that $\Ad_{\theta_{\gamma,\f{D}^{\s{A}^{\prin}}}}(\vartheta_{p,\wt{\sQ}})=\vartheta_{p,\wt{\sQ'}}$ for all $p\in \Theta^{\prin,\midd}$, and applying $\rho$ yields $\vartheta_{p,\sQ'}\in\kk(M)_{\sQ'}$, as desired.
\end{proof}

For any $m_1,\ldots,m_s\in \Theta^{\midd}$, one sees from the above discussion that the structure constant $\alpha(m_1,\ldots,m_s;m)$ defined  by
\begin{align}\label{eq:str_const_A_merged}
\vartheta_{m_1,\sQ}\cdots \vartheta_{m_s,\sQ} = \sum_{m\in M} \alpha(m_1,\ldots,m_s;m) \vartheta_{m,\sQ}
\end{align}
are independent of $\sQ\in \Xi$. Indeed, we see that these structure constants are given by
\begin{align}\label{eq:str-const-A}
\alpha(m_1,\ldots,m_s;m) = \sum_{\wt{m}\in \rho^{-1}(m)} \rho(\alpha(\wt{m}_1,\ldots,\wt{m}_s;\wt{m})),
\end{align}
where the sum is over all $\wt{m}\in \rho^{-1}(m)$, and for each $i=1,\ldots,s$, $\wt{m}_i$ is any fixed choice of lift of $m_i$ by $\rho$.  Here, the constants $\alpha(\wt{m}_1,\ldots,\wt{m}_s;\wt{m})$ can be computed as in Proposition \ref{StructureConstants}. In particular, \eqref{eq:str_const_A_merged} is a finite sum since the theta functions that appear are positive Laurent polynomials and the structure constants are positive. 

These structure constants canonically determine an algebra structure on 
\begin{align*}
\s{A}^{\midd}:=\bigoplus_{m\in \Theta^{\midd}} \kk \cdot \vartheta_m.
\end{align*}
There is a morphism
\begin{align}\label{nu}
\nu:\s{A}^{\midd}&\rar \s{A}^{\up}\subset \kk[M] \nonumber\\
\sum_{m\in \Theta^{\midd}} a_m \vartheta_{m} &\mapsto \sum_{m\in \Theta^{\midd}} a_m \vartheta_{m,\sQ}
\end{align}
where $\sQ$ is any generic point in $C_{\sd}^+=\rho(C_{\sd^{\prin}}^+)$.  \cite{gross2018canonical} conjectures that $\nu$ is always injective, i.e., that $\{\vartheta_{m,\sQ}|m\in \Theta^{\midd}\}$ is linearly independent. By \cite[Thm 0.3(7)]{gross2018canonical}, $\nu$ is at least injective if some seed mutation equivalent to $\sd$ has strongly convex $M^{\oplus}$.

On the other hand, even when the Injectivity Assumption possibly fails, we can define the cluster complex $\s{C}$ in $M_{\bb{R}}$ by applying $\rho$ to the cones of the principal coefficients cluster complex $\s{C}^{\prin}\subset M_{\bb{R}}^{\prin}$.  Then for $\sQ\in C_{\sd}^+$ and $m\in \s{C}\cap M$, $\vartheta_{m,\sQ} \in \s{A}^{\up}$ is equal to the cluster monomial with extended $g$-vector $m$.  We therefore always have
\begin{align*}
\s{A}^{\ord}\subset \nu(\s{A}^{\midd})\subset \s{A}^{\up}.
\end{align*}

We will write $\?{\s{C}}$ to mean the closure of the support of $\s{C}$.

\begin{prop}\label{indep}
If there is a linear relation
\begin{align*}
	\sum_{m\in \Theta^{\midd}} a_m \nu(\vartheta_m) = 0
\end{align*} in $\s{A}^{\up}$, then $a_m=0$ for all $m\in \?{\Xi}\cap \Theta^{\midd}$.  In particular, the set $\{\nu(\vartheta_m)|m\in \?{\Xi}\cap \Theta^{\midd}\}$ is linearly independent.
\end{prop}

This is a generalization of \cite[Thm 7.20]{gross2018canonical}, which is the analogous result with $\s{C}$ in place of the larger set $\?{\Xi}\supset \?{\s{C}}$.  In particular, loc. cit. showed linear independence of cluster monomials (previously proven in \cite{CKLP}).  The key to the proof of Proposition \ref{indep} is the following lemma:
\begin{lem}\label{pmQ}
Let $p\in \Theta^{\midd}$ and $m\in \Theta^{\midd}\setminus \{p\}$.  Then for $\sQ\in \Xi$ sufficiently close to $p$, the $z^p$-coefficient of $\vartheta_{m,\sQ}$ is $0$.  On the other hand, the $z^p$-coefficient of $\vartheta_{p,\sQ}$ is nonzero for all $\sQ\in \Xi$ and $1$ for all $\sQ\in \Xi$ sufficiently close to $p$.
\end{lem}

Under the Injectivity Assumption, or more generally whenever the cone $M^{\oplus}_{\bb{R}}:=\bb{R}_{\geq 0}\langle \omega_1(e_i)|i\in I\setminus F\rangle$ is convex, Lemma \ref{pmQ} is known and used crucially in \cite{gross2018canonical} and elsewhere, e.g., in the proof of \cite[Prop 6.4(3)]{gross2018canonical} (the classical case of Proposition \ref{StructureConstants}).  In that setting, no containment in $\Theta^{\midd}$ or $\?{\Xi}$ is assumed because the convexity of $M^{\oplus}$ already gives the finiteness needed for the proof.  We note that for punctured surfaces with non-empty boundary, convexity of $M_{\bb{R}}^\oplus$ was shown by \cite[Theorem 6.8]{geiss2020generic}.

\begin{myproof}[Proof of Lemma \ref{pmQ}]
Let $\sQ\in \Xi$. Assume $p\in \?{\Xi}$, and write $\sQ$ as $\sQ=p+v$ for some $v\in M_{\bb{R}}$. Since $m\in \Theta^{\midd}$ and $\sQ\in \Xi$, $\vartheta_{m,\sQ}$ has only finitely many terms.  In particular, there are only finitely many broken lines $\gamma_1,\ldots,\gamma_s$ in $M_{\bb{R}}^{\prin}$ with ends $(\wt{m},\wt{\sQ})$ --- using tildes to denote fixed choices of lifts --- and final attached exponent in $\rho^{-1}(p)$. Since $m\neq p$, these broken lines must all bend somewhere before attaining this final attached exponent.

Let us first prove that the $z^p$-coefficient of $\vartheta_{m,\sQ}$ is $0$ for $\sQ$ sufficiently near $p$. If $p=0$, then the final attached exponent of any broken line contributing to the $z^p$-coefficient (for any $\sQ$) would lie in $\rho^{-1}(0)=0\oplus N$ and would therefore commute with all wall-crossings by \eqref{LambdaN}.  But such a broken line could not have bent, so the $z^0$-coefficient must in fact vanish.

Assume $p\neq 0$ from now on. Then for each of these broken lines $\gamma_i$, $i=1,\ldots,s$, the image under $\rho$ of the final straight segment $L_i$ will be a line segment $\rho(L_i)=\{\sQ+tp|t\in [0,K_i]\}$ for some finite $K_i\in \bb{R}_{>0}$.  Note that these images under $\rho$ are independent of the choice of lift $\wt{\sQ}$ of $\sQ$, because broken lines with ends $(\wt{m},\wt{\sQ}_1)$ and $(\wt{m},\wt{\sQ}_2)$ for $\wt{\sQ}_1$ and $\wt{\sQ}_2$ denoting two lifts of $\sQ$ are related by translation by $\wt{\sQ}_2-\wt{\sQ}_1\in (0,N_{\bb{R}})$ and therefore have the same projection under $\rho$.  Fix a generic $K\in \bb{R}$ with $K>\max_i K_i$.  Let $\sQ'=\sQ+Kp=(K+1)p+v$.  By construction, $\sQ'$ is not contained in any $\rho(L_i)$ for any $i=1,\ldots,s$. Thus, there cannot be any broken lines with ends $(\wt{m},\wt{\sQ'})$ and final exponent in $\rho^{-1}(p)$ because the last straight segment of such a broken line could be extended to end at a lift of $\sQ$, but the image under $\rho$ of the final straight segment will contain $\sQ'$ and therefore cannot be one of the $L_1,\ldots,L_s$ considered above.  Since all walls are cones, we can rescale $\wt{\sQ'}$ to find $\vartheta_{{m},{(K+1)p+v}} = \vartheta_{{m},{p+\frac{1}{K+1}v}}$.  Thus, the $z^p$-coefficient of $\vartheta_{{m},{p+\frac{1}{K+1}v}}$ is $0$ for large $K$, as claimed.

For the statement regarding the $z^p$-coefficient of $\vartheta_{p,\sQ}$, note that there is always a straight broken line with ends $(\wt{p},\wt{\sQ})$ contributing $z^{\wt{p}}$ to $\vartheta_{\wt{p},\wt{\sQ}}$. By Lemma \ref{BLpos} (the positivity of broken lines), other broken lines will not cancel with this term, so the $z^p$-coefficient of $\vartheta_{p,\sQ}$ will be nonzero.  The fact that this coefficient is $1$ when $\sQ$ is sufficiently close to $p$ follows from essentially the same arguments used in the previous paragraph.
\end{myproof}

\begin{proof}[Proof of Proposition \ref{indep}]
Suppose we have a relation 
\begin{align}\label{depend}
	\sum_{m\in \Theta^{\midd}} a_m \nu(\vartheta_{m}) = 0.  
\end{align}
For $\sQ\in \Xi$, let 
\begin{align*}
	f_{\sQ}:=\sum_{m\in \Theta^{\midd}} a_m \vartheta_{m,\sQ}.
\end{align*}
By Lemma \ref{lem:thetaQQ}, the assumption that $f_{\sQ}=0$ for $\sQ\in C_{\sd}^+$ implies that $f_{\sQ}=0$ for all $\sQ\in \Xi$.  Let $p\in \?{\Xi}\cap \Theta^{\midd}$. By Lemma \ref{pmQ}, if $\sQ\in \Xi$ is sufficiently close to $p$ (note that $p\in \?{\Xi}$ ensures that such $\sQ$ exists), then the $z^p$-coefficient of $f_{\sQ}$ (which is $0$ since $f_{\sQ}=0$) must equal $a_p$.  Hence, $a_p=0$, as desired.
\end{proof}

When the Injectivity Assumption fails, the \textbf{full Fock-Goncharov conjecture} is modified slightly to consist of the following conditions:
\begin{itemize}
\item $\nu:\s{A}^{\midd}\rar \s{A}^{\up}$ is injective;
\item $\Theta^{\midd}=M$;
\item $\nu(\s{A}^{\midd})=\s{A}^{\up}$.
\end{itemize}
By \cite[Prop. 0.14]{gross2018canonical}, Proposition \ref{FGc} can be generalized as follows:
\begin{prop}\label{FGc2}
If $M^{\oplus}$ is strongly convex for some seed mutation equivalent to $\sd$, and if the cluster complex $\s{C}$ is not contained in a half-space, then the full Fock-Goncharov conjecture holds.
\end{prop}

Proposition \ref{indep} gives us the following without having to worry about convexity.
\begin{prop}\label{dense-FG}
If $\s{C}$ is dense in $M_{\bb{R}}$, then $\nu:\s{A}^{\midd}\rar \s{A}^{\up}$ is injective and $\Theta^{\midd}=M$. 
\end{prop}
We expect $\nu(\s{A}^{\midd})=\s{A}^{\up}$ to hold in these cases as well, but without the convexity of $M^{\oplus}$ we are not sure how to prove this.
\begin{proof}
The injectivity of $\nu$ in such cases is immediate by Proposition \ref{indep}, so it remains to prove that $\Theta^{\midd}=M$.  The cluster complex for $\sd^{\prin}$ is $\rho^{-1}(\s{C})$, and this cannot be contained in a half-space since $\s{C}$ is dense.  So Proposition \ref{FGc2} implies that $\Theta^{\prin,\midd}=M^{\prin}$, hence $\Theta^{\midd}=\rho(\Theta^{\prin,\midd})=M$, as desired. 
\end{proof}

This density of $\s{C}$ holds for most cluster algebras from marked surfaces:

\begin{prop}[\cite{yurikusa2020density}, Thm. 1.2]\label{gdense}
Let $\Sigma$ be a connected triangulable surface.  If $\Sigma$ contains a single puncture and no other markings, then the closure of the cluster complex is a half-space in $M_{\bb{R}}$ (and the complementary half-space is the closure of the notched-arc cluster complex).  Otherwise the cluster complex is dense in $M_{\bb{R}}$.
\end{prop}

We note that \cite{yurikusa2020density} does not have boundary arcs / frozen vectors, but the results are easily extended to our setup because adding frozen vectors just replaces each cone $\sigma$ in the cluster complex with $\sigma \oplus M_F$ where $M_F$ is the $\bb{R}$-span of the frozen vectors.

We also note that Proposition \ref{gdense} is easily extended to disconnected surfaces --- if $\Sigma$ is a disjoint union of connected triangulable surfaces, then the corresponding cluster complex is the product of the cluster complexes for the components (this follows easily using, e.g.,  the constructions of \S \ref{sec:union}).

\begin{thm}\label{fFG-surfaces}
Let $\Sigma$ be a triangulable marked surface.  If no component of $\Sigma$ is closed (i.e., with empty boundary), then the full Fock-Goncharov conjecture holds.  If some components of $\Sigma$ are possibly closed, but no component is a once-punctured closed surface, then the full Fock-Goncharov conjecture holds for $\s{A}_t^{\prin}$ and $\s{X}_t$. In any case (for any triangulable $\Sigma$), it is at least true that $\nu$ is injective, $\s{A}^{\midd}=\s{A}^{\can}$,  and $\?{\Xi}=M_{\bb{R}}$.
\end{thm}

In the once-punctured closed surface cases, the proofs of the above claims will depend on Theorem \ref{thm:punctured_bracelet_theta} (the result relating tagged bracelets to theta functions).  We therefore take care to not utilize these cases while proving Theorem \ref{thm:punctured_bracelet_theta}.

\begin{proof}[Proof of Theorem \ref{fFG-surfaces}]
If every component of $\Sigma$ has non-empty boundary, then $M^{\oplus}$ is convex by \cite[Theorem 6.8]{geiss2020generic}, and $\s{C}$ is not contained in a half space (Proposition \ref{gdense}), so the full Fock-Goncharov conjecture follows from Proposition \ref{FGc2}.   

Now suppose that $\Sigma$ possibly does have some closed components. We assume for simplicity that $\Sigma$ is connected (the general cases easily follow).  If $\Sigma$ has at least two punctures, then $\s{C}$ is dense in $M_{\bb{R}}$ by Proposition \ref{gdense}, so the full Fock-Goncharov conjecture for $\s{A}_t^{\prin}$ and $\s{X}_t$ follows from Proposition \ref{FGc}, while the claims for $\s{A}$ follow from Proposition \ref{dense-FG}.

Suppose $\Sigma$ has only one puncture. We will see in Theorem \ref{thm:punctured_bracelet_theta} that the theta functions coincide with tagged bracelets (up to scaling in the case of the once-punctured torus), and since they are elements of $\Sk^{\Box}(\Sigma)$, the tagged bracelets belong to $\s{A}^{\up}$ by Proposition \ref{SkAPropPun}.  The equality $\Theta^{\midd}=M$ follows.

We next show that $\?{\Xi}=M_{\bb{R}}$.  We always have $\?{\Xi}\supset \?{\s{C}}$, so the claim is immediate when $\?{\s{C}}=M_{\bb{R}}$.  This leaves only the once-punctured closed surface setting.  Here, we have seen that all bracelets (hence all theta functions by our upcoming results) are Laurent polynomials in each cluster. As before, we can apply the automorphism of Remark \ref{rmk:tag-switch} which changes all tags at the puncture, and this immediately yields that the bracelets are Laurent polynomials for the notched triangulation clusters as well.

The injectivity of $\nu$ for once-punctured closed surfaces now follows from Proposition \ref{indep}. 
\end{proof}

We will use Theorem \ref{fFG-surfaces} to identify $\s{A}^{\can}$ with its image $\nu(\s{A}^{\can})\subset \s{A}^{\up}$.  In particular, we will write $\nu(\vartheta_p)$ as simply $\vartheta_p$.

We note that the full Fock-Goncharov conjecture is also known for $\s{A}^{\prin}$ and $\s{X}$ in the once-punctured torus case by \cite{zhou2020cluster}.  On the other hand (as we previously noted in Remark \ref{rmk:once-torus}), \cite{zhou2020cluster} also showed that $\s{A}^{\can}$ is a \textit{proper} subalgebra of $\s{A}^{\up}$ in the case of the once-punctured torus, so the full Fock-Goncharov conjecture fails for $\s{A}$ in this case. Interestingly, if one defines $\s{A}^{\up,\Box}$ to be the intersection of $\s{A}^{\up}$ with the clusters associated to the notched arc cluster structure, then \cite{zhou2020cluster} finds that one does obtain $\s{A}^{\can}=\s{A}^{\up,\Box}$ in this case.  One naturally expects the same to be true for all other once-punctured positive-genus surfaces.

We do not currently know how to show this, but we will see in Theorem \ref{thm:tag_sk_up_cl_alg}(i) that we do always have $\s{A}^{\can}=\Sk^{\Box}(\Sigma)$ (with the caveat that in the once-punctured torus cases, one must either work over $\kk\supset \bb{Q}$ or insert an extra wall in the scattering diagram).  The argument is essentially that $\Sk^{\Box}(\Sigma)$ is generated by tagged bracelets, and tagged bracelets will turn out to be theta functions.

\subsection{Compatibility of theta bases for $\s{A}$ and $\s{X}$}

Recall the map $\omega_1:N\rar M$ and the induced map $\omega_1:\kk[N]\rar \kk[M]$.
It is well-known (cf. \cite[Prop. 2.2]{FG1}) that $\omega_1$ is compatible with mutations and thus extends to a map $\omega_1:\s{X}^{\up}\rar \s{A}^{\up}$.  Note that, under the Injectivity Assumption, $\omega_1$ also induces a map $\omega_1:\wh{\s{X}} \rar \wh{\s{A}}$ for $\wh{\s{X}}:=\kk\llb N\rrb$ and $\wh{\s{A}}:=\kk\llb M\rrb$ (still commuting with mutations). 

\begin{prop}\label{prop:omega-theta}
For each $n\in \Theta_{\s{X}}^{\midd}$, consider $\vartheta_n\in \s{X}^{\up}$.  Then $\vartheta_{\omega_1(n)}\in \s{A}^{\up}$, and we have \begin{align}\label{omega-theta}
	\omega_1(\vartheta_n)=\vartheta_{\omega_1(n)}.
\end{align}
Now suppose that the Injectivity Assumption holds so that $\vartheta_m$ is well-defined in $\wh{\s{A}}$ for all $m\in M$.  Then \eqref{omega-theta} applies for all $n\in N$.

In the quantum setting, if $\Lambda$ satisfies the compatibility condition \eqref{Lambda-B} for all $i\in I$ (as opposed to just for $i\in I\setminus F$), then there is a $\kk$-algebra homomorphism $\omega_1:\wh{\s{X}}_t\rar \wh{\s{A}}_t$ for $\wh{\s{X}}_t:=\kk_t\llb N\rrb$, $\wh{\s{A}}_t:=\kk_t\llb M\rrb$, determined by $z^n\mapsto z^{\omega_1(n)}$ and $t\mapsto t^d$. This morphism commutes with mutations to induce $\omega_1:\s{X}_t^{\up}\rar \s{A}_t^{\up}$.  Furthermore, this $\omega_1$ induces $\omega_1:\s{X}_t^{\can}\rar \s{A}_t^{\can}$ with $\omega_1(\vartheta_n)=\vartheta_{\omega_1(n)}$. 

\end{prop}
\begin{proof}
The statements in the classical setting are an immediate consequence of Definition \ref{X-Aprin}, Equation \eqref{theta-rho}, and the easy observation that $\omega_1=\rho\circ \xi$.  The statement in the quantum setting with $d=1$ is \cite[Thm. 1.2(8)]{davison2019strong}. The argument for general $d$ is essentially the same, with the need for $t\mapsto t^d$ evident in \eqref{Lambda-v1v2}.  The key observation is that this $\omega_1$ maps the scattering functions $\Psi_t(z^{e_i})$ of \eqref{DXin} to the scattering functions $\Psi_{t^d}(z^{\omega_1(e_i)})$ of \eqref{DAin}.
\end{proof}

\subsection{Positivity and the scattering atlas}\label{sec:positivity_atlas}

Recall the algebra $\s{A}_t^{\can}$ generated by the theta functions as in \S \ref{StrC}, and recall that each generic $\sQ\in M_{\bb{R}}$ determines an inclusion $\iota_{\s{Q}}:\s{A}_t^{\can} \hookrightarrow \kk_t\llb M\rrb$, $\vartheta_p\mapsto \vartheta_{p,\sQ}$ as in Lemma \ref{lem:fQ}.  Given any $f\in \s{A}_t^{\can}$, one says that $f$ is \textbf{universally positive with respect to the scattering atlas}, or \textbf{theta positive}\footnote{We thank Greg Muller for suggesting this shortened terminology.} for short, if for every generic $\sQ\in M_{\bb{R}}$, all nonzero coefficients of 
\begin{align}\label{fmQ}
\iota_{\sQ}(f)=\sum_{m\in M} a_{f,m,\sQ} z^m
\end{align} are positive elements.

One says that nonzero $f\in \s{A}^{\can}_t$ is \textbf{atomic with respect to the scattering atlas}, or {\bf theta atomic} for short, if it is theta positive and cannot be decomposed as a sum of two other nonzero theta positive elements.

If the Injectivity Assumption does not necessarily hold, we can still define a similar notion as follows.  For $f\in \s{A}^{\midd}$, we can define $a_{f,m,\sQ}$ as in \eqref{fmQ} for $\sQ\in \Xi$.  We say that $f$ is \textbf{$\Xi$-positive} if $a_{f,m,\sQ}\in \bb{Z}_{\geq 0}$ for all $m\in M$ and all $\sQ\in \Xi$.  Then we say nonzero $f\in \s{A}^{\midd}$ is {\bf $\Xi$-atomic} if $f$ cannot be decomposed as a sum of two other nonzero $\Xi$-positive elements.

\begin{prop}\label{AtomicProp}
Suppose the Injectivity Assumption holds.  Then all theta functions are universally positive with respect to the scattering atlas, and the basis of $\s{A}_t^{\can}$ consisting of the theta functions is strongly positive.  Furthermore, the theta functions are precisely the atomic elements of $\s{A}_t^{\can}$ with 
respect to the scattering atlas.

Suppose now that the Injectivity Assumption does not necessarily necessarily hold.  The theta functions $\{\vartheta_m\}_{m\in \Theta^{\midd}}$ are still $\Xi$-positive and a strongly positive basis.  Furthermore, if $\?{\Xi}\supset \Theta^{\midd}$,  then the theta functions $\{\vartheta_m\}_{m\in \Theta^{\midd}}$ are the $\Xi$-atomic elements\footnote{We note that \cite[Examples 2]{ManAtomic} discussed theta atomicity in the case of the Markov quiver (the cluster algebra associated to the once-punctured torus).  The arguments there do not fully address the failure of the Injectivity Assumption in that case, so our Lemma \ref{pmQ} and the subsequent proof of Proposition \ref{AtomicProp} fill a minor gap in loc. cit.} of $\s{A}^{\midd}$.
\end{prop}
\begin{proof}
As observed in \cite{gross2018canonical,davison2019strong}, the theta positivity is immediate from Lemma \ref{BLpos}.  Similarly, the strong positivity is immediate from Lemma \ref{BLpos} and Proposition \ref{StructureConstants}.  Under the Injectivity Assumption, the atomicity is proven in \cite{ManAtomic} in the classical setting and \cite[Thm. 3.16]{davison2019strong} in the quantum setting. In the possible absence of the Injectivity Assumption, the $\Xi$-positivity again follows from Lemma \ref{BLpos} (and the definition \eqref{theta-rho}), and the strong positivity similarly follows from Lemma \ref{BLpos} combined with \eqref{eq:str-const-A}.

The atomicity arguments from the cases where the Injectivity Assumption holds can be extended to cases where the Injectivity Assumption possibly fails by using Lemma \ref{pmQ}, assuming that $\?{\Xi}\supset \Theta^{\midd}$.  In these cases, let $f\in \s{A}^{\midd}$ be an arbitrary $\Xi$-positive element.  We wish to show that the expansion $f=\sum_{m\in \Theta^{\midd}} a_m\vartheta_m$ has non-negative coefficients.  By Lemma \ref{pmQ}, if $\sQ\in \Xi$ is sufficiently close to a given $m$, then $a_m=a_{f,m,\sQ}$ where $a_{f,m,\sQ}$ is defined as in \eqref{fmQ} (the assumption $\?{\Xi}\supset \Theta^{\midd}$ ensures such $\sQ$ exist for all $m\in \Theta^{\midd}$).  Since $a_{f,m,\sQ}$ is non-negative by the definition of theta positivity, the claim follows.
\end{proof}

By Proposition \ref{Chambers}, universal positivity with respect to the scattering atlas implies universal positivity with respect to the cluster atlas.\footnote{\cite{FG1} conjectured that the elements which are atomic with respect to the cluster atlas form a basis, but the set of such elements was shown to be linearly dependent in \cite{LLZ} for all rank $2$ cluster algebras which are not of finite or affine type.  It follows that universal positivity with respect to the scattering atlas is generally a strictly stronger condition.}  Of course, the two notions (as well as $\Xi$-positivity) agree whenever $\s{C}$ is dense in $M_{\bb{R}}$.

Consider a marked surface $\Sigma$ with an ideal triangulation $\Delta$. We say that an element $z\in \s{A}^{\up}(\sd_{\Delta})$ is universally positive with respect to the \textbf{ideal triangulation atlas} if, for any ideal triangulation $\Delta'$, $z$ is a positive Laurent polynomials in the elements $[\gamma]$ for arcs $\gamma\in \Delta'$. Atomicity with respect to the ideal triangulation atlas is then defined naturally. We similarly define positivity and atomicity with respect to the {\bf tagged triangulation atlas}. By Theorem \ref{gdense}, if $\Sigma$ is not a once-punctured closed surface, then the scattering atlas, cluster atlas, and tagged triangulation atlas are all equivalent.  We will see in \S \ref{sec:closed_surface} that the tagged triangulation and scattering atlases are equivalent for once-punctured closed surfaces as well, except in the case of the once-punctured torus.

Theorem \ref{fFG-surfaces} and Propositions \ref{gdense} and \ref{AtomicProp} (plus \S \ref{sec:closed_surface}) thus imply the following:

\begin{cor}\label{cor:surface-theta-atomic}
Let $\Sigma$ be any triangulable marked surface with triangulation $\Delta$.  The theta functions are precisely the $\Xi$-atomic elements of $\s{A}^{\midd}(\sd_{\Delta})$.  If no component of $\Sigma$ is a once-punctured closed surface, then $\Xi$-positivity, positivity with respect to the scattering atlas, positivity with respect to the cluster atlas, and positivity with respect to the tagged triangulation atlas are all equivalent.  For general triangulable $\Sigma$, $\Xi$-positivity is equivalent to positivity with respect to the scattering atlas, and this is equivalent to positivity with respect to the tagged triangulation atlas except in the case of a once-punctured torus.
\end{cor}

One may similarly define universal positivity with respect to the scattering atlas, or theta positivity, for elements of $\s{X}^{\can}_t$, using generic $\sQ\in N_{\bb{R}}$ and the scattering diagram $\f{D}^{\s{X}_t}$ as in \cite[\S 4.3]{davison2019strong}.  Furthermore, \cite{davison2019strong} defines $f\in \s{X}^{\can}_t\subset \s{A}^{\prin,\can}_t$ (cf. \S \ref{XupTheta}) to be \textbf{universally positive with respect to the principal coefficients scattering atlas}, or \textbf{principally positive} for short, if the nonzero coefficients of \begin{align*}
f_{\sQ}=\sum_{n\in \xi(N)\subset M^{\prin}} a_{f,n,\sQ} z^n
\end{align*} are positive for all $n\in \xi(N)$ (for $\xi$ as in \eqref{xi}) and all generic $\sQ\in M_{\bb{R}}^{\prin}$.  The nonzero principally positive elements which cannot be decomposed into sums of other nonzero principally positive elements are said to be \textbf{atomic with respect to the principal coefficients scattering atlas}.  By \cite[Thm. 1.1]{davison2019strong}, the theta functions in $\s{X}_t^{\can}$ are precisely the atomic elements with respect to either the scattering atlas or the principal coefficients scattering atlas, so these two notions of positivity and atomicity (i.e., with respect to the scattering atlas or the principal coefficients scattering atlas) are in fact equivalent.

The advantage of the principal coefficients scattering atlas perspective is that this atlas includes all of the $\s{X}$-space clusters, so this atlas contains the cluster atlas considered by Fock and Goncharov \cite{FockGoncharov06a,FG1}.  Proposition \ref{gdense} lifts to the principal coefficients setting (because walls of $\f{D}^{\s{A}^{\prin}}$ are parallel to $\ker(\rho)=(0,N_{\bb{R}})$), so we immediately obtain the following: 

\begin{prop}\label{prop:X-pos}
In general, the theta functions $\{\vartheta_n\}_{n\in N}\subset \s{X}_t^{\can}$ are precisely the elements which are atomic with respect to scattering atlas, or equivalently, the elements which are atomic with respect to the principal coefficients scattering atlas.

Now let $\Sigma$ be a triangulable marked surface, and consider the associated $\s{X}_t^{\can}$.  If no component of $\Sigma$ is a once-punctured closed surface, then positivity with respect to the principal coefficients scattering atlas (equivalently, with respect to the scattering atlas) is equivalent to positivity with respect to the $\s{X}$-space cluster atlas. 
\end{prop}

We now focus on the classical setting.  The set $\Theta^{\midd}$ is equal to $\Theta^{\midd}_{\bb{R}}\cap M$ for some convex cone $\Theta^{\midd}_{\bb{R}}\subset M_{\bb{R}}$.  Let $\Theta_{\bb{R}} \subset \Theta^{\midd}_{\bb{R}}$ be a subcone, and denote $\Theta:=\Theta_{\bb{R}}\cap M$ and
\begin{align*}
\s{A}^{\Theta}:=\bigoplus_{m\in \Theta} \kk\cdot \vartheta_m \subset \s{A}^{\midd}.
\end{align*}
Suppose $\s{A}^{\Theta}$ is closed under multiplication, i.e., forms a subalgebra of $\s{A}^{\midd}$.  This holds, for example, if $\Theta_{\bb{R}}=\Theta^{\midd}_{\bb{R}}$.  It also holds if $\Theta_{\bb{R}}$ is convex and either
\begin{itemize}
\item $\Theta_{\bb{R}}$ closed under addition by elements of $M^+$, or
\item $\Theta_{\bb{R}}\setminus \{0\}$ is closed under addition by elements of $M^+$.    
\end{itemize}
If $M^{\oplus}$ is not strongly convex, we also require $\Theta_{\bb{R}}\subset \?{\Xi}$ for $\Xi$ as in \eqref{eq:Xi}.

We say that an element $f:=\sum_{m\in \Theta} a_m \vartheta_m \in \s{A}^{\Theta}$ is \textbf{universally positive with respect to the $\Theta_{\bb{R}}$-atlas} if 
\begin{align*}
f_{\sQ}:=\sum_{m\in \Theta} a_m \vartheta_{m,\sQ} = \sum_{m\in M} a_{m,\sQ} z^m
\end{align*}
has non-negative integer coefficients $a_{m,\sQ}\in \bb{Z}_{\geq 0}$ for each $m\in M$ and for all generic $\sQ\in \Theta_{\bb{R}}$.  Here, if $M^{\oplus}$ is strongly convex, then $f_{\sQ}\in \kk\llb M\rrb$.  Otherwise, the assumption that $\Theta_{\bb{R}}\subset \?{\Xi}$ ensures that $f_{\sQ}$ is defined as an element of $\kk[M]$.  One now defines the notion of \textbf{atomic with respect to the $\Theta_{\bb{R}}$-atlas}, or $\Theta_{\bb{R}}$-atomic for short, in the obvious way.

\begin{prop}\label{Theta-pos}
The set $\{\vartheta_m\}_{m\in \Theta}$ consists precisely of the $\Theta_{\bb{R}}$-atomic elements of $\s{A}^{\Theta}$.
\end{prop}
The argument is essentially the same as in the atomicity proofs of \cite{ManAtomic} and \cite{davison2019strong}.
\begin{proof}
Let $f=\sum_{m\in \Theta} a_m \vartheta_m$.  Under the hypotheses, the $z^m$-coefficient of $f_{\sQ}$ is $a_m$ whenever $\sQ$ is sufficiently close to $m$ --- in the case of convex $M^{\oplus}_{\bb{R}}$ this is shown in \cite{gross2018canonical}, and for non-strongly convex $M^{\oplus}$ with $\Theta\subset \?{\Xi}$ it holds via the same argument used in the proof of Lemma \ref{pmQ}.  Now, if $f$ is universally positive with respect to the $\Theta_{\bb{R}}$-atlas, it follows that $a_m$ is non-negative for each $m$.   So the atomic elements must be the theta functions, as desired.
\end{proof}

Now consider a triangulable marked surface $\Sigma$ with ideal triangulation $\Delta$, and consider the untagged bracelets basis for the untagged skein algebra $\Sk(\Sigma)\subset \Sk^{\Box}(\Sigma)\subset \s{A}^{\up}(\sd_{\Delta})$ as in \S \ref{sec:bracelet_band}.  As we have noted, we will later show (cf. \S \ref{sec:theta_punctured}) that these bracelets are theta functions. We can take $\Theta$ to be the $g$-vectors of these untagged bracelets.  This $\Theta$ equals $\Theta_{\bb{R}}\cap M$ for some cone $\Theta_{\bb{R}}$, and $\s{A}^{\Theta}$ is indeed closed under multiplication since it equals $\Sk(\Sigma)$.  Since $\Theta^{\midd}=M$ and $\?{\Xi}=M_{\bb{R}}$ (Theorem \ref{fFG-surfaces}), we have $\Theta\subset \Theta^{\midd}_{\bb{R}}\cap \?{\Xi}$.  Furthermore, the clusters associated to generic points $\sQ\in \Theta_{\bb{R}}$ are precisely the charts of the ideal triangulation atlas---indeed, for each arc $\gamma$ in an ideal triangulation, $[\gamma]$ is a cluster variable in the corresponding cluster, unless $\gamma$ is a noose bounding an arc $\gamma'$, in which case $\frac{[\gamma]}{[\gamma']}$ is the corresponding cluster variable (cf. Example \ref{ex:noose}).  Proposition \ref{Theta-pos} thus yields the following:

\begin{lem}\label{lem:untag_skein_atom_basis}
Assume we know that the untagged bracelets basis for $\Sk(\Sigma)$ consists of theta functions.  Then it consists precisely of the atomic elements of $\Sk(\Sigma)$ with respect to the ideal triangulation atlas.
\end{lem}

\subsection{Positivity of bracelets}\label{sec:positivity_skein_bracelets}
The following positivity properties for bracelets are immediate consequences of results in the literature.

\begin{thm}[\cite{Thurst}, Thm. 1]
\label{Thurst-Brac-Bas} The bases $\Brac(\Sigma)$ and $\?{\Brac}(\Sigma)$
for $\Sk(\Sigma)$ and $\?{\Sk}(\Sigma)$, respectively,
are strongly positive. 
\end{thm}
\begin{proof}
The $\?{\Brac}(\Sigma)$-case is \cite[Thm. 1]{Thurst}.
Extending to $\Brac(\Sigma)$ is easy because the boundary curves
have no crossings with other curves (up to homotopy) and so multiplication
by possibly negative powers of these elements is easily understood. 
\end{proof}

\begin{lem}\label{ClassicalPosNoPunct}
In the classical setting, if $\Sigma$ has no punctures, then the bracelets are universally positive with respect to the scattering atlas.
\end{lem}
\begin{proof}
For unpunctured $\Sigma$, the bracelets include all the cluster monomials (they are monomials in the arcs of ideal triangulations; cf. Proposition \ref{SkAProp}), and Theorem \ref{Thurst-Brac-Bas} says that the classical bracelets basis is strongly positive.  So Lemma \ref{PosImplications} implies that the bracelets are universally positive with respect to the cluster atlas.  The cluster and scattering atlases agree by Proposition \ref{gdense}, so the claim follows.
\end{proof}

\begin{lem}\label{Lem:pos-1-p}
Let $\Sigma$ be a once-punctured closed surface.  In the classical setting, the untagged bracelets for $\Sk(\Sigma)$ are universally positive with respect to the cluster atlas.
\end{lem}
\begin{proof}
For once-punctured closed surfaces, the plain bracelets again include all cluster monomials (cf. Proposition \ref{SkAPropPun}), so the argument follows as in the proof of Lemma \ref{ClassicalPosNoPunct}.
\end{proof}

\subsection{Bar-invariance of quantum theta functions}\label{sec:bar}
We note one more property of quantum theta functions which will be useful to us.  Let $f$ be an element of $\kk_t$, or of a quantum torus algebra like $\kk_t[N]$ or $\kk_t[M]$, or of a completion of such an algebra like $\kk_t\llb N\rrb$ or $\kk_t\llb M\rrb$.  Then $f$ is said to be \textbf{bar-invariant} if it is invariant under the involution $t\mapsto t^{-1}$.  
Scattering automorphisms preserve bar-invariance by \cite[Thm. 2.15 and Lem. 3.4]{davison2019strong}, so for any $f\in \s{A}_t^{\can}$ or $\s{X}_t^{\can}$, $\iota_{\sQ}(f)$ (defined like in Lemma \ref{lem:fQ}) is bar-invariant for one generic $\sQ$ if and only if it is bar-invariant for all generic $\sQ$.  In this case, we say that $f$ is bar-invariant.
\begin{lem}[\cite{davison2019strong}, Thm. 1.1]\label{lem:theta-bar-inv}
Quantum theta functions are bar-invariant.
\end{lem}

\begin{lem}\label{lem:1theta-implies-qtheta}
Left $f$ be a nonzero, bar-invariant, theta positive element of $\s{A}_t^{\can}$ or $\s{X}_t^{\can}$.  Then $f$ is a quantum theta function if and only if $\lim_{t\rar 1}(f)$ is a classical theta function in $\s{A}^{\can}$ or $\s{X}^{\can}$, respectively.
\end{lem}
\begin{proof}
The bar-invariance and theta positivity imply that $f=\sum_p a_p \vartheta_p$ for some bar-invariant elements $a_p\in \bb{Z}_{\geq 0} [t^{\pm 1}]$.  If this gives a theta function $\vartheta_{p_0}|_{t=1}$ in the classical limit, then $a_{p_0}|_{t=1}=1$, and $a_p|_{t=1}=0$ for $p\neq p_0$.  The only possibility is that $a_{p_0}=1$ and $a_{p}=0$ for $p\neq p_0$, so $f=\vartheta_{p_0}$ as claimed.
\end{proof}

\section{Operations on seeds}\label{sec:seed_change}
We have seen that the classical skein algebras $\Sk(\Sigma)$ and quantum skein algebras $\Sk_t(\Sigma)$ (defined for unpunctured $\Sigma$) are contained in an associated (quantum) upper cluster algebras with coefficients (i.e., frozen variables) associated to the boundary arcs. From the viewpoint of cluster algebras, it is natural to consider cluster algebras defined by seeds with other coefficients. In this section, we discuss some ways to manipulate the coefficients of a cluster algebra, and we explore some general properties of these manipulations.  This will allow us, for example, to quantize skein algebras for punctured surfaces by first extending to principal coefficients.

Additionally, we will see that operation of gluing two boundary arcs of a marked surface corresponds to the operation of identifying two frozen indices, followed by unfreezing the new index.  We will discuss what happens to theta functions when seeds are changed in this way.

\subsection{Dominance order and pointedness}
\label{sec:pointedness}
Let there be given any seed $\sd=(N,E,I,F,\omega)$. Recall the notation $$N_{\sd}^{\oplus} := N^{\oplus}:= \bigoplus_{k\in I_{\uf}} \bb{Z}_{\geq 0} e_k$$
and
$$M_{\sd}^\oplus := M^\oplus := \omega_1(N^{\oplus}).$$
For the purpose of the following notions (Definitions \ref{defn:pointedness}-\ref{defn:supp}), we assume that $M_{\sd}^\oplus$ is strongly convex. This assumption is satisfied, for example, when the Injectivity Assumption holds for $\sd$. Then we can define the dominance order on $M$ as the following:
\begin{defn}[Dominance order \cite{qin2017triangular}]\label{def:dom}
For any $m,m'\in M$, we say $m'$ is \textbf{dominated} by $m$, denoted by $m'\preceq_{\sd} m$, if $m'\in m+M_{\sd}^{\oplus}$.
\end{defn}

We also consider the \textbf{quantum seed} $\sd=(N,E,I,F,\omega,\Lambda)$ --- writing the compatible form $\Lambda$ as part of the seed data --- whose compatible pair determines the positive number $d$ given by \eqref{Lambda-B}. Recall that the Injectivity Assumption is automatically satisfied by a quantum seed.

Recall the completions 
\begin{eqnarray*}
\kk_{t}\llb N\rrb & = & \kk_{t}[N]\otimes_{\kk_{t}[N^{\oplus}]}\kk_{t}\llb N^{\oplus}\rrb\\
\kk_{t}\llb M\rrb & = & \kk_{t}[M]\otimes_{\kk_{t}[M^{\oplus}]}\kk_{t}\llb M^{\oplus}\rrb
\end{eqnarray*}
as in \S \ref{qtoralg}.
We sometimes denote $ \kk_{t}\llb M\rrb$ by $\kk_{t}\llb M_\sd\rrb$ to emphasize its dependence on $\sd$. Recall that, for working with the classical case, it suffices to evaluate $t=1$.

\begin{defn}[Degree and pointed elements {\cite{qin2017triangular}}]\label{defn:pointedness}
Assume that a given element $Z\in\kk_{t}\llb M_\sd\rrb$ takes the form
\begin{align}\label{eq:pointed_func}
	Z = \sum_{n\in N^\oplus}c_{n}(t)z^{m+\omega_1(n)}
\end{align}
for some $m\in M$ and some coefficients $c_n(t)\in \kk_{t}$. If $c_0\neq 0$, then $Z$ is said to have \textbf{degree} $m$, denoted by $\deg Z=m$, and its $F$-function is defined as the formal Laurent series $F_{Z}:=\sum c_n(t) z^{\omega_1 (n)}\in \kk_t\llb M^{\oplus}\rrb$.

If further $c_0=1$, then $Z$ is said to be \textbf{$m$-pointed}.   A set $\{Z_{m}|m\in M\}$ is said to be \textbf{$M$-pointed} if its elements
$Z_{m}$ are $m$-pointed.
\end{defn}

For an $m$-pointed element $Z$, we may refer to its degree $m$ as the \textbf{$g$-vector} of $Z$ and denote $g(Z)=m$ (this generalizes the extended $g$-vectors of \S \ref{sub:g}). By Proposition \ref{TopBasis}, a theta function $\vartheta_{m,\sQ}$ is $m$-pointed, i.e., has $g$-vector $m$, and this is independent of the choice of $\sQ$.

Note that $Z$ has degree $m$ if and only if $m$ is the unique $\prec_{\sd}$-maximal degree for the monomial terms of $Z$.

\begin{defn}[Support]\label{defn:supp}
Assume that the Injectivity Assumption holds. The \textbf{support} of $Z\in\kk_{t}\llb M_\sd\rrb$ (in $I_{\ufv}$) is defined as
\begin{align}\label{eq:supp}
	\supp Z=\supp_{I_{\ufv}} Z=\{i\in I_{\ufv} | \exists n=(n_i)\in \Z^{I_\ufv},\ c_n(t)\neq 0,\ n_i>0\}.
\end{align}
\end{defn}

\subsection{Similar seeds}
\label{sec:similarity}

In order to change the coefficients of a given seed, we recall the notion of similarity following \cite{Qin12,qin2017triangular}. 

Let there be given two seeds $\sd=(N,E,F,I,\omega)$ and $\sd'=(N',E',F',I',\omega')$ (or qauntum seeds $\sd$, $\sd'$ with the Lambda matrices $\Lambda$ and $\Lambda'$, which determine positive numbers $d$ and $d'$ respectively via \eqref{Lambda-B}). If there is a bijection $\sigma$ from $I_{\uf}=I\setminus F$ to $I'_{\uf}=I'\setminus F'$, for any vector $p=(p_i)_{i\in I_\ufv}$ with coordinates $p_i\in \R$, define $\sigma p:=(\sigma p_{i'})_{i'\in I'_\ufv}$ such that $\sigma p_{\sigma i}=p_i$. I.e., we may view $\sigma$ as identifying $N^{\oplus}$ with $N'^{\oplus}$ and $M^{\oplus}$ with $M'^{\oplus}$.

\begin{defn}[Similar seeds {\cite{qin2017triangular}}]\label{def:similar_seed}
The seed $\sd'$ is said to be \textbf{similar} to $\sd$ up to an identification $\sigma:I_\ufv\simeq I'_\ufv$, if they share the
same principal $B$-matrix up to $\sigma$, i.e. $b_{i,j}(\sd)=b_{\sigma i, \sigma j}(\sd')$ for $i,j\in I_{\ufv}$. 
\end{defn}

Let $\s{A}^{\bullet}(\sd)$ and $\s{A}^{\bullet}(\sd')$ denote the (quantum) cluster algebras associated to the (quantum) seeds $\sd$ and $\sd'$ respectively, where $\bullet$ stands for $\ord$, $\midd$, $\can$ or $\up$. 
\begin{defn}[Similar cluster algebras]\label{def:sim_alg}
If $\sd$ and $\sd'$ are similar, then  $\s{A}^{\bullet}(\sd)$ and $\s{A}^{\bullet}(\sd')$ are said to be \textbf{similar} or of the same \textbf{type}.
\end{defn}

From now on, assume that $\sd$ and $\sd'$ are similar up to a permutation $\sigma$ and, moreover, $M^\oplus$ and ${M'}^\oplus$ are strongly convex. We work over a ring $\kk_t=\kk[t^{\pm \frac{1}{D}}]$ such that $d|D$ and $d'|D$, see \S \ref{qtoralg}.  Recall the notation $\pr_{I_{\uf}}:\bb{Z}^I\rar \bb{Z}^{I_{\uf}}$ as in \S \ref{sub:g}.

\begin{defn}[Similar pointed elements]\label{def:similar_element}
An $m'$-pointed element $Z'\in\kk_{t}\llb M'_{\sd'}\rrb$ is said to be similar to the $m$-pointed
element $Z\in\kk_{t}\llb M_\sd \rrb$ in \eqref{eq:pointed_func} if $\pr_{I'_{\uf}}m'=\sigma(\pr_{I_{\uf}}m)$ and $Z'$ takes the form\footnote{The decomposition coefficients $c_n$ for $Z$ in \eqref{eq:pointed_func} might not be unique when $M^\oplus$ is strongly convex but the Injectivity Assumption fails. Nevertheless, $Z'$ is uniquely determined by $Z$ and $m'$.} 
\begin{eqnarray*}
	Z' = \sum_{n}c_{n}(t^{\frac{d'}{d}})z^{m'+\omega'_1(\sigma n)}.
\end{eqnarray*}
\end{defn}
Note that the similar element $Z'$ has the same $F$-function as that of $Z$ (after replacing $t$ by $t^{\frac{d'}{d}}$). $Z'$ is thus determined by $Z$ up to a frozen factor---more precisely, up to a shift of all exponents of $Z'$ by a vector in the span of $\{(e_i')^*|i\in F'\}$.

Let $\f{D}$ and $\f{D}'$ denote the scattering diagrams associated to the initial seeds $\sd$ and $\sd'$, respectively, as in \eqref{DIn} and \eqref{DAt}. Then they provide the following similar elements:
\begin{lem}\label{lem:similar_theta_function}
If $\vartheta_{m}$ is an $m$-pointed theta function for $\f{D}$, $m\in M$, then
the similar pointed elements in $\kk_t\llb M'\rrb$ are the theta functions $\vartheta_{m'}$ for $\f{D}'$, where $\pr_{I'_{\uf}}m'=\sigma(\pr_{I_{\uf}}m)$.
\end{lem}
\begin{proof}
Since the injectivity assumption holds, we may fix a compatible $\Lambda$ and $\Lambda'$ if these have not been fixed already.   Recall $M_F=N_{\uf}^{\perp}\subset M$; i.e., $M_F$ denotes the span of $\{e_i^*|i\in F\}$.  In general, the exponents on scattering functions of $\f{D}$ lie in $M^+=\omega_1(N^+)\subset \omega_1(N_{\uf})$, and by \eqref{Lambda-B}, $\Lambda_2(\omega_1(N_{\uf}))\subset N_{\uf}$.  It follows that all walls of $\f{D}$ are closed under addition by $M_{F,\bb{R}}$, and that all scattering automorphisms act trivially on $\kk_t[M_F]$. This further implies that, for any $u\in M_F$ and any generic $\sQ,\sQ'\in M_{\bb{R}}$ with $\sQ-\sQ'\in M_{F,\bb{R}}$, there is a bijection between broken lines contributing to $\vartheta_{m+u,\sQ'}$ and broken lines contributing to $\vartheta_{m,\sQ}$ with attached monomials on the former being obtained from the corresponding monomials on the latter by adding $u$ to the exponents.  Hence, $\vartheta_{m+u,\sQ'}$ is obtained from $\vartheta_{m,\sQ}$ by adding $u$ to every exponent.

Relabeling the basis vectors of $\sd'$ if necessary, we can assume that the permutation $\sigma$ is trivial.  Now consider the quantum seed $$\wt{\sd}:=(\wt{N}\coloneqq N\oplus N_{F'},I\sqcup F',\wt{E}=\{\wt{e}_i\}_{i\in I\sqcup F'},F\sqcup F',\wt{\omega},\wt{\Lambda}).$$  Here, $\wt{\omega}(\wt{e}_i,\wt{e}_j)$ is defined to equal $\omega(e_i,e_j)$ if $i,j$ are both in $I$, $\omega'(e_i,e_j)$ if $i,j$ are both in $I'=I_{\uf}\sqcup F'$, and $0$ if one index is in $F$ while the other is in $F'$.  Similarly, we extend $\Lambda$ and $\Lambda'$ to $\wt{M}$ via $\Lambda(m,e_j^*)\coloneqq 0$ for $j\in F'$ and $\Lambda'(m,e_j^*)\coloneqq 0$ for $j\in F$.  Then the sum $\wt{\Lambda}$ of these extended forms is a compatible bilinear form for $\wt{s}$ such that $\wt{\Lambda}_2(\wt{\omega}_1 (e_k))=(d+d') e_k$, $k\in I_{\uf}$.
  
The arguments from the first paragraph of the proof apply to this seed to show that, for $\sQ,\sQ'$ generic and differing by a point in $\wt{M}_{F\sqcup F',\bb{R}}$, and for any $u\in M_{F\sqcup F'}$, $\vartheta_{m+u,\sQ'}$ can be obtained from $\vartheta_{m,\sQ}$ by adding $u$ to every exponent.  Let $\pr_I$ and $\pr_{I'}$ denote the natural projections of $\wt{M}$ onto $M$ and $M'$, respectively.  These induce morphisms of the quantum torus Lie algebras $\pr_I:\f{g}^t(\wt{\sd})\rar \f{g}^t(\sd)$ and $\pr_{I'}:\f{g}^t(\wt{\sd})\rar \f{g}^t(\sd')$, as well as morphisms of $\kk$-modules $\pr_I : \kk_t\llb \wt{M}\rrb\rar \kk_t\llb M\rrb$ and $\pr_{I'}:\kk_t\llb \wt{M}\rrb\rar \kk_t\llb M'\rrb$ which intertwine the respective Lie algebra actions, via applying $\pr_I$ or $\pr_{I'}$ to the exponents while setting $\pr_I:t\mapsto t^{\frac{d}{d+d'}}$ and $\pr_{I'}:t\mapsto t^{\frac{d'}{d+d'}}$.   Note that $$\pr_I(\vartheta_{\wt{m},\sQ})=\vartheta_{\pr_I(\wt{m}),\pr_I(\sQ)}\quad \text{and} \quad \pr_{I'}(\vartheta_{\wt{m},\sQ})=\vartheta_{\pr_{I'}(\wt{m}),\pr_{I'}(\sQ)}$$
    for all $\wt{m}\in \wt{M}$ and generic $\sQ\in \wt{M}_{\bb{R}}$.  We thus see that, for any $\wt{m}\in \wt{M}$ and generic $\sQ\in \wt{M}_{\bb{R}}$, $\vartheta_{m,\pr_I(\sQ)}$ for $m=\pr_I(\wt{m})$ is similar to every element of the form $\vartheta_{\pr_{I'}(\wt{m}+u),\pr_{I'}(\sQ')}$ for $u\in \wt{M}_{F\sqcup F'}$ and $\sQ'-\sQ\in M_{\bb{R}}$.  These elements $\pr_{I'}(\wt{m}+u)$ are precisely the $m'\in M'$ with $\pr_{I'_{\uf}} m' = \sigma(\pr_{I_{\uf}} m)$, so the claim follows.
\end{proof}

Now suppose that a classical seed $\sd$ satisfies the condition that $M^\oplus$ is strongly convex, but consider a similar seed $\sd'$ which might not satisfy this convexity condition.  We can still adapt Definition \ref{def:similar_element} for $\sd'$. More precisely, for any $m$-pointed Laurent polynomial $Z=\sum_{n\in N^{\oplus}} c_n z^{m+\omega_1(n)}\in \kk[M_\sd]$, $c_n\in \kk$, and any $m'$ such that $\pr_{I'_{\uf}}m'=\sigma(\pr_{I_{\uf}}m)$, we can construct a Laurent polynomial $Z'\in\kk[M_{\sd'}]$ by:
\begin{align}\label{eq:projection}
Z':=& \sum_{n\in N^{\oplus}}c_{n}z^{m'+\omega'_1(\sigma n)}\\
=&\sum_{p\in {M'}^{\oplus}}\left(\sum_{\sigma n\in {\omega'_1}^{-1}(p)} c_n\right)z^{m'+p}.\nonumber
\end{align}
E.g., the restriction map $\rho$ on theta functions (see \eqref{theta-rho}) is the special case of this construction with $M^{\prin}$ in place of $M$ and with $M$ in place of $M'$.

We see that the $t=1$ version of Lemma \ref{lem:similar_theta_function} holds in this setting as well by essentially the same arguments, but with the added assumption that $m\in \Theta^{\midd}$.

\begin{lem}\label{lem:similar_theta_function_2}
If $\vartheta_{m}$ is an $m$-pointed theta function for $\f{D}$, $m\in \Theta^{\midd}\subset M$, then
the similar elements in $\kk[M']$ as in \eqref{eq:projection} are the theta functions $\vartheta_{m'}$ for $\f{D}'$ where $\pr_{I'_{\uf}}m'=\sigma(\pr_{I_{\uf}}m)$.
\end{lem}

\subsection{Gluing frozen vertices}\label{sec:gluing_frozen_vertices}
Let $\sd=(N,I,F,E,\omega)$ be a (classical) seed. 

Let $b_1,b_2$ denote two frozen vertices of the quiver associated to $\sd$ (cf. Remark \ref{QS}). We can glue $b_1,b_2$ into a single frozen vertex, which we denote by $b$. If this gluing results in any loops at $b$ (i.e., arrows from $b$ to $b$), we simply remove these loops from the quiver (this happens if there were arrows between $b_1$ and $b_2$, but this will not be the case in any examples we care about).

Correspondingly, we can construct the seed $\?{\sd}=(\?{N},\?{I},\?{F},\?{E},\?{\omega})$ by gluing $\sd$ as follows. Replace $N$ with $\?{N}:=N/\bb{Z}\langle e_{b_1}-e_{b_2}\rangle$, replace $b_1$ and $b_2$ in $I$ and $F$ with $b$, take $e_b\in E$ to be the projection of either $e_{b_i}$ ($i=1,2$) to $\?{N}$, and replace $\omega$ with $\?{\omega}$ given as follows: if $i,j\neq b$, then $\?{\omega}(e_i,e_j)=\omega(e_i,e_j)$, and
\begin{align*}
\?{\omega}(e_i,e_b)=\omega(e_i,e_{b_1}+e_{b_2}).
\end{align*}
The dual lattice $\?{M}$ is the sublattice of $M$ spanned by $\{e^*_i\}_{i\in I\setminus \{b_1,b_2\}} \cup \{e_b^*:=e^*_{b_1}+e^*_{b_2}\}$.

In general, for any equivalence relation on $F$, we can glue the vertices $i$ in each equivalent class $[i]$ by iterating the above gluing process. The resulting seed $\?{\sd}$ satisfies
\begin{align}\label{eq:gluing_omega}
	\?{\omega}(e_{[i]},e_{[j]})=\sum_{i\in [i],j\in [j]}\omega(e_i,e_j),
\end{align}
where we denote $[k]=k$ for $k\in I\setminus F$. In particular, the resulting seed does not depend on the order of our gluing process. Notice that the natural embedding from $\?{M}$ to $M$ sends $e_{[i]}^*$ to $\sum_{i\in [i]}e_i^*$.

As we shall see in \S \ref{Sec:glue}, the gluing operation appears naturally on surfaces.

We define a linear map 
\begin{align}\label{piM}
	\pi_M:M\mapsto \?{M},\quad m\mapsto \?{m}    
\end{align}
by specifying that $\?{e}^*_i=e^*_{[i]}$.  
Elements $\pi_M(m)$ will typically be denoted as just $\?{m}$. If $\sd$ and $\?{\sd}$ are quantum seeds, let $d$ and $\?{d}$ denote the numbers determined by their respective compatible pairs. We let $\pi_{M}$ denote the induced $\kk$-linear map from $\kk_t\llb M \rrb$ to $\kk_t\llb \?{M}\rrb$ such that $\pi_M(t)=t^{\?{d}/{d}}$. We do NOT require that $\pi_M$ respects multiplication.

Note that $\sd$ and $\?{\sd}$ are similar seeds. We have the following observation:
\begin{align}\label{pi_cluster_poisson}
	\pi_M(\omega_1(n))=\?{\omega}_1(\?{n}).
\end{align}

Similarly, consider the principal coefficients seeds $\sd^{\prin}$ and $\?{\sd}^{\prin}=(\?{\sd})^{\prin}$ as in \S \ref{prinsub}. We glue $\sd^{\prin}$ into a seed $\?{\sd^{\prin}}$ by gluing two elements of $I_2$ whenever the corresponding elements in $I_1$ are glued. Let $\pi_{M^{\prin}}$ denote the induced $\kk$-linear map from $\kk\llb M^{\prin} \rrb$ to $\kk\llb M^{\?{\sd^{\prin}}} \rrb$.  Note that $\?{\sd^{\prin}}$ and $\?{\sd}^{\prin}$ agree except for the fact that $\omega^{\?{\sd}^{\prin}}(e_{[j]},e_{[j']})=1$ when $j\in F_1$ and $j'$ is the corresponding element of $I_2$, whereas $\omega^{\?{\sd^{\prin}}}(e_{[j]},e_{[j']})$ equals the cardinality of $[j]$).  We thus naturally identify $\kk\llb M^{\?{\sd^{\prin}}} \rrb$ with $\kk\llb M^{\?{\sd}^{\prin}} \rrb$. Then, for any $n\in \Z^{I_{\uf}}$, we have
\begin{align}\label{eq:pi_principal}
	\pi_{M^{\prin}}(\omega_1 ^{\sd^{\prin}}(n) )=\omega^{\?{\sd}^{\prin}}_1(\?{n}).
\end{align}

We deduce that for any $m$-pointed element $Z$ in $\kk\llb M^{\prin} \rrb$ (or in $\kk_t\llb M^{\prin} \rrb$), $\pi_{M^{\prin}}(Z)$ is similar to $Z$. In particular, using Lemma \ref{lem:similar_theta_function}, 
\begin{align}\label{eq:piM-prin}
	\text{$\pi_{M^{\prin}}$ sends theta function for $\sd^{\prin}$ to theta functions for $\?{\sd}^{\prin}$.}
\end{align} Consequently, using \eqref{theta-rho}, for all $m\in M$ and generic $\sQ\in C_{\sd}^+$ and $\?{\sQ}\in C_{\?{\sd}}^+$,
\begin{align}\label{pi-theta}
	\pi_M(\vartheta_{m,\sQ}) = \vartheta_{\?{m},\?{\sQ}}
\end{align}
in $\s{A}^{\can}(\?{\sd})$ (or in $\s{A}_{t}^{\can}(\?{\sd})$ if $\sd$ and $\?{\sd}$ are quantum seeds).

\begin{rem}
	We note that \eqref{pi-theta} can alternatively be derived as a consequence of \cite[Lem. 2.11]{davison2019strong}.  Indeed, loc. cit. applies to relate the scattering diagrams for $\sd$ and $\?{\sd}$, and then it is straightforward to apply the construction of broken lines and theta functions to reach the desired equality.
\end{rem}

\subsection{Union of seeds}\label{sec:union}

Given two seeds $\sd_i=(N_i,I_i,E_i,F_i,\omega^{(i)})$, $i=1,2$, define their union to be the seed 
\begin{align}\label{eq:seed-union}
	\sd=(N_1\oplus N_2,I_1\sqcup I_2, E_1\sqcup E_2, F_1\sqcup F_2,\omega^{(1)}\oplus \omega^{(2)}).    
\end{align}
Similarly, for quantum seeds $\sd_1,\sd_2$ with the same multiplier $d$, we define their union with the additional matrix $\Lambda(\sd)=\Lambda_1\oplus \Lambda_2$. Let $\s{A}_{t}^{\can}(\sd_i)$, $\s{A}_{t}^{\can}(\sd)$ denote the corresponding (quantum) canonical cluster algebras, and similarly the corresponding ordinary, middle, and upper (quantum) cluster algebras. For any $g_i\in M_i$, $i=1,2$, let $\vartheta_{g_i}$ and $\vartheta_{g_1+g_2}$ denote the corresponding theta function in $\s{A}_{t}^{\can}(\sd_i)$ and $\s{A}_{t}^{\can}(\sd)$ respectively.

Note that we can naturally view $\kk_t[ M_i]$ as a subalgebra of $\kk_t[M_1\oplus M_2 ]$, $i=1,2$. If $\sd_i$, $i=1,2$, satisfies that $M_i^\oplus$ is strongly convex, then we can define $\kk_t\llb M_i\rrb$, which can be viewed as subalgebras of $\kk_t\llb M_1\oplus M_2 \rrb$.
\begin{lem}\label{lem:theta-union-product}
	(1) If $\sd_1$ and $\sd_2$ satisfy the Injectivity Assumption, then $\vartheta_{g_1}\vartheta_{g_2}=\vartheta_{g_1+g_2}$ (in both the quantum and classical settings).
	
	(2) Even in the possible absence of the Injectivity Assumption, if $\vartheta_{g_i}\in \s{A}^{\midd}(\sd_i)$, $i=1,2$, then $\vartheta_{g_1}\vartheta_{g_2}=\vartheta_{g_1+g_2}$ in $\s{A}^{\midd}(\sd)$.
\end{lem}
\begin{proof}
	(1) We observe that the scattering diagram for $\sd$ is given by the direct sum of the scattering diagrams for $\sd_1$ and $\sd_2$, i.e.
	\begin{align*}
		\f{D}(\sd) &= \f{D}(\sd_1)\oplus \f{D}(\sd_2)\\
		&:= \{(\f{d}\oplus M_{2,\bb{R}},g) | (\f{d},g)\in\f{D}(\sd_1) \} \cup \{(M_{1,\bb{R}}\oplus \f{d},g)  | (\f{d},g)\in \f{D}(\sd_2)\}.
	\end{align*}
	For the initial scattering diagram this follows immediately from the definition \eqref{DIn}, and then the equality for the corresponding consistent scattering diagrams comes from the consistency of $\f{D}(\sd_1)$ and $\f{D}(\sd_2)$, plus the fact that the two sub Lie algebras associated to $\sd_1$ and $\sd_2$ commute with each other.

	Now fix $p=(p_1,p_2)\in M_1\oplus M_2$ and generic $\sQ=(\sQ_1,\sQ_2)\in M_{\bb{R}}=M_{1,\bb{R}}\oplus M_{2,\bb{R}}$.  Then the broken lines contributing to $\vartheta_{p,\sQ}$ are given as follows: the supports are $\ell:(-\infty,0]\rar M_{\bb{R}}$ with $\ell=(\ell_1,\ell_2)$, where $\ell_1$ and $\ell_2$ are the supports of broken lines for $\vartheta_{p_i,\sQ_1}$ and $\vartheta_{p_2,\sQ_2}$, respectively.  The attached monomials are the products of the pairs of attached monomials for the corresponding segments of $\ell_1$ and $\ell_2$.  It follows that $\vartheta_{p,\sQ}=\vartheta_{p_1,\sQ_1}\vartheta_{p_2,\sQ_2}$, as desired.
	
	(2) We first apply statement (1) to the principal coefficient seed $\sd_i^{\prin}$. Then (2) follows from applying the projections as in \eqref{theta-rho}. 
\end{proof}

\begin{rem}\label{rem:surface-union}
	If a triangulable marked surface $\Sigma$ contains multiple connected components, then the seed associated to $\Sigma$ is the union (in the above sense) of the seeds associated to the components of $\Sigma$.  Thus, when investigating theta functions for cluster algebras from marked surfaces, Lemma \ref{lem:theta-union-product} allows us to easily reduce to cases where the surface is connected.
\end{rem}

\subsection{Unfreezing}\label{sec:unfreezing}

Let $\sd'$ denote a (quantum) seed, and $\sd$ the (quantum) seed obtained from $\sd'$ by freezing some $b\in I'_{\uf}$ (so $b\in F$ but not $F'$). The seeds $\sd$ and $\sd'$ are NOT similar in our previous sense, but the unfreezing operation will be important when we treat surfaces later.

Note that we have a natural identification $M_{\sd}=M_{\sd'}$ inducing 
\begin{align}\label{eq:unfreezing_identification}
	\f{i}:\kk_t[M_{\sd}]=\kk_t[ M_{\sd'}]
\end{align}
and the natural inclusion $\f{i}:M_{\sd}^\oplus \hookrightarrow M_{\sd'}^\oplus$. Correspondingly, if $M_{\sd}^\oplus$ and $M_{\sd'}^\oplus$ are strongly convex, we have the natural inclusion 
\begin{align}\label{eq:unfreezing_inclusion}
	\f{i}:\kk_t\llb M_{\sd}\rrb\hookrightarrow\kk_t\llb M_{\sd'}\rrb.
\end{align}
In this case, for any $m$-pointed element $Z\in \kk_t\llb M_{\sd}\rrb$, $\f{i}(Z)$ remains $m$-pointed with the same $F$-function.

We assume for the rest of this subsection that $\sd$ and $\sd'$ satisfy the Injectivity Assumption, or at least that $M_{\sd'}^{\oplus}$ (hence $M_{\sd}^{\oplus}$) is strongly convex.  Let $\f{D}$ and $\f{D}'$ denote the scattering diagrams associated to $\sd$ and $\sd'$ respectively. For $m\in M_{\sd}$, let $\vartheta_m$ and $\vartheta'_{m}$ denote the corresponding $m$-pointed theta functions in $\f{D}$ and $\f{D}'$ respectively.
We will present sufficient conditions for $\f{i}(\vartheta_m)=\vartheta_{m}$.

Since $\f{i}(\vartheta_m)$ is an $m$-pointed element in $\kk_t\llb M_{\sd'}\rrb$, we have a unique decomposition (convergent in $\kk_t\llb M_{\sd'}\rrb$, see \cite[Rmk. 2.5]{davison2019strong},\cite[\S 4]{qin2019bases}):
\begin{align}\label{eq:decompose_unfreeze}
	\f{i}(\vartheta_m)=\sum_{p\in m+M_{\sd'}^\oplus} c_{p}\vartheta'_{p}.
\end{align}
with $c_{m}=1$.

\begin{lem}\label{lem:positive_unfreeze_theta}
	Assume that ${M'}^{\oplus}$ is strongly convex. If the coefficients $c_{p}$ in \eqref{eq:decompose_unfreeze} are non-negative in the classical ($t=1$) limit, then $\f{i}(\vartheta_m)=\vartheta'_{m}$.
\end{lem}
\begin{proof}
	Let us compute the theta functions using broken lines. In terms of scattering diagrams, unfreezing a frozen vertex results in a new initial wall, hence more broken lines, all of which still have positive coefficients. So the Laurent expansion of $$\vartheta'_m-\f{i}(\vartheta_m)=\sum_{p\in m+M_{\sd'}^+} (-c_{p}\vartheta'_{p})$$ must have non-negative coefficients. Given the Laurent positivity of the theta functions $\vartheta'_{p}|_{t=1}$ and the assumption that the coefficients $c_{p}|_{t=1}$ are non-negative, the only possibility is that $c_p=0$ for all $p\in m+M_{\sd'}^+$.
\end{proof}

The following result is a special case of the application of the freezing operators in \cite{qin2022freezing} which suffices for our purpose. Recall that the seed $\sd$ is obtained from $\sd'$ by freezing the vertex $b$.  We use the notation $\supp$ as in \eqref{eq:supp}.
\begin{lem}\label{lem:same_support_theta_func}
	Assume that $s'$ satisfies the Injectivity Assumption. If for some $m\in M$, we have $b\notin\supp \vartheta'_m$, then $\f{i}(\vartheta_m)=\vartheta'_m$.
\end{lem}
\begin{proof}
	In terms of scattering diagrams, unfreezing a frozen vertex $b$ results a new incoming wall with normal direction $e_b$ and, correspondingly, more walls whose normal directions $n$ satisfying $n_b\neq 0$ (see the construction of consistent scattering diagrams in \cite[Appx. C]{gross2018canonical}). Therefore, $b\notin \supp  \vartheta'_{m}$ implies that the broken lines appearing in $\vartheta'_{m}$ only bend at those walls already appearing in $\f{D}$. Consequently, we get $\vartheta'_{m}=\f{i}(\vartheta_m)$.
\end{proof}

\begin{prop}\label{prop:cluster-commute-theta}
    Let $\sd'$ be a quantum seed.  Consider a quantum theta function $\vartheta_p$ and a cluster monomial $\vartheta_m$ (so $m\in \s{C}\cap M$) for $\sd'$ such that $\vartheta_p$ and $\vartheta_m$ $t$-commute.\footnote{Two elements $a$ and $b$ are said to $t$-commute if $ab=t^kba$ for some number $k$.} Then $t^{\alpha} \vartheta_p\vartheta_m =\vartheta_{p+m}$ where $\alpha=-\Lambda'(p,m)$ (to make the product bar-invariant).
\end{prop}
\begin{proof}
	One inductively reduces to the case where $\vartheta_m$ is a cluster variable (or possibly the inverse of a frozen cluster variable).  By changing our base seed $\sd'$, we may assume then that $m=e_i^*$ for some $i\in I$ (or possibly $-e_i^*$ if $i\in F$).  The case $i\in F$ is a simple general fact, cf. the first paragraph from the proof of Lemma \ref{lem:similar_theta_function}.  So we may assume $i\in I\setminus F$.
	
	Let $\sd$ be the seed obtained by freezing the vertex $i$.
	The $t$-commutativity implies that $i\notin \supp \vartheta_{p,\sQ}$ (for $\sQ$ in the positive chamber).  So, by the positivity of theta functions, since  $\vartheta_{p+e_i^*}$ is a summand of $t^{\alpha}\vartheta_p\vartheta_{e_i^*}$, $\vartheta_{p+e_i^*,\sQ}$ must also have support outside of $i$.  So the broken lines of $\vartheta_{p,\sQ}$ and $\vartheta_{p+e_i^
	*,\sQ}$ only bend at the walls that are present in $\f{D}(\sd)$, and as a result are the same for $\sd'$ as for $\sd$.  We thus reduce to checking the equality for $\sd$, where $i$ is frozen, and as noted above, the equality in this case follows as in the first paragraph from the proof of Lemma \ref{lem:similar_theta_function}.
\end{proof}

\section{Quantum tagged bracelets with coefficients}\label{sec:bracelet_cluster_alg}

Here we apply the results of \S \ref{sec:seed_change} to motivate definitions of bracelets bases for cluster algebras from surfaces with arbitrary coefficients.  Even when considering punctured surfaces, coefficients can be chosen so that the Injectivity Assumption is satisfied.  This allows for a definition and characterization of $g$-vectors.  Furthermore, these cluster algebras satisfying the Injectivity Assumption can then be quantized. In \S \ref{sec:q_tag_bracelet}, we construct quantum bracelets bases for these quantized cluster algebras with coefficients.\footnote{For the quantum bracelet elements associated to notched arcs in once-punctured closed surfaces, we resort to defining these bracelets to simply be the corresponding quantum theta functions.  This is partially justified in Remark \ref{rem:quantization_tagged_arc} using the DT-transformation, which will be discussed in \S \ref{sec:DT}. A more natural definition is still desirable.} For this construction, we rely heavily on our knowledge about theta functions and bracelets in later sections.

\subsection{Surface type cluster algebras}

Take any tagged triangulation $\Delta$ of $\Sigma$ and let $\sd_\Delta$ denote the corresponding seed ($\sd_\Delta$ is said to have boundary coefficients). Recall that $\Sk(\Sigma)$ (or $\Sk_t(\Sigma)$ for unpunctured $\Sigma$) are contained in the (quantum) upper cluster algebras $\s{A}^{\up}$ (or $\s{A}_t^{\up}$).

Recall that two seeds are similar if they share the same principal $B$-matrix up to relabelling vertices, see Definition \ref{def:similar_seed}.
\begin{defn}[Surface type cluster algebra]\label{def:surface_type}
	If $\sd$ is a (quantum) seed similar to $\sd_\Delta$, we say the corresponding (quantum) cluster algebra $\s{A}^{\bullet}(\sd)$ (or $\s{A}^{\bullet}_t(\sd)$) is of type $\Sigma$, where $\bullet$ stands for $\ord$, $\midd$, $\can$, or $\up$.
\end{defn}

Recall that, by taking the similar elements as in Definition \ref{def:similar_element}, a basis for a cluster algebra might provide a basis for a similar cluster algebra, see Lemma \ref{lem:similar_theta_function}. Correspondingly, we are interested in the algebra elements of a cluster algebra of type $\Sigma$ which are similar to the bracelet elements in $\Sk(\Sigma)\subset \s{A}^{\up}(\sd_\Delta)$ (or $\Sk_t(\Delta)\subset \s{A}_t^{\up}(\sd_\Delta)$).

\begin{defn}[Bracelet elements for similar seeds]\label{def:bracelets_cluster_elem}
	Assume that a (quantum) seed $\sd$ satisfies the Injectivity Assumption, and we have defined a bracelet element $\beta$ in $\kk \llb M_{\sd} \rrb$ (or in $\kk_t \llb M_{\sd} \rrb$). If $\sd'$ is a (quantum) seed similar to $\sd$ which satisfies the Injectivity Assumption, then the elements in $\kk \llb M_{\sd'} \rrb$ (or $\kk_t \llb M_{\sd'} \rrb$) similar to $\beta$ in the sense of Definition \ref{def:similar_element} will be called bracelet elements. If $\sd'$ is a classical seed similar to the classical seed $\sd$ but which does not satisfies the Injectivity Assumption, then the elements in $\kk[M_{\sd'}]$ constructed from $\beta$ via \eqref{eq:projection} will be called bracelet elements.
\end{defn}
We recall that the bracelet elements constructed from $\beta$ in Definition \ref{def:bracelets_cluster_elem} are uniquely determined up to a frozen factor. Moreover, if $\beta$ is contained in $\s{A}^{\up}(\sd)$ (or $\s{A}^{\up}_t(\sd)$), then the similar elements are also contained in $\s{A}^{\up}(\sd')$ (or $\s{A}^{\up}_t(\sd')$), see \cite[Lem. 4.2.2(iii)]{qin2017triangular}.

Recall that for unpunctured $\Sigma$, the seed $\sd_\Delta$ satisfies the Injectivity Assumption, and the bracelet elements in $\Sk(\Sigma)$ and $\Sk_t(\Sigma)$ have been defined in \S \ref{sec:bracelet_band}. Correspondingly, we obtain the definition of bracelet elements for all (quantum) cluster algebras of type $\Sigma$.

Let $\Sigma$ denote a marked surface, possibly with punctures, and $\Delta$ any tagged triangulation. Then the Injectivity Assumption for the corresponding seed $\sd_{\Delta}$ might fail. In particular, one can not make $\sd_{\Delta}$ into a quantum seed by adding a compatible $\Lambda$. In view of Definition \ref{def:bracelets_cluster_elem}, to construct bracelet elements for cluster algebras of type $\Sigma$, it suffices to construct the bracelet elements for one seed $\sd$ similar to $\sd_{\Delta}$ such that $\sd$ satisfies the Injectivity Assumption.

We will see in \S \ref{subsub:MSW} that our tagged bracelets with coefficients agree with those of \cite{musiker2013bases}.

\subsection{Classical bracelets with coefficients}\label{sec:class-brac-coeff}
	(Tagged) bracelets for $\s{A}^{\up}(\sd_\Delta)$ have been defined in \S \ref{Sec:tag-brac}. In the following, we will extend the construction to $\s{A}^{\up}(\sd)$ for any seed $\sd$ similar to $\sd_\Delta$: we will define $\BracE{C}^\sd$ for any weighted tagged simple multicurve $C$. The superscript $\sd$ is often omitted when the context is clear.

\subsubsection{Classical bracelets for loops with coefficients}\label{subsub:class-loops}

Let $L=\bigcup_i w_i L_i$ denote a simple curve in $\Sigma$ whose components $L_i$ are simple non-peripheral loops (up to frozen variable factors, the approach here also applies when the components $L_i$ are allowed to include arcs which do not end at punctures).  In \S \ref{sec:mod_PGL2}, we will see that such $L$ (viewed as a lamination $e(L)$) determines\footnote{Actually, $\bb{I}(wL)$ is only defined if $\pi(wL)\in N$.  However, \eqref{eq:classical_trace} allows us to extend the definition to $wL$ with $2\pi(wL)\in N$ by taking the elements $X_{i_k}^{\pm 1/2}$ of loc. cit. to correspond to $z^{\pm e_i/2}\in \frac{1}{2}N$.  Alternatively, specifying that $\bb{I}(2wL)$ should equal $T_2(\bb{I}(wL))$ for $T_2$ as in \eqref{Cheb1}, and specifying that $\bb{I}(wL)$ should be positive, suffices to uniquely determine $\bb{I}(wL)$ from $\bb{I}(2wL)$.} an element $\bb{I}(L)\in \s{X}(\sd_{\Delta})$.  By applying $p^*=\omega_1^{\sd_\Delta}$, we recover an element $\bb{I}^{\vee}(L)\in \s{A}^{\up}(\sd_{\Delta})$ which equals the bracelet $\langle L\rangle_{\Brac}$ defined in \S \ref{sec:bracelet_band}; cf. \S \ref{sec:mod_PGL2} and \S \ref{sec:dec_SL2_moduli}, particularly Lemma \ref{PullbackLem}.

Similarly, for the seed $\sd_{\Delta}^{\prin}$ of the principal coefficients, we define the bracelet element 
\begin{align}\label{Lprin-def}
\langle L\rangle_{\Brac}^{\prin} := \xi(\bb{I}(L))
\end{align}
where $\xi=\omega_1^{\sd_\Delta^{\prin}}$ is the map acting on exponents via $N\rar N\oplus M$, $n\mapsto (\omega_1(n),n)$, cf. \eqref{xi}.

Since $\sd_{\Delta}^{\prin}$ satisfies the Injectivity Assumption, $\langle L\rangle_{\Brac}^{\prin}$ determines a bracelet element $\langle L\rangle_{\Brac}^{\sd}$, up to frozen factors, for arbitrary seed $\sd$ similar to $\sd_{\Delta}$, see Definition \ref{def:bracelets_cluster_elem}. Correspondingly, we define the following particular bracelet element which is similar to $\langle L\rangle_{\Brac}^{\prin}$, equal to $\langle L\rangle_{\Brac}^{\prin}$ in the case $\sd=\sd_{\Delta}^{\prin}$, and equal to $\langle L\rangle_{\Brac}$ in the case $\sd=\sd_{\Delta}$:
	$$\langle L\rangle_{\Brac}^{\sd}:= \omega_1^{\sd}(\bb{I}(L)).$$

	Later, we will see that $\langle L\rangle_{\Brac}^{\prin} $ is a theta function (Lemma \ref{lem:loop_puncture_bracelet}) and $\bb{I}(L)$ is a theta function (Theorem \ref{thm:ThetaX}). They are 
	related to $\BracE{L}$ as in Definition \ref{X-Aprin} and \eqref{theta-rho}. Then $\langle L\rangle_{\Brac}^{\sd}$ is a theta function by Lemma \ref{lem:similar_theta_function} or \ref{lem:similar_theta_function_2}.

\subsubsection{Bracelets for arcs}\label{subsub:brace-arcs}
A geometric characterization of cluster variables for surface type cluster algebras with coefficients is given in \cite[Thm. 15.6]{fomin2018cluster}.  Here, a cluster algebra with coefficients $\s{A}^{\up}(\sd)$
is associated to a triangulable surface $\Sigma$ equipped with a choice of multi-lamination $\textbf{L}$, i.e., a finite multiset of integer unbounded laminations.  Each tagged arc $\gamma$ is then represented by a ``laminated lambda length'' $x_{\bf L}(\gamma)=:\BracE{\gamma}$, a function on a ``laminated Teichm\"uller space.''  These laminated lambda lengths, other than those associated to notched arcs on once-punctured closed surfaces, yield the set of cluster variables for the cluster algebra.  In particular, taking ${\bf L}$ to be the set of elementary laminations ${\bf L}=\{e(\gamma)|\gamma\in \Delta_{\uf}\}$ yields the case of principal coefficients, cf. \cite[Prop. 17.3]{fomin2018cluster}.

We do not review the details of the \cite{fomin2018cluster}-construction here since it is rather involved and not needed for our purposes---knowing that the laminated lambda lengths are cluster variables already suffices to imply that they are theta functions.  For the case of notched arcs on laminated once-punctured closed surfaces, we can understand the associated laminated lambda lengths by working with a $k$-sheeted covering space $\wt{\Sigma}$ of $\Sigma$, $k\in \bb{Z}_{\geq 2}$.  Since $\wt{\Sigma}$ has $k\geq 2$ punctures, we know that the laminated lambda lengths for $\wt{\Sigma}$ are cluster variables, hence theta functions.  We apply this in \S \ref{sub:bracelets-1p} to show that, except for once-punctured tori, notched arcs correspond to theta functions for the once-punctured closed $\Sigma$.  

Now, for any weighted tagged arc $w\gamma$ in $\Sigma$, we denote the corresponding bracelet element in the cluster algebra with coefficients $\s{A}^{\up}(\sd)$ by $\BracE{w\gamma}=(\BracE{\gamma})^w$.  

\subsubsection{Arbitrary bracelets}\label{sec:classical_general_bracelets}

For any weighted simple multicurve $C\in\wSMulti(\Sigma)$, we have $C=\bigcup_i w_iC_i$ as a union of non-homotopic weighted components $C_i$, each of which is a weighted loop or weighted tagged arc.  The corresponding bracelet elements of $\s{A}^{\up}(\sd)$ is $\prod_i \BracE{w_iC_i}$.  In particular, if $C=\bigcup_i w_i\gamma_i$ is a collection of weighted tagged arcs $w_i\gamma_i$ with each $\gamma_i$ corresponding to a cluster variable, then $\BracE{C}$ is a cluster monomial.

\subsection{$g$-vectors}\label{sec:g-vector}

Let $C$ denote any weighted tagged simple multicurve $\bigcup_i w_i C_i$. When a seed $\sd$ similar to $\sd_{\Delta}$ satisfies the Injectivity Assumption, by our construction, $\BracE{C}$ has a unique $\prec_{\sd'}$-maximal degree term with respect to any chosen seed $\sd'$ of $\s{A}^{\up}(\sd)$ (it is pointed when $C_i$ are not doubly-notched arcs in a once-punctured torus, see  \S \ref{sub:bracelets-1p}). Its degree will be called the (extended) \textbf{$g$-vector} and denoted by $g(C)$. Let $\Delta'$ denote the tagged triangulation corresponding to $\sd'$ (in particular, $\sd'$ is similar to $\sd_{\Delta'}$). Let $e(C_i)$ denote the elementary laminate of $C_i$, see \S \ref{sec:shear_coord}. Then we have
	\begin{align}\label{eq:g_vector_shear}
		\pr_{I_\ufv}g(C_i)=-b^{\Delta'}(e(C_i)).
	\end{align}
For loops, the equality follows from \eqref{eq:intersection_shear_loop}, \eqref{eq:ell-pi-pointed}, and the construction of $\BracE{C_i}$; for tagged arcs, it follows from Proposition \ref{prop:shear_g}.

Assume $\Delta$ has no self-folded triangles. For any closed non-peripehral loop $L$, $\bb{I}(L)$ is pointed at $-\pi(L)$; cf. \eqref{eq:ell-pi-pointed}. So by \eqref{eq:intersection_shear_loop} and Lemma \ref{PullbackLem}, $\BracE{L}$ is pointed at 
\begin{align}\label{gLoop}
    g(L) = -\omega_1^{\sd}(\pi(L)).
\end{align} 
Since $2\pi(L)\in N^+$, we obtain
\begin{align}\label{eq:loop_deg_monoid}
	-2g(L)\in M^+.
\end{align}

	We sometimes consider a seed $\sd$ for which the Injectivity Assumption might fail, for example, $\sd=\sd_\Delta$. For any $C$, let $g^{\sd^{\prin}}(C)$ denote the $g$-vector of $C$ for the seed $\sd^{\prin}$. Then we define the $g$-vector of $C$ as the projection $\rho (g^{\sd^{\prin}}(C))$ for $\rho$ as in \eqref{rhodef}. Notice that \eqref{eq:g_vector_shear} still holds (it can be deduced from $\sd^{\prin}$ by removing the framing vertices).  In the case where $\sd=\sd_{\Delta}$, \eqref{eq:g_vector_shear} recovers \eqref{eq:shear_g}.

\subsection{Quantum tagged bracelets}\label{sec:q_tag_bracelet}

We now extend the constructions of \S \ref{sec:class-brac-coeff} to the quantum setting. Let $\sd$ denote a quantum seed which is similar to $\sd_\Delta$ at the classical level.

\subsubsection{Quantum bracelets for loops}\label{sec:qbracelet_loop}

	Let $L=\bigcup_i w_i L_i$ denote a collection of weighted non-peripheral loops. 
	When $\Sigma$ has no punctures, the bracelet elements $\BracE{L}$ in the quantum skein algebra $\Sk_t(\Sigma)\subset \s{A}^{\up}_t(\sd_\Delta)$ has been defined in \S \ref{sec:bracelet_band}. We will see that it is a quantum theta function in Theorem \ref{thm:bracelet-theta-no-punct}.
	
	For general $\Sigma$ and quantum seed $\sd$, we apply the same construction used in \S \ref{subsub:class-loops} but with the quantum element $\bb{I}_t(L)$ as in \S \ref{sec:X-q-can} in place of $\bb{I}(L)$. The resulting element $\omega_1^{\sd}(\bb{I}_t(L))$ will turn out to be a theta function by \S \ref{sec:ThetaX}. We define the quantum bracelet element 
	\begin{align}\label{eq:q_loop_bracelet}
	\BracE{L}:=\omega_1^{\sd}(\bb{I}_t(L)).
	\end{align}
	When $\Sigma$ has no punctures and $\sd=\sd_\Delta$, this coincides with the above bracelet defined in the skein algebra (\S \ref{sec:bracelet_band}), see Lemmas \ref{PullbackLem} and \ref{lem:unpunctured_Poisson_theta}.

\subsubsection{Quantum bracelets for tagged arcs}\label{subsub:qtag}

Suppose $w_i\gamma_i$ is a weighted tagged arc in $\Sigma$, other than a notched arc in a once-punctured closed surface.  Then in the classical setting we noted that $\BracE{w_i\gamma_i}$ is a cluster monomial.  We define $\BracE{w_i\gamma_i}$ in the quantum setting to simply be the corresponding quantum cluster monomial.

More generally, let $C=\bigcup_i w_i\gamma_i$ be a collection of pairwise-compatible weighted tagged arcs $\gamma_i\in \Delta'$ for some tagged triangulation $\Delta'$, with no $\gamma_i$ being a notched arc in a once-punctured closed surface.  We know that $\BracE{C}$ is a cluster monomial in the classical setting, so for the quantum setting we simply take $\BracE{C}$ to be the corresponding quantum cluster monomial.

	Next, suppose $\Sigma$ is a once-punctured closed surface other than a torus. Let $C=\bigcup_i w_i\gamma_i$ be a collection of pairwise-compatible doubly-notched arcs for some tagged triangulation. We will see that $\BracE{C}$ is a theta function in the classical setting (Theorem \ref{thm:1p-bracelet}). In the quantum setting, we define $$\BracE{C}\coloneqq \DT(\BracE{C^{\diamond}})$$ where $C^{\diamond}$ is obtained from $C$ by switching all taggings from notched to plain, and $\DT$ is the (quantum) DT-transformation as in \S \ref{sec:DT}. Proposition \ref{DT-monomial} implies that the resulting  $\BracE{C}$ is a quantum theta function, and Corollary \ref{prop:DT} will ensure that the classical limit really is $\BracE{C}|_{t=1}$.
	
	\begin{rem}\label{rem:quantization_tagged_arc}
		It is desirable to find a more natural and fundamental definition of quantum tagged arcs which also applies to the once-puncture torus. To the best of authors' knowledge, there exists no definition of quantum tagged arcs in literature. A generalization of the perfect matching formula in \cite{musiker2013matrix,huang2022expansion} might provide one alternative definition.
	\end{rem}

\subsubsection{Arbitrary quantum bracelets}\label{sec:general_quantum_bracelet}

Let $C=C_1\cup C_2\in \wSMulti(\Sigma)$ where $C_1=\bigcup_i w_iL_i$ and $C_2=\bigcup_j w_j\gamma_j$, with each $L_i$ being a distinct homotopy-class of loop and each $\gamma_j$ being a distinct class of tagged arc, none of which are notched arcs in once-punctured closed tori. Then we define
	$$\BracE{C}:=\BracE{C_1} \BracE{C_2}.$$
This product is bar-invariant by the Lemma \ref{lem:commute_arc_loop}, which says that $\BracE{C_1}$ and $\BracE{C_2}$ commute.

\begin{rem}

	We do not currently have a characterization of quantum bracelets which included contributions of notched arcs in once-punctured closed surfaces.  The covering space trick used in \S \ref{subsub:brace-arcs} is not expected to work in the quantum setting because the folding construction of \S \ref{GenFoldApp} does not carry over to the quantum setting.

	We also do not currently include a description of the classical bracelet elements with coefficients associated to a notched arc in a once-punctured torus.  However, for $\gamma_i$ a weight-one notched arc in a once-punctured torus, the corresponding classical theta function in the coefficient-free setting is known (cf. \cite[\S 5.3]{zhou2020cluster} and \eqref{eq:theta12}), and all coefficients equal $1$.  So, given the corresponding classical theta function with, say, principal coefficients, bar-invariance is sufficient to uniquely determine the corresponding quantum theta function.

\end{rem}

\subsection{Surface cutting}\label{sec:cut}

In this subsection, we prove a useful lemma which will allows us to understand loops in punctured surfaces $\Sigma$ by cutting along arcs to produce an unpunctured surface $\Sigma'$ where the loops may be easier to understand.  We then apply this lemma to give an alternative characterization of quantum bracelets associated to loops in punctured surfaces.

Fix a marked surface $\Sigma=(\SSS,\MM)$. Let
$\gamma$ denote any internal arc. We obtain
a new surface $\SSS'$ by cutting $\SSS$ along $\gamma$ so that
the two copies $\gamma^{1},\gamma^{2}$ of $\gamma$ become part of
its boundary. The set $\MM'$ of the marked points for $\SSS'$ is
defined by $$\MM':=\left(\MM\setminus \{\text{endpoints of }\gamma\}\right)\cup \{\text{endpoints of $\gamma^1$ and $\gamma^2$}\}.$$

Let $\Delta$ denote any ideal triangulation containing $\gamma$.
Then $\Delta':=(\Delta\backslash\gamma)\cup\{\gamma^{1},\gamma^{2}\}$
becomes a triangulation for $\Sigma':=(\SSS',\MM')$. Consider the
corresponding seeds $\sd_{\Delta}$ and $\sd_{\Delta'}$. We see that, if $\gamma$ is not an edge in a self-folded triangle in $\Delta$, then $\sd_{\Delta}$ is obtained from $\sd_{\Delta'}$ by gluing the frozen elements $\gamma^1,\gamma^2$ as in \S \ref{sec:gluing_frozen_vertices}, and then unfreezing the resulting frozen element $\gamma'$ as in \S \ref{sec:unfreezing}.

Let $\eta$ denote any arc or closed loop in $\SSS$ such that $\eta$ does
not intersect $\gamma$. By examining the Kauffman skein relation (Figure \ref{SkeinFig}), we see that the Laurent expansion of $[\eta]$ in the skein algebra $\Sk(\Sigma)$ (or the quantum skein algebra $\Sk_t(\Sigma)$ for unpunctured $\Sigma$) with respect to the triangulation $\Delta$ is the same as the Laurent
expansion of $[\eta]$ in the skein algebra $\Sk(\Sigma')$ (or $\Sk_t(\Sigma')$) with
respect to the triangulation $\Delta'$, after we identify $[\gamma^1],[\gamma^2]$ with $[\gamma]$.

\begin{lem}\label{lem:triangulation}
	Let $\Sigma=(\SSS,\MM)$ denote a triangulable punctured surface, and let $C=\bigcup_{i=1}^r L_i$ denote a disjoint union of non-isotopic, non-contractible (in $\SSS\setminus \MM$), non-peripheral simple loops $L_i$.   There exist non-isotopic interior arcs $\gamma_1,\ldots,\gamma_s$, such that 
	\begin{itemize}
		\item $\gamma_i$ does not intersect $C$;
		\item we can extend $\{\gamma_1,\ldots,\gamma_s\}$ into an ideal triangulation $\Delta$ which does not have self-folded triangles;
		\item the marked surface $\Sigma'$ obtained from $\Sigma$ by cutting along the arcs $\gamma_1,\ldots,\gamma_s$ is triangulable and unpunctured.
	\end{itemize}
\end{lem}
\begin{proof}
	We do this by induction on the number of punctures in $\Sigma$.  For the base case with $0$ punctures, we may take $\{\gamma_1,\ldots,\gamma_s\}$ to be the empty set and apply \cite[Lem. 2.13]{FominShapiroThurston08} to say that an ideal triangulation without self-folded triangles exists.
	
	Now suppose $\Sigma$ does contain some punctures.  It clearly suffices to deal with the case where $\Sigma$ is connected, since otherwise we can just deal with each connected component separately.  We further assume that $C$ is non-empty since otherwise we may take $\{\gamma_1,\ldots,\gamma_s\}$ to be the arcs of any ideal triangulation without self-folded triangles.

	With these assumptions, fix any puncture $p$.  Consider an arc $\alpha$ obtained by starting at $p$, traveling to a neighborhood of some component $L_i$ of $C$, going around parallel to $L_i$, and then returning to $p$ (cf. Figure \ref{DLoop} for an illustration of a similar construction).  Note that $L_i$ being non-peripheral and non-contractible ensures that $\alpha$ does not bound an unpunctured or once-puntured monogon.  Also, any component cut out by $\alpha$ contains a boundary component with a single marked point (a copy of $\alpha$ with a copy of the marked point $p$), so we do not obtain an unmarked component, an unmarked boundary component, a sphere, or a bigon.  Thus, the marked surface $\Sigma_{\alpha}$ obtained via cutting along $\alpha$ is triangulable.
	
	The new surface $\Sigma_{\alpha}$ contains one less puncture than $\Sigma$.  By the inductive assumption, $\Sigma_{\alpha}$ admits arcs $\gamma_1,\ldots,\gamma_s$ satisfying the desired conditions.  It is clear that these arcs, together with $\alpha$, satisfy the desired conditions for $\Sigma$.
\end{proof}

Let $L=\bigcup w_i L_i$ denoted a weighted simple curve in $\Sigma$ whose components $L_i$ are simple non-peripheral loops. We shall now use a cut-and-paste approach to construct the quantum bracelet element $\langle L\rangle_{\Brac}$ in a quantum cluster algebra of type $\Sigma$---specifically, in the corresponding quantum cluster algebra with principal coefficients.

For $C=\bigcup_i L_i$, choose a collection of interior arcs $\gamma_1,\ldots,\gamma_s$ and triangulation $\Delta$ as in Lemma \ref{lem:triangulation}.  Cutting along these arcs $\gamma_i$ yields an unpunctured surface $\Sigma'$ with triangulation $\Delta'$ obtained by replacing each $\gamma_i\in \Delta$ with the two copies $\gamma_i^{(1)}$ and $\gamma_i^{(2)}$ of $\gamma_i$ that appear on the boundary of $\Sigma'$.  We know how to associate to $L$ a bracelet element $\langle L\rangle_{\Sigma'}$ of $\s{A}_t^{\up}(\sd_{\Delta'})$ constructed using $\Lambda$ as in \eqref{Lambda-ei-ej} with $d=4$.  Let $\s{A}_t^{\up}(\sd^{\prin}_{\Delta'})$ be the corresponding upper quantum cluster algebra with principal coefficients, constructed using $\Lambda^{\prin}$ as in \eqref{LambdaPrin} (with multiplier $d'=1$).  Note that the seed $\sd^{\prin}_{\Delta'}$ is similar to $\sd_{\Delta'}$.  We take $\langle L\rangle'$ to be the unique $m'$-pointed element of $\s{A}_t^{\up}(\sd^{\prin}_{\Delta'})$ similar to $\langle L\rangle_{\Sigma'}\in\s{A}_t^{\up}(\sd_{\Delta'})$ (as in Definition \ref{def:similar_element}) with $m'=\pr_{I_\ufv}m'$.

 Now note that $\sd_{\Delta}^{\prin}$ is related to $\sd_{\Delta'}^{\prin}$ via applying frozen vertex-gluing of \S \ref{sec:gluing_frozen_vertices} to each pair $\gamma_i^{(1)},\gamma_i^{(2)}$ to obtain frozen vertices $\?{\gamma}_i$, gluing each corresponding pair of principal coefficient vertices to obtain $\?{\gamma}_{i'}$,  removing all but one of the arrows between $\?{\gamma}_{i'}$ and $\?{\gamma}_i$ as in the discussion following \eqref{pi_cluster_poisson}, and then unfreezing each $\?{\gamma}_i$ as in \S \ref{sec:unfreezing}. We now define the following quantum bracelet element in $\s{A}^{\up}_t(\sd_{\Delta}^{\prin})$: 
\begin{align}\label{eq:LBrac}
	\BracE{L}:=\f{i}\circ \pi_{M^{\prin}}(\langle L\rangle')
\end{align}
for $\pi_{M^{\prin}}$ as in \S \ref{sec:gluing_frozen_vertices} and $\f{i}$ as in \S \ref{sec:unfreezing}.

The fact that this construction yields the same elements as that of \S \ref{sec:qbracelet_loop} (and thus is independent of the choices of arcs $\gamma_1,\ldots,\gamma_s$) follows, e.g., from Lemma \ref{lem:Xglue}.  This lemma says that $\bb{I}_t$ respects the gluing of boundary arcs used here; it is a consequence of the State Sum Property of the quantum trace map $\Tr_q$ used in the definition of $\bb{I}_t$.

\subsubsection{Relation to the bracelets of Musiker-Shiffler-Williams}\label{subsub:MSW}

We note that our (tagged) bracelets basis agrees with the set $\s{B}$ defined in \cite{musiker2013bases} for a surface-type cluster algebra with full-rank exchange matrix.  Indeed, the elements associated to tagged arcs in either set are cluster variables (with the exception of notched arcs in once-punctured closed surfaces, but here the corresponding elements in either setting can be obtained by applying the digon relation in a covering space). 
 
 Now it remains to compare the elements corresponding to loops.  For these we can use the surface cutting technique of this subsection to reduce to unpunctured surfaces.  In fact, as we will see in \S \ref{sec:annular_loops}-\S \ref{sec:non-annular-loop}, we can reduce to (1) the case of a twice-marked annulus or (2) the case of an unpunctured surface with a single boundary marking (so a single boundary component).  These cases satisfy the Injectivity Assumption even without the boundary-arc coefficients: 
\begin{enumerate}
\item The twice-marked annulus corresponds to the Kronecker quiver; cf. Example \ref{AnnEx}.
\item The compatibility condition \eqref{Lambda-B} implies that $\Lambda(i,\cdot)$ is nonzero for each non-boundary arc $i$.  In case (2), letting $\alpha_L$ denote the boundary arc, we note that $\Lambda(i,\alpha_L)=0$ for all interior arcs $i$.  So in fact $\Lambda$ as in \eqref{Lambdaij} is non-degenerate in this setting even without the boundary-arc coefficient.  Compatibility then implies $\omega|_{N_{\uf}}$ is non-degenerate as well.
\end{enumerate}
Now, it is already known that $\s{B}$ coincides with the bases of \cite{FockGoncharov06a} (hence our bracelets bases) in the coefficient-free setting, cf. \cite[the comments preceeding Thm. 1.1]{musiker2013bases}.  To pass to other full-rank coefficient systems, we just need that the elements of $\s{B}$ for different such coefficient systems are similar pointed elements.  In the principal coefficients setting, this follows from \cite[Thm. 5.1]{musiker2013bases}.  For other coefficient systems, it is essentially by design (based on the separation formula of \cite{FominZelevinsky07}) that this holds for $\s{B}$ as constructed in \cite[\S 7]{musiker2013bases}.

	\section{Bracelets are theta functions: unpunctured cases}\label{S:BracTheta}

	In this section, unless otherwise specified, we assume that $\Sigma$ is an unpunctured triangulable marked surface (although some arguments do work equally well in the punctured setting). Hence, the quantum skein algebra $\Sk_t(\Sigma)$ can be defined, and the Injectivity Assumption holds for the corresponding seeds with boundary coefficients.  Let $\Delta$ be an ideal triangulation of $\Sigma$.  Recall as in \eqref{Eq:BracDef} that, for any weighted simple multicurve $C=\bigcup w_i C_i$, we have the corresponding bracelet $\langle C \rangle_{\Brac}$ in the quantum skein algebra $\Sk_t(\Sigma)\subset \s{A}_t^{\up}(\sd_{\Delta})$, which we denote by  $\langle C \rangle$ for simplicity. The aim of this section is to show that $\langle C \rangle$ is a theta function for $\s{A}_t^{\up}(\sd_{\Delta})$, cf. Theorem \ref{thm:bracelet-theta-no-punct}.

	\subsection{The gluing lemma for surfaces}\label{Sec:glue}
	
	In this subsection we prove a key lemma which will allow us to prove the correspondence between bracelets and theta functions for $\Sigma'$ by first cutting $\Sigma'$ along arcs and then proving the correspondence on the cut surface $\Sigma$.
	
	Let $\Sigma=(\SSS,\MM)$ be a triangulable marked surface (possibly disconnected).  Let $b_1,b_2$ be the closures of two distinct components of $\partial \SSS\setminus (\partial\SSS\cap \MM)$ --- we also denote the corresponding skein algebra elements (associated to the corresponding boundary arcs) by $b_1,b_2$.  Let $\Sigma'=(\SSS',\MM')$ be the marked surface obtained by gluing $b_1$ to $b_2$ (consistently with the orientation of $\SSS$), identifying the markings at the ends of $b_1$ with the corresponding markings\footnote{If, say, $b_1$ has both ends at the same marking, then this gluing will identify the markings at the ends of $b_2$ with each other, even if they were previously distinct.} at the ends of $b_2$. For simplicity, we assume that $\Sigma$ and $\Sigma'$ are both unpunctured.
	
	Let us choose a triangulation $\Delta'$ for $\Sigma'$ which contains the common image $b$ of $b_1$ and $b_2$ in $\Sigma'$. This $\Delta'$ induces a triangulation $\Delta$ on $\Sigma$. Consider the (quantum) seeds $\sd=\sd_{\Delta}$, $\sd'=\sd_{\Delta'}$ as in \eqref{sdDelta},\eqref{Lambda-ei-ej}. 
	Let $\?{\sd}$ denote the seed obtained from $\sd$ by gluing $b_1,b_2$ (viewed as vertices of the quiver associated to $\sd$) into a single frozen vertex $b$ as in \S \ref{sec:gluing_frozen_vertices}.  Note then that $\sd'$ can be obtained from $\?{\sd}$ by unfreezing $b$ as in \S \ref{sec:unfreezing}.  Recall that we have the $\kk$-linear map $\pi_{M_{\sd}}:\kk_t[M_{\sd}]\rightarrow \kk_t[M_{\?{\sd}}]$ for the gluing process sending $z^{e^*_{b_1}}$ and $z^{e^*_{b_2}}$ to $z^{e^*_b}$ in \eqref{piM}---in fact, $\pi_{M_{\sd}}$ is $\kk_t$-linear in this case since $\?{d}=d$ (in general $t\mapsto t^{\frac{\?{d}}{{d}}}$).  Also recall the natural identification $\f{i}:\kk_t[M_{\?{\sd}}]\simeq \kk_t[M_{\sd'}]$ for the unfreezing process as in \eqref{eq:unfreezing_identification}.\footnote{The composition $\f{i}\circ\pi_{M_{\sd}}$ is a $\kk_t$-linear map, but it does not generally preserve the multiplication of quantum torus algebras.}

	Let $C$ denote a weighted simple multicurve in $\Sigma$ such that the sum of the weights of $b_1$ and $b_2$ is non-negative. Note that gluing identifies $C$ with a simple multicurve $C'$ in $\Sigma'$ which does not intersect the interior of $b$. We have the corresponding skein algebra elements $[C]\in \Sk_t(\Sigma)$ and $[C']\in \Sk_t(\Sigma')$ respectively. Identifying $[C]$ and $[C']$ with their Laurent expansions using the triangulations $\Delta$ and $\Delta'$, respectively, we have the following result:
	\begin{lem}\label{lem:gluing_Skein_element}
		We have $\f{i}\circ\pi_{M_{\sd}}([C])=[C']$.
	\end{lem}
	\begin{proof}
		Let $C$ be a weighted simple multicurve. By Corollary \ref{Deltak}, we know that there exists a finite collection of quantum cluster monomials $[\gamma_i]$ represented by some simple multicurves $\gamma_i$, $0\leq i\leq r$, such that the components of each $\gamma_i$ are arcs in $\Delta$ and such that we have the following decomposition in $\Sk_t(\Sigma)$:
		\begin{align*}
			[\gamma_0][C]=\sum_{i=1}^r c_i [\gamma_i]
		\end{align*}
		with coefficients $c_i\in \kk_t$ which we may assume are all nonzero.  This decomposition is computed using the skein relations which resolve the intersections between $C$ and $\gamma_0$. 
		
		Let $\gamma_i'$ denote the simple multicurves in $\Sigma'$ corresponding to $\gamma_i$, $i=0,\ldots,r$, under the gluing process. Then the components of $\gamma_i'$ belong to $\Delta'$. We observe that $[\gamma_i]$ and $[\gamma'_i]$ are initial cluster monomials $z^{m_i}\in \kk_t[M_{\sd}]$ and $z^{m'_i}\in \kk_t[M_{\sd'}]$ respectively, such that $m'_i=\f{i}\circ \pi_{M_{\sd}}(m_i)$. Moreover, using the skein relations which resolve the intersection between $C'$ and $\gamma_0'$, and noticing that passing from $\Sigma$ to $\Sigma'$ preserves the homotopy equivalence between curves, we still have the following decomposition in $\Sk_t(\Sigma')$:
		\begin{align*}
			[\gamma'_0][C']=\sum_{i=1}^r c_i t^{h_i} [\gamma'_i],
		\end{align*}
		where extra factors $t^{h_i}$ appear because the $\Lambda$-pairing for calculating products of arcs may change. 

		The Laurent expansion of $[C]$ in $\kk_t[M_{\sd}]$ is given by
		\begin{align*}
			[C]&=\sum_{i=1}^r c_i (z^{-m_0}*z^{m_i})\\
			&=\sum_{i=1}^r c_i t^{k_i} z^{m_i-m_0}
		\end{align*}
		for some $k_i\in \bb{Z}$. Similarly, the Laurent expansion of $[C']$ in $\kk_t[M_{\sd'}]$ is given by
		\begin{align*}
			[C']&=\sum_{i=1}^r c_i t^{h_i}(z^{-m'_0}*z^{m'_i})\\
			&=\sum_{i=1}^r c_i t^{h_i+k_i'} z^{m'_i-m'_0}
		\end{align*}
		for some $k_i'\in \bb{Z}$, $i=1,\ldots,r$. 
		
		Since $[C]$ and $[C']$ are bar-invariant, the coefficients $c_i t^{k_i}$ and $c_i t^{h_i+k_i'}$ are bar-invariant, which implies that $t^{k_i}=t^{h_i+k_i'}$ for each $i$. Consequently, we obtain $[C']=\f{i} \circ \pi_{M_{\sd}}([C])$, as claimed.
	\end{proof}

	Consider the bracelet elements $\langle C\rangle_{}\in \Sk_t(\Sigma)$ and  $\langle C'\rangle_{}\in \Sk_t(\Sigma')$ parametrized by $C$ and $C'$ respectively. Recall our assumption that $\Sigma,\Sigma'$ are unpunctured. We have the following result.

	\begin{lem}[The Gluing Lemma]\label{glue}
	       If the bracelet $\langle C\rangle_{}\in \Sk_t(\Sigma)$ is a theta function in $\s{A}_{t}^{\up}(\sd_{\Delta})$, then $\langle C'\rangle_{}$ is a theta function in $\s{A}_{t}^{\up}(\sd_{\Delta'})$ as well (similarly in the classical $t=1$ setting).
	\end{lem}
	
	\begin{proof}
		First, Lemma  \ref{lem:gluing_Skein_element} implies the equality $\f{i} \circ \pi_{M_{\sd}}( \E{C})=\E{C'_i}$. More precisely, denote $C=\bigcup_i w_i C_i$ and, correspondingly, $C'=\bigcup_i w_i C'_i$. Then $\f{i} \circ \pi_{M_{\sd}}$ sends $\E{C_i}=[C_i]$ to $\E{C'_i}=[C'_i]$ by Lemma \ref{lem:gluing_Skein_element}. Recall that $\E{w_iC_i}$ equals $\E{C_i}^{w_i}$ if $C_i$ is an arc and $T_{w_i}(\E{C_i})$ if $C_i$ is a loop. Moreover, $\E{C}$ is a product of the factors $\E{w_iC_i}$. Similarly for $\E{C'}$. Thus, $\langle C\rangle$ is a $\kk_t$-linear combination of weighted multicurves, and $\langle C'\rangle$ is the corresponding $\kk_t$-linear combination of the images of these weighted multicurves under $\f{i}\circ \pi_{M_{\sd}}$. The equality now follows from the $\kk_t$-linearity of $\f{i}\circ \pi_{M_{\sd}}$.

		Recall that the gluing process $\pi_{M_{\sd}}$ sends theta function to theta functions by \eqref{pi-theta}, so $\pi_{M_{\sd}}(\langle C\rangle)$ is a theta function.  Furthermore, since $\f{i}\circ\pi_{M_{\sd}}(\langle C\rangle) = \langle C'\rangle$ is theta positive in the $t=1$ setting by Lemma \ref{ClassicalPosNoPunct},  Lemma \ref{lem:positive_unfreeze_theta} ensures that $\f{i}$ applied to $\pi_{M_{\sd}}(\langle C\rangle)$ yields a theta function; i.e., $\langle C'\rangle$ is a theta function, as desired.

	\end{proof}

	By the definition of (quantum) tagged bracelet elements, we have the following useful observation as a special case of Lemma \ref{lem:theta-union-product} (also cf. Remark \ref{rem:surface-union}).
	\begin{lem}\label{union-product}
		Consider two marked surfaces (possibly with punctures) $\Sigma_1$ and $\Sigma_2$. For $i=1,2$, choose a weighted tagged simple multicurve $C_i$ on $\Sigma_i$, endow $\Sigma_i$ with an ideal triangulation $\Delta_i$, and let $\sd_i$ denote a (quantum) seed similar to $\sd_{\Delta_i}$. Let $\sd$ be the union of $\sd_1$ and $\sd_2$ as in \eqref{eq:seed-union} or the quantum seed analog.  Let $\Sigma=\Sigma_1\sqcup \Sigma_2$ with triangulation $\Delta=\Delta_1\sqcup \Delta_2$, so $\sd$ is similar to $\sd_{\Delta}$. If the (quantum) tagged bracelet elements $\E{L_i}$ are theta functions in $\s{A}^{\up}_t(\sd_i)$ respectively, then $\E{L_1\sqcup L_2}\in \s{A}^{\up}_t(\sd)$ is a theta function as well.
	\end{lem}

	\subsection{Annular loops in unpunctured surfaces}
	\label{sec:annular_loops}

	Recall that, for any given seed satisfying the injectivity assumption, we use $g(C)$ to denote the degree ($g$-vector) of a tagged bracelet element $\langle C\rangle$; cf. Definition \ref{defn:pointedness} and \S \ref{sec:g-vector}.
	
	\begin{lem}[Loops in the twice-marked annulus]\label{AnnLoop}
		Suppose $\Sigma$ is an annulus with one marking on each boundary component as in Example \ref{AnnEx}.  Let $L$ be the simple loop in $\Sigma$.  Then $\langle kL\rangle = \vartheta_{kg(L)}$ for each $k \in \bb{Z}_{\geq 0}$.
	\end{lem}
	\begin{proof}
		It suffices to prove this in the quantum setting.  Let $\alpha_2$ be any arc in $\Sigma$. The two tringulations containing $\alpha_2$ are pictured in Figure \ref{annuli_fig}, with the other arcs in these triangulations labelled $\alpha_1$ and $\alpha_{3}$.  By the skein relation (Figure \ref{SkeinFig}), we see $[\alpha_2]  [L] = q[\alpha_1]+q^{-1}[\alpha_{3}]$.  Recall the notation $$v_1:=\omega_1(e_1)=(0,2,-1,-1) \quad \mbox{and} \quad v_2:=\omega_1(e_2)=(-2,0,1,1)$$ as in \eqref{v1v2}.  In the seed associated to the triangulation $\Delta_{\ufv}=\{\alpha_1,\alpha_2\}$, we have the following $g$-vectors (using \eqref{fprime} to find $g(\alpha_3)$ and \eqref{gLoop} for $g(L)$):
		\begin{align*}
			g(\alpha_1)=(1,0,0,0), \quad g(\alpha_2)=(0,1,0,0), \quad  g(\alpha_{3})=-(1,0,0,0)+v_2 = (-1,0,1,1)
		\end{align*}
		and
		\begin{align*}
		    g(L)=(1,-1,0,0).
		\end{align*}
		So to show that $L=\vartheta_{g(L)}$, it suffices to show that \begin{align}\label{eq:theta-product}
		\vartheta_{(0,1,0,0)}\vartheta_{(1,-1,0,0)}=q\vartheta_{(1,0,0,0)}+q^{-1}\vartheta_{(-1,0,1,1)}.    
		\end{align}
		This can be checked by computing broken lines as in Example \ref{Ann-broken} and Figure \ref{Kr2q-broken}.  Indeed, for $\sQ$ in the positive chamber, $\vartheta_{(1,0,0,0),\sQ}$ and $\vartheta_{(0,1,0,0),\sQ}$ are the monomials $z^{(1,0,0,0)}$ and $\vartheta_{(0,1,0,0),\sQ}$, respectively, while $\vartheta_{(-1,0,1,1),\sQ}$ is the mutated cluster variable $z^{(-1,0,1,1)}+z^{(-1,2,0,0)}$. Noting that $(1,-1,0,0)=\frac{1}{2}(v_1+v_2)$, consideration of broken lines yields \begin{align*}
		    \vartheta_{(1,-1,0,0),\sQ} = z^{(1,-1,0,0)}+z^{(-1,1,0,0)}+z^{(-1,-1,1,1)};
		\end{align*}
		the broken lines contributing these three terms are illustrated in Figure \ref{fig:theta-limiting-wall}.
		The equality \eqref{eq:theta-product} now follows (referring to Example \ref{AnComp} for the $\Lambda$-matrix).
		
		Finally, the claim for higher $k$ follows from \eqref{ChebAnn} and the definition of the bracelet $\langle kL\rangle$  in \eqref{wLBrac}.
		
		\begin{figure}[htb]
	\centering
	\begin{tikzpicture}
		\draw
		(-2,0) -- (2,0) node[right] {$(e_2^{\perp},\Psi_{t^4}(z^{v_2}))$}
		(0,-2) -- (0,2) node[left] {$(e_1^{\perp},\Psi_{t^4}(z^{v_1}))$}
		(0,0) -- (2,-2) 
		(0,0) -- (1,-2) 
		(0,0) -- (1.333,-2)
		(0,0) -- (1.5,-2)
		(0,0) -- (1.6,-2)
		(0,0) -- (1.667,-2)
		(0,0) -- (1.714,-2)
		(0,0) -- (1.75,-2)
		(0,0) -- (1.778,-2)
		(0,0) -- (1.8,-2)
		(0,0) -- (1.818,-2)
		(0,0) -- (1.833,-2)
		(0,0) -- (1.846,-2)
		(0,0) -- (1.857,-2)
		(0,0) -- (1.867,-2)
		(0,0) -- (1.875,-2)
		(0,0) -- (2,-1) 
		(0,0) -- (2,-1.333)
		(0,0) -- (2,-1.5)
		(0,0) -- (2,-1.6)
		(0,0) -- (2,-1.667)
		(0,0) -- (2,-1.714)
		(0,0) -- (2,-1.75)
		(0,0) -- (2,-1.778)
		(0,0) -- (2,-1.8)
		(0,0) -- (2,-1.818)
		(0,0) -- (2,-1.833)
		(0,0) -- (2,-1.846)
		(0,0) -- (2,-1.857)
		(0,0) -- (2,-1.867)
		(0,0) -- (2,-1.875);
		\draw (1,0.6) node 
		{$\bullet$};
		\draw (1.3,.8) node 
		{$\sQ$};
		\draw[line width=.5mm] (0.4,-2) -- (-1.6,0) -- (0,1.6) -- (1,0.6);
		\draw[line width=.5mm] (2,-0.4) -- (1,0.6);
		\draw[line width=.5mm] (2,-1) -- (1,0) -- (1,0.6);
	\end{tikzpicture}
	\caption{
		\label{fig:theta-limiting-wall}}
\end{figure}
	\end{proof}
	\begin{figure}
		\centering
		\def\svgwidth{450pt}
		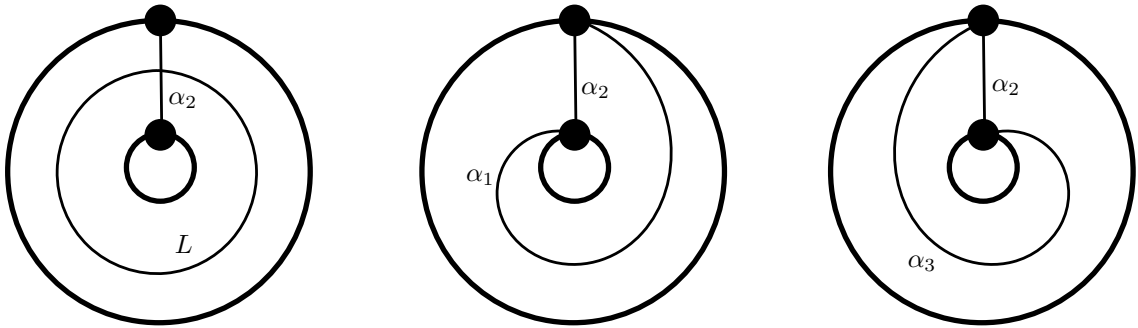
		\caption{$\alpha_2  L = q\alpha_1 + q^{-1}\alpha_{3}$.}
		\label{annuli_fig}
	\end{figure}
	
	We note that \cite[Thm. 2.8]{SZ} has previously shown for the classical cluster algebra of the Kronecker quiver (as in Example \ref{QuivEx} but without the frozen vertices; i.e., the twice-marked annulus without boundary coefficients) that there is an atomic basis satisfying the Chebyshev recursion.
	
	\begin{lem}\label{ArcLoop1Lem}
		If $L$ is a simple loop in an unpunctured surface $\Sigma$, and if there exists an arc $\gamma$ which intersects $L$ exactly one time, then $\langle kL\rangle\in\Sk_t(\Sigma)$ is equal to $\vartheta_{g(kL)}$ for each $k\in \bb{Z}_{\geq 0}$.
	\end{lem}
	\begin{proof}
		Let $U_L\subset \SSS$ be a small tubular neighborhood of $L$, and let $U_{\gamma}\subset \SSS$ be a small tubular neighborhood of $\gamma$ which is pinched at the endpoints.  Then $\Sigma_L:=U_L\cup U_{\gamma}$ is an annulus $\A$ containing $L$ and $\gamma$ and with one marked point on each boundary component, cf. Figure \ref{AnnCut} (in general, the two marked points might come from the same point in $\Sigma$).  Let $b_1,b_2$ be the two boundary components of $U_L\cup U_{\gamma}$.  The arcs $\gamma,b_1,b_2$ form part of a triangulation of $\Sigma$.  By Lemma \ref{AnnLoop}, $\langle kL\rangle=\vartheta_{g(kL)}$ for each $k\in \bb{Z}_{\geq 0}$ if we work in the cluster algebra associated to the marked annulus $\A$. This equality $\langle kL\rangle = \vartheta_{g(kL)}$ extends to the full surface $\Sigma$ by Lemma \ref{glue}.
		
		\begin{figure}[htb]
			\def\svgwidth{400pt}
			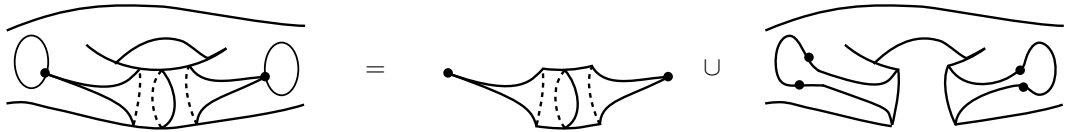
			\caption{Excising an annulus containing a bracelet.  The first component on the right-hand side is the annulus $\A$. \label{AnnCut}}
		\end{figure}
	\end{proof}

	In view of Lemmas \ref{AnnLoop} and \ref{ArcLoop1Lem}, we call a simple loop $L$ an \textbf{annular loop} if there exists an arc $\gamma$ in $\Sigma$ which intersects $L$ exactly once. 
		In this case, for any chosen marked point $r$ in the same connected component of $\SSS$ as $L$, we can choose $\gamma$ so that $r$ is at least one of its endpoints.

	\subsection{Non-annular loops in unpunctured surfaces}\label{sec:non-annular-loop}
	
		Unfortunately, not all non-contractible loops are annular. A non-contractible loop which is not annular will be called a \textbf{non-annular loop}. In view of Lemma \ref{union-product}, let us assume $\Sigma$ is connected.  Our goal in this subsection is to prove that if $L$ is a non-annular loop, then the bracelets $\langle kL\rangle$ are theta functions.  Since the argument is somewhat complicated, we summarize it here.
	
	We begin in \S \ref{subsub:cut} by cutting $\Sigma$ down to a simpler surface $\Sigma_L$ with only one boundary component and one marked point, and with $L$ homotopic to the boundary (Figure \ref{DLoop}).
	
	Then in \S \ref{subsub:nonannular1} we show that $[L]$ is a theta function for $\Sigma_L$.  To do this, we first show that $[L]$ is universally positive (in the quantum setting), hence a sum of $s\geq 1$ theta functions (a priori with powers of $t$ for coefficients).  We also find that $[L]$ times an interior arc $[\gamma]$ is a sum of three theta functions (Figure \ref{ArcLoop2}), so strong positivity of theta functions implies $s\leq 3$. More careful inspection rules out $s=2,3$, implying that $[L]$ really is a theta function.
	
	In \S \ref{subsub:mcg} we review the mapping class group $\Gamma_{\Sigma}$ and apply results of \cite{gross2018canonical,davison2019strong} to show that $\Gamma_{\Sigma}$ acts equivariantly on theta functions; cf. Lemmas \ref{GammaTheta} and \ref{cor:tau-m}.  This is applied in \S \ref{subsub:Dehn} to show that $[L]^s$ is a linear combination of theta functions $\vartheta_{kg(L)}$ for $k\in \bb{Z}_{\geq 0}$ (Lemma \ref{kg}), essentially because $\langle kL\rangle$ are the only bracelets (up to multiplication by the boundary arc) on which the mapping class group $\Gamma_{\Sigma_L}$ acts trivially (or even with finite orbit).
	
	We give a partial description of $[L]^k$ times an interior arc $[\gamma]$ in Lemma \ref{lem:const_gammaLk} using the skein relations and induction.  By combining this with Lemmas \ref{kg} and \ref{Zp}, we deduce that, for $\sQ$ sufficiently near $g(L)$, we have $\vartheta_{g(L),\sQ}=z^{g(L)}+z^{-g(L)}+[\text{higher order terms]}$.  Finally, in Lemma \ref{ChebyLem}, we use the structure constants formula (Proposition \ref{StructureConstants}) and the characterization of the Chebyshev polynomial $T_k$ in Lemma \ref{Lem:Tk} to conclude that $\vartheta_{kg(L)}=T_k(\vartheta_{g(L)})$, hence $\langle kL\rangle = \vartheta_{kg(L)}$.

	\subsubsection{Cutting out a simpler surface $\Sigma_L$}\label{subsub:cut}
	
	\begin{lem}\label{non-annular-L}
		Let $L$ be a non-annular loop in a connected marked surface $\Sigma$.  Let $r\in \MM$ be a marked point. Then there is an arc $\alpha_L$ disjoint from $L$ such that:
		\begin{itemize}
			\item both ends of $\alpha_L$ are at $r$;
			\item The connected component $\SSS_L^{\circ}$ of $\SSS\setminus \alpha_L$ which contains $L$ has positive genus and contains no marked points in its closure except for $r$;
			\item $L$ is homotopic to the image of $\alpha_L$ in $\Sigma$.
		\end{itemize}
	\end{lem}
	This situation is illustrated in Figure \ref{DLoop}.
	\begin{figure}[htb]
		\def\svgwidth{420pt}
		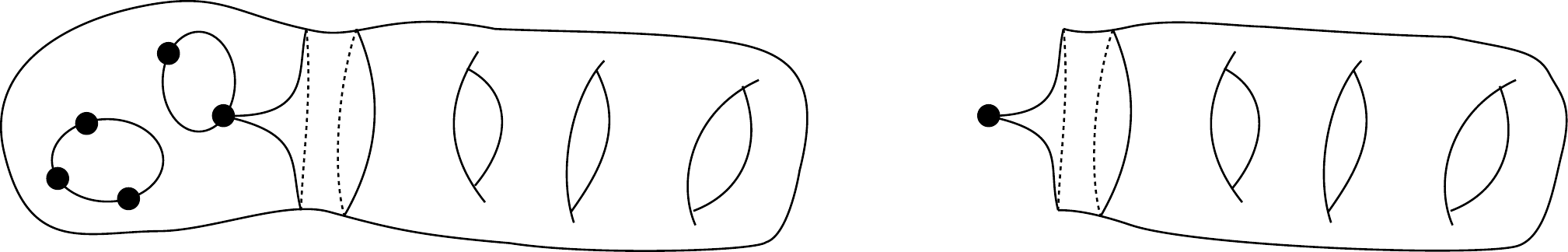
		\caption{If there is not an arc intersecting $L$ exactly once, then we can cut along an arc $\alpha_L$ to simplify the picture.\label{DLoop}}
	\end{figure}
	\begin{proof}
		Consider the compactification of $\SSS\setminus L$ obtained by adding two copies $L_1$ and $L_2$ of $L$ to the boundary.  If $\SSS\setminus L$ were connected, then there would be paths from $r$ to $L_i$ for each $i=1,2$, and these could be glued in $\SSS$ to construct an arc intersecting $L$ exactly once.  Similarly, if $\SSS\setminus L$ is disconnected but each component contains marked points, then one could find an arc between the two components (when glued to form $\SSS$) which intersects $L$ exactly once.
		
		So a non-annular loop $L$ must disconnect $\SSS$, and one of the resulting connected components must have no marked points.  Let $U_L$ be a small tubular open neighborhood of $L$, and let $\alpha$ be a path from $r$ to $\partial \overline{U_L}$ which does not intersect $U_L$.  We then construct $\alpha_L$ by following $\alpha$ from $r$ to $\partial \overline{U_L}$, wrapping around this component of $\partial \overline{U_L}$, and then following along $\alpha$ back to $r$. The fact that this $\alpha_L$ satisfies the desired conditions is clear from the construction.
	\end{proof}
	We shall continue to use the notation $\alpha_L$ as in Lemma \ref{non-annular-L}. We let $\SSS_L$ denote the closure of $\SSS_L^{\circ}$ in $\SSS$, and let $\Sigma_L=(\SSS_L,\{r\})$. Note that $\Sigma_L$ is a unpunctured triangulable marked surface.

	\subsubsection{A single non-annular loop with weight $1$}\label{subsub:nonannular1}
	\begin{lem}\label{non-annular-theta}
		$[L]\in\Sk_t(\Sigma_L)$ is a theta function.
	\end{lem}
	We will prove Lemma \ref{non-annular-theta} by induction on the genus of $\Sigma_L$.  We therefore assume for the rest of \S \ref{subsub:nonannular1} that either $\Sigma_L$ has genus $1$, or that Lemma \ref{non-annular-theta} holds whenever $\Sigma_L$ is of lower positive genus (i.e., the induction hypothesis holds).
	\begin{lem}[Quantum universal positivity for simple loops]\label{lem:q-univ-pos}
		For any simple loop $L'\in \Sigma_L$, $[L']\in \Sk_t(\Sigma_L)$ is universally positive with respect to the scattering atlas.
	\end{lem}
	\begin{proof}
		The positivity for the annular loops is known by Lemma \ref{ArcLoop1Lem}. Non-annular loops $L'$ in $\Sigma_L$ which are not homotopic to $L$ exist only if the genus of $\Sigma_L$ is larger than $1$---in this case we can construct $\Sigma_{L'}$ inside $\Sigma_L$.  Then the induction assumption for Lemma \ref{non-annular-theta} implies that  $[ L']$ is a theta function for $\Sigma_{L'}$, hence a theta function for $\Sigma_L$ using the gluing lemma (Lemma \ref{glue}).  So it only remains to show that $L$ is universally positive.
				
		Note that for every isotopy class of non-boundary arc in $\Sigma_L$, the minimal number of intersections of a representative of the isotopy class with $L$ is exactly two.  Let $\gamma$ be any such non-boundary arc.  Then $[\gamma]  [L]$ is a sum of three terms as in Figure \ref{ArcLoop2} (the fourth term coming from the skein relations is equivalent to $0$ because it contains a contractible arc; the $q$-coefficients are suppressed in the figure).  In the quantum setting, we have
		\begin{align}\label{gammaL}
			[\gamma]  [L]=q^2[\tau_L^{1/2}(\gamma)]+q^{-2}[\tau_L^{-1/2}(\gamma)] + [\alpha_L][L_0].
		\end{align}
		The first two terms on the right-hand side in \eqref{gammaL} are arcs (the motivation for the notation $\tau_L^{\pm 1/2}(\gamma)$ will become clear when we discuss Dehn twists in \S \ref{subsub:Dehn}), which correspond to (quantum) cluster variables, and these are universally positive by \cite{DavPos}.  The third term is the product of the frozen variable $\alpha_L$ and a simple loop $L_0$ non-homotopic to $L$, which are both universally positive.
		
		Hence, $[\gamma] [L]$ is universally positive for any interior arc $\gamma$. For any triangulation $\Delta$, we may choose $\gamma$ to be an interior arc in $\Delta$, so then $[\gamma]$ is a cluster variable in $\sd_\Delta$ and the positivity of $[L]$ in $s_{\Delta}$ follows from that of $[\gamma] [L]$. Since this works for any triangulation $\Delta$, it follows that $[L]$ is universally positive with respect to the cluster atlas (which equals the scattering atlas by Proposition \ref{gdense}) as desired.
	\end{proof}

	\begin{figure}[htb]
		\def\svgwidth{425pt}
		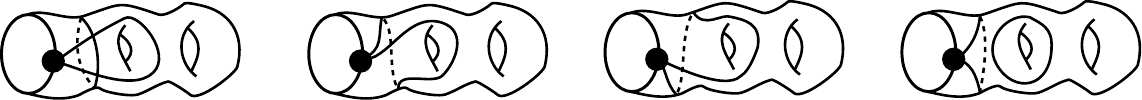
		\caption{An arc times a loop resulting in a sum of two arcs plus the product of an arc and a disjoint loop.
			\label{ArcLoop2}}
	\end{figure}

		\begin{lem}\label{lem:leading_term_not_loop}
			In \eqref{gammaL}, we have $g(\gamma)+g(L)\neq g(\alpha_L)+g(L_0)$. 
		\end{lem}

		\begin{proof}
			Consider $\bb{Z}$ acting on $\Sk_t(\Sigma)$ via Dehn twisting by $L$.  The elements $[L]$, $[\alpha_L]$, and $[L_0]$ are all invariant under this action, but the orbit of $[\gamma]$ is infinite.  Thus, under the induced action on $g$-vectors (see \S \ref{subsub:mcg}), the orbit of $g(\gamma)+g(L)$ is infinite while $g(\alpha_L)+g(L_0)$ is fixed, so these two sums cannot be equal.
		\end{proof}
		
		We will use the following general fact about theta functions:
		\begin{lem}\label{lem:prod_2_not_cv}
			If $\vartheta_{g_1}$ and $\vartheta_{g_2}$ are theta functions such that $\pr_{I_\ufv}(g_1),\pr_{I_\ufv}(g_2)\neq 0$, then $\vartheta_{g_1}\vartheta_{g_2}$ is not a cluster variable (and in the quantum setting it is not a power of $t$ times a cluster variable).
		\end{lem}
		\begin{proof}
			By positivity, it suffices to prove this in the classical setting.  If $g_1$ and $g_2$ share a chamber of $\s{C}$, then $\vartheta_{g_1}\vartheta_{g_2}=\vartheta_{g_1+g_2}$ is the cluster monomial $z^{g_1+g_2}$, but this cannot be a cluster variable by our assumption on $g_1,g_2$. Now suppose that $g_1$ and $g_2$ do not share a chamber of $\s{C}$.  Let $C$ be an arbitrary chamber of $\s{C}$ and let $\sQ\in C$ be generic.  Then at least one $g_i$ is not in $C$, in which case there are at least two broken lines contributing to $\vartheta_{g_i,\sQ}$.  One is the straight broken line with ends $(g_i,\sQ)$. Since the traveling direction of a broken line is bent by tangent vectors of the walls, see \S \ref{sec:broken_lines}, there is also a broken line with initial end $g_i$ which is straight until crossing the boundary of $C$, where it bends non-trivially. So then (using universal positivity) $\vartheta_{g_1,\sQ}\vartheta_{g_2,\sQ}$ has at least two terms, hence is not a cluster variable (or even a cluster monomial) in the cluster associated to $C$.  Since $C$ was an arbitrary chamber of $\s{C}$, this proves the claim.
		\end{proof}

		Now, we prove Lemma \ref{non-annular-theta} under the assumption that either $\Sigma_L$ has genus $1$ or the induction hypothesis holds.  That is, we will show that $[L]\in \Sk_t(\Sigma_L)$ is a quantum theta function. Our arguments  are based on the decomposition of $[\gamma] [L]$ as in Figure \ref{ArcLoop2} together with positivity.
		
		\begin{proof}[Proof of Lemma \ref{non-annular-theta}]
			As in the proof of Lemma \ref{lem:q-univ-pos}, we may select an arc $\gamma$ which intersects $L$ twice, and then the product $[\gamma][L] $ in the quantum skein algebra is given as in \eqref{gammaL}.  The first two terms here are cluster variables (times $q^{\pm 2}$), and the third term is also a theta function by Lemma \ref{ArcLoop1Lem} (because the product of a theta function with a frozen variable is still a theta function by Lemma \ref{lem:similar_theta_function})\footnote{Generally, a theta function times a frozen variable may actually be a power of $t$ times a theta function (which would still be fine for our purposes here).  However, $g$-vectors of boundary arcs always lie in $\ker(\Lambda)$, so this power of $t$ does not appear in the present setting.}.

			Since $[L]$ is universally positive and the theta functions are atomic, $L$ must be a positive linear combination of theta functions, i.e., a sum of theta functions with coefficients in $\bb{Z}_{\geq 0}[t^{\pm 1}]$.  So we can write
			\[
			[ L] =\sum_{i=1}^s t^{a_i} \vartheta_{p_i}
			\]
			for some collection $p_1,\ldots,p_s\in M$ (possibly with repetition) and $a_1,\ldots,a_s\in \bb{Z}$ with $s\geq 1$.  By strong positivity, $[\gamma] $ times any $\vartheta_{p_i}$ is again a positive linear combination of theta functions.  By Lemma \ref{lem:prod_2_not_cv}, a theta function $\vartheta_{p_i}$ times the cluster variable $[\gamma]$ will never be just a single cluster variable times a power of $t$ unless\footnote{Lemma \ref{lem:prod_2_not_cv} would allow $p_i\neq 0$ if at least $\pr_{I_{\uf}}(p_i)=0$, but in this case $\vartheta_{p_i}$ is a power of a frozen cluster variable while $[\gamma]$ is a non-frozen cluster variable, so the product $[\gamma]\vartheta_{p_i}$ would be a cluster monomial but not a cluster variable.} $p_i=0$, i.e., unless $\vartheta_{p_i}=1$.  But no $p_i$ here can equal $0$ since $\gamma$ is not a term in $[\gamma] [L]$, so the fact that $[\gamma][ L]$ is a sum of three theta functions (times powers of $t$) including two cluster variables implies that $s$ equals $1$ or $2$.  
			
			The case $s=1$ yields the claim (by bar-invariance, $a_1=0$ in this case), so it suffices to rule out the possibility of $s=2$.

			If $s=2$, then we can write 
			\[
			[ L]=\vartheta_{g(L)}+\vartheta_p
			\]
			for some $p\in M$ --- there must be a term $\vartheta_{g(L)}$ by the pointedness of $[L]$, 
			while the bar-invariance of bracelets and theta functions (cf. Lemma \ref{lem:theta-bar-inv}) implies that the coefficient of $\vartheta_p$ is also $1$. As previously noted, we know that $p\neq 0$ since $[\gamma]$ is not a term in $[\gamma][L]$.
			
			Note that, since neither $[\gamma]\vartheta_{g(L)}$ nor $[\gamma]\vartheta_p$ is a cluster variable (by Lemma \ref{lem:prod_2_not_cv} again), one of these products must equal the first two terms (the sum of the cluster variables) in \eqref{gammaL}, and the other one must be the remaining theta function $[\alpha_L][L_0]$, which we denote by $\vartheta_u$. So $[\gamma]\vartheta_v=\vartheta_u$ with $v$ being either $g(L)$ or $p$. Lemma \ref{lem:leading_term_not_loop} implies that $v\neq g(L)$, so $v=p$. Since the leading monomial in $[\gamma]\vartheta_p=\vartheta_u$ has coefficient $t^{\Lambda(g(\gamma),p)}=1$, we have $\Lambda(g(\gamma),p)=0$.

			Notice that we have $\Lambda(g(\alpha_L),\ )=0$. By the above discussion, $\Lambda(g(\gamma),p)=0$ for any non-boundary arc $\gamma$. Hence, $\Lambda(\ ,p)=0$. Since $\Lambda$ is compatible with the $B$-matrix and there is only one frozen variable, we deduce that $\ker(\Lambda_2)$ has dimension $1$ and equals $\bb{Z} g(\alpha_L)$, so $p\in \bb{Z}g(\alpha_L)$.  But then $\vartheta_p$ would be a power of a frozen cluster variable, so $[\gamma]\vartheta_v$ would be a cluster monomial, not the term $\vartheta_u=[\alpha_L][L_0]$.  This contradiction implies $s=1$, i.e., $[L]=\vartheta_{g(L)}$, as desired.
		\end{proof}
		
		\subsubsection{The mapping class group}\label{subsub:mcg}
		
		We briefly recall how mapping class groups acts on cluster algebras arising from (possibly punctured) marked surfaces.  Here, the mapping class group $\Gamma_{\Sigma}$ of $\Sigma$ is the group of isotopy classes of orientation-preserving diffeomorphisms $\SSS\rar \SSS$ which take $\MM$ to $\MM$.
  
		Recall that cluster mutations give isomorphisms between the skew-fields of fractions of the clusters associated to $\sd$ and $\sd'$, where $\sd$ and $\sd'$ are any two seeds related by mutations. Correspondingly, any rational function in the cluster variables of $\sd$ is sent to a rational function in the cluster variables of $\sd'$. 
		
		Note that every element of $\Gamma_{\Sigma}$ takes a triangulation $\Delta$ to another triangulation $\Delta'$. Correspondingly, we obtain a new seed $\sd'=\sd_{\Delta'}$ from $\sd=\sd_{\Delta}$. Note that the seeds $\sd_{\Delta}$ and $\sd_{\Delta'}$ are isomorphic, so we can also view this action as an automorphism of the cluster algebra by identifying the cluster variables of $\sd_{\Delta}$ with those of $\sd_{\Delta'}$ (i.e., the mapping class group is contained in the cluster modular group of \cite{FG1}).
	
		Thus, for $\tau\in \Gamma_{\Sigma}$, we have $\tau':\s{A}_t^{\can}(\sd_{\Delta})\risom \s{A}_t^{\can}(\sd_{\Delta'})$, $\vartheta_{m}\mapsto \vartheta_{\tau'(m)}$, where $\tau':M_{\sd_{\Delta}}\rar M_{\sd_{\Delta'}}$ is defined to be the linear map determined by $e_{\sd_{\Delta},i}^*\mapsto e_{\sd_{\Delta'},i}^*$ for all $i\in I$.  On the other hand, let $\jj$ be a sequence of elements in $I\setminus F$ such that $(\sd_{\Delta})_{\jj}=\sd_{\Delta'}$ in the notation of \S \ref{sec:seed-mut}. By \cite[Cor. A.4]{davison2019strong} (or \cite[Thm. 1.24]{gross2018canonical} in the classical setting), there is an isomorphism $\psi_{\jj}:\s{A}_t^{\can}(\sd_{\Delta}) \risom \s{A}_t^{\can}(\sd_{\Delta'})$ given by $\vartheta_p \mapsto \vartheta_{T_{\jj}(p)}$ for $T_{\jj}$ as in \eqref{Tjj}.

		Thus, the composition $\psi_{\jj}^{-1}\circ \tau':\vartheta_m\mapsto \vartheta_{T_{\jj}^{-1}\circ \tau'(m)}$ is an automorphism of $\s{A}_{\sd_{\Delta}}^{\can}$.  By \cite[Prop. 4.9(2)]{davison2019strong}, $\vartheta_{T_{\jj}^{-1}\circ \tau'(e_{\gamma}*)} = \vartheta_{g(\tau(\gamma))}$ for each $\gamma\in \Delta$.  So the corresponding automorphism of $\Sk_t(\Sigma)$ maps $[\gamma]$ to $[\tau(\gamma)]$ for each $\gamma\in \Delta$.  Hence, $\psi_{\jj}^{-1}\circ \tau'$ coincides with the natural action of $\tau$ on $\Sk_t(\Sigma)$.  We therefore denote $\tau=\psi_{\jj}^{-1}\circ \tau'$.  We also let $\tau$ denote the piecewise-linear automorphism of $M$ given by $T_{\jj}^{-1}\circ \tau'$.  To summarize, we have the following:
		
		\begin{lem}\label{GammaTheta}
			The mapping class group $\Gamma_{\Sigma}$ acts on $M$ via piecewise-linear automorphisms.  Furthermore, for each $\tau\in \Gamma_{\Sigma}$, the action of $\tau$ on $\Sk_t(\Sigma)$ maps theta functions to theta functions via $\vartheta_{m}\mapsto \vartheta_{\tau(m)}$.
		\end{lem}

	\begin{lem}\label{cor:tau-m}
	    If $\tau(\vartheta_m)=\vartheta_{m'}$, then $m'=\tau(m)$.  If $\tau \in \Gamma_{\Sigma}$ fixes $\vartheta_{m_1},\ldots,\vartheta_{m_s}$, then for any collection $k_1,\ldots,k_s\in \bb{Z}_{\geq 0}$, $\tau$ must permute the theta functions $\vartheta_p$ appearing in the expansion $f\coloneqq \prod_{i=1}^s \vartheta_{m_i}^{k_i} = \sum_p c_p \vartheta_p$.  Thus, $\tau$ acts with finite orbit on each index $p$ of these theta functions, in particular for $p=\sum_{i=1}^s k_i m_i$.
	\end{lem}
	\begin{proof}
	    The first statement follows from the fact that $\tau(\vartheta_m)=\vartheta_{\tau(m)}$ in Lemma \ref{GammaTheta} and the fact that, if $\vartheta_m=\vartheta_{m'}$, then $m=m'$ (a consequence of Lemma \ref{indep}).
	    
	    For the second statement of the lemma, since $\tau$ fixes each $\vartheta_{m_i}$, it must fix $f$, so we see (using the linear independence of theta functions from Lemma \ref{indep} again) that $\tau$ must indeed permute the theta functions appearing in the summation.  Combining this with the first statement of the lemma yields that the elements $p\in M$ with $c_p\neq 0$ must be permuted.  Since there are only finitely many such $p$, the orbits of these elements are finite.  Finally, the fact that $a_p\neq 0$ for $p=\sum_{i=1}^s k_i m_i$ is a consequence of Corollary \ref{cor:alpha-sum}.
	\end{proof}

        Note that $\tau$ sends the cluster variable $[\gamma]$ for $\gamma\in \Delta$ to $[\tau(\gamma)]$.  Since all other elements of $\Sk_t(\Sigma)$ can be expressed a Laurent polynomials in the cluster variables for $\Delta$, we obtain the following:
        \begin{lem}
        For any weighted simple multicurve $C$, the action of $\tau$ on $\Sk_t(\Sigma)$ maps $\E{C}$ to $\E{\tau(C)}$.
        \end{lem}

    Recall from \S \ref{sec:g-vector} that for unpunctured $\Sigma$, or more generally, for a seed $\sd$ of surface type satisfying the Injectivity Assumption, every bracelet element $\beta$ is pointed.  We write $\beta_m$ for the $m$-pointed bracelet element. 
    \begin{lem}\label{lem:infinite-orbit}
        For unpunctured $\Sigma$, or more generally, for $\sd$ of surface type with $M_{\sd}^{\oplus}$ strongly convex, if a bracelet element $\beta_g$ has infinite orbit under the action of some $\tau\in \Gamma_{\Sigma}$, then $g$ also has infinite orbit under the action of $\tau$.
    \end{lem}
    \begin{proof}
        By the pointedness of the bracelets, if we expand $\vartheta_g$ in the bracelets basis\footnote{In the surface type setting with strongly convex $M_{\sd}^{\oplus}$, the bracelets still form a basis over $\kk[M_F]$, so this does not affect the proof except that the coefficients $c_m$ may like in $\kk[M_F]$.} as $\vartheta_g=\sum_m c_m \beta_m$, then $c_g=1$.  Let ${\bf B}\ni \beta_g$ be the set of all bracelets which contribute non-trivially to the bracelets basis expansion of at least one element of the form $\tau^k(\vartheta_g)$, $k\in \bb{Z}$.  If the orbit of the $\tau$-action on $\vartheta_g$ were finite, then ${\bf B}$ would be finite, but ${\bf B}$ must be infinite since it contains $\tau^k(\beta_g)$ for all $k\in \bb{Z}$.  So the orbit of $\vartheta_g$ must be infinite, hence the orbit of $g$ must be infinite as well.
    \end{proof}

		\subsubsection{Dehn twists and non-annular loops}\label{subsub:Dehn}
		
		Given any loop $L$ with a fixed choice of orientation, let $\tau_L\in \Gamma_{\Sigma}$ denote the corresponding Dehn twist of $\Sigma$ (denoted $\tw_L$ in \S \ref{sec:Dehn_twists}).  As in \S \ref{subsub:mcg}, $\tau_L$ acts on the associated algebras and lattices.  We now return to the setting of the surface $\Sigma_L$ as in \S \ref{subsub:cut}--\S\ref{subsub:nonannular1}.

		Let $\gamma\in \Sk_t(\Sigma_L)$ be any non-boundary arc of $\Sigma_L$, so $\gamma$ intersects $L$ twice (after applying an isotopy to minimize the number of intersections). Recall that, as pictured in Figure \ref{ArcLoop2}, we have \eqref{gammaL}: 
		\begin{align*}
			[\gamma] [L]=q^2[\tau_L^{1/2}(\gamma)]+q^{-2}[\tau_L^{-1/2}(\gamma)] + [\alpha_L][L_0]
		\end{align*}
		for a pair of arcs $\tau_L^{\pm 1/2}(L)$ and some non-contractible simple loop $L_0$ non-isotopic to $L$ and disjoint from $L$ --- here, the notation $\tau_L^{\pm 1/2}(L)$ for the two arcs can be viewed formally, but it is chosen because, as we shall see next, applying similar skein relations to compute $[\tau_L^{\pm 1/2}(\gamma)][L]$ yields the positive and negative Dehn twists of $\gamma$ by $L$ (we assume the orientation of $L$ is chosen so that the signs work out as indicated). Correspondingly, let us call $\tau_L^{1/2}(L)$ and $\tau_L^{-1/2}(L)$ the positive arc and negative arc, respectively in the decomposition of $[\gamma] [L]$.
		
		Using the skein relations again, one computes the decomposition $[\tau_L^{\pm 1/2}(\gamma)][L]$ to find:
		\begin{align*}
			[\tau_L^{1/2}(\gamma)][L] &= q^2[\tau_L(\gamma)] + q^{-2}[\gamma ]+ [\alpha_L][L_{\hf}]\\    
			[\tau_L^{-1/2}(\gamma)][L] &= q^2[\gamma] + q^{-2}[\tau_L^{-1}(\gamma)]+ [\alpha_L][L_{-\hf}]
		\end{align*}
		for other non-peripheral simple loops $L_{\pm\hf}$ non-isotopic to $L$ and disjoint from $L$. Therefore, by repeatedly multiplying $[\gamma]$ by $L$, one finds that
		\begin{align}\label{eq:gammaL2}
			[\gamma]\cdot[L]^2 = 2[\gamma ]+ q^4[\tau_L(\gamma) ]+ q^{-4}[\tau_L^{-1}(\gamma)]+ q^2[\alpha_L][L_{\hf}] + [\alpha_L ][L][L_0]+q^{-2}[\alpha_L][L_{-\hf}].
		\end{align}

		Similarly, we can recursively define arcs $\tau_L^{d}(\gamma)$, $d\in\hf\Z$, via the relation
		\begin{align}\label{eq:gammaLd}
			[\tau_L^{d}(\gamma)][L]& = q^2[\tau^{d+\hf}_L(\gamma)] + q^{-2}[\tau_L^{d-\hf}\gamma ]+ [\alpha_L][L_{d}]
		\end{align}
		for some non-peripheral simple loop $L_{d}$ non-isotopic to $L$ and disjoint from $L$.\footnote{We observe that $L_{d}=L_{d+1}$ for any $d\in \frac{1}{2}\Z$. This observation will not be used.} As before, we call the first two terms in the decomposition the positive arc and the negative arc respectively. We note that, when $d\in \bb{Z}$, $\tau_L^{d}(\gamma)$ is indeed equal to the $d$-th power of the Dehn twist.
		
		The following result partially generalizes \eqref{eq:gammaL2}.
		\begin{lem}\label{lem:const_gammaLk}
			For any $k\in\Z_{>0}$, we have a decomposition into bangles:
			\begin{align*}
				[\gamma] [L]^{k} = \sum_{\substack{d\in \frac{1}{2}\bb{Z} \\ |d|\leq \frac{k}{2}}}c_d[\tau_L^{d}(\gamma)]+\sum_{\substack{r\in \bb{Z}_{\geq 0}, s\in \frac{1}{2}\bb{Z} \\ 0\leq r\leq k-1 \\ |s|\leq \frac{k-1}{2}}}l_{(r,s)}[\alpha_L][L]^r[L_{s}].
			\end{align*}
			where the coefficients $c_d$ and $l_{(r,s)}$ are Laurent polynomials in $q$ with non-negative integer coefficients. Moreover, we have $c_0=\binom{k}{k/2}$ if $2$ divides $k$ and $c_0=0$ otherwise.
		\end{lem}
		\begin{proof}
			The first claim follows from a straightforward induction argument using \eqref{eq:gammaLd}.

			Note that the coefficients $c_d$ are the combined contributions of the coefficients from iterations of the first two terms appearing in the decomposition \eqref{eq:gammaLd}. In particular, we can calculate the coefficient $c_0$ as follows: when multiplying $[\gamma]$ by $[L]$ repeatedly $k$ times, we choose either the positive arc or the negative arc in \eqref{eq:gammaLd} at each step ($k$ steps in total). In order to finally obtain the arc $[\gamma]$, the number of times where we choose the positive arc must equal the number of times where we choose the negative arc, and so $c_0$ counts the number of ways to choose the positive arc exactly $\frac{k}{2}$ times.  The claim about the value of $c_0$ follows.
		\end{proof}

		\begin{lem}\label{lem:product_loops_decompose}
			Let $L_1,\ldots,L_r$ denote a collection of disjoint simple loops (possibly isotopic) in $\Sigma_L$. When decomposing the product $[\alpha_{L}]\prod_i [L_i]$ into a finite linear combination of theta functions, no theta functions $\vartheta_g$ appearing can correspond to an internal arc.
		\end{lem}
		\begin{proof}
			Every internal arc has infinite orbit under the action of $\tau_L$, but $\tau_L$ acts trivially on the product $[\alpha_{L}]\prod_i [L_i]$. The claim thus follows from Lemma \ref{cor:tau-m}.
		\end{proof}

		\begin{lem}\label{kg}
			For each $s\in \bb{Z}_{\geq 0}$, $[L]^s$ is a finite linear combination of theta functions of the form $\vartheta_{kg(L)}$ for $k\in \bb{Z}_{\geq 0}$.
		\end{lem}
		\begin{proof}
			Note that $L$ and the boundary arc $\alpha_L$ are the only arcs or loops in $\SSS$ which are invariant under the action of the mapping class group of $\Sigma$.  All others have infinite orbits under some Dehn twists.  So by Lemma \ref{lem:infinite-orbit}, the only elements in $M$ which have finite orbit under the mapping class group action are those of the form $ag(\alpha_L)+bg(L)$ for $a\in \bb{Z}$ and $b\in \bb{Z}_{\geq 0}$. Hence, the theta functions invariant under the action of the mapping class group are of the form $\vartheta_{ag(\alpha_L)+bg(L)}$ for $a\in \bb{Z}$ and $b\in \bb{Z}_{\geq 0}$, while all other theta functions have infinite orbits under this action. It therefore follows from Lemma \ref{cor:tau-m} that $\vartheta_{g(L)}^s$ is a linear combination of theta functions of the form $\vartheta_{ag(\alpha_L)+bg(L)}$ for $a\in \bb{Z}$ and $b\in \bb{Z}_{\geq 0}$.
			
			On the other hand, by \eqref{eq:loop_deg_monoid}, $g(L)\in M_{\uf,\bb{Q}}$. So, since bends of broken lines are in $M_{\uf}$, Proposition \ref{StructureConstants} implies that theta functions contributing to $\vartheta_{g(L)}^s$ must all have $g$-vector in $M_{\uf,\bb{Q}}$.  But $g(\alpha_L)$ cannot be in $M_{\uf,\bb{Q}}$---indeed, $\Lambda$ is non-degenerate on $M_{\uf}$, and one can easily see from the definition of $\Lambda$ (cf. \eqref{Lambdaij} and \eqref{Lambda-ei-ej}) that $g(\alpha_L) \in \ker(\Lambda)$.  The claim follows.
		\end{proof}

		The following is a general lemma on theta functions which satisfy the conclusion of Lemma \ref{kg}:

		\begin{lem}\label{Zp}
			Let $\vartheta_p$ be a theta function such that, for all $s\in \bb{Z}_{\geq 0}$, $\vartheta_p^s$ is a linear combination of theta functions of the form $\vartheta_{kp}$ with $k\in \bb{Z}_{\geq 0}$.  Fix $\ell\in \bb{Z}_{\geq 1}$.  Then for any generic $\sQ$ sufficiently close to $p$, the only broken lines in $M_{\bb{R}}$ with ends $(p,\sQ)$ 
			and final exponent in $M\setminus (p+\ell M^+)$ are the straight broken line and broken lines with final exponents of the form $-bp$ for $b\in \bb{Z}_{\geq 1}$.
		\end{lem}
		
		\begin{proof}	
			\textit{Ruling out exponents $bp$ with $b\geq 0$:} Let us consider $bp\in M$ with $b\in \bb{Q}_{\geq 0}$. For the given $\ell\in \bb{Z}_{\geq 1}$ and any $\sQ$ sharing a chamber of $\f{D}_{\ell}$ with $p$, the only broken line with respect to $\f{D}_{\ell}$ ending at $\sQ$ with the final exponent $bp$ is the straight broken line with initial direction $bp$. This broken line contributes to $\vartheta_p$ if and only if $b=1$.

			So now it suffices to show that the broken lines with ends $(p,\sQ)$ for generic $\sQ$ sufficiently close to $p$ have final exponent in $\bb{Z}p$.
			
			\textit{Ruling out exponents in $M\setminus \bb{Z}p$:} Consider any $m\in M$.  Choose $\ell_m\in \bb{Z}_{\geq 1}$ such that 
			\begin{align}\label{eq:mnotin}
				m\notin p+\ell_m M^+.
			\end{align}  
			Then any broken line $\gamma_{p,\sQ,m}$ with ends $p$ and $\sQ$ and final exponent $m$ will only bend at walls of $\f{D}_{\ell_m}$.  Note that, for sufficiently large $c\in \bb{Z}_{\geq 1}$, $m+cp$ will share a chamber of $\f{D}_{\ell_m}$ with $p$.  Further observe that $m+cp\notin (c+1)p+\ell_m M^+$ (by adding $cp$ to both sides of \eqref{eq:mnotin}), 
			i.e., $\ell_m$ is sufficiently large for applying Lemma \ref{lem:fQ} to the coefficient $a_{\vartheta_p^{c+1},\sQ,m+cp}$ (in the notation of loc. cit.; see \eqref{eq:large_l_from_max_deg}).  So take $\sQ$ sharing a chamber of $\f{D}_{\ell_m}$ with $m+cp$, and suppose there exists a broken line $\gamma_{p,\sQ,m}$ with ends $(p,\sQ)$ and final exponent $m$.  Then Proposition \ref{StructureConstants} implies that the $\vartheta_{m+cp}$-coefficient of the theta function decomposition of $\vartheta_p^{c+1}$ is nonzero because it will receive a positive contribution from any $(c+1)$-tuple of broken lines consisting of $\gamma_{p,\sQ,m}$ together with $c$ copies of the straight broken line with ends $(p,\sQ)$.  So by assumption, $m+cp$ must be in $(\bb{Z}_{\geq 0})p$, hence $m\in \bb{Z}p$.

			\textit{Existence of sufficiently close $\sQ$:} Given $\ell$ as in the statement of the lemma, the set $$M_{p,\ell}:=(p+M^{\oplus})\setminus (p+\ell M^+)$$ is finite.  So we can take ``sufficiently close'' in the statement of the lemma to mean sharing a chamber of $\f{D}_{\ell'}$ with $p$ for $\ell'=\max \left(\{\ell_m|m\in M_{p,\ell}\}\cup \{\ell\}\right)$.  The claim follows.
			\end{proof}

		We are at last ready to prove that bracelets for weighted non-annular loops are theta functions.
		
		\begin{lem}\label{ChebyLem}
			Let $L$ be a simple non-annular loop in the unpunctured triangulable surface $\Sigma$.  Then $\langle kL\rangle =\vartheta_{kg(L)}\in \Sk_t(\Sigma)$ for each $k\in \Z_{\geq 1}$.
		\end{lem}
		\begin{proof}
			By the gluing lemma (Lemma \ref{glue}), it suffices to prove the claim for $\Sigma_L$.  Furthermore, we already know from Lemma \ref{non-annular-theta} that $\langle L\rangle =\vartheta_{g(L)}$, so it remains to show that $\vartheta_{kg(L)}=T_k(\vartheta_{g(L)})$ in $\Sk_t(\Sigma_L)$ for $T_k$ as in \eqref{Cheb1} and for any given $k\geq 2$.   Let us denote $p\coloneqq g(L)$.
			
			As in Lemma \ref{Lem:Tk}, define integers $\nu_{a,k}\in \bb{Z}_{\geq 0}$ for $a=0,\ldots,k$ by specifying that $\nu_{a,k}$ is the coefficient of $z^{kp}$ in the Laurent expansion of $(z^p+z^{-p})^k$. By Lemma \ref{kg}, $(\vartheta_p)^k=\sum_{0\leq a\leq k} \nu'_{a,k} \vartheta_{ap}$ for some coefficients $\nu'_{a,k}\in \bb{Z}_{\geq 0}[t^{\pm 1}]$. By Lemma \ref{Lem:Tk}, it suffices to show that $\nu'_{a,k}=\nu_{a,k}$ for all $a$.
			
			(i) Let us study $\vartheta_p$.  Fix $\ell\in \bb{Z}_{\geq 1}$ large enough so that $0\notin kp+\ell M^+$. By Lemma \ref{Zp}, for $\sQ$ sufficiently close to $p$, the only broken lines with ends $(p,\sQ)$ and final exponent in $M\setminus (p+\ell M^+)$ will have the final exponent $p$ or $-bp$ for $b\in \Z_{\geq 1}$. Note that the $z^p$-coefficient for $\vartheta_p$ is $1$. Denote the $z^{-bp}$-coefficient of this $\vartheta_{p,\sQ}$ by $c_b$.
			
			We prove that $c_1=1$ and $c_b$ vanishes for $b>1$. Note that $\sQ$ is always sufficiently close to $0$ for the application of Proposition \ref{StructureConstants} (i.e. $\sQ$ and $0$ always share the same chamber).  In particular, it follows that the $\vartheta_0$-coefficient (i.e., the constant coefficient) of $\vartheta_{p}^2$ is equal to $2c_1$.  So we begin by showing that this $\vartheta_0$-coefficient of $[L]^2$ equals $2$.

			We see from \eqref{eq:gammaL2} that the $\vartheta_{g(\gamma)}$-coefficient of $[\gamma] [L]^2$ is equal to $2$ (using Lemma \ref{lem:product_loops_decompose} to ensure that the $\vartheta_{g(\gamma)}$-coefficient of the term $[\alpha_L][L][L_0]$ equals $0$).  Thus, we indeed find that the $\vartheta_0$-coefficient of $[L]^2$ is equal to $2$ (using Lemma \ref{lem:const_coeff} and noticing that $[\gamma]=\vartheta_{g(\gamma)}$ is a cluster variable). So $c_1=1$, as claimed. 
			
			It remains to show $c_b=0$ for $b\geq 2$. If some such $c_b$ were nonzero (hence positive), then Proposition \ref{StructureConstants} implies that the $\vartheta_0$-coefficient $\nu'_{0,b+1}$ of $\vartheta_p^{b+1}$ would be larger than the constant term $\nu_{0,b+1}$ of $(z^p+z^{-p})^{b+1}$. But Lemma \ref{lem:const_gammaLk} and Lemma \ref{lem:product_loops_decompose} imply that the $\vartheta_{g(\gamma)}$-coefficient of $[\gamma][L]^{b+1}$ must be $\nu_{0,b+1}$, so the constant term of $[L]^{b+1}$ must equal $\nu_{0,b+1}$ as well. This contradiction shows that $c_b=0$ for $b\geq 2$.

			(ii) Now let us study $(\vartheta_p)^k$. By (i), for $\sQ$ sufficiently close to the ray through $p$, $\vartheta_{p,\sQ}=z^p+z^{-p}+\ldots$ where the remaining terms have exponents in $M\setminus (p+\ell M^+)$.  Recall that $\ell$ was chosen to be large enough that $0\notin kp+\ell M^+$, so $ap\notin kp+\ell M^+$ for any $a\geq 0$ because $p\in -M^{\oplus}_{\bb{Q}}$ by \eqref{eq:loop_deg_monoid}.  Thus, any term in $\vartheta_{p,\sQ}^k$ involving a contribution from a term other than $z^p$ or $z^{-p}$ must have exponent in $kp+\ell M^+\not\ni ap$, hence cannot contribute to the structure constants $\nu'_{a,k}$.  Thus, applying Proposition \ref{StructureConstants}, $\nu'_{a,k}$ is indeed equal to the coefficient $\nu_{a,k}$ of $z^{ap}$ in the expansion of $(z^p+z^{-p})^k$.
		\end{proof}
		
		\subsection{Bracelets in unpunctured surfaces}\label{Disj-union-section}
		
		Our aim in this subsection is to prove the following:
		
		\begin{thm}\label{thm:bracelet-theta-no-punct}
			The bracelets basis and theta basis coincide for all unpunctured surfaces $\Sigma$ (in both the quantum and classical settings).
		\end{thm}

		Let $\Sigma$ denote a marked surface, possibly with punctures, with an ideal triangulation $\Delta$. Let $C=\bigsqcup_{i=1}^s w_i C_i$ denote a tagged simple multicurve, $w_i>0$. Denote $\beta_i=w_i C_i$ and define the collection $\beta=\{\beta_1,\ldots,\beta_s\}$. 
		
		In the following, unless otherwise specified, we will assume that $\Sigma$ is unpunctured, in which case we will simply take $\sd=\sd_\Delta$.
		
		We know from Proposition \ref{SkAProp} that weighted arcs are cluster monomials, which are theta functions.  Also, weighted non-peripheral loops are theta functions by Lemmas \ref{ArcLoop1Lem} and \ref{ChebyLem}. Moreover, the quantum bracelet $\BracE{C}$ is a bar-invariant element in $\Sk_t(\Sigma)$. Therefore, Theorem \ref{thm:bracelet-theta-no-punct} will follow once we prove Lemma \ref{LemDisjoint}:

		\begin{lem}\label{LemDisjoint}
			For unpunctured $\Sigma$, we have \begin{align}\label{ProdDisjointThetas}
				t^{\alpha}\vartheta_{g(\beta_1)}\cdots\vartheta_{g(\beta_s)}=\vartheta_{g(\beta_1)+\ldots+g(\beta_s)}.
			\end{align}
			where the factor $t^{\alpha}$ is chosen so that the product is pointed, i.e., $\alpha=-\sum_{1\leq i<j\leq s} \Lambda(g(\beta_i),g(\beta_j))$
		\end{lem}

		To prove this, we will need some more lemmas.  We begin by showing that it suffices to prove a simplified version of Lemma \ref{LemDisjoint}. The collection $\beta={\beta_1,\ldots,\beta_s}$ is said to be \textbf{maximal} if one cannot extend $C$ to a weighted simple multicurve $C\sqcup w_{s+1}C_{s+1}$ with one more component $C_{s+1}$ (the existence of maximal $\beta$ follows, e.g., from Lemma \ref{lem:lin-ind} below).

		\begin{lem}\label{lem:max-and-boundary}
			To prove Lemma \ref{LemDisjoint}, it suffices to prove the cases where $t=1$ and the collection $\beta$ is maximal.  Furthermore, one may assume that this maximal $\beta$ includes no interior arcs.
		\end{lem}
		\begin{proof} To reduce to the $t=1$ setting, note from Corollary \ref{cor:alpha-sum} that the coefficient of $\vartheta_{g(\beta_1)+\ldots+g(\beta_s)}$ in the product $\vartheta_{g(\beta_1)}\cdots \vartheta_{g(\beta_s)}$ is indeed $t^{-\alpha}$ for $\alpha$ as in Lemma \ref{LemDisjoint}.  If any other theta functions contributed to this product, strong positivity ensures that they would have coefficients in $\bb{Z}_{\geq 0}[t^{\pm 1}]$, so these theta functions would still contribute in the $t=1$ setting.  So proving Lemma \ref{LemDisjoint} when $t=1$ would rule this out.
		
		So now fix $t=1$.  To show that we can assume $\beta$ is maximal, it suffices to show that if $f_1,f_2$ are theta positive and $f_1f_2$ is a theta function, then each of the factors $f_1$ and $f_2$ must be theta functions. For $i=1,2$, let $f_i=\sum_{j=1}^{s_i} a_{ij}\vartheta_{m_{ij}}$ be a positive linear combination of theta functions, so $s_i \geq 1$ and each $a_{ij}\in \bb{Z}_{\geq 1}$.  By strong positivity, if $\sum_{j=1}^{s_i} a_{ij} > 1$ for either $i=1,2$, then the sum of the coefficients in the theta function expansion of the product $f_1f_2$ will also be larger than $1$. Therefore, $f_1f_2$ being theta function indeed implies that both $f_1$ and $f_2$ are theta functions.  In particular, a product of bracelets being a theta function implies that any product of a subset of the bracelets is a theta function.

			The second claim follows from Lemma \ref{glue}, the gluing lemma, because one can always cut $\Sigma$ along any interior arcs in $\beta$ to obtain a surface in which all arcs of $\beta$ are boundary arcs.
		\end{proof}

		Let $e(\beta_i):=w_i e(C_i)$ denote the laminate associated to $\beta_i=w_i C_i$, where $e(C_i)$ denotes the elementary laminate as in Definition \ref{def:el-lam}.
		\begin{lem}\label{lem:lin-ind}
			Consider $\Sigma$ possibly with punctures (in which case we work with a quantum seed $\sd$ similar to $\sd_{\Delta}$) and a collection $\beta=\{\beta_1,\ldots,\beta_s\}$ consisting of compatible weighted loops and arcs. Let $\beta'$ denote the subset consisting only of weighted loops. Then for any ideal triangulation $\Delta$, the following statements are true:
			\begin{enumerate}[label=(\roman*), noitemsep]
				\item The shear coordinate vectors $b^\Delta(e(\beta_i))$, $\beta_i\in \beta'$, are linearly independent.
				
				\item If the only arcs in $\beta$ are boundary arcs, then the corresponding $g$-vectors $g(\beta_1),\ldots,g(\beta_s)$ are linearly independent.
			\end{enumerate}
		\end{lem}
		\begin{proof}
			It suffices to prove the cases where each $\beta_i$ has weight $1$ since changing the weight just scales the shear coordinates and $g$-vector.  
			
			(i)
			If this were not the case, it would mean that there is an equality
			\begin{align}\label{g-vect-eq}
				\sum_{j\in J_1} a_j  b^\Delta(e(\beta_j)) = \sum_{j\in J_2} a_j  b^\Delta(e(\beta_j))
			\end{align}
			for some non-empty disjoint subsets $J_1,J_2 \subset \{1,\ldots,s\}$ with $\beta_j\in \beta'$ whenever $j\in J_1\sqcup J_2$, and some positive integers $a_j\in \bb{Z}_{\geq 1}$, $j\in J_1\sqcup J_2$.  Correspondingly, we have two unbounded integral laminations $L_1$, $L_2$, given by $L_i=\bigsqcup_{j\in J_i} a_je(\beta_j)$ for $j=1,2$, respectively.. Then $L_1$ and $L_2$ are non-homotopic, but they have the same shear coordinates.  But, by Proposition \ref{prop:shear_coord} (\cite[Theorems 12.3, 13.6]{fomin2018cluster}), non-homotopic $L_1$ and $L_2$ must have distinct shear coordinates. This is a contradiction.
			
			(ii) 
			Notice that the quantum seed $\sd$ satisfies the Injectivity Assumption.  The $g$-vectors of frozen variables are distinct elements of $\{e_i^*|i\in F\}$, which are in particular linearly independent.  The claim follows now from (i) and the equality $\pr_{I_{\uf}}(\beta_i)=-b^\Delta(e(\beta_i))$ (see \eqref{eq:g_vector_shear}).
		\end{proof}

		Recall that we can also associate an intersection coordinate vector $\pi(L)\in N_{\bb{Q}}$ to an integer bounded lamination $L\in \cA_{L}(\Sigma,\Q)$ by \eqref{pil}.
		\begin{lem}\label{lem:independ_intersection_coord}
			For $\Sigma$ a possibly punctured surface and $\beta$ a collection of compatible weighted loops and arcs, and for any ideal triangulation $\Delta$, the corresponding intersection coordinate vectors $\pi(e(\beta_1)),\ldots$ $,\pi(e(\beta_s))$ are linearly independent.
		\end{lem}
		\begin{proof}
			The proof is similar to that of Lemma \ref{lem:lin-ind} except that we work with $N_\Q$ instead of $\Z^{I_\ufv}$, and we use Proposition \ref{prop:intersection_coordinate} instead of Proposition \ref{prop:shear_coord} for seeing that distinct laminations have different intersection coordinate vectors.
		\end{proof}

		Given $\ell\in \bb{Z}_{\geq 1}$, let $C_j^\circ$ denote the connected components of $M_\R\setminus \f{D}_{\ell}$ (indexed by $j$) and $C_j$ the corresponding closures, called \textbf{chambers}. Then the chambers $C_j$ are convex cones, and these chambers together with their faces form a complete fan in $M_\R$, see \cite[Thm. 3.1]{reading2020scattering} (alternatively, for $\Sigma$ not a once-punctured surface, $\f{D}_{\ell}$ is equivalent to Bridgeland's stability scattering diagram \cite{bridgeland2017scattering} which gives a fan structure, see \cite{qin2019bases}).

		The following Lemma is crucial for our proof.
		\begin{lem}\label{lem-sigma-beta}
			Consider $\Sigma$ possibly with punctures, and assume that the collection $\beta$  consists only of weighted loops and boundary arcs (we do not require here that $\beta$ is maximal).
			Let $\beta'\subset \beta$ be the subset consisting of weighted loops.  Let $\sigma_{\beta}$ denote the cone in $M_{\bb{R}}$ spanned by $$\{g(\beta_1),\ldots,g(\beta_s)\}\cup\{-g(\beta_i)|\beta_i\in \beta \setminus \beta'\}.$$  
			Then for any $\ell\in \bb{Z}_{\geq 1}$, there exists a chamber of $\f{D}_{\ell}$ which contains both $\sigma_{\beta}$ and the cluster complex chamber $C_{\Delta'}$ associated to some ideal triangulation $\Delta'$ without self-folded triangles.  Furthermore, if $\Delta_0=\{\gamma_1,\ldots,\gamma_k\}$ is a set of interior ideal arcs, none of which are nooses, such that the elements of $\beta\cup \Delta_0$ are pairwise compatible, then $\Delta'$ may be chosen to contain $\Delta_0$.
		\end{lem}
		\begin{proof} 
			Let $\sigma_{\beta'}\subset \sigma_{\beta}$ be the subcone spanned by $\{g(\beta_i)|\beta_i\in \beta'\}$.  Note that $\sigma_{\beta'}\subset M_{\uf,\bb{R}}:=\omega_1 (N_{\uf}) \otimes \R$ by \eqref{eq:loop_deg_monoid}. It suffices to show that $\sigma_{\beta'}$ is contained in a chamber of $\f{D}_{\ell}$, because one then extends the result from $\sigma_{\beta'}$ to $\sigma_{\beta}$ by noting that, in general, every chamber of $\f{D}_{\ell}$ includes $\pm g(A_i)$ for every frozen variable $A_i$, hence for every boundary arc.
			
			(i) Here we consider the fan structure induced by $\f{D}_{\ell}$ on the relative interior $\sigma_{\beta'}^{\circ}$ of $\sigma_{\beta'}$.
			
			For any collection of chambers $C_j$ of $\f{D}_{\ell}$, $j\in J$, define the the intersection $C_J=\bigcap_{j\in J} C_j$. Further define the restriction $\uC_J=C_J \cap \sigma_{\beta'}^\circ$. Then $\uC_J$ is still a convex cone. Note that if $\uC_J=\uC_{J'}$, then $\uC_J=\uC_{J'}=\uC_{J\cup J'}$. Correspondingly, we say $J$ is maximal for a cone $\uC_J$ if whenever $\uC_{J'}=\uC_J$, we have $J'\subset J$. Similarly, note that $J'\subset J$ implies $\uC_J\subset \uC_{J'}$.  Since cones of the form $C_J$ form a complete fan for $M_\R$, their restrictions $\uC_J$ form a complete fan for $\sigma_{\beta'}^\circ$ as well.
			
			(ii) Suppose that $\sigma_{\beta'}$ was not contained in any $C_J$. Let us find two points $m_\pm(\lambda)$ in $\sigma_{\beta'}$ which are close to each other but do not belong to a common chamber.
			
			Note that $\sigma_{\beta'}$ and the cones $C_J$ are closed in $M_\R$. By our assumption, the relative interior $\sigma_{\beta'}^\circ$ is not contained in any $\uC_{J}$. Take any non-empty maximal cone $\uC_{J}$ in $\sigma_{\beta'}^\circ$ with maximal $J$, and choose a point $m$ with rational coordinates which is in $\sigma_{\beta'}^{\circ}$ and on the relative boundary of $\uC_{J}$ (the assumption $\uC_J\neq \sigma_{\beta'}^{\circ}$ ensures that such points exist). Take a point $\delta$ with rational coordinates in the relative interior $\uC_{J}^\circ$. Then for any $\lambda>0$ sufficiently small, the point $m_+(\lambda):=m+\lambda \delta $ lies in the interior of $\uC_J$, and the point $m_-(\lambda)=m-\lambda \delta$ lies outside $\uC_J$ but still in $\sigma_{\beta'}^{\circ}$.
			
			We claim that $m_-(\lambda)\notin C_j$ for any $j\in J$, i.e. $m_-(\lambda)$ and $m_+(\lambda)$ are not contained in a common chamber $C_j$. For $p\in \sigma_{\beta'}^\circ$, define $J(p)=\{j|p\in \uC_j\}$. Then $J(m_+(\lambda))=J$ because $m_+(\lambda)$ is an interior point of the top dimensional cone $\uC_J$ and $J$ is maximal. Define $J'=J(m_-(\lambda))$. If our claim was false, then $J\cap J'\neq\emptyset$. Moreover, $\uC_J$ is \textit{properly} contained in $\uC_{J\cap J'}$ since the latter contains $m_-(\lambda)$. This contradicts our assumption that $\uC_J$ is maximal, so our claim holds.
			
			(iii) Take the points $m$, $m_\pm(\lambda)$, the chamber $\uC_J$, and the vector $\delta=\sum \delta_i g(\beta_i)$, $\delta_i\in\Q$ as above. We will show that some points near $m_+(\lambda)$ must be contained in the same chamber as some points near $m_-(\lambda)$ and thus obtain a contradiction.
			
			Let $\Delta$ denote an ideal triangulation without self-folded triangles such that $\Delta\supset \Delta_0$, and let $\gamma_1,\ldots,\gamma_r$ be its internal arcs. Define the matrix $(c_{ij})=(e(\beta_i)\cdot \gamma_j)_{1\leq i\leq s,1\leq j\leq r}$, where we write $e(\beta_i)\cdot \gamma_j$ to denote the crossing number $\cc(e(\beta_i),\gamma_j)$. The row vectors of $(c_{ij})$ are linearly independent by Lemma \ref{lem:independ_intersection_coord}. Consequently, the column rank of $(c_{ij})$ is $s$.
			
			Let $S$ be the set of $i=1,\ldots,s$ such that $\beta_i\in \beta'$.  Let us denote $$m=\sum_{i\in S} b_i g(\beta_i)=\sum_{i\in S} \left(\frac{b_i}{\sum_j c_{ij}}\right)\left(\sum_j c_{ij}\right)g(\beta_i),$$ so each $b_i\in \Q_{>0}$. Further define $w_i\in \NN_{>0}$ for each $i\in S$ via $w_i= \nu\left(\frac{b_i}{\sum_j c_{ij}}\right)$ for all $i\in S$ and for some fixed $\nu\in \NN_{>0}$. Then we can rewrite
			\[
			m=\frac{1}{\nu}\sum_{i\in S} w_i\left(\sum_j c_{ij}\right)g(\beta_i).
			\]
			
			Since the column rank of $(c_{ij})$ equals $s$, its column vectors span $\Q^s$ over $\Q$. Recall the numbers $\delta_i\in \Q$ from Step (ii). For the vector $(\frac{\delta_i}{w_i})_{1\leq i\leq s}\in \Q^s$, we can find rational numbers $\frac{d_j}{d}\in \Q$, where $1\leq j\leq r$, $d_j\in \bb{Z}$, $d\in \NN_{>0}$, such that $\frac{\delta_i}{w_i}=\sum_j c_{ij}\frac{d_j}{d}$ for any $i$. 
			
			Then we have, for any $D\in \NN_{>0}$:
			\begin{align*}
				m_\pm(\lambda)&=m\pm\lambda \sum_{i\in S} \delta_i g(\beta_i)\\
				&=\frac{1}{\nu }\sum_{i\in S} w_i\left(\sum_j c_{ij}\right)g(\beta_i)\pm\frac{\lambda}{d} \sum_{i\in S} \left(w_i\sum_j c_{ij}d_j\right) g(\beta_i)\\
				&=\frac{1}{\nu D}\sum_{i\in S} w_i\left(\sum_j c_{ij}D\right)g(\beta_i)\pm\frac{\lambda}{d} \sum_{i\in S} \left(w_i\sum_j c_{ij}d_j\right) g(\beta_i).
			\end{align*}
			Now set $\lambda=\frac{d}{\nu D}$ with a sufficiently large $D$ such that $\lambda$ is sufficiently small and $D>|d_j|$ for all $j$. We obtain
			\begin{align*}
				m_\pm(\lambda)=\frac{1}{\nu D}\sum_{i\in S} w_i\left(\sum_j c_{ij}(D\pm d_j)\right)g(\beta_i)
			\end{align*}
			
			Define weighted simple multicurves $\gamma_\pm=\bigcup_j (D\pm d_j)\gamma_j$. Then $g(\gamma_\pm)$ are interior points in the chamber  $C_\Delta$  associated to $\Delta$ in the scattering diagram $\f{D}$.
			
			Consider the composition of Dehn twists $\tau=\prod_{\beta_i\in \beta'} \tau^{w_i}_{\beta_i}$. For any given $K\in \NN_{>0}$, define $\Delta(K):=\tau^K (\Delta)$ and $\gamma_\pm(K):=\tau^K (\gamma_\pm)$. Then $g(\gamma_\pm(K))$ belong to the chamber $C_{\Delta(K)}$ in $\f{D}$ associated to $\Delta(K)$. In addition, for each $\beta_i\in \beta$, $e(\beta_i)$ is $\tau$-invariant, so $e(\beta_i) \cdot \gamma_\pm (K)=e(\beta_i) \cdot \gamma_\pm =\sum_j c_{ij}(D\pm d_j)$.
			
			By Theorem \ref{Thm:Yurikusa_Dehn}, \eqref{eq:shear_g}, and Remark \ref{rmk:tw}, there exists some $K_0\in \NN_{>0}$, such that, modulo frozen coordinates $\Z^F$,
			\begin{align*}
				g(\gamma_\pm(K+K_0))&\equiv g(\gamma_\pm(K_0))+\sum_{i\in S} K w_i(e(\beta_i)\cdot \gamma_\pm(K_0)) g(\beta_i)\\
				&=g(\gamma_\pm(K_0))+\sum_{i\in S} K w_i\left(\sum_j c_{ij}(D\pm d_j)\right) g(\beta_i)\\
				&=g(\gamma_\pm(K_0))+K \nu D m_\pm(\lambda).
			\end{align*}
			Consequently, we get 
			\begin{align*}
				m_\pm(\lambda)+\frac{1}{K \nu D}g(\gamma_\pm(K_0))\in \frac{1}{K\nu D} g(\gamma_\pm(K+K_0))+\Z^F \subset C_{\Delta(K+K_0)}.
			\end{align*}
			Note that $C_{\Delta(K+K_0)}$ must be contained in a chamber $C'$ of $\f{D}_{\ell}$.
			
			Recall that $m_+(\lambda)\in \uC_J^\circ$ with maximal $J$, and for all $j\in J$, $m_-(\lambda)\notin \uC_j$.  By choosing $K$ sufficiently large, we still have $m_+(\lambda)+\frac{1}{K \nu D}g(\gamma_+(K_0))\in \uC_J^\circ$ and, for all $j\in J$, $m_-(\lambda)+\frac{1}{K \nu D}g(\gamma_-(K_0))\notin \uC_j$ --- hence, $m_\pm(\lambda)+\frac{1}{K \nu D}g(\gamma_\pm(K_0))$ are not contained in a common chamber in $\f{D}_{\ell}$. But they are equal to $\frac{1}{K \nu D}g(\gamma_\pm(K+K_0))$ respectively, which both belong to $C_{\Delta(K+K_0)}\subset C'$. This contradiction implies that the assumption in (ii) is false, i.e., $\sigma_{\beta'}$ must be contained in a chamber of $\f{D}_{\ell}$.
			
			  (iv)  Finally, take a top dimension cone $\uC_J=C_J \cap \sigma_{\beta'}^\circ$ in $\sigma_{\beta'}^\circ$ as in step (i). We have shown that $\uC_J=\sigma_{\beta'}^\circ$. Our arguments in step (iii) show that for any interior point $m_+(\lambda)$ in $\uC_J$, we can construct a chamber $C_{\Delta(K+K_0)}$ of the cluster complex containing a point of $\sigma_{\beta'}^{\circ}$ near to $m_+(\lambda)$. Recall that $C_{\Delta(K+K_0)}$ is contained in a chamber $C'$ of $\f{D}_{\ell}$. So $C'$ contains an interior point of $\uC_J$, and we deduce $C_J\subset C'$.  It follows that $\sigma_{\beta'}\subset C_J\subset C'$ and $C_{\Delta'}\subset C'$ for $\Delta'=\Delta(K+K_0)$.  Since the Dehn twists act trivially upon $\Delta_0$, we have $\Delta_0\subset \Delta'$, as desired.
		\end{proof}

		\begin{lem}\label{lem:int-closed}
			Under the assumptions of Lemma \ref{lem-sigma-beta}, plus the assumption that $\beta$ is maximal, we have that for any $m_1,m_2\in \sigma_{\beta}\cap M$, the decomposition $$\vartheta_{m_1}\vartheta_{m_2}=\sum_{m\in M} a_m \vartheta_m$$ satisfies $a_m=0$ whenever $m\notin \sigma_{\beta}\cap M$. 
		\end{lem}
		\begin{proof}	
			Since the elements of $\beta$ are non-intersecting, Dehn twists by any loop in $\beta$ will act trivially on each $\beta_i\in \beta$, so they will act with finite orbit on each $g\in \sigma_{\beta}$ by Lemma \ref{cor:tau-m}.  In particular, each such Dehn twist acts with finite orbit on $m_1$ and $m_2$, hence on $\vartheta_{m_1}\vartheta_{m_2}$.
			
			On the other hand, for $m\in M\setminus (\sigma_{\beta}\cap M)$, the maximality of $\beta$ implies that the bracelet $\beta_m$ with $g$-vector $m$ has positive crossing number with some $\beta_i\in \beta$, and this $\beta_i$ is a loop by the assumption that $\beta$ contains no interior arcs.  So by Lemma \ref{lem:infinite-orbit}, the action of $\tau_{\beta_i}$ on such $m$ has infinite orbit.  Thus, $\vartheta_m$ with $m\notin \sigma_{\beta}\cap M$ cannot appear in the theta function expansion of $\vartheta_{m_1}\vartheta_{m_2}$.
		\end{proof}

		\begin{lem}\label{loop-broken-pm}
			Let $\Sigma$ and $\beta$ be given as in Lemma \ref{lem-sigma-beta} with $\Sigma$ unpunctured. For $\ell' \gg \ell \gg 1$, fix generic $\sQ$ in a chamber of $\f{D}_{\ell'}$ which contains the cone $\sigma_{\beta}$ as in Lemma \ref{lem-sigma-beta}.  Then for each $\beta_i\in \beta'$, there are precisely two broken lines with ends $(g(\beta_i),\sQ)$ and final exponents in $M\setminus (p+\ell M^+)$, and their final attached monomials are $z^{g(\beta_i)}$ and $z^{-g(\beta_i)}$.  If $\beta_i$ is a weighted boundary arc, then the only broken line with ends $(g(\beta_i),\sQ)$ is the straight broken line with final attached monomial $z^{g(\beta_i)}$.
		\end{lem}
		\begin{proof}
			For $\beta_i$ a weighted loop, the fact that $\vartheta_{kg(\beta_i)}=T_k(\vartheta_{g(\beta_i)})$ for each $k\in \Z_{\geq 1}$ (by Lemmas \ref{ArcLoop1Lem} and \ref{ChebyLem}) implies that $\vartheta_{g(\beta_i)}^k=\sum_{b=0}^k a_{k,b}\vartheta_{bg(\beta_i)}$ for some $a_{k,b}\in \bb{Z}_{\geq 0}$ (cf. \cite[Lem. 3.18]{Allegretti2015duality}). So by Lemma \ref{Zp}, $\vartheta_{g(\beta_i),\sQ}$ has the form $z^{g(\beta_i)}+\sum_{b\in \bb{Z}_{\geq 1}} c_b z^{-bg(\beta_i)}+\ldots$ with each $c_b$ in $\bb{Z}_{\geq 0}$ and with any remaining terms having exponents in $p+\ell M^+$ coefficients in $\bb{Z}_{\geq 0}$.  By the Chebyshev relations, the constant term in $\vartheta_{g(\beta_i)}^k$ equals $0$ if $k$ is odd and equals $\binom{k}{k/2}$ if $k$ is even.  It follows that $c_1=1$ and $c_k=0$ for $k\geq 2$, as desired.

			Finally, the claim for boundary arcs follows from the fact that boundary arcs correspond to frozen cluster variables and that all walls are parallel to $M_F:=\bb{Z}\langle e_i^*|i\in F\rangle$.
		\end{proof}
		
		\begin{myproof}[Proof of Lemma \ref{LemDisjoint}]
				We know from Corollary \ref{cor:alpha-sum}  that the theta function expansion of the product on the left-hand side of \eqref{ProdDisjointThetas} includes the term on the right-hand side, so our goal is just to show that there are no contributions from any other theta functions.
			
			We apply Lemma \ref{lem:max-and-boundary} to assume that $\beta$ is maximal and includes no interior arcs, and that $t=1$.  By Lemma \ref{lem:int-closed}, $$\vartheta_{g(\beta_1)} \cdots \vartheta_{g(\beta_s)} = \sum_{m\in \sigma_{\beta}\cap M} a_m \vartheta_m$$
			with each $a_m\in \bb{Z}_{\geq 0}$.  Take any $m$ such that $a_m\neq 0$. 
			Choose $\ell,\ell'$ as in Lemma \ref{loop-broken-pm} with $\ell$ large enough so that $m-g(\beta_1)-\ldots-g(\beta_s)\notin \ell M^+$.  Then choose a generic $\sQ$ in a chamber $\sigma$ of $\f{D}_{\ell'}$ which contains $\sigma_{\beta}$ as in Lemma \ref{lem-sigma-beta}.  So by Lemma \ref{loop-broken-pm}, if $\gamma_1,\ldots,\gamma_s$ is a tuple of broken lines with ends $(g(\beta_i),\sQ)$, $i=1,\ldots,s$ respectively, contributing to $a_m$ as in Proposition \ref{StructureConstants}, then the final exponent of $\gamma_i$ is $g(\beta_i)$ if $\beta_i$ is a weighted boundary arc, and it is $\epsilon_i g(\beta_i)$ for $\epsilon_i=\pm 1$ if $\beta_i$ is a weighted loop.  By the linear independence of the $g$-vectors from Lemma \ref{lem:lin-ind}, the sum of these final exponents can only lie in $\sigma_{\beta}$ if each $\epsilon_i$ is non-negative.  This means that $$m=g(\beta_1)+\ldots+g(\beta_s),$$
			as desired.
		\end{myproof}

		\section{Results for general cluster algebras of surface type}\label{sec:general_surface}
  		\subsection{Bracelets are theta functions: punctured cases}\label{sec:theta_punctured}

		Let $\Sigma=(\SSS,\MM)$ denote a triangulable surface, possibly with punctures. By Lemmas \ref{lem:similar_theta_function} and \ref{lem:similar_theta_function_2}, in order to show that weighted tagged bracelet elements are theta functions for (quantum) cluster algebras, it suffices to work with any (quantum) cluster algebra of type $\Sigma$ with $M_{\sd}^{\oplus}$ strongly convex. In view of Lemma \ref{union-product}, we can work with connected $\Sigma$ without loss of generality. 
		
		Let there be given any weighted simple multicurve $L=\bigcup w_i L_i$, such that $L_i$ are loops. By Lemma \ref{lem:triangulation}, we can choose internal plain arcs $\gamma_1,\ldots,\gamma_s$ which do not intersect $L_i$, such that the surface $\Sigma'$ obtained by cutting $\Sigma$ along $\gamma_i$ is unpunctured, and, moreover, $\{\gamma_1,\ldots,\gamma_s\}$ can be extended to an ideal triangulation $\Delta$ without self-folded triangles. Let $\Delta'$ denote the triangulation of $\Sigma'$ corresponding to $\Delta$. We associate to it the principal coefficient seed $\sd'=\sd_{\Delta'}^{\prin}$. 
		
		Following \S \ref{sec:gluing_frozen_vertices}, for $1\leq i\leq s$, gluing the frozen vertices $\gamma_i^{(1)},\gamma_i^{(2)}$ corresponding to $\gamma_i$  (and also gluing the corresponding principal coefficient indices to $\gamma_i'$ and changing $\omega(\gamma_i,\gamma_i')$ to $1$), we obtain a seed $\?{\sd'}$ from $\sd'$. Following \S \ref{sec:unfreezing}, let $\sd$ denote the seed obtained from $\?{\sd'}$ by unfreezing $\gamma_i$. Then $\sd$ is of full-rank and is similar to $\sd_\Delta$. In fact, $\sd=\sd_{\Delta}^{\prin}$.

		We make $\sd'$ into a quantum seed by choosing the canonical bilinear form $\Lambda_{\sd'}$ as in \eqref{LambdaPrin}. We similarly choose a compatible $\Lambda$ for $\sd$ as in \eqref{LambdaPrin} (viewing $\sd$ as $\sd_{\Delta}^{\prin}$), and this induces a compatible pair for $\?{\sd'}$. As before, we have a $\kk$-linear map $\pi_{\sd'}:\kk_t[\sd']\rightarrow \kk_t[\?{\sd'}]$ for the gluing and a $\kk_t$-linear map $\f{i}:\kk_t[\?{\sd'}]\rightarrow \kk_t[\sd]$ for the unfreezing.
		
		Recall that $L$ represents a bracelet element $\E{L}$ in the (quantum) upper cluster algebra $\s{A}^{\up}_t(\sd)$, see \S \ref{sec:qbracelet_loop} and \eqref{eq:LBrac}.
		
		\begin{lem}\label{lem:loop_puncture_bracelet}
			The bracelet element $\E{L}$ is a theta function of $\s{A}^{\up}_t(\sd)$.
		\end{lem}
		\begin{proof}
			Notice that $L$ can be naturally identified with a weighted simple multicurve $L'$ in $\Sigma'$ consisting of non-intersecting weighted loops and corresponding to a bracelet element $\langle L'\rangle$ in $\s{A}_t^{\up}(\sd')$. By construction, the corresponding bracelet elements satisfy $\E{L}=\f{i}\circ \pi_{\sd'}\E{L'}$, cf. \eqref{eq:LBrac}. 
			
			Since $\Sigma'$ is an unpunctured surface, Theorem \ref{thm:bracelet-theta-no-punct} and Lemma \ref{lem:similar_theta_function_2} imply that $\E{L'}$ is a theta function for $\s{A}^{\up}_t(\sd')$. Then $\pi_{\sd'}\E{L'}$ is a theta function for $\?{\sd'}$, see \eqref{pi-theta}.  So by Lemma \ref{lem:positive_unfreeze_theta},  in order to show that $\E{L}=\f{i}(\pi_{\sd'}\E{L'})$ is a theta function, it suffices to show that $\langle L\rangle$ is theta positive when $t=1$.  That is, we wish to show in the classical setting that in the expansion
			\begin{align}\label{eq:EL}
				\E{L}=\sum _{g\in g(L)+M^\oplus} c_g \vartheta_g,
			\end{align}
            the coefficients $c_g$ always lie in $\bb{Z}_{\geq 0}$.  Note that it suffices to check this positivity in the case where $L$ consists of a single weighted component, so we assume this is the case.
            
            Suppose $\Sigma$ is a once-punctured closed surface.  Let $H$ denote the hyperplane in $M_{\R}$ with the normal vector $\sum_{i\in I_{\ufv}} e_i$. Observe that $M_\ufv:=\omega_1 (N_\ufv)$ is contained in $H$. In addition, we have $g(L)\in M_\ufv$ by \eqref{eq:loop_deg_monoid}. We deduce that all $g$ appearing in \eqref{eq:EL} are contained in $H$.  Thus, for any such $g$, we can take a base point $\sQ$ which is arbitrarily close to $g$ and belongs to a chamber of the (plain arc) cluster complex (by Proposition \ref{gdense}).
            
            We wish to show that the same is true when $\Sigma$ is not a once-punctured closed surface.  Suppose $c_g\neq 0$.  Then $\?{g}\coloneqq \pr_{I_1}(g)$ is the extended $g$-vector of a tagged bracelet element $\beta_{\?{g}}$.  We claim that no component of $\beta_{\?{g}}$ can be an interior tagged arc with an end at a puncture; in particular, no component can be an arc with a notch.  It then follows from Lemma \ref{lem-sigma-beta} that we can find $\sQ$ arbitrarily close to $g$ and contained in a chamber $C_{\Delta'}$ associated to an ideal triangulation $\Delta'$ without self-folded triangles.
            
            To show our claim, suppose some component of $\beta_{\?{g}}$ is an interior tagged arc $\gamma$ with at least one end at a puncture $P$.  If $\SSS\setminus \gamma$ is connected, then there exists a loop $L_{\gamma}$ disjoint from $L$ which intersects $\gamma$ exactly once.  If $\SSS\setminus \gamma$ is not connected, then both ends of $\gamma$ must lie at the same puncture $P$.  Let $\SSS_1,\SSS_2$ be the components of $\SSS\setminus \gamma$ which contain $L$ and do not contain $L$, respectively.  Since $\gamma$ does not cut out an unpunctured or once-punctured monogon, we can find another arc $\alpha\in \?{\SSS_2}$ with both ends at $P$ such that $\gamma$ and $\alpha$ are not homotopic.  By gluing the ends of $\alpha$ together and moving this glued point slightly into $\SSS_1$, we again construct a loop $L_{\gamma}$ in $\Sigma$ having essential intersection with $\gamma$ and no intersections with $L$.  In either case, we have that the orbit of $\gamma$, hence of $\beta_{\?{g}}$, hence of $g$ (by Lemma \ref{lem:infinite-orbit}) under the action of the Dehn twist $\tau_{L_{\gamma}}$ is infinite.  Since $\tau_{L_{\gamma}}$ acts trivially on $\E{L}$, this situation cannot happen (the expansion \eqref{eq:EL} in these cases is finite by Theorem \ref{fFG-surfaces}).
            
            We have thus shown that for any $g$ with $c_g\neq 0$, we can find $\sQ$ arbitrarily close to $g$ and contained in $C_{\Delta'}$ for some ideal triangulation $\Delta'$ without self-folded triangles.  So by Lemma \ref{lem:fQ}, to prove the positivity of $c_g$, it suffices to show that $\E{L}$ is positive in every cluster associated to such a triangulation.  This follows immediately from \eqref{Lprin-def} and the description of $\bb{I}(L)$ given in \eqref{eq:ell-pi-pointed} (in this setting, $X^n$ in \eqref{eq:ell-pi-pointed} is identified with $z^n$).
		\end{proof}

It now follows from Lemma \ref{lem:similar_theta_function} or \ref{lem:similar_theta_function_2} that $\langle L\rangle_{\Brac}^{\sd}$ is a theta function for arbitrary $\sd$ similar to $\sd_{\Delta}$, not just for $\sd_{\Delta}^{\prin}$.

	\begin{lem}\label{lem:commute_arc_loop}
		Let $\gamma$ denote a tagged arc and $L$ a simple loop not intersecting $\gamma$.  Then $\BracE{\gamma}$ and $\BracE{L}$ commute in $\s{A}_t^{\up}(\sd)$.  Here, if $\Sigma$ is a once-punctured torus, then we require that $\gamma$ is not a doubly-notched arc so that $\BracE{\gamma}$ is defined.
	\end{lem}

\begin{proof}

(i) We first assume that $\gamma$ has a plain tag at one of its ending.
	
	Recall from \S \ref{S:BracTheta} that one of the following two cases applies:
	\begin{enumerate}
		\item (annular type, Figure \ref{AnnCut}) $L$ is contained in an annulus $\Sigma_L$ with one marked point on each boundary component.
		
		\item (non-annular type, Figure \ref{DLoop}) $L$ is contained in a marked surface $\Sigma_L$ with one component and one boundary marked point, such that $L$ is homotopic to the boundary component. 
	\end{enumerate}
	Since $\gamma$ has a plain ending, in both cases, we can further choose $\Sigma_L$ such that it does not contain $\gamma$ and its boundary arcs are compatible with $\gamma$ (in the sense of \S \ref{S:tagged_sk}).

	Choose an ideal triangulation $\Delta'$ of $\Sigma_L$. Extend $\Delta'\cup\{\gamma\}$ to a tagged triangulation $\Delta$ of $\Sigma$. Let $(\sd,\Lambda)$ be a quantum seed with $\sd$ similar to $\sd_{\Delta}$.  Since $\Delta'$ is an ideal triangulation, we can understand $\bb{I}(L)$ using \eqref{eq:classical_trace} and \eqref{eq:ell-pi-pointed}.
 
     Now let $\gamma_k\in \Delta$ denote any vertex belonging to the support of the $g(L)$-pointed element $\BracE{L}$ in the quantum torus algebra $\s{A}_t^{\sd}$. Recall that $g(L)=-\omega_1(\pi(L))$ where $\pi(L)$ denotes the intersection coordinates \eqref{eq:intersection_shear_loop}. Since $L$ does not intersect $\gamma$, the $e_{\gamma}$ component of $\pi(L)$ is $0$, hence by \eqref{Lambda-B},
    \begin{align*}
    \Lambda(e^*_{\gamma},g(L))=\langle \pi(L),e^*_{\gamma}\rangle = 0.
    \end{align*}

    So now it suffices to show that $\Lambda(e^*_{\gamma},\omega_1(e_{\gamma_k}))=0$, or using \eqref{Lambda-B} again, that $\langle e_{\gamma_k},e^*_{\gamma}\rangle=0$.  By construction \eqref{eq:q_loop_bracelet} and using the description of $\bb{I}(L)$ coming from \eqref{eq:classical_trace} and \eqref{eq:ell-pi-pointed}), $\BracE{L}$ in the $t=1$ setting is supported on the vertex $\gamma_k\in \Delta$ if and only if $L$ intersects $\gamma_k$.  This carries over to the quantum setting by Lemma \ref{lem:I1} and \cite[Thm. 1.1]{cho2020laurent} (quantum universal positivity of $\bb{I}_t(L)$).  So $L$ not intersecting $\gamma$ implies $\langle e_{\gamma_k},e^*_{\gamma}\rangle=0$, as desired.

(ii) It remains to treat a doubly-notched arc $\gamma$ (so we now assume that $\Sigma$ is not a once-punctured torus).  Let $\gamma^{\diamond}$ denote the corresponding plain arc.  Let $\Delta$ be an ideal triangulation containing $\gamma^{\diamond}$, and work with an initial seed $\sd$ similar to $\sd_{\Delta}$.  Suppose that $\Sigma$ is not a once-punctured closed surface.  Applying the automorphism $\DT$ of $\mr{\s{A}}_t^{\sd}$ as in \S \ref{sec:DT}, it suffices to show that $\DT\BracE{\gamma}$ and $\DT\BracE{L}$ commute.  By Proposition \ref{prop:DT_tagged_rotation}, $\DT\BracE{\gamma}$ and $\BracE{\gamma^{\diamond}}$ agree up to frozen components, and by Proposition \ref{prop:surface-DT}, $\DT\BracE{L} = \BracE{L}$.  We know from (i) that $\BracE{L}$ commutes with $\BracE{\gamma^{\diamond}}$. 
 Also, it follows from \eqref{Lambda-B} that $\BracE{L}$ always commutes with frozen variables because all exponents in Laurent expansions of $\BracE{L}$ lie in $\omega_1(N_{\uf})$.  The desired commutativity follows.
 
 Now suppose that $\Sigma$ is a once-punctured closed surface.  Applying $\DT^{-1}$, it suffices to show that $\DT^{-1}\BracE{\gamma}$ and $\DT^{-1}\BracE{L}$ commute. Proposition \ref{prop:DT_tagged_rotation} implies that $\DT^{-1}\BracE{\gamma}$ and $\BracE{\gamma^{\diamond}}$ agree up to frozen components.  Moreover, Lemma \ref{lem:support_DT_L} says that $\DT^{-1}\BracE{L}$ is bi-pointed at $g(L)$ and $-g(L)$, i.e., has the form of the right-hand side of \eqref{eq:bipointed}.  Hence, all exponents appearing in the Laurent expansion of $\DT^{-1}\BracE{L}$ lie in $\omega_1(N_L)$ where $N_L$ denotes the $\bb{Z}$-span of $$\{e_{\gamma_i}|\gamma_i \text{ an arc in $\Delta$ intersecting }L\}.$$ Since $L$ does not intersect $\gamma^{\diamond}$, it follows from \eqref{Lambda-B} again that $\DT^{-1}\BracE{L}$ commutes with $\BracE{\gamma^{\diamond}}$ and with all frozen variables.  The desired commutativity follows.

\end{proof}

		Let $L\cup C$ be a weighted tagged simple multicurve in a triangulable surface $\Sigma$ with $L=\bigcup_j w_jL_j$ consisting of weighted loops and $C=\bigsqcup_i w_i C_i$ consisting of weighted tagged arcs. 
		\begin{lem}\label{lem:cluster_tagged_bracelet_theta}
		If no component of $\Sigma$ is a once-punctured closed surface or if $C$ does not contain a doubly-notched arc, then the (quantum) bracelet element $\BracE{L\sqcup C}$ is a theta function in $\s{A}^{\up}_t(\sd)$.
		\end{lem}
		\begin{proof}
		    By Lemma \ref{union-product}, it suffices to treat connected $\Sigma$. It also suffices to work in the quantum setting.  So $\BracE{C}$ is a quantum cluster monomial, and $\BracE{L}$ is a quantum theta function by Lemma \ref{lem:loop_puncture_bracelet}.  By the construction in \S \ref{sec:general_quantum_bracelet}, $\BracE{L\sqcup C} = \BracE{L} \BracE{C}$.  The two factors here commute by Lemma \ref{lem:commute_arc_loop}, so Proposition \ref{prop:cluster-commute-theta} implies that $\E{L\sqcup C}$ is a theta function.
		\end{proof}

		Once-punctured closed surfaces will be treated in \S \ref{sub:bracelets-1p} by working with their covering spaces (Theorem \ref{thm:1p-bracelet}, Lemma \ref{lem:1p_general_bracelet}).

		\begin{thm}\label{thm:punctured_bracelet_theta}
			Let $L\cup C$ be a weighted tagged simple multicurve in a triangulable surface $\Sigma$ with $L=\bigcup_j w_jL_j$ consisting of weighted loops and $C=\bigsqcup_i w_i C_i$ consisting of weighted tagged arcs.  If no component of $C$ is a notched arc in a once-punctured closed torus, then the (quantum) bracelet element $\BracE{L\sqcup C}$ is a theta function in $\s{A}^{\up}_t(\sd)$.
			
			Otherwise, let $\gr(g(C))$ be the sum of the weights of all notched arcs in once-punctured closed torus components of $\Sigma$.  Then in the coefficient-free classical setting, we have
			$$\BracE{L\cup C}=4^{\gr (g(C))}\vartheta_{g(L\sqcup C)}.$$
		\end{thm}

		\begin{proof}
		    By Lemma \ref{union-product}, it suffices to treat connected $\Sigma$. We have seen that $\E{L\sqcup C}$ is a theta function under the assumption of Lemma \ref{lem:cluster_tagged_bracelet_theta}.
		    
		    Now suppose that $\Sigma$ is a once-punctured closed surface and $C$ includes (doubly) notched arcs.  In the classical $t=1$ setting, the desired results (including the torus case) are Lemma \ref{lem:1p_general_bracelet}.
		    
		    Finally, consider the quantum setting for $\Sigma$ a once-punctured closed surface of genus $\geq 2$ and $C$ consisting of notched arcs.  Then $\BracE{L}$ is a quantum theta function by Lemma \ref{lem:loop_puncture_bracelet}, and $\BracE{C}$ is also a quantum theta function (see \S \ref{subsub:qtag}). So, as their product, $\BracE{L\cup C}$ is theta positive. Finally, by Lemma \ref{lem:1theta-implies-qtheta}, $\BracE{L\cup C}$ is a quantum theta function because it is bar-invariant by Lemma \ref{lem:commute_arc_loop} and its classical limit is a theta function by the previous paragraph.
		\end{proof}

		\subsection{Skein algebras and their atomic bases}\label{sec:skein_atomic_bases}

		By Theorem \ref{thm:bracelet-theta-no-punct} combined with Theorem \ref{fFG-surfaces} and Propositions \ref{SkAProp}, \ref{gdense} and \ref{AtomicProp}, we have the following result:

		\begin{thm}\label{thm:sk_unpunct_basis}
			When $\Sigma$ is unpunctured, its quantum localized skein algebra $\Sk_t(\Sigma)$ coincides with the corresponding quantum upper cluster algebra $\s{A}^{\up}_t(\sd_\Delta)$. Moreover, the quantum bracelets $\Brac_t(\Sigma)$ coincide with the quantum theta functions, and they form the atomic basis for $\Sk_t(\Sigma)$ with respect to the cluster atlas.
		\end{thm}

		\begin{thm}\label{thm:tag_sk_up_cl_alg}
			Let $\Sigma$ denote a connected marked surface, possibly with punctures. Let $\Sk(\Sigma)$ (resp. $\Sk^{\Box}(\Sigma)$) denote the (resp. tagged) classical localized skein algebra.  Suppose $\Sigma$ is not a once-punctured torus.  Then the following claims are true:
			\begin{enumerate}[label=(\roman*),ref=(\roman*), noitemsep]
				\item $\nu: \s{A}^{\midd}(\sd_\Delta)\rightarrow \s{A}^{\up}(\sd_\Delta)$ is injective, $\s{A}^{\midd}(\sd_\Delta)=\s{A}^{\can}(\sd_\Delta)$, and $\Sk^{\Box}(\Sigma) = \nu(\s{A}^{\midd}(\sd_\Delta))$.\label{item:nu_injectivity}
				
				\item The tagged bracelets form the atomic basis for $\Sk^{\Box}(\Sigma)$ with respect to the scattering atlas, or equivalently, with respect to the tagged triangulation atlas.
				
				\item The bracelets form the atomic basis for $\Sk(\Sigma)$ with respect to the ideal triangulations atlas.
			\end{enumerate}
            Now suppose $\Sigma$ is a once-punctured torus.  Then the above claims still hold if, for this case, we construct the theta functions, $\s{A}^{\midd}$ and $\s{A}^{\can}$, $\nu$, and the scattering atlas using a certain scattering diagram $\f{D}'$ in place of $\f{D}^{\sd^{\prin}_{\Delta}}$ (recall that frozen variables are set to $1$ after computing the theta functions).  This $\f{D}'$ has the form $\f{D}^{\sd^{\prin}_{\Delta}} \sqcup \{(H,f_H)\}$ for $H$ the hyperplane orthogonal to $n_0\coloneqq\sum_{i\in I_{\uf}} e_i$ and $f_H\in \f{g}_{m_0}^{\parallel}$, $m_0\coloneqq(0,n_0)$, satisfying $\Ad_{f_H}^{-\langle n_0,m\rangle}(z^m)=z^m(P_H)^{|\langle n_0,m\rangle|}$ where $P_H$ is a polynomial in $1+z^{m_0}\bb{Z}_{\geq 0}[z^{m_0}]$ whose coefficients add up to $4$.
		\end{thm}
		
		\begin{proof}
			
			(i)  By Theorem \ref{fFG-surfaces}, $\s{A}^{\midd}(\sd_\Delta)=\s{A}^{\can}(\sd_\Delta)$, and $\nu$ is injective.  Recall from \S \ref{S:tagged_sk} that $\Sk^{\Box}(\Sigma)$ is the subalgebra of $\s{A}^{\up}(\sd_{\Delta})$ generated by the tagged bracelets. By Theorem \ref{thm:punctured_bracelet_theta}, the tagged bracelets coincide with the theta functions, so $\Sk^{\Box}(\Sigma)$ is in fact equal to $\nu(\s{A}^{\midd}(\sd_\Delta))$.

			(ii) The claim follows from (i), Theorem \ref{thm:punctured_bracelet_theta}, and Corollary \ref{cor:surface-theta-atomic} (the atomicity of theta functions).

			(iii) Now that we know the bracelets are theta functions, this follows from Lemma \ref{lem:untag_skein_atom_basis}.

            Finally, suppose that $\Sigma$ is a once-punctured torus.  Take the scattering diagram $\f{D}'$ from the statement of the theorem to be $\?{\f{D}^{{\sd^{\prin}_{\wt{\Delta}}}}}$ as in \S \ref{sec:unfolding-closed-surfaces}.  We see from Lemma \ref{lem:H}, \eqref{eq:adfh}, and Theorem \ref{thm:1p-bracelet} that this $\f{D}'$ indeed has the desired from.  Furthermore, using this $\f{D}'$ and Lemma \ref{lem:1p_general_bracelet}, the arguments from the cases where $\Sigma$ was not a once-punctured torus now apply in this setting as well.
		\end{proof}

			Let there be given an initial seed $\sd$ and some $i\in F$. As in \cite{qin2022freezing}, we say a Laurent polynomial $Z=\sum c_m z^m \in \kk[M]$ is \textbf{regular on $A_i=0$} if $\langle e_i, m\rangle \geq 0$ whenever $c_m\neq 0$. By the mutation rule, this property for $Z\in \s{A}^{\up}$ is independent of the choice of the initial seed $\sd$ (see \cite[\S A]{qin2022freezing}, or using the viewpoint of valuations and tropicalization \cite{CMMM}).  Recall the sets $\?{\Brac}_t(\Sigma)$ and $\?{\Brac}^{\Box}(\Sigma)$ of \S \ref{sec:bracelet_band} and \S \ref{Sec:tag-brac}, respectively.

			\begin{cor}\label{cor:tag_sk_up_cl_alg}
			Let $\Sigma$ denote a connected marked surface, possibly with punctures. Then the following claims are true (with the same caveats for the once-punctured torus case as in Theorem \ref{thm:tag_sk_up_cl_alg}).
			\begin{enumerate}[label=(\roman*),ref=(\roman*), noitemsep]
				\item $\?{\Brac}^{\Box}(\Sigma)$ coincides with $\{\vartheta_m|\text{$\vartheta_m$ is regular on $A_i=0$ for $i\in F$}\}$.

				\item The set $\?{\Brac}^{\Box}(\Sigma)$ is the atomic basis for $\?{\Sk}^{\Box}(\Sigma)$ with respect to the scattering atlas, or equivalently, with respect to the tagged triangulation atlas.
				
				\item The set $\?{\Brac}(\Sigma)$ is the atomic basis for $\?{\Sk}(\Sigma)$ with respect to the ideal triangulations atlas.
			\end{enumerate}
		\end{cor}
		\begin{proof}
(i)
If $\Sigma$ has empty boundary, then this reduces to the claim that $\Brac^{\Box}(\Sigma)$ coincides with the set of theta functions (Theorem \ref{thm:tag_sk_up_cl_alg}).  So now assume that $\Sigma$ has non-empty boundary.

Take any weighted tagged simple multicurve $C\sqcup L$, such that $C$ denotes a union of weighted tagged arcs and $L$ a union of weighted simple loops. Since $\Sigma$ has non-empty boundary, $\E{C}$ must be a cluster monomial in some tagged triangulation $\Delta'$. Then the Laurent expansion of $\E{C}$ in $\sd_{\Delta'}$ is regular on the frozen $A_i=0$ if and only if the multiplicity of the boundary arc $A_i$ in $C$ is non-negative.

Choose $\Delta$ any ideal triangulation. Then the Laurent expansion of $\E{L}$ in $\sd_{\Delta}$ is regular on $A_i=0$ for $i\in F$ by Corollary \ref{Deltak}. Moreover, since $g(L)$ has vanishing frozen coordinates by \eqref{eq:loop_g}, and since the coefficient of $z^{g(L)}$ in $\vartheta_{g(L),\sQ}$ is nonzero for all $\sQ$, the Laurent expansion of $\E{L}$ in $\sd_{\Delta'}$ (or any cluster) is not divisible by $A_i$ for any $i\in F$.

Therefore, $\E{C\sqcup L}=\E{C}\E{L}$ is regular on $A_i=0$ for frozen $i$ if and only if $C$ is regular, or equivalently, if $\E{C}\in \?{\Sk}^{\Box}(\Sigma)$.  This is equivalent to $\E{C\sqcup L} \in \?{\Brac}(\Sigma)$, as desired.

(ii)  We know that $\Brac^{\Box}(\Sigma)$ forms a basis for $\Sk^{\Box}(\Sigma)$ by Theorem \ref{thm:tag_sk_up_cl_alg}, so the elements of $\?{\Brac}^{\Box}(\Sigma)$ are linearly independent.  The fact that $\?{\Brac}^{\Box}(\Sigma)$ spans $\?{\Sk}^{\Box}(\Sigma)$ is Lemma \ref{lem:Brac-spans}.

Now take any element $Z\in \?{\Sk}^{\Box}(\Sigma)$ which is universally positive with respect to the scattering atlas (equivalence between this and the tagged triangulation atlas is part of Corollary \ref{cor:surface-theta-atomic}).  Then Theorem \ref{thm:tag_sk_up_cl_alg}(ii) implies we can write $Z$ as a linear combination
\begin{align}\label{eq:Z}
Z=\sum a_i \E{C_i} \text{ for $\E{C_i}\in \Brac^{\Box}(\Sigma)$ and $a_i\in \bb{Z}_{\geq 0}$.}
\end{align}
Since elements of $\Brac^{\Box}(\Sigma)$ are linearly independent and $\?{\Brac}(\Sigma)$ spans $\?{\Sk}^{\Box}(\Sigma)$, the $\E{C_i}$ appearing in \eqref{eq:Z} must in fact all be in $\?{\Brac}^{\Box}(\Sigma)$.  So $\?{\Brac}^{\Box}(\Sigma)$ is indeed an atomic basis with respect to the scattering atlas.

(iii) The proof is similar to (ii), but using Lemma \ref{Lem:bracelet_basis} in place of Lemma \ref{lem:Brac-spans}.

\end{proof}

		\begin{cor}\label{cor:q_cluster_skein_equal}
		    For $\Sigma$ unpunctured, $\?{\Sk}_t(\Sigma) = \?{\s{A}}_t^{\up}(\sd_{\Delta})$. Moreover,  $\?{\Brac}_t(\Sigma)$ coincides with the set of theta functions which are regular on $A_i=0$ for all $i\in F$.
		\end{cor}
		\begin{proof}
This follows from Theorem \ref{thm:sk_unpunct_basis} and Corollary \ref{cor:tag_sk_up_cl_alg}(i) (positivity ensures that $\vartheta_m$ is regular on $A_i=0$ if and only if its $t=1$ limit is).
\end{proof}

\subsubsection{The tagged skein algebra from generators and relations}\label{subsub:tagged-skein}

The following alternative construction of $\?{\Sk}^{\Box}(\Sigma)$ was suggested to us by Greg Muller. Recall that a generalized tagged multicurve is the same as a multicurve, except that the arcs are allowed to be generalized tagged arcs.  Let $\kk^{\Box}(\Sigma)$ be the free $\kk$-module generated by homotopy equivalence classes of generalized tagged multicurves.  Let $R^{\Box}\subset \kk^{\Box}(\Sigma)$ be module of relations generated by the following (applied locally):
\begin{itemize}
    \item Contractible arcs are equivalent to $0$;
    \item Contractible loops are equivalent to $-2$;
    \item Peripheral loops are equivalent to $+2$;
    \item The $q=1$ case of the skein relation from Figure \ref{SkeinFig};
    \item The local digon relation of Figure \ref{fig:localDigon}.
\end{itemize}

\begin{cor}\label{cor:tag-skein}
    The tagged skein algebra $\?{\Sk}^{\Box}(\Sigma)$ from \S \ref{S:tagged_sk} can be equivalently defined as the $\kk$-algebra $\kk^{\Box}(\Sigma)/R^{\Box}$, equipped with $\sqcup$ as the product.
\end{cor}
\begin{proof}
    We saw in \S \ref{S:tagged_sk} that $\?{\Sk}^{\Box}(\Sigma)$ is generated by the same elements of $\kk^{\Box}(\Sigma)$ and satisfies all the relations coming from $R^{\Box}$.  We thus have a map $\kk^{\Box}(\Sigma)/R^{\Box}\rar \?{\Sk}^{\Box}(\Sigma)$.  It remains to check that this map has trivial kernel (i.e., we want to ensure that $\?{\Sk}(\Sigma)$ does not include any additional relations).  The proof of Lemma \ref{lem:Brac-spans} applies to show that the bracelets span $\kk^{\Box}(\Sigma)/R^{\Box}$.  So if there were a nontrivial element in the kernel, this would correspond to a nontrivial linear combination of bracelets in $\?{\Sk}^{\Box}(\Sigma)$ which equals $0$, contradicting the claim in Corollary \ref{cor:tag_sk_up_cl_alg} that the bracelets form a basis for $\?{\Sk}^{\Box}(\Sigma)$.
\end{proof}

    \subsubsection{Thurston's Conjecture}
	The construction of the quantum bracelets basis in \S \ref{sec:bracelet_band} applies for all unpunctured surfaces $\Sigma$, including those which are not triangulable.  D. Thurston conjectured that these quantum bracelets bases should be strongly positive for all (not necessarily triangulable) unpunctured marked surfaces $\Sigma$ \cite[Conj. 4.20]{Thurst}.  We prove this for cases with non-empty boundary.

    \begin{thm}\label{Thurst-conj}
    Let $\Sigma$ be an unpunctured marked surface, every component of which has non-empty boundary.  The quantum bracelets basis for $\?{\Sk}_t(\Sigma)$ is strongly positive.
    \end{thm}
    \begin{proof}
    By inserting extra marked points on the boundary components, we can form a triangulable unpunctured marked surface $\Sigma'$.  Note that there is a natural inclusion $\?{\Sk_t}(\Sigma)\hookrightarrow \?{\Sk}_t(\Sigma')$ which takes quantum bracelets to quantum bracelets.  The quantum bracelets basis for $\?{\Sk}_t(\Sigma')$ coincides with the quantum theta basis by Theorem \ref{thm:sk_unpunct_basis} and Corollary \ref{cor:q_cluster_skein_equal}, hence is strongly positive by Proposition \ref{AtomicProp}.  Strong positivity of the quantum bracelets basis for $\?{\Sk}_t(\Sigma)$ now follows.
    \end{proof}

\section{Canonical functions on cluster Poisson varieties from surfaces} \label{sec:cluster_poisson}
	
	\begin{ntn}
		When considering the cluster Poisson algebra $\s{X}^{\up}$ associated to a surface $\Sigma$, we shall write $\s{X}^{\up}_{\uf}$ for upper cluster algebra obtained by forgetting the frozen variables.  Note that $\s{X}^{\up}_{\uf}$ is a subalgebra of $\s{X}^{\up}$ via $N_{\uf}\hookrightarrow N$, and this inclusion identifies the theta functions in $\s{X}^{\up}_{\uf}$ with theta functions in $\s{X}^{\up}$.  Recall the map $\pi:\s{A}_L(\SSS,\bb{Q})\risom N_{\bb{Q}}$ as in \eqref{pil}.  Denote 
		\begin{align*}
			\s{A}_{\Sigma}(\bb{Z}^t):= \s{A}_{L}(\Sigma,\bb{Z})\cap \pi^{-1}(N) \quad \text{and} \quad  \s{A}_{\Sigma}^0(\bb{Z}^t):= \s{A}_{L}(\Sigma,\bb{Z})\cap \pi^{-1}(N_{\uf}).
		\end{align*} 
	\end{ntn}

	The skein algebras $\Sk(\Sigma)$ considered in \S \ref{SkeinSection} can be interpreted in terms of rational functions on the moduli spaces of decorated, twisted, $\SL_2$-local systems on $\Sigma$.  This perspective was developed in \cite{FockGoncharov06a}, with the connection to skein relations appearing in \cite[proof of Thm. 12.2]{FockGoncharov06a} (also cf. \cite[Prop. 4.12]{musiker2013matrix} which extends this connection to the principal coefficients setting).
	We shall review this interpretation in \S \ref{sec:dec_SL2_moduli}.
	
	On the other hand, \cite{FockGoncharov06a,FG4} also gives a moduli-theoretic interpretation for the cluster $\s{X}$-variety $\s{X}^{\up}_{\uf}$ associated to a marked surface $\Sigma$ and uses this to construct canonical functions on $\s{X}^{\up}_{\uf}$, parametrized by $\s{A}_{\Sigma}^0(\bb{Z}^t)\cong N_{\uf}$ and defined in terms of traces and eigenvalues of certain monodromies.  This moduli-theoretic interpretation is extended to the quantum setup in \cite{Allegretti2015duality} using the quantum Teichm\"uller spaces of \cite{CF}.  Allegretti and Kim \cite{Allegretti2015duality} then consider quantizations of the canonical functions of Fock-Goncharov, defined using the quantum trace maps of \cite{BW}.  We shall briefly review the classical version of these constructions in \S \ref{sec:mod_PGL2}.  We review the quantization in \S \ref{sec:Chek-Fock} and then review some useful properties for this quantum analog in \S \ref{sec:X-q-can}. The main result of this section---Theorem \ref{thm:ThetaX}, which says that the (quantum) canonical functions of \cite{FockGoncharov06a,Allegretti2015duality} are the (quantum) theta functions---will be proved in \S \ref{sec:ThetaX}.

	We shall work with the full $\s{X}^{\up}$ rather than just $\s{X}^{\up}_{\uf}$, and so we consider bases parametrized by $N\cong \s{A}_{\Sigma}(\bb{Z}^t)$, extending those parameterized by $\s{A}_{\Sigma}^0(\bb{Z}^t)$. See Remark \ref{pin} for a short summary of how \cite{goncharov2019quantum} uses ``pinnings'' to incorporate the frozen variables into the moduli theoretic viewpoint.

	\subsection{Moduli of framed $\PGL_2$-local systems}\label{sec:mod_PGL2}
	
	Let $\MM^{\partial}:=\partial \SSS\cap \MM$; i.e., $\MM^{\partial}$ is the set of marked points which are not punctures.  Let $\hat{\SSS}$ denote the surface with boundary obtained from $\Sigma$ by viewing punctures in $\SSS$ as unmarked boundary circles.  Let $\partial' \hat{\SSS}$ denote the \textbf{punctured boundary} of $\hat{\SSS}$, i.e., $$\partial'\hat{\SSS}:=\partial \hat{\SSS} \setminus \MM^{\partial}.$$
	\begin{dfn}[\cite{FockGoncharov06a}, Def. 1.2]\label{PGL2Moduli}
		A \textbf{framed $\PGL_2$-local system} on a marked surface $\Sigma=(\SSS,\MM)$ is the data of a pair $(\s{L},\beta)$, where $\s{L}$ is a $\PGL_2$-local system on $\SSS$ (i.e., a principal $\PGL_2$-bundle with a flat connection), and $\beta$ is a flat section of the induced $(\PGL_2/B)$-local system $(\s{L}/B)|_{\partial' \hat{\SSS}}$.  Here, $B$ is a maximal Borel subgroup of $\PGL_2$ (e.g., the projection of the subgroup of upper triangular matrices), so $\PGL_2/B\cong \bb{P}^1$.  The space $\s{X}_{\Sigma}$ is the moduli space of framed $\PGL_2$-local systems on $\Sigma$. 
	\end{dfn}
	
	An equivalent definition of $\s{X}_{\Sigma}$ can be given as follows.  Choose a hyperbolic structure with geodesic boundary on $\SSS$.  Let $F_{\infty}(\Sigma)$ denote the preimage of the punctured boundary $\partial' \hat{\SSS}$ in the universal cover.  Note that there is a natural action of $\pi_1(\SSS)$ on $F_{\infty}(\Sigma)$.  
	
	Consider the data of a pair $(\rho,\psi)$, where $\rho:\pi_1(\SSS)\rar \PGL_2$ is a group homomorphism and $\psi:F_{\infty}(\Sigma)\rar \bb{P}^1$ is a $(\pi_1(\SSS),\rho)$-equivariant map.  That is,
	\begin{align*}
		\psi(\gamma c)=\rho(\gamma)\psi(c)
	\end{align*}
	for each $\gamma\in \pi_1(\SSS)$ and $c\in F_{\infty}(\SSS)$.  By \cite[Lem. 1.1]{FockGoncharov06a} (also cf. \cite[Def. 2.7]{Allegretti2015duality}), a framed $\PGL_2$-local system on $\Sigma$ is equivalent to the data of a pair $(\rho,\psi)$ as above, modulo the action of $\PGL_2$ (via conjugation on $\rho$ and the tautological action on $\bb{P}^1$).  So $\s{X}_{\Sigma}$ can be interpreted as the moduli space of these pairs $(\rho,\psi)$ up to this $\PGL_2$-action.

	Consider the upper cluster algebra $\s{X}^{\up}_{\uf}$ associated to $\sd_{\Delta}$ for $\Delta$ an ideal triangulation of $\Sigma$. Fock and Goncharov \cite[\S 9]{FockGoncharov06a} have shown that $\s{X}_{\uf}^{\up}$ naturally contains and is birational to\footnote{Rather, \cite[\S 9]{FockGoncharov06a} shows that clusters associated to ideal triangulations form charts on $\s{X}_{\Sigma}$.  The extension to clusters associated to tagged triangulations follows via the $(\bb{Z}/2\bb{Z})^n$-action described in \cite[\S 12.6]{FockGoncharov06a}.} the coordinate ring $\Gamma(\s{X}_{\Sigma},\s{O}_{\s{X}_{\Sigma}})$ of $\s{X}_{\Sigma}$.

    In particular, each interior arc $i$ in $\Sigma$ determines a rational coordinate function $X_i$ on $\s{X}_{\Sigma}$.\footnote{The $X_j$ here are different from our $X$-variables $z^{e_j}$ in \eqref{Xi} when there exist self-folded triangles. We keep the symbol $X_j$ as in most literature.}  Briefly, the coordinate $X_i$ for an interior arc $i$ is geometrically characterized as follows: For $Q$ the quadrilateral in $\Delta$ containing $i$ as a diagonal, let $Q'$ be the quadrilateral obtained by rotating the vertices of $Q$ along $\partial \SSS$ by a small amount in the direction of the orientation (i.e., replacing the arcs with associated elementary laminates as in Definition \ref{def:el-lam}). The flags associated to the vertices of $Q'$ can all be parallel transported in $Q'$ to the fiber $(\s{L}/B)|_p$ at a vertex $p$ of the elementary laminate $e(i)$ and identified in cyclic counterclockwise order with the points $\infty,-1,0,X_i\in \bb{P}^1$; cf. \cite[\S 1.4; \S 9.4]{FockGoncharov06a}.

		Following \cite[Def. 9.2(ii)]{allegretti2020monodromy}, we identify $z^{e_j}$ with $X_j$ if $j$ is not an interior arc of a self-folded triangle; otherwise, we identify $z^{e_j}$ with $X_j X_k$ where $k$ is the noose containing $j$. The skew-symmetric form $\varepsilon$ in \cite[\S 12]{FockGoncharov06a} is our $\omega=B^T$.  The map $p^*$ of loc. cit. (which we review in \S \ref{sec:dec_SL2_moduli}) is our $\omega_1$.

    \begin{rem}[Pinnings and frozen variables]\label{pin}
    Following \cite{goncharov2019quantum}, one may incorporate frozen variables by introducing the additional data of a pinning.  A \textbf{pinning} of a framed $\PGL_2$-local system $(\s{L},\beta)\in \s{X}_{\Sigma}$ is a choice of lift $\alpha\in (\s{L}/U)|_{\partial \SSS \setminus \MM^{\partial}}$ of $\beta|_{\partial \SSS \setminus \MM^{\partial}}$, where $U$ is the maximal unipotent subgroup of $B$ \cite[Def. 2.3, Lem.-Def. 3.7]{goncharov2019quantum}.  One considers the moduli space $\s{P}_{\Sigma}$ of framed $\PGL_2$-local systems with pinning.  By \cite[Ex. 2, bottom of page 34]{goncharov2019quantum}, $\PGL_2/U$ parametrizes pairs $(v,\omega)$ for $v\in \bb{A}^2\setminus \{0\}$ and $\omega\in \Lambda^2(\bb{A}^2)^*\setminus \{0\}$ (i.e., $\omega$ is a symplectic form on $\bb{A}^2$), modulo the $\bb{G}_m$-action $(v,\omega) \mapsto (tv,t^{-2}\omega)$.  Now consider a triangle in the ideal triangulation $\Delta$ with edges labeled $1,2,3$ in counterclockwise order, with the edge $E_3$ being a boundary edge.  Let $L_i$, $i=1,2,3$, be the elements of $\PGL_2/B$ associated to the three edges, viewed as lines in $\bb{A}^2$. Let $A_3=(v_3,\omega)$ be the lift of $L_3$ to $\PGL_2/U$ given by the pinning.  By loc. cit. and \cite[Eq. 343]{goncharov2019quantum}, the associated frozen variable $X_{E_3}$ is given by $$X_{E_3}=\frac{\omega(l_1,l_2)}{\omega(v,l_1)\omega(v,l_2)}$$ for $l_1,l_2$ arbitrary nonzero points in $L_1,L_2$, respectively.  By \cite[Thm. 8.11]{goncharov2019quantum}, these coordinates, along with those associated to interior arcs as above, yield a cluster Poisson structure $\s{X}_{\sd_{\Delta}}$ on $\s{P}_{\Sigma}$.
    \end{rem}
	
	Given an integer bounded lamination $\ell \in \s{A}_{\Sigma}(\bb{Z}^t)$, one associates the following canonical function $\bb{I}(\ell)\in \s{X}^{\up}$: 
	
	\begin{enumerate}
		\item If $\ell$ is a peripheral loop of weight $w$ with underlying unweighted loop $\?{\ell}$ clockwise-oriented around a puncture $p$, then $\bb{I}(\ell):=\lambda_{p}^w$, where $\lambda_p$ is the eigenvalue of the monodromy $\rho(\?{\ell})\in \PGL_2$ whose associated eigenspace is the monodromy-invariant flag associated to the point $p$ by the framing data.  Note: $\ell\in \s{A}_{\Sigma}(\bb{Z}^t)$ implies that $w$ is even, so the sign ambiguity in identifying $\lambda_p$ is irrelevant after raising to the $w$-th power.  		
		
		Equivalently (e.g., using \eqref{eq:classical_trace}), when expressed in the cluster associated to a fixed choice of ideal triangulation $\Delta$ without self-folded triangles, one has $$\rho(\ell)=X^{-\pi(\ell)}$$
		for $\pi$ the map as in \eqref{pil} associated to $\Delta$. 
		\item If $\ell$ is a non-peripheral loop of weight $w$ and underlying unweighted loop $\?{\ell}$, then $\bb{I}(\ell):=\Tr(\rho(\?{\ell})^w)$. As before, there is a sign-ambiguity when defining $\rho(\?{\ell})$, but this is resolved when applying the trace of the $w$-th power.
		\item If $\ell=w\?{\ell}\in \s{A}_{\Sigma}(\bb{Z}^t)$  for $\?{\ell}$ a bounded arc-laminate, let $\gamma$ be the ideal arc obtained by translating the ends of $\?{\ell}$ counterclockwise along $\partial \SSS$ until they are at points of $\MM$, so $\ell=e(\gamma)$. Let $\Delta$ be a tagged triangulation containing $\gamma$; note then that $-\wt{b}^{\Delta}(\ell)\in we_{\gamma}^*+M_F\subset C_{\sd_{\Delta}}^+$. In view of Lemma \ref{lem:intersection_shear_loop}, define $\bb{I}(\ell):=z^{-\pi(\ell)}$ in the cluster associated to $\Delta$.
		\item If $\ell=\sum_i w_i\ell_i$ where $\ell_i$ are the curves of $\ell$, with each homotopy class of curves appearing at most once in the sum and $w_i\in \bb{Z}$, then
		\begin{align*}
			\bb{I}(\ell):=\prod_i \bb{I}(w_i\ell_i).
		\end{align*}
	\end{enumerate}

Restricting to $\bb{I}:\s{A}_{\Sigma}^0(\bb{Z}^t)\rar \Gamma(\s{X}_{\Sigma},\s{O}_{\s{X}_{\Sigma}})\subset \s{X}^{\up}_{\uf}$ recovers the canonical functions of  \cite{FG4}.

Given a fixed choice of ideal triangulation $\Delta$, the monodromies $\rho(\?{\ell})$ considered above admit the following description in terms of the variables $X_i$, cf. \cite[\S 2.2]{Allegretti2015duality} or \cite[Eq. 12.19]{FockGoncharov06a}.  For each $i$, fix a square-root $X_i^{1/2}$ (the choice of square-root will not matter thanks to the definition of $\pi$ and the requirement that $\pi(\ell)\in N$).  Let $\ell$ be a component of a lamination of weight $w$ with underlying unweighted curve $\?{\ell}$.  Assume by possibly deforming $\ell$ that it intersects each arc of $\Delta$ in the minimal possible number of points (after truncating the parts spiraling around punctures).  Let $i_1,\ldots,i_s$ denote the arcs of $\Delta$ which $\?{\ell}$ crosses, in order, possibly with repetition.  After crossing an arc $i_k$, $\?{\ell}$ enters a triangle $T$ of $\Delta$ and then turns either left or right before exiting the triangle through the next arc $i_{k+1}$ (cyclically ordered).  If this turn is to the left, let 
\begin{align*}
	M_{k}=\left(\begin{matrix}
		X_{i_k}^{1/2} & X_{i_k}^{1/2} \\
		0 & X_{i_k}^{-1/2}
	\end{matrix}
	\right)
	=X_{i_k}^{-1/2}\cdot\left(\begin{matrix}
		X_{i_k} & X_{i_k} \\
		0 & 1
	\end{matrix}
	\right),
\end{align*}
and if it is to the right, then let
\begin{align*}
	M_{k}=\left(\begin{matrix}
		X_{i_k}^{1/2} & 0 \\
		X_{i_k}^{-1/2} & X_{i_k}^{-1/2}
	\end{matrix}
	\right)
	=X_{i_k}^{-1/2}\cdot\left(\begin{matrix}
		X_{i_k} & 0 \\
		1 & 1
	\end{matrix}
	\right).
\end{align*}
Then
\begin{align}\label{eq:classical_trace}
	\rho(\?{\ell})=M_{\?{\ell}}:=M_1 \cdots M_s.
\end{align}

A straightforward induction argument reveals that for any product of matrices of the form $\left(\begin{matrix}
	X_{i_k} & X_{i_k} \\
	0 & 1
\end{matrix}
\right)$ and $\left(\begin{matrix}
	X_{i_k} & 0 \\
	1 & 1
\end{matrix}
\right)$, the bottom-right entry will be of the form $1+\sum_{n\in N^+} a_n X^n$ for some coefficients $a_n\in \bb{Z}_{\geq 0}$, while the top-left entry will be of the form $\sum_{n\in N^+} b_n X^n$ for $b_n\in \bb{Z}_{\geq 0}$.  Thus, for $\ell$ a non-peripheral loop, the above description implies that
	\begin{align}\label{eq:ell-pi-pointed}
		\bb{I}(\ell)=\Tr(\rho(\?{\ell})^w)=X^{-\pi(\ell)}\left(1+\sum_{n\in N^+} c_n X^n\right)
	\end{align}
for some coefficients $c_n\in \bb{Z}_{\geq 0}$. 

Assume there exist no self-folded triangles. It is now clear from our definitions that $\bb{I}(\ell)$ is $(-\pi(\ell))$-pointed for all $\ell\in \s{A}_{\Sigma}(\bb{Z}^t)$.

\subsection{Moduli of twisted decorated $\SL_2$-local systems}\label{sec:dec_SL2_moduli}

We now briefly review the definition of twisted decorated $\SL_2$-local systems on $\Sigma$ as in \cite[Def. 2.4]{FockGoncharov06a} (also cf. \cite[Def. 10.1]{GonSh}).

Consider a marked surface $\Sigma = (\SSS,\MM)$.  Let $T'\SSS$ denote the punctured tangent space to $\SSS$, i.e., the complement of the $0$-section in $T\SSS$.  For any chosen base point $x\in \SSS$, $\pi_1(T'_x\SSS,\bb{Z})\cong \bb{Z}$.  Let $O$ denote a curve representing a chosen generator for $\pi_1(T'_x\SSS,\bb{Z})$.  A twisted $\SL_2$-local system $\s{L}$ on $\SSS$ is an $\SL_2$-local system on $T'\SSS$ whose monodromy around $O$ is $-\Id$.

Given such an $\s{L}$, the associated decorated flag bundle is $\s{L}_{\s{A}}:=\s{L}/U$ where $U$ is a maximal unipotent subgroup of $\SL_2$ (e.g., upper triangle matrices with $1$'s on the diagonal --- one assumes that $U$ is chosen to be the unipotent radical of the maximal Borel subgroup $B$ considered in Definition \ref{PGL2Moduli}).

For each component $C$ of $\partial'\hat{\SSS}$, let $\sigma:C\rar T'C$ denote the canonical-up-to-isotopy section of $T'\SSS|_C$ given by tangent vectors to $C$ directed according to the orientation induced on $C$ by the orientation of $\SSS$.

\begin{dfn}[\cite{FG1}, Def. 2.4]
	A twisted decorated $\SL_2$-local system on $\Sigma$ is a pair $(\s{L},\alpha)$, where $\s{L}$ is a twisted $\SL_2$-local system on $\SSS$, and $\alpha$ is a flat section of $\s{L}_{\s{A}}|_{\sigma(C)}$ for each component $C$ of $\partial'\hat{\SSS}$.  The moduli space of twisted decorated $\SL_2$-local systems on $\Sigma$ is denoted by $\s{A}_{\Sigma}$.
\end{dfn}

As shown in \cite[\S 10]{FockGoncharov06a} (and reviewed in \cite{FG4}), $\s{A}$ admits a cluster structure --- more precisely, the upper cluster algebra associated to $\Sigma$ in \S \ref{sub:ClSk} can be identified with the ring of global regular functions on (at least a dense subset of) $\s{A}_{\Sigma}$.

Recall the set $\s{X}_L(\Sigma,\bb{Z})$ of integral unbounded laminations as in Definition \ref{def:lamination}.  As reviewed in \S \ref{sec:shear_coord} (cf. Figure \ref{fig:elementary_laminates}), one associates an elementary laminate $e(\gamma)$ to each tagged arc $\gamma$, and in this way one naturally constructs bijection between $\s{X}_L(\Sigma,\bb{Z})$ and the set of tagged bracelets on $\Sigma$ which do not include boundary arcs.  In \cite[Def. 12.4]{FockGoncharov06a}, Fock and Goncharov define canonical coordinates on $\s{A}_{\Sigma}$ associated to the $\s{X}$-laminations $\s{X}_L(\Sigma,\bb{Z})$---as alluded to in Remark \ref{rem:FG-Z2Z}, laminates which spiral counterclockwise into punctures (i.e., associated to arcs with notched ends) are understood using the $(\bb{Z}/2\bb{Z})^k$-action of \cite[\S 12.6]{FockGoncharov06a}.  These canonical coordinates on $\s{A}_{\Sigma}$ are precisely the bracelets basis elements (without boundary coefficients), cf. \cite[Prop. 4.12]{musiker2013matrix}---indeed, the relationship to the skein relations was already observed in \cite[Proof of Thm. 12.2]{FockGoncharov06a}.

In particular, if $\ell=\sum_i w_i \ell_i$ is a weighted multicurve consisting of pairwise compatible and non-isotopic loops $\ell_i$ with weight $w_i$, then the associated bracelets basis element can equivalently be interpreted as the function on $\s{A}_{\Sigma}$ given by
\begin{align}\label{eq:HolA}
	\bb{I}^{\vee}(\ell) = \prod_i \Tr(\Mono_{\ell_i}^{w_i})
\end{align}
where $\Mono_{\ell_i}$ denotes the monodromy of the twisted $\SL_2$-local system around $\ell_i$ (or rather, the lift of $\ell_i$ to $T'C$).

We extend $\bb{I}^{\vee}$ from $\s{X}_L(\Sigma,\bb{Z})$ to $\wt{\s{X}}_L(\Sigma,\bb{Z})$ as in \eqref{def:lamination} as follows.  Let $\ell=\ell_0+\ell_1$ for $\ell_0\in \s{X}_L(\Sigma,\bb{Z})$ and $\ell_1=\sum w_ie(\gamma_i)$ for boundary arcs $\gamma_i$.  Then $$\bb{I}^{\vee}(\ell)=\bb{I}^{\vee}(\ell_0)\prod_i z^{w_ie_{\gamma_i}^*}.$$
Note that these factors $z^{w_ie_{\gamma_i}^*}$ are just the associated boundary coefficients.  I.e., we extend $\bb{I}^{\vee}$ so that the elements $\bb{I}^{\vee}(\ell)$ for $\ell\in \wt{\s{X}}_L(\Sigma,\bb{Z})$ are precisely the tagged bracelets bases with frozen variables for boundary arcs.

Let $p:\s{A}_{\Sigma}\rar \s{X}_{\Sigma}$ denote the canonical projection induced by $\SL_2\rar \PGL_2$.  By \cite[Prop. 9.1]{FockGoncharov06a}, the induced map on functions $p^*:\Gamma(\s{X}_{\Sigma},\s{O}_{\s{X}_{\Sigma}}) \rar \Gamma(\s{A}_{\Sigma},\s{O}_{\s{A}_{\Sigma}})$ can be identified with the map induced by $\omega_1$ acting on exponents.  One finds the following:
\begin{lem}\label{PullbackLem}
	Let $\gamma=\sum_i w_i \gamma_i$ be a weighted simple multicurve consisting of pairwise compatible and non-isotopic loops and tagged arcs $\gamma_i$ with weights $w_i$.  Assume that none of the arcs in $\gamma$ have any ends at a puncture (specifying that the arcs are tagged just serves to exclude nooses).  Associate to $\gamma$ the integral bounded lamination $\ell:=\sum_i w_i e(\gamma_i)$ for $e(\gamma_i)$ the elementary laminate as in \S \ref{sec:shear_coord}. Suppose $\ell\in \s{A}_{\Sigma}(\bb{Z}^t)$, so $\bb{I}(\ell)$ is defined.  Then 
	\begin{align}\label{eq:I}
		\bb{I}^{\vee}(\gamma)z^{f_{\gamma}} = p^*(\bb{I}(\ell))
	\end{align}
	where $z^{f_{\gamma}}\in \kk[M]$ is some monomial with $f_{\gamma}\in M_F=\bb{Z}\langle e^*_i|i\in F\rangle$.  If $\gamma$ contains no arcs, then $f_{\gamma}=0$.
\end{lem}

We note that this is essentially \cite[Thm. 12.2(4)]{FockGoncharov06a} with frozen variables. 

\begin{proof}
	For $\gamma$ a combination of weighted loops, the claims are immediate from the definitions in terms of traces of monodromies.  

	For $\gamma$ consisting entirely of arcs, work in a cluster associated to an ideal triangulation $\Delta$ which contains all the arcs of $\gamma$ and no self-folded triangles.  Then we have $p^*(\bb{I}(\ell))=z^{\omega_1(-\pi(\ell))}$, which by \eqref{eq:intersection_shear_loop} is equal to $z^{-\wt{b}^{\Delta}(\ell)}$.  By Proposition \ref{prop:shear_g}, if $\gamma$ consists only of interior arcs, then 
	this agrees with $z^{g^{\sd_\Delta}(\gamma)}$ up to a frozen factor, and $z^{g^{\sd_\Delta}(\gamma)}$ is precisely the corresponding bracelet element, i.e., $\bb{I}^{\vee}(\ell)$, as desired.
	
	Now consider some $\gamma_i$ which is a boundary arc.  Then both $g^{\sd}_{\Delta}(\gamma_i)$ and $b^{\Delta}(e(\gamma_i))$ are seen to be contained in $M_F$, so the claim extends to cases with boundary arcs.
	
	One extends to general $\gamma$ by observing that adding weighted simple multicurves with pairwise-compatible non-isotopic components corresponds to multiplication of the associated elements on either side of \eqref{eq:I}.
\end{proof}

\subsection{Quantization: The Chekhov-Fock Algebra}\label{sec:Chek-Fock}

The quantum upper cluster algebra $\s{X}_q^{\up}$ associated to $\Sigma$ can be interpreted via the following construction of Chekhov-Fock \cite{CF}.  Let $\Delta$ be an ideal triangulation of $\Sigma$.  Let $T_i$ be a triangle of $\Delta$.  Then the triangle algebra $\s{T}_{T_i}^t$ associated to $T_i$ is the quantum torus algebra $$\s{T}_{T_i}^t=\kk_t^{\omega_{T_i}}[N_{T_i}]:=\kk_t[z^n|n\in N_{T_i}]/\langle z^{n_1}z^{n_2}:=t^{\omega_{T_i}(n_1,n_2)}z^{n_1+n_2}\rangle$$ where $N_{T_i}:=\bb{Z}\langle e_{i1},e_{i2},e_{i3}\rangle$ and $\omega(e_{ij},e_{i(j+1)}):=-1$ for $e_{i1},e_{i2},e_{i3}$ denoting the sides of $T_i$ in clockwise order.

	Now define the tensor product algebra
	\begin{align*}
		\s{T}_{\Delta}:=\bigotimes_{T_i\subset \Delta} \s{T}_{T_i}^t.
	\end{align*}
	We omit the factors $1$ when describing the elements of $\s{T}_{\Delta}$.
	If $e\in \Delta$ separates two triangles $T_i$ and $T_j$, and if $z^{e_{ia}}=:X_{ia}\in \s{T}_{T_i}$ and $z^{e_{jb}}=:X_{jb}\in \s{T}_{T_j}$ are the associated elements of $\s{T}_{T_i}\subset \s{T}$ and $\s{T}_{T_j}\subset \s{T}$, then define
	\begin{align*}
		X_e:=X_{ia}\otimes X_{jb}\in \s{T}_{\Delta}.
	\end{align*}
	If $e$ is the arc in the center of a self-folded triangle $T_i$, corresponding to the vectors $e_{i1},e_{i2}\in N_{T_i}$, then $$X_e:=z^{e_{i1}+e_{i2}}\in \s{T}_{T_i}\subset \s{T}_{\Delta}.$$
	For $e\in T_i$ a boundary arc of $\Sigma$ with associated element $z^{e_{ia}}\in \s{T}_i$, we take $X_e:=z^{e_{ia}}\in \s{T}_i\subset \s{T}_{\Delta}$.
	\begin{dfn}
		The \textbf{Chekhov-Fock algebra} $\s{X}_{t}^{\Delta}$ is the subalgebra of $\s{T}_{\Delta}$ generated by the elements $X_e$ and their inverses.
	\end{dfn}
	Equivalently, $\s{X}_t^{\Delta}$ is the quantum torus algebra $\s{X}_t^{\sd_{\Delta}}$ for the seed $\sd_{\Delta}$ associated to $\Delta$.
	
	Given different ideal triangulations $\Delta_1,\Delta_2$, \cite{CF} shows that the fraction fields $\mr{\s{X}}_{t}^{\Delta_1}$ and $\mr{\s{X}}_{t}^{\Delta_2}$ are related by a certain isomorphism $$\Phi_{\Delta_1\Delta_2}^t:\mr{\s{X}}_{t}^{\Delta_2}\risom \mr{\s{X}}_{t}^{\Delta_1},$$
	see \cite[Proposition 5]{liu2009quantum} for a precise definition. These isomorphisms agree with the corresponding quantum mutation maps $\mu_{\jj}^{\s{X}}$ considered in \eqref{mujjX} (for $\jj$ corresponding to a sequence of flips relating the two triangulations).  
	
	\begin{rem}
		We recall that some clusters of $\s{X}_t^{\up}$ might be related only to tagged triangulations, not ideal triangulations.  D. Allegretti has suggested to us that, by interpreting tagged triangulations as equivalence classes of signed triangulations (as defined in \cite[\S 8]{BS}), one could associate $q$-deformed canonical functions to laminations and signed (hence tagged) triangulations via essentially the same procedure used in \cite{Allegretti2015duality} for the ideal triangulations.
	\end{rem}

	\subsection{Quantum canonical coordinates}\label{sec:X-q-can}
	
	In \cite[Thm. 11]{BW}, Bonahon and Wong defined a ``quantum trace map'' $\Tr_{\Delta}^q$ as a certain homomorphism from a skein algebra on $\Sigma$ to $\mr{\s{X}}_{\Delta}^{t^{1/4}}$ (also cf. \cite{Le-qtrace} for a construction of the quantum trace map based on Muller's quantum skein algebra \cite{muller2016skein}).  Using this, \cite[Definitions 3.1, 3.4, 3.8, and 3.11]{Allegretti2015duality} defines quantum analogs of the Fock-Goncharov canonical coordinates as follows:\footnote{\cite{Allegretti2015duality} restricts to surfaces without boundary, so arc-laminates as in (3) here do not actually appear. 
	}

	\begin{enumerate}
		\item If $\ell$ is a peripheral loop, then $\bb{I}_t(\ell):=z^{-\pi(\ell)}$ for $\pi$ as in \eqref{pil}.   
		\item If $\ell$ is a non-peripheral loop of weight $w$ and underlying unweighted loop $\?{\ell}$, then  $\bb{I}_t(\ell):=T_w(\Tr_{\Delta}^q(\?{\ell}))$, where $T_w$ denote the $w$-th Chebyshev polynomial of the first kind as in \eqref{Cheb1}.  Here, $\Tr_{\Delta}^q(\?{\ell})$ is interpreted as applying $\Tr_{\Delta}^q$ to a certain framed link projecting to $\?{\ell}$ (one with constant elevation and vertical framing); cf. \cite[Def. 3.4]{Allegretti2015duality}.
		\item If $\ell=w\?{\ell}\in \s{A}_{\Sigma}(\bb{Z}^t)$ for $\?{\ell}$ a bounded arc-laminate, let $\gamma$ be the ideal arc obtained by translating the ends of $\?{\ell}$ along $\partial \SSS$ against the orientation until they are at points of $\MM$, so $\?{\ell}=e(\gamma)$. Let $\Delta'$ be an ideal triangulation containing $\gamma$. Define $\bb{I}_t(\ell)$ to be the element given by the monomial $z^{-\pi(\ell)}$ in the cluster associated to $\Delta'$.  I.e., $\bb{I}_t(\ell)=\Phi_{\Delta\Delta'}^t(z^{-\pi(\ell)})$.
		\item Let $\ell=\sum_{j\in J} w_j\ell_j\in \s{A}_{\Sigma}(\bb{Z}^t)$ with each laminate $\ell_j$ appearing at most once in the sum.   Let $\{\ell_j| j\in J'\subset J\}$ be the arc-laminates contributing to $\ell$, and let $\Delta'$ be an ideal triangulation such that, for each $j\in J'$, there is some $\gamma_j\in \Delta'$ with $\ell_j=e(\gamma_j)$.   Then
		\begin{align}\label{eq:It-ell}
			\bb{I}_t(\ell):=\Phi^t_{\Delta\Delta'}(z^{\sum_{j\in J'} -\pi(w_j\ell_j)})\prod_{j\in J\setminus J'} \bb{I}_t(w_j\ell_j).
		\end{align}
		\end{enumerate}

	We note that the factors in the product over $J\setminus J'$ commute by \cite[Lem 3.10]{Allegretti2015duality}. We will see in the proof of Lemma \ref{lem:no-pun-class-pos} that these factors also commute with the factor $\Phi^t_{\Delta\Delta'}(z^{\sum_{j\in J'} -\pi(w_j\ell_j)})$.
	
	Lemmas \ref{lem:I1}, \ref{lem:Xglue}, and \ref{lem:no-pun-class-pos} review the properties of $\bb{I}_t$ which shall be useful to us.
	
	\begin{lem}\label{lem:I1}
		For all $\ell\in \s{A}_{\Sigma}(\bb{Z}^t)$, $\bb{I}_1(\ell) = \bb{I}(\ell)$. 
	\end{lem}
	\begin{proof}
		This is clear for arcs and peripheral loops.  The case of non-peripheral loops is part of \cite[Prop. 3.12]{Allegretti2015duality}.  The extension to disjoint unions / products is straightforward.
	\end{proof}

	Let $\Sigma$ be a (not necessarily connected) marked surface with triangulation $\Delta$.  Suppose we cut $\Sigma$ along some interior arcs of $\Delta$ to form a new surface $\Sigma'$ with triangulation $\Delta'$ as in \S \ref{sec:cut}.  It is clear that this induces a projection $\f{glue}:\s{T}_{\Delta'}\rar \s{T}_{\Delta}$ --- if $e_1,e_2\in \Delta'$ glue together to form an edge $e$, then the projection maps $X_{e_1}\mapsto X_{e}$ and $X_{e_2}\mapsto X_e$.  Note that $\f{glue}$ is the identity on Laurent polynomials in the $X$-variables which do not correspond to any of the edges being glued.
	
	\begin{lem}\label{lem:Xglue}
		Consider $\Sigma',\Delta'$ obtained from $\Sigma,\Delta$ via cutting as above.  Let $\ell$ be a lamination in $\s{A}_{\Sigma}(\bb{Z}^t)$ which is disjoint from the arcs being cut, so $\ell$ can also be viewed as an element $\ell'\in \s{A}_{\Sigma'}(\bb{Z}^t)$.  Then
		\begin{align*}
			\f{glue}(\bb{I}_t(\ell')) = \bb{I}_t(\ell).
		\end{align*}
	\end{lem}
	\begin{proof}
		For non-peripheral loops, this is an easy consequence of the State Sum Property of the quantum trace map \cite[Thm. 11(1)]{BW}.  The claim for weighted peripheral loops, weighted arcs, and disjoint unions of non-isotopic laminations is then immediate from the definitions.
	\end{proof}

	\begin{lem}\label{lem:no-pun-class-pos}
		If $\Sigma$ has no punctures and $\ell\in \s{A}_{\Sigma}(\bb{Z}^t)$, $\bb{I}_t(\ell)$ is bar-invariant and universally positive with respect to the scattering atlas.
	\end{lem}
	\begin{proof}
		Suppose that $\ell$ consists only of weighted loops.  In unpunctured cases, all clusters correspond to ideal triangulations.  Let $\Delta$ be an arbitrary such ideal triangulation.  Let $\Sigma'$ be another surface, with triangulation $\Delta'$, which can be glued to $\Sigma$ along boundary edges to produce a new surface $\wt{\Sigma}$ with triangulation $\wt{\Delta}$ such that $\wt{\Sigma}$ has no boundary components (except for punctures).  Then $\bb{I}_t(\ell)\in \s{X}_t^{\wt{\Delta}}$ is positive (i.e., has coefficients in $\bb{Z}_{\geq 0}[t^{\pm 1}]$) by \cite[Thm. 1.1]{cho2020laurent}.  So by Lemma \ref{lem:Xglue}, $\bb{I}_t(\ell)$ is positive as an element of $\s{X}_t^{\Delta}\subset \s{X}_t^{\Delta\sqcup \Delta'}$.  Since $\Delta$ here is arbitrary, universal positivity with respect to the cluster atlas follows. Universal positivity with respect to the scattering atlas then follows from Proposition \ref{prop:X-pos}.    
		Bar-invariance follows via a similar argument using \cite[Thm. 1.2(4)]{Allegretti2015duality}.
		
		One the other hand, combinations of weighted arcs are defined to be cluster monomials, hence theta functions, so these are always bar-invariant and universally positive with respect to the scattering atlas.  By \eqref{eq:It-ell}, we now see that $\bb{I}_t(\ell)$ for arbitrary $\ell\in\s{A}_{\Sigma}(\bb{Z}^t)$ is a product of theta positive elements, hence the general $\bb{I}_t(\ell)$ is theta positive. 
        
        To see the bar-invariance, it remains to check that the factors $\bb{I}_t(w_j\ell_j)$ in \eqref{eq:It-ell} for $\ell_j$ a loop commute with the monomial $z^{\sum_{j\in J'} -\pi(w_j\ell_j)}$ (this can be checked in any cluster, so we may assume $\Delta=\Delta'$).  By the theta positivity and Lemma \ref{lem:I1}, the exponents appearing in the Laurent expansion of $\bb{I}_t(w_j\ell_j)$ are the same as those appearing $\bb{I}_1(w_j\ell_j)$.  The desired commutativity can now be deduced from the description of $\rho(\ell_j)$ in \eqref{eq:classical_trace}.
	\end{proof}

	\subsection{The canonical functions are theta bases}\label{sec:ThetaX}
	
	Our goal is to prove the following:
	\begin{thm}\label{thm:ThetaX}
		The elements $\bb{I}_t(\ell)$ for $\ell \in \s{A}_{\Sigma}(\bb{Z}^t)$ are precisely the quantum theta function $\{\vartheta_n|n\in N\}$ in $\s{X}_t^{\up}$.  More precisely, $\bb{I}_t(\ell)=\vartheta_{-\pi(\ell)}$ for each $\ell$.
	\end{thm}

	\begin{lem}\label{lem:no-pun-classical}
		Theorem \ref{thm:ThetaX} holds for $t=1$ when $\Sigma$ is an unpunctured surface. 
	\end{lem}
	\begin{proof}
		Note that for unpunctured surfaces, $\wt{\s{X}}_L(\Sigma,\bb{Z})$ is the same as $\s{A}_L(\Sigma,\bb{Z})$, see Definition \ref{def:lamination}. Given $n\in N$, let  $\ell:=\pi^{-1}(n)\in \s{A}_{\Sigma}(\bb{Z}^t)$ be the corresponding integer bounded lamination. By Lemma \ref{PullbackLem} and the fact that the elements $\bb{I}^{\vee}(\ell)$ are bracelets, hence theta functions, we have that $p^*(\bb{I}(\ell))=z^f\vartheta_m$ for some $m\in M$, $f\in M_F$.   So by Lemma \ref{lem:similar_theta_function} (multiplying a theta function by $z^f$ for $f\in M_F$ yields another theta function), $p^*(\bb{I}(\ell))$ is a theta function.  By Proposition \ref{prop:omega-theta}, $\omega_1$ takes $\s{X}$-type theta functions to $\s{A}$-type theta functions, hence $p^*$ (which equals $\omega_1$) takes $\s{X}$-type theta functions for $\sd_{\Delta}$ to $\s{A}$-type theta functions for $\sd_{\Delta}$.  Since the Injectivity Assumption is satisfied for unpunctured surfaces, we know that $\omega_1|_{-\pi(\ell)+N_{\uf}}$ is injective.  So since all exponents of $\vartheta_{-\pi(\ell)}$ and $\bb{I}(\ell)$ lie in $-\pi(\ell)+N^\oplus$ (because they are $(-\pi(\ell))$-pointed) and both map under $\omega_1$ to $\vartheta_{\omega_1(-\pi(\ell))}$, we must have $\vartheta_{-\pi(\ell)}=\bb{I}(\ell)$, as desired.
	\end{proof}
	
	\begin{lem}\label{lem:unpunctured_Poisson_theta}
		Theorem \ref{thm:ThetaX} holds when $\Sigma$ has no punctures.
	\end{lem}
	\begin{proof}
		By Lemma \ref{lem:no-pun-class-pos}, the elements $\bb{I}_t(\ell)$ are bar-invariant and theta positive, and by Lemmas \ref{lem:I1} and  \ref{lem:no-pun-classical}, they yield theta functions in the classical limit.  The claim now follows from Lemma \ref{lem:1theta-implies-qtheta}.
	\end{proof}

	\begin{proof}[Proof of Theorem \ref{thm:ThetaX}]
		By Lemma \ref{lem:triangulation}, given an integer bounded lamination $\ell$ without peripheral loops, we can find an ideal triangulation $\Delta$ of $\Sigma$ without self-folded triangles and then cut $\Sigma$ along arcs of $\Delta$ which do not intersect $\ell$ to obtain an unpunctured surface $\Sigma'$ with induced triangulation $\Delta'$.  Then $\ell$ corresponds to a quantum theta function in $\s{X}_t^{\Delta'}$ by Lemma \ref{lem:unpunctured_Poisson_theta}. Notice that $\ell$ does not cross any arcs being cut/glued, so by Lemma \ref{lem:Xglue}, $\bb{I}_t(\ell)$ has the same Laurent expansion whether viewed as an element of $\s{X}_t^{\Delta}$ or $\s{X}_t^{\Delta'}$.

		Choose initial quantum seeds associated to $\Delta$ and $\Delta'$ as in \S \ref{sec:theta_punctured}. Similarly, the Laurent expansions of the corresponding theta functions $\vartheta_{-\pi(\ell)}$ do not change under the gluing---otherwise their $p^*$-projections $p^*(\vartheta_{-\pi(\ell)})=z^{f_{\ell}}\vartheta_{g(\ell)}$ ($f_{\ell}\in M_F$ is given in Lemma \ref{PullbackLem}) in $\s{A}_t^{\up}$ would have different $F$-polynomials after gluing, and we know this is not the case since these theta functions are bracelets by Theorem \ref{thm:punctured_bracelet_theta}. Thus, $\ell$ corresponds to a theta function in $\s{X}_t(\sd_{\Delta})\subset \s{X}_t^{\Delta}$, as desired.

		To extend to cases with peripheral loops, note that for such a loop $\ell_j$, $\pi(\ell_j)\in \ker(\omega_1)$ by \cite[Lem 3.9]{Allegretti2015duality}.  Thus, $z^{-\pi(\ell_j)}$ commutes with all wall-crossings, and all walls are closed under addition by $\bb{R}\pi(\ell_j)$.  It follows that for $\ell_j$ a peripheral loop and any $n\in N$, we have $\vartheta_{-w\pi(\ell_j)}\vartheta_n=z^{-w\pi(\ell_j)}\vartheta_n=\vartheta_{n-w\pi(\ell_j)}$. The claim now follows from the definition $\bb{I}_t(\ell):=z^{-\pi(\ell)}$ and from \eqref{eq:It-ell}.
	\end{proof}

	\appendix
	
	\section{Folding of cluster structures}\label{GenFoldApp}

	In this section we consider the operation of ``folding'' for cluster algebras, as previously considered in \cite{felikson2012cluster,huang2018unfolding}, but now understood from the perspective of scattering diagrams.  A relationship between folding and scattering diagrams was previously examined in \cite{zhou2020cluster}; our goal is to show that a slightly weaker version of \cite[Thm. 2.20]{zhou2020cluster} applies to our more general version of folding.  This will be applied in \S \ref{sec:closed_surface} to relate the scattering diagram for a surface $\Sigma$ to the scattering diagram for a covering space $\wt{\Sigma}$.
	
	\subsection{Skew-symmetrizable cluster algebras}\label{sec:skew-symmetrizable}
	
	So far we have focused on skew-symmetric seeds because these are sufficient for understanding cluster algebras from surfaces.  However, the results of this appendix naturally apply to the more general skew-symmetrizable setup, and since these results may be of independent interest, we choose to work with skew-symmetrizable seeds here.  We begin by reviewing the modifications needed for this setup. For technical simplicity, our convention differs from the standard one, see Remark \ref{rem:compare_seed_convention}.
	
	A (Langlands dual)\footnote{In \eqref{omega-di}, one would typically have a factor of $d_j$ on the right-hand side of \eqref{omega-di} rather than $d_i$.  However, our folding construction will require the equation to be as in \eqref{omega-di}.  This change amounts to working with the Langlands dual of a seed as in \cite{FG1}.  Using $-\omega(\ ,e_i)$ in place of $\omega(e_i,\ )$ in \eqref{DIn-gen-class} (i.e., taking the Langlands dual scattering diagram) would allow us to avoid taking the Langlands dual here. 
	See Remark \ref{rem:compare_seed_convention} for more on  Langlands dual seeds.} \textbf{skew-symmetrizable} seed $\sd$ is data $(N,I,E=\{e_i\}_{i\in I},F,\omega^{\circ},\{d_i|i\in I\})$, where $N$, $I$, $E$, and $F$ are as in the skew-symmetric setup of \S \ref{Cluster_Section}, $\omega^{\circ}$ is a $\bb{Q}$-valued skew-symmetric form on $N$, and the values $d_i$ are positive rational numbers such that the form $\omega$ defined by
	\begin{align}\label{omega-di}
		\omega(e_i,e_j)=d_i\omega^{\circ}(e_i,e_j)
	\end{align}
	satisfies $\omega(e_i,e_j)\in \bb{Z}$ whenever $i$ and $j$ are not both in $F$.  We shall use the notation $\omega_1:N\rar M_{\bb{Q}}$, $n\mapsto \omega(n,\cdot)$ as before.  Note that the integrality condition implies that $\omega_1|_{N_{\uf}}$ has image in $M$.  We again define the bilinear form $B:=\omega^T$, i.e., $B(e_i,e_j)=\omega(e_j,e_i)=d_j\omega^{\circ}(e_j,e_i)$. We may also view $B$ as the matrix $(B_{ij})$, $B_{ij}=B(e_i,e_j)$.    Given a skew-symmetrizable seed $\sd$ as above, a compatible form $\Lambda$ is a $\bb{Q}$-valued skew-symmetric form on $M$ such that, for all $i\in I\setminus F$, we have
	\begin{align}\label{Lambda-gen}
		\Lambda_2(\omega_1(e_i)) = d_i'e_i
	\end{align}
	where $d'_i=\alpha d_i\in \Q_{>0}$ for some fixed\footnote{If the associated quiver $Q$ is connected, then the condition $d'_i=\alpha d_i$ is forced by the skew-symmetry and \eqref{Lambda-gen}.  More generally, one could allow different choices of $\alpha$ for each connected component of $Q$. See \cite[Lemma 2.1.11(1)]{qin2020dual} (which only looks different from \eqref{Lambda-gen} because we use $M$ instead of $M^{\circ}$) and \cite[Thm. 2.1]{GekhtmanShapiroVainshtein05}.} $\alpha\in \bb{Q}_{>0}$.  Note that \eqref{Lambda-gen} is equivalent to $\Lambda_2(\omega_1^{\circ}(e_i))= \alpha e_i$ for all $i\in I\setminus F$; i.e.,
	\begin{align}\label{eq:Lambda-circ}
		\Lambda_2(\omega_1^{\circ}(n))=\alpha n
	\end{align}
	for all $n\in N_{\uf}$.

	We note that the data of a skew-symmetric seed with compatible $\Lambda$ as in \S \ref{sec:review} is equivalent, up to re-scaling $\alpha$, to the data of a skew-symmetrizable seed with compatible $\Lambda$ having each $d_i$ equal to $d/\alpha$ for $d$ as in \eqref{Lambda-B}. 
		
		We will focus mostly on the classical (as opposed to quantum) setup, so the consideration of a compatible form $\Lambda$ is not strictly necessary.  Nevertheless, such $\Lambda$ will exist whenever the Injectivity Assumption is satisfied, and this form is still useful in understanding the classical setup, so we will use it here for convenience and for consistency with the rest of the paper.
		
		Now given a skew-symmetrizable seed $\sd$ with compatible $\Lambda$ as above, one may define associated torus algebras, dilogarithms, scattering diagrams, and theta functions almost exactly as in \S \ref{Section_Scat}.\footnote{The elements $\hat{z}^m$ of \eqref{hatz} are generalized as 	$\hat{z}^m\coloneqq${\large$\frac{z^m}{t^{|\Lambda_2(m)|}-t^{-|\Lambda_2(m)|}}$}. Here, for $n\in N_{\bb{Q}}$, $|n|:=\frac{1}{D}|Dn|$ for $D\in \bb{Z}_{\geq 1}$ such that $Dn\in N$ and $|Dn|$ is the index in $N$, so in particular, $|\Lambda_2(\omega_1(e_i))|=d'_i$.  The elements $\hat{z}^m$ generate the sub Lie algebra $\f{g}^t$, and the classical limit from $\f{g}^t$  to $\f{g}$ is defined such that $\hat{z}^m\mapsto \frac{1}{\alpha |m|}z^m$.}
	Then the initial scattering diagram $\f{D}^{\s{A}_t}_{\In}$ is defined by replacing the $d$ in \eqref{DIn} with $d'_i$ --- that is, one defines
		\begin{align}\label{DIn-gen}
			\f{D}^{\s{A}_t}_{\In} = \{(e_i^{\perp},\Psi_{t^{d'_i}}(z^{\omega_1(e_i)})) |i \in I\setminus F\}.
		\end{align}
		Unfortunately, the positivity results of \cite{davison2019strong} (and thus the proofs of most of the other main results of \cite{davison2019strong}) do not extend to the quantum skew-symmetrizable setup,\footnote{It is known that quantum positivity fails in some skew-symmetrizable cases, cf. \cite[\S 2.4.1]{cheung2020quantization} and \cite[\S 5.3]{nakanishi2022pentagon} for counterexamples motivated by the non-positivity of quantum greed bases \cite[\S 3]{LLRZpnas}.} and the generalized folding arguments below also do not apply in the quantum setting. 
		
		We therefore restrict to the classical limit.  By \eqref{eq:classical_dilogarithm}, the initial scattering diagram is
		\begin{align}\label{DIn-gen-class}
			\f{D}_{\In}^{\s{A}} = \{(e_i^{\perp},\Psi(z^{\omega_1(e_i)})^{1/d'_i})|i\in I\setminus F\};
		\end{align}
		as in \eqref{eq:mu-class}, mutation is given by 
		\begin{align}\label{eq:mut-Ad}
			(\mu_i^{\s{A}})^{-1}(z^p)=\Ad^{-1}_{\Psi(z^{\omega_1(e_i)})^{1/d'_i}}(z^p)=z^p(1+z^{\omega_1(e_i)})^{-\Lambda(\omega_1(e_i),p)/d'_i} = z^p(1+z^{\omega_1(e_i)})^{\langle e_i,p\rangle}.
		\end{align}
		Classical skew-symmetrizable cluster algebras are within the generality of \cite{gross2018canonical}, so the classical limits of the results of \S \ref{Section_Scat} will all still hold as in the skew-symmetric cases.

			\begin{rem}\label{rem:compare_seed_convention}
				Let us recall the standard convention of a seed $\sd$, see \cite[\S 2.1]{qin2019bases} \cite[Lem. 2.1.6]{qin2020dual}. It has symmetrizers $d_i \in \Z_{>0}$ and a $\Q$-valued skew-symmetric bilinear form $\{\  ,\ \}$ on $N(\sd)=\bigoplus_{i\in I} \Z e_i$. In addition, the bilinear form $B$ defined by $B(e_j,e_i)=B_{ji}=\{e_i,e_j\}d_j$ should satisfy $B_{ji}\in \Z$ whenever $i,j$ are not both in $F$. Denote the matrix $B=(B_{ij})$. It is further a quantum seed if it is endowed with a compatible $\Q$-valued skew-symmetric bilinear form $\Lambda$ on $M(\sd)$ such that $\Lambda_2(\{e_k,\ \})= d'_k d_k e_k$ for some multipliers $d'_k\in \Q_{>0}$ for all $k\in I\setminus F$. Denote $v_i=\{e_i,\ \} =\sum_{j\in I}B_{ji}\frac{1}{d_j}e_j^*$. The corresponding quantum scattering diagram $\f{D}(\sd)$ has the initial walls $(e_i^\bot, \Psi_{t^{d'_i}}(z^{v_i}))$.
				
				Let $D$ denote the least common multiple of $\{d_i|i\in I\}$. The Langlands dual seed $\sd^\vee$ has the basis vectors $e_i^\vee:=d_i e_i$, the symmetrizers $d_i^\vee:=\frac{D}{d_i}$, and the bilinear form $\{\ ,\ \}^\vee:=\frac{1}{D}\{\ ,\ \}$. Then we have
				\begin{align*}
					d_i^\vee \{e_i^\vee,e_j^\vee\}^\vee =\frac{1}{d_i}\{d_i e_i, d_j e_j\} =B_{ji}.
				\end{align*}
				Notice that $(e_j^\vee)^*=\frac{1}{d_j}e_j^*$ and $v_i=d_i^\vee\{e_i^\vee,\ \}^\vee=\frac{1}{d_i}\{e_i^\vee,\ \}$.
				
				Our basis vector $e_i$ is identified with the above $e_i^\vee$. The matrix $B$, the bilinear forms $B$, $\Lambda$, and multipliers $d_i'$ in our convention coincide with those above. The symbols $d_i$ and $\omega^\circ$ in our convention could be identified with $d_i^\vee$ and  $\{\ ,\ \}^\vee$, or with $\frac{1}{d_i}$ and $\{\ ,\ \}$ above. In either case, $\omega_1(\sum_{i\in I} n_i e_i)=\omega^{\circ}_1(\sum_{i\in I} n_i d_i e_i)$ in our convention coincides with $\{\sum_{i\in I} n_i e_i,\ \}$ above. The walls $((e_i)^\bot, \Psi_{t^{d'_i}}(z^{\omega_1(e_i)}))$ in our convention coincide with the initial walls of $\f{D}(\sd)$ as defined in this remark.
			\end{rem}

		\subsection{Covering of seeds}\label{sec:folding_seeds}

		Fix a skew-symmetrizable seed $\sd=(N,I,E,F,\omega^{\circ},\{d_i\}_{i\in I})$.  Let $\Pi$ be a partition of $I$, and for each $i\in I$, let $\Pi i\subset I$ denote the part of the partition containing $i$. We assume that $\Pi$ satisfies the following:
		\begin{enumerate}
			\item For each $i\in I$ and $i'\in \Pi i$, we have $d_i=d_{i'}$;
			\item For each $i,j\in I$ and all $i'\in \Pi i$,
			\begin{align}\label{FoldEqn}
				\sum_{j' \in \Pi j} \omega(e_{i},e_{j'}) = \sum_{j' \in \Pi j} \omega(e_{i'},e_{j'}).
			\end{align}
			We note that Condition (1) implies that \eqref{FoldEqn} is equivalent to the equation obtained by replacing each $\omega$ by $\omega^{\circ}$.
			\item Each $\Pi i$ is contained either entirely in $I\setminus F$ or entirely in $F$.
		\end{enumerate}

		\begin{eg}\label{eg:group-cover}
			Consider a finite group $\Pi$ acting on $I$, so the orbits of $\Pi$ yield a partition of $I$, also denoted $\Pi$.  Assume that Conditions (1) and (3) are satisfied for this partition.  Suppose that for all $g\in \Pi$ and $i,j\in I$ we have $\omega(e_{gi},e_{gj})=\omega(e_i,e_j)$. Then Condition (2) will be satisfied as well.  In our examples of primary interest (i.e., covering spaces), $\Pi$ will be a group of deck transformations.  We note that in the quiver folding considered in \cite{zhou2020cluster}, one imposes the stronger condition that $\omega(e_{g_1 i},e_{g_2 j})=\omega(e_i,e_j)$ for all $g_1,g_2\in \Pi$.
		\end{eg}

		We define a new skew-symmetrizable seed $\?{\sd}$ as follows: let $\?{I}$ be a set indexing the components of the partition $\Pi$.  Let $\?{N}=\bb{Z}^{\?{I}}$ with the natural basis $\?{E}=\{\?{e}_i|i\in \?{I}\}$.  Let $\?{F}\subset \?{I}$ be the indices corresponding to subsets of $F$.  For each $i\in \?{I}$, let 
		\begin{align*}
			\?{d}_i:=\frac{d_i}{|\Pi i|}.
		\end{align*}
		Define $\?{\omega}$ by 
		\begin{align}\label{omega-bar}
			\?{\omega}(e_{\Pi i},e_{\Pi j}) := \sum_{j' \in \Pi j} \omega(e_{i},e_{j'})  = \frac{1}{|\Pi i|}\sum_{\substack{i'\in \Pi i \\  j'\in \Pi j}} \omega(e_{i'},e_{j'}) = \?{d}_i \sum_{\substack{i'\in \Pi i \\  j'\in \Pi j}} \omega^{\circ}(e_{i'},e_{j'}).
		\end{align}
		Equivalently, one defines
		\begin{align}\label{omega-circ-bar}
			\?{\omega}^{\circ}(e_{\Pi i},e_{\Pi j}) = \sum_{\substack{i'\in \Pi i \\  j'\in \Pi j}} \omega^{\circ}(e_{i'},e_{j'}).
		\end{align}
		
		Following \cite[Def. 2.4]{huang2018unfolding}, one says that $\sd$ is a \textbf{covering} of $\?{\sd}$. We will also say that $\sd$ is a covering of a seed $\?{\sd}$ if just the data $\?{N}, \?{I}, \?{E}, \?{F}$, and $\?{\omega}$ associated to $\?{\sd}$ is as above (so $\?{\omega}^{\circ}$ may differ by some $\bb{Q}_{>0}$ re-scaling, and the $\?{d}_i$'s by the inverse re-scaling, but this re-scaling has no effect in the classical setup).  We say the partition $\Pi$ realizes $\sd$ as a covering of $\?{\sd}$.  Being an unfolding will require the additional condition that composite mutations of $\sd$ are also coverings of the corresponding mutations of $\?{\sd}$, cf. \S \ref{sec:fold}.
		
		\begin{example}\label{ex:Pi-prin}
			Let $\sd$ be any seed, possibly not satisfying the Injectivity Assumption, and let $\Pi$ be a partition of $I$ satisfying the conditions (1)-(3) above, thus realizing $\sd$ is a covering of some $\?{\sd}$. Then we extend $\Pi$ to a partition of $I^{\prin}=I\sqcup I$ satisfying (1)-(3) by taking the partition of each copy of $I$ to be the same as for the original $\Pi$.  This extension of $\Pi$ thus realizes $\sd^{\prin}$ as a covering of $\?{\sd}^{\prin}$.  
		\end{example}
		
		Assume now that $\sd$ satisfies the Injectivity Assumption.  Fix a compatible form $\Lambda$ and let $d'_i=\alpha d_i$ for some fixed $\alpha\in \bb{Q}_{>0}$ as in \eqref{Lambda-gen}.  Our goal is to show that the scattering diagram $\f{D}^{\s{A}}(\?{\sd})$ is naturally contained in a ``slice'' of the scattering diagram $\f{D}^{\s{A}}(\sd)$.
		
		We consider the embedding $\iota:\?{N}\hookrightarrow N$ defined by
		\begin{align*}
			\iota:e_{\Pi i} \mapsto \sum_{i'\in \Pi i} e_{i'}.
		\end{align*}
		Let $\?{M}:=\Hom(\?{N},\bb{Z})$, and let $\iota^*:M\rar \?{M}$ be the projection dual to $\iota$. It follows that 
		\begin{align}\label{eq:iota-star}
			\iota^*(e^*_i)=e^*_{\Pi i}.
		\end{align}
		We have the following diagram
		\[
		\begin{tikzcd}
			\?{N}\times \?{N} \arrow{r}{\?{\omega}}\arrow{d}{\iota\times\iota}\arrow[r, shift right=2, "\?{\omega}^\circ"']& \Q\\
			N\times N \arrow{r}{\omega} \arrow[r, shift right=2, "\omega^\circ"']& \Q
		\end{tikzcd}
		\]

		Consider the bilinear pairing $\iota^*\omega^{\circ}$ on $\?{N}$ given by $$(\iota^*\omega^{\circ})(n_1,n_2):=\omega^{\circ}(\iota(n_1),\iota(n_2))$$
		and similarly for $\iota^*\omega$.  We see from \eqref{omega-circ-bar} that \begin{align*}
			\iota^*\omega^{\circ} = \?{\omega}^{\circ},
		\end{align*}
		We similarly see from the second expression for $\?{\omega}$ in \eqref{omega-bar} that
		\begin{align*}
			\?{\omega}(e_{\Pi i},e_{\Pi j})= \frac{1}{|\Pi i|}(\iota^*\omega)(e_{\Pi i}, e_{\Pi j}),
		\end{align*}
		so
			\begin{align*}
				\?{\omega}_1(e_{\Pi i})= \frac{1}{|\Pi i|} (\iota^* \omega)(e_{\Pi i},\cdot),
			\end{align*}
		or equivalently, \begin{align}\label{i-omega-i}
			|\Pi i|\?{\omega}_1(e_{\Pi i}) = \iota^*(\omega_1(\iota(e_{\Pi i}))).
		\end{align}
		
		Similarly, the first expression for $\?{\omega}$ in \eqref{omega-bar} implies 
		\begin{align}\label{iota-omega}
			\?{\omega}_1(e_{\Pi i}) = \iota^*(\omega_1(e_i))
		\end{align}
		for each $i\in I$.  
		
		We also consider the inclusion
		\begin{align}\label{eq:rho-def}
			\kappa:\?{M}_{\bb{R}}\rightarrow M_{\bb{R}}, \qquad
			e_{\Pi i}^* \mapsto \frac{1}{|\Pi i|} \sum_{i'\in \Pi i} e_{i'}^*
		\end{align}
		with dual
		\begin{align*}
			\kappa^*:N_{\bb{R}}\rar \?{N}_{\bb{R}}, \qquad e_i\mapsto \frac{1}{|\Pi i|} e_{\Pi i}.
		\end{align*}
		Note that $\kappa$ is a section for the projection $\iota^*$, i.e., \begin{align}\label{eq:iota-rho-Id}
			\iota^* \circ \kappa = \Id_{\?{M}_{\bb{R}}}.    
		\end{align}
		We have the following 
		diagram:
		\[
		\begin{tikzcd}[every arrow/.append style={shift left}]
			\?{N} \arrow{r}{\iota}\arrow{d}{\?{\omega}_1} & N \arrow{d}{\omega_1} \\
			\?{M_{\bb{R}}} \arrow{r}{\kappa} & M_{\bb{R}} \arrow{l}{\iota^*} 
		\end{tikzcd}
		\]
		Note that $\kappa$ and $\iota$ are compatible with the dual pairings, i.e, for $\?{n}\in \?{N}$ and $\?{m}\in \?{M}_\R$, we have 
		\begin{align}\label{iota-rho}
			\langle \?{n},\?{m}\rangle = \langle \iota(\?{n}),\kappa(\?{m})\rangle.
		\end{align}
		Let $\kappa^{-1}(Z)$ denote the preimage of any subset $Z\subset M_\R$ in $\?{M}_\R$. Note that $\kappa^{-1}(Z)$ may be viewed as the intersection of $Z\subset M_\R$ with $\kappa (\?{M}_\R)$.  
		By \eqref{iota-rho}, for each $i\in I$, we have
		\begin{align}\label{rho-perp}
			\kappa^{-1}(e_i^{\perp}) = e_{\Pi i}^{\perp}.
		\end{align}

		We also note that for all $i,j\in I$ and each $j'\in \Pi j$, $$\omega(\iota(e_{\Pi i}),e_{j'})=-\frac{d_i}{d_j} \omega\left(e_{j'},\sum_{i'\in \Pi i} e_{i'}\right),$$ and by \eqref{FoldEqn} this is the same for each $j' \in \Pi j$.  Hence, $\omega_1\circ \iota$ has image in $\kappa(\?{M}_\R)$.  Since $\iota^*\circ \kappa = \Id_{\?{M}_{\bb{R}}}$ by \eqref{eq:iota-rho-Id}, we have $\kappa \circ \iota^*\circ \kappa = \kappa$, i.e., $\kappa\circ \iota^*$ restricts to the identity on $\kappa(\?{M}_\R)$.  Hence 
		\begin{align}\label{rho-iota-omega-iota}
			\kappa(\iota^*(\omega_1(\iota(e_{\Pi i})))) = \omega_1(\iota(e_{\Pi i})),    
		\end{align}
		so by \eqref{i-omega-i} we have
		\begin{align}\label{rho-omega-bar}
			|\Pi i|\kappa(\?{\omega}_1(e_{\Pi i})) = \omega_1(\iota(e_{\Pi i})).
		\end{align}

		By \eqref{Lambda-gen}, the compatible form $\Lambda$ on $M$ satisfies $$\Lambda_2(\omega_1(e_i))=d_i'e_i$$
		for $i\in I\setminus F$, $d'_i=\alpha d_i$ for some fixed $\alpha\in \bb{Q}_{>0}$.  We claim then that $\?{\Lambda}:=\kappa^*\Lambda$ is a compatible form for $\?{\sd}$ with $\?{d}'_i = \alpha \?{d}_i = \frac{d'_i}{|\Pi i|}$. Indeed, for any $\Pi i \in \?{I}\setminus \?{F}$, $\Pi j\in \?{I}$, we have
		\begin{align*}
			(\kappa^*\Lambda)(e^*_{\Pi j},\?{\omega}_1(e_{\Pi i})) &= \Lambda(\kappa(e_{\Pi j}^*),\kappa(\?{\omega}_1(e_{\Pi i}))) \\
			&= \frac{1}{|\Pi i|}\Lambda(\kappa(e_{\Pi j}^*),\omega_1(\iota(e_{\Pi i})))  &\text{(by \eqref{rho-omega-bar})}\\
			&= \frac{d'_i}{|\Pi i|}\langle \iota(e_{\Pi i}), \kappa(e_{\Pi j}^*)\rangle &\text{(by \eqref{Lambda-gen})} \\
			&= \?{d}'_i\langle e_{\Pi i},e_{\Pi j}^*\rangle &\text{(by \eqref{iota-rho})}
		\end{align*}
		as desired.  In particular, the compatibility implies that the seed $\?{\sd}$ satisfies the Injectivity Assumption.

		\subsection{Restriction of scattering diagrams}\label{sec:restriction_scattering_diagram}

		Now, recall from \eqref{DIn-gen-class} that the initial scattering diagram for $\sd$ is
		\begin{align}\label{D-nobar-in}
			\f{D}^{\sd}_{\In} = \{(e_i^{\perp},\Psi(z^{\omega_1(e_i)})^{1/d'_i})|i\in I\setminus F\}.
		\end{align}
		Similarly the initial scattering diagram for $\?{\sd}$ is
		\begin{align}\label{D-bar-in}
			\f{D}^{\?{\sd}}_{\In} &= \{(e_{\Pi i}^{\perp}, \Psi(z^{\?{\omega}_1(e_{\Pi i})})^{1/\?{d}'_i})|\Pi i\in \?{I}\setminus \?{F}\} \nonumber\\
			&= \{(\kappa^{-1} (e_i^{\perp}),\Psi(z^{\iota^*(\omega_1(e_{ i}))})^{1/\?{d}'_i})|\Pi i\in \?{I}\setminus \?{F}\}  \qquad \mbox{~(by \eqref{iota-omega} and \eqref{rho-perp})} \nonumber\\
			&\equiv \{(\kappa^{-1} (e_i^{\perp}),\Psi(z^{\iota^*(\omega_1(e_{ i}))})^{1/d'_i})|i\in I\setminus F\}
		\end{align}
		where the ``$\equiv$'' in the last line denotes equivalence of scattering diagrams --- here, for each $\Pi i\in \?{I}\setminus \?{F}$, we factor the corresponding wall on the left-hand side of the equivalence into $|\Pi i|$-many walls on the right-hand side which are identical to each other.  
		
		As usual, one may consider the corresponding consistent scattering diagrams $\f{D}^{\sd}$ and $\f{D}^{\?{\sd}}$, respectively.  We note that the walls of $\f{D}^{\?{\sd}}_{\In}$ may be obtained from those of $\f{D}^{\sd}_{\In}$ by applying $\kappa^{-1}$ to the supports and $\iota^*$ to the attached group elements.  Our goal now is to extend this operation on $\f{D}^{\sd}$ to construct a scattering diagram $\?{\f{D}^{\sd}}$ which is closely related to (and often equivalent to) $\f{D}^{\?{\sd}}$.

		\subsubsection{The support of $\?{\f{D}^{\sd}}$}\label{subsub:support}

		Viewed as a cone complex, the cones of the finite scattering diagram $\?{\f{D}^{\sd}_k}$ are defined to be the non-empty cones of the form $\kappa^{-1}(\sigma)$ for $\sigma$ a cone in $\f{D}^{\sd}_k$.  The codimension-one (in $\?{M}_{\bb{R}}$) such cones will be the supports of the walls of $\?{\f{D}^{\sd}_k}$.  Note that the codimension of a cone might change when applying $\kappa^{-1}$.
		
		On the other hand, walls of $\f{D}^{\sd}_k$ are always contained in $v^{\Lambda\perp}$ for some $v\in M^+$. By Lemma \ref{lem:compatible-Lambda-perp} below, we have $\kappa^{-1}(v^{\Lambda \perp})=(\iota^* v)^{\?{\Lambda}\bot}$ whenever $v\in M^+$. Notice that $\iota^*(M^+)=\?{M}^+$ by \eqref{iota-omega}, so in particular, $(\iota^* v)^{\?{\Lambda}\bot}$ has codimension-one for $v\in M^+$. 
		Thus, $\kappa^{-1}$ of any wall or higher-codimension cone in $\f{D}^{\sd}_k$ will have codimension at least one in $\?{M}_{\bb{R}}$.

			\begin{lem}\label{lem:compatible-Lambda-perp}
				For any $\?{m}\in \?{M}$ and $v\in M^+$, we have $\Lambda(\kappa(\?{m}),v)=\?{\Lambda}(\?{m},\iota^*v)$.
			\end{lem}
			
		\begin{proof}
		It suffices to check the equality for $\?{m}=e^*_{\Pi j}$ and $v=\omega_1(e_i)$. We have
		\begin{align*}
		\Lambda(\kappa(e^*_{\Pi j}),\omega_1(e_i))&=\frac{1}{|\Pi j|}\Lambda\left(\sum_{j'\in \Pi j}e^*_{j'}, \omega_1(e_i)\right)\\
		&=\frac{1}{|\Pi j|}d_j' \delta_{\Pi j,\Pi i}.
		\end{align*}
		By \eqref{iota-omega}, we have
		\begin{align*}
		\?{\Lambda}(e^*_{\Pi j},\iota^* \omega_1(e_i))=	\?{\Lambda}(e^*_{\Pi j},\?{\omega}_1 (e_{\Pi i}))=\?{d}_j'\delta_{\Pi j,\Pi i}.
		\end{align*}
		The claim follows from $\?{d}_j'=\frac{d_j}{|\Pi j|}$.
		\end{proof}

		\subsubsection{The invariant Lie sub algebra}\label{subsub:inv-Lie}
		
		It is clear that $\iota^*$ gives a morphism $\f{g}_{\sd}\rar \f{g}_{\?{\sd}}$ as $\kk$-modules (via action on the exponents).  Furthermore, it is evident from \eqref{iota-omega} that $\iota^*$ maps $\f{g}_{\sd}^{\geq k}$ to $\f{g}_{\?{\sd}}^{\geq k}$, hence gives a map $(\f{g}_{\sd})_k\rar (\f{g}_{\?{\sd}})_k$ on the level of modules (notation as in \eqref{eq:gk}).  However, it is not generally true that these maps of modules respect the Lie brackets as required to give a map of Lie algebras.
		
		To get a map of Lie algebras, first note (as in Example \ref{eg:group-cover}) that we may view the partition $\Pi$ as the set of orbits in $I$ under the action of a finite group which we also denote $\Pi$ (abusing notation).   Let $\f{g}_{\sd}^{\Pi}$ denote the sub Lie algebra of $\f{g}_{\sd}$ consisting of the $\Pi$-invariant elements (under the induced action of $\Pi$ on $\f{g}_{\sd}$).  
		
		\begin{lem}\label{lem:iota-g-inv}
			The restriction $\iota^*:\f{g}_{\sd}^{\Pi}\rar \f{g}_{\?{\sd}}$ of $\iota^*$ respects the Lie brackets and thus gives a well-defined morphism of Lie algebras.
		\end{lem}
		
		Note then that $\iota^*$ also induces Lie algebra morphisms $\iota^*:(\f{g}^{\Pi}_{\sd})_k\rar (\f{g}_{\?{\sd}})_k$ and $\iota^*:\hat{\f{g}}_{\sd}^{\Pi}\rar \hat{\f{g}}_{\?{\sd}}$.
		
		\begin{proof}
			Note that the $\Pi$-invariant elements of $\kk[M]\supset \f{g}_{\sd}$ are spanned by those of the form $\sum_{u'\in \Pi u} z^{u'}$ for $u\in M$, $\Pi u$ denoting the orbit of $u$ under the action of $\Pi$ on $M$.  Recall that the Lie brackets on $\f{g}_{\sd}$ is the Poisson bracket $\{z^u,z^v\}=\Lambda(u,v)z^{u+v}$, and similarly on $\f{g}_{\?{\sd}}$ using the compatible form $\kappa^* \Lambda$.  Note that $\iota^*(u')=\iota^*(u)=:\?{u}$ for each $u'\in \Pi u$.  We have
			\begin{align*}
				\iota^*\left(\left\{\sum_{u'\in \Pi u} z^{u'},\sum_{v'\in \Pi v} z^{v'}\right\}\right) &= \iota^*\left(\sum_{u'\in \Pi u, v'\in \Pi v} \Lambda\left(u',v'\right)z^{u'+v'}\right)\\
				&= \sum_{u'\in \Pi u, v'\in \Pi v} \Lambda(u',v')z^{\?{u}+\?{v}} \\
				&= \Lambda\left(\sum_{u'\in \Pi u} u',\sum_{v'\in \Pi v} v'\right)z^{\?{u}+\?{v}}
			\end{align*}
			and
			\begin{align*}
				\left\{\iota^*\left(\sum_{u'\in \Pi u} z^{u'}\right),\iota^*\left(\sum_{v'\in \Pi v} z^{v'}\right)\right\} &= \left\{\sum_{u'\in \Pi u} z^{\iota^*(u')},\sum_{v'\in \Pi v} z^{\iota^*(v')}\right\} \\
				&= \sum_{u'\in \Pi u, v'\in \Pi v} \kappa^*\Lambda\left(\iota^*(u'),\iota^*(v')\right)z^{\?{u}+\?{v}}\\
				&=\Lambda\left(\kappa\circ\iota^*\left(\sum_{u'\in \Pi u} u'\right), \kappa\circ\iota^*\left(\sum_{v'\in \Pi v} v'\right)\right) z^{\?{u}+\?{v}}
			\end{align*}
			So it suffices to check that
			\begin{align}\label{eq:wts}
				\Lambda\left(\sum_{u'\in \Pi u} u',\sum_{v'\in \Pi v} v'\right) =  \Lambda\left(\kappa\circ\iota^*\left(\sum_{u'\in \Pi u} u'\right), \kappa\circ\iota^*\left(\sum_{v'\in \Pi v} v'\right)\right)
			\end{align}
			Since $\sum_{u'\in \Pi u} u'$ is $\Pi$-invariant, it must lie in the $\Pi$-invariant part of $M$, i.e., in $\kappa(\?{M}_{\bb{Q}})\cap M$, so it equals $\kappa(u_0)$ for some $u_0\in \?{M}_{\bb{Q}}$.  Similarly, $\sum_{v'\in \Pi v} v' = \kappa(v_0)$ for some $v_0\in \?{M}_{\bb{Q}}$.  We now see that both sides of \eqref{eq:wts} are equal to
			\begin{align*}
				\Lambda(\kappa(u_0),\kappa(v_0)) 
			\end{align*}
			where for the right-hand side we use \eqref{eq:iota-rho-Id}.  This proves the claim.
		\end{proof}

		Recall that $\f{g}^{\Pi}_{\sd}$ and $\f{g}_{\?{\sd}}$ act on $\kk\llb M\rrb$ and $\kk\llb \?{M}\rrb$ respectively as in \eqref{eq:classical_Lie_mod}. By Lemma \ref{lem:compatible-Lambda-perp}, these actions are compatible with the maps $\iota^*$, $\kappa$; that is, for $g\in \f{g}_{\sd}^{\Pi}$ and $\?{m}\in \?{M}\cap \kappa^{-1}(M)$, we have $\kappa(\iota^*(g).z^{\?{m}}) = g.\kappa(z^{\?{m}})$.

		\subsubsection{Scattering functions for $\?{\f{D}^{\sd}}$.}
		
		Let $\?{\f{d}}=\kappa^{-1}(\f{d})$ be the support of a wall in $\?{\f{D}^{\sd}_k}$ as in \S \ref{subsub:support}.  We define the scattering function $f_{\?{\f{d}}}\in (\f{g}_{\?{\sd}})_k$ of the the wall $(\?{\f{d}},f_{\?{\f{d}}})\in \?{\f{D}^{\sd}_k}$ as follows: 
		
		We identify $\?{M}_{\bb{R}}$ with its image $\kappa(\?{M}_{\bb{R}})\subset M_{\bb{R}}$.   Let $\gamma$ be a path in $\?{M}_{\bb{R}}$ crossing $\?{\f{d}}$ transversely at time $\tau$ and intersecting no other codimension-one cells of $\?{\f{D}^{\sd}_k}$. Then we define $f_{\?{\f{d}}}$ so that the path-ordered product $\theta_{\gamma,\?{\f{D}^{\sd}}}$  agrees with $\iota^*(\theta_{\wt{\gamma}})$, where $\wt{\gamma}$ is a small generic perturbation of $\kappa(\gamma)$ in $M_{\bb{R}}$ with the same endpoints as $\gamma$, but which crosses the walls of $\f{D}^{\sd}_k$ transversely while avoiding joints.  That is,
        \begin{align}\label{fprod}
        f_{\?{\f{d}}} = \iota^*(\theta_{\wt{\gamma}})^{\sign \Lambda(-\gamma'(\tau),v_{\?{\f{d}}})}
        \end{align}
        where $v_{\?{\f{d}}}$ is the primitive element of $\?{M}^+$ with $\?{\f{d}}\subset v_{\?{\f{d}}}^{\?{\Lambda}\perp}$. We note that for $\sigma=\bigcap_i \f{d}_i$ the smallest cone of $\f{D}^{\sd}_k$ containing $\?{\f{d}}$, $f_{\?{\f{d}}}$ is $\iota^*(f_{\sigma})$ where $f_{\sigma}$ is the scattering function attached to the cone $\sigma$ in the perspective on consistent scattering diagrams considered in \cite[\S 2.3]{mou2019scattering}.

		To see that this is well-defined, let $\{(\f{d}_i,f_{\f{d}_i})\}_i$ be the walls crossed by $\wt{\gamma}$. Note that for each $i$, we have $\iota^*(v_{\f{d_i}})\in \?{M}^+:=\?{\omega}_1(\?{N}^+)$ because $\iota^*(M^+)=\?{M}^+$ as an easy consequence of \eqref{iota-omega}.    Now notice that $\kappa(\?{\f{d}})\subset \f{d}_i\subset v_{\f{d}_i}^{\Lambda \bot}$, so by Lemma \ref{lem:compatible-Lambda-perp}, we have  $\?{\f{d}}\subset (\iota^* v_{\f{d}_i})^{\?{\Lambda} \bot}$. Hence, all the vectors $\iota^*(v_{\f{d}_i})$ are parallel to the element $v_{\?{\f{d}}}\in \?{M}^+$. 
		Since multiplication in $G_k^{\sd}$ respects the $M$-grading, and since $\iota^*$ maps degree-$m$ elements to degree-$\iota^*(m)$ elements, it follows that $f_{\?{\f{d}}}$ lies in $G_{v_{\?{\f{d}}}}^{\parallel}$ for $v_{\?{\f{d}}}^{\?{\Lambda}\perp}\supset \?{\f{d}}$, as desired.

        From now on, we assume that the action of the finite group $\Pi$ respects the bilinear form $\omega|_{N_{\uf}}$, i.e., $\omega(e_i,e_j)=\omega(ge_i,ge_j)$ for all $g\in \Pi$ and $i,j\in I\setminus F$ (cf. Example \ref{eg:group-cover}). In this case, we say $\sd$ has $\Pi$-symmetry. This assumption will be satisfied in \S \ref{sec:closed_surface} where $\Pi$ consists of deck transformations.

		\begin{lem}\label{lem:theta-Pi-inv}
			Assume that $\sd$ has $\Pi$-symmetry.  Let $\gamma$ be any smooth path in $M_{\bb{R}}$ with endpoints in $\kappa(\?{M}_{\bb{R}})$ which avoids the joints of $\f{D}^{\sd}_k$. Then $\theta_{\gamma,\f{D}^{\sd}}\in (\f{g}_{\sd}^{\Pi})_k$.
		\end{lem}
		\begin{proof}
			Consider the natural induced actions of $\Pi$ on paths in $M_{\bb{R}}$, on scattering diagrams over $\f{g}^{\sd}_k$ in $M_{\bb{R}}$, and on the unipotent group $G^{\sd}_k$.  Note (using the $\Pi$-symmetry assumption) that for any $g\in \Pi$ and any path $\gamma_0$, we have
			\begin{align}\label{eq:gtheta1}
				\theta_{g\cdot \gamma_0,g \cdot\f{D}^{\sd}_k}=g\cdot \theta_{\gamma_0,\f{D}^{\sd}_k}.
			\end{align}
			
			Now take $\gamma_0:=g^{-1}\cdot \gamma$.  Then \eqref{eq:gtheta1} becomes 
			\begin{align}\label{eq:gtheta2}
				\theta_{\gamma,g\cdot \f{D}^{\sd}_k}=g\cdot \theta_{\gamma_0.\f{D}^{\sd}_k}.
			\end{align}
			Since $\f{D}^{\sd}_k$ is consistent over $\f{g}^{\sd}_k$, and since the endpoints of $\gamma$ and $\gamma_0$ are the same, we can replace $\gamma_0$ in \eqref{eq:gtheta2} with $\gamma$.  Furthermore, we note that the scattering diagram $\f{D}^{\sd}_k$ is invariant under the $\Pi$-action (using the $\Pi$-symmetry assumption again).  Thus, \eqref{eq:gtheta2} becomes  
			\begin{align*}
				\theta_{\gamma,\f{D}^{\sd}_k}=g\cdot \theta_{\gamma.\f{D}^{\sd}_k}.
			\end{align*}
			Since this holds for all $g\in \Pi$, the claim follows.
		\end{proof}

		\begin{thm}\label{thm:folding-D}
			Assume that $\sd$ has $\Pi$-symmetry. Then there is a well-defined, unique-up-to-equivalence scattering diagram $\?{\f{D}^{\sd}}$ in $\?{M}_{\bb{R}}$ such that the walls of the finite sub scattering diagram $\?{\f{D}^{\sd}_k}$ are as described above, up to equivalence.  Furthermore, $\?{\f{D}^{\sd}}$ is consistent. Up to equivalence preserving the positivity of the walls, $\f{D}^{\?{\sd}}\subset \?{\f{D}^{\sd}}$.
		\end{thm}
		\begin{proof}
			We showed the well-definedness above. The uniqueness statement is automatic since every wall $\?{\f{d}}$ lies in $\?{\f{D}^{\sd}_{k}}$ for all $k$ greater than or equal to some $k_{\?{\f{d}}}$.
			
			For consistency, it suffices to check consistency of $\?{\f{D}^{\sd}_k}$ for each $k$.  By construction, for any generic closed loop $\gamma \subset \?{M}_{\bb{R}}$, the path-ordered $\theta^{\?{\f{D}^{\sd}_k}}_{\gamma}$ is equal to $\iota^*(\theta^{{\f{D}^{\sd}_k}}_{\gamma'})$ for some perturbation $\gamma'$ of $\kappa(\gamma)$ in $M_{\bb{R}}$, and the consistency of $\f{D}^{\sd}$ ensures that $\theta^{{\f{D}^{\sd}_k}}_{\gamma'}=1$.  Here we use Lemmas \ref{lem:iota-g-inv} and \ref{lem:theta-Pi-inv} to ensure that the action of $\iota^*$ on the path-ordered products is via homomorphism.
			
			Finally, we see from  \eqref{D-nobar-in} and \eqref{D-bar-in} that, up to equivalence, the incoming walls of $\?{\f{D}^{\sd}}$ include all the incoming walls of $\f{D}^{\?{\sd}}$.  The claim that $\f{D}^{\?{\sd}}\subset \?{\f{D}^{\sd}}$ follows.
		\end{proof}

		The following general result will not be used in this paper.
		\begin{cor}\label{cor:fold-Theta-mid}
			Assume that $\sd$ has $\Pi$-symmetry.  Then $\kappa^{-1}(\Theta^{\midd}_{\bb{R}}(\sd))\subset \Theta^{\midd}_{\bb{R}}(\?{\sd})$.
		\end{cor}
		\begin{proof}
			Let $\sQ\in C_{\?{\sd}}^+$ and let $\sQ'\in C_{\sd}^+$ be a generic point (with respect to $\f{D}^{\sd}$) very near $\kappa(\sQ)$.  For any $m\in \kappa^{-1}(M)$, any broken line $\Gamma$ in $\kappa(\?{M}_{\bb{R}})$ with ends $(m,\sQ)$ and final monomial $c_{\Gamma}z^{m_{\Gamma}}$ with respect to $\?{\f{D}^{\sd}}$ can be deformed in $\?{M}_{\bb{R}}$ to a broken line $\Gamma'$ with ends $(\kappa(m),\sQ')$ and final monomial $c_{\Gamma}z^{\kappa(m_{\Gamma})}$.  If $\kappa(m)\in \Theta^{\midd}(\sd)$, then $\vartheta^{\f{D}^{\sd}}_{\kappa(m),\sQ'}$ is a finite Laurent polynomial (writing the relevant scattering diagram in the exponent for clarity), so $\vartheta^{\?{\f{D}^{\sd}}}_{m,\sQ}$ must be as well (because we just saw that any broken line contributing to  $\vartheta^{\?{\f{D}^{\sd}}}_{m,\sQ}$ also yields a broken line contributing to $\vartheta^{\f{D}^{\sd}}_{\kappa(m),\sQ'}$).
			
			Since $\f{D}^{\?{\sd}}\subset \?{\f{D}^{\sd}}$ as positive scattering diagrams by Theorem \ref{thm:folding-D}, we similarly have that every broken line contributing to $\vartheta^{\f{D}^{\?{\sd}}}_{m,\sQ}$ must also contribute to $\vartheta^{\?{\f{D}^{\sd}}}_{m,\sQ}$, so finiteness of $\vartheta^{\?{\f{D}^{\sd}}}_{m,\sQ}$ implies finiteness of $\vartheta^{\f{D}^{\?{\sd}}}_{m,\sQ}$.  Putting all this together, we see that $\kappa(m)\in \Theta^{\midd}(\sd)$ implies $m\in \Theta^{\midd}(\?{\sd})$.  The claim follows.
		\end{proof}

		\begin{lem}\label{lem:iota-star-theta}
			Suppose $\?{m}=\iota^*(m)$.  Fix $\sQ,\sQ'\in M_{\bb{R}}$ generic and sharing chambers of $\f{D}^{\sd}_k$ with $\kappa(\?{\sQ})$, $\kappa(\?{\sQ}')$, respectively, for some generic $\?{\sQ},\?{\sQ}'\in \?{M}_{\bb{R}}$.  If $\vartheta^{\?{\f{D}^{\sd}}}_{\?{m},\?{\sQ}}\equiv \iota^*(\vartheta^{\f{D}^{\sd}}_{m,\sQ})$ (modulo $m+kM^+$), then $\vartheta^{\?{\f{D}^{\sd}}}_{\?{m},\?{\sQ}'}\equiv \iota^*(\vartheta^{\f{D}^{\sd}}_{m,\sQ'})$ (modulo $m+kM^+$).
		\end{lem}
		\begin{proof}
			This follows from Lemma \ref{CPS} and the fact that, by construction, $\theta_{\?{\gamma}}^{\?{\f{D}}^{\sd}_k}=\iota^*(\theta_{\gamma}^{\f{D}^{\sd}_k})$ for any paths $\?{\gamma},\gamma$ from $\?{\sQ}$ to $\?{\sQ}'$ and $\sQ$ to $\sQ'$, respectively.
		\end{proof}
		
		\subsection{Examples of covering}
		
		\begin{example}\label{eg:folding_cyclic_A3}
			Consider the seed $\sd$ given such that $I=\{1,2,3\}$, $d_i=1$ for all $i$, and $\omega^\circ$ is given by
			\[
			\omega^\circ(e_i,e_j)=
			\begin{pmatrix}
				0&1&-1\\
				-1&0&1\\
				1&-1&0
			\end{pmatrix} 
			\]
			Let $\sd^{\prin}$ denote the corresponding principal coefficient seed with framing vertices $j'$ for each $j\in I$ such that $\omega^{\circ}(i,j')=\delta_{ij}$. Choose the compatible form $\Lambda$ given by
			\[
			\Lambda=
			\begin{pmatrix}
				0& \id_3\\
				-\id_3&-\omega^\circ
			\end{pmatrix}
			\]
			so $\Lambda (\omega^{\prin})^T=\Id_6$.

			Choose $\Pi$ such that $\Pi i=\{1,2,3\}$ and $\Pi i'=\{1',2',3'\}$. Then the corresponding seed $\?{\sd^{\prin}}$ has an unfrozen vertex $\Pi 1$ and a frozen vertex $\Pi 1'$. Its skew-symmetric bilinear form in the basis $e^*_{\Pi 1}$ and $e^*_{\Pi 1'}$ is given by 
			\[
			\?{\omega}^\circ=\begin{pmatrix}
				0&3\\
				-3&0
			\end{pmatrix}
			\]
			and we have $\?{d}_{\Pi 1}=\?{d}_{\Pi 1'}=\frac{1}{3}$.  It is compatible with the matrix $\kappa^*\Lambda$ given by
			\[
			\kappa^*\Lambda=
			\begin{pmatrix}
				0&\frac{1}{3}\\
				-\frac{1}{3}&0
			\end{pmatrix}. 
			\]
			
			Identify $i$ with $i+3$ for simplicity (i.e., assume $I$ is cyclically ordered).  The walls $\f{d}$ for $\f{D}^{\sd^{\prin}}$ are the three initial walls $(e_i^\perp,\Psi(z^{v_i}))$ where $v_i\coloneqq \omega^{\prin}_1(e_i)=e_{i+1}^*-e_{i-1}^*+e_{i'}^*$, and the three non-initial walls $(\f{d}_i,\Psi(z^{u_i}))$ for $\f{d}_i=(e_{i+1}^{\perp}\cap e_{i-1}^{\perp})+\R_{\geq 0}(- u_i)$ and $$u_i:=v_{i+1}+v_{i-1}=e_{i-1}^*-e_{i+1}^*+e_{(i+1)'}^*+e_{(i-1)'}^*.$$   
			For $m=-\sum_{i=1}^{3} e_i^*$, one can compute that
			\begin{align*}
					\vartheta_{m}=z^{m}\prod_{i=1}^{3}(1+z^{v_i}+z^{v_i+v_{i-1}}).
			\end{align*}

			The space $\?{M}_{\R}$ is spanned by $e^*_{\Pi 1}=\kappa^{-1}(\frac{1}{3}\sum _i e^*_i)$ and $e^*_{\Pi 1'}=\kappa^{-1}(\frac{1}{3}\sum_i e^*_{i'})$. The only wall of $\f{D}_{\?{\sd^{\prin}}}$ (up to equivalence) is $(\?{\f{d}},f_{\?{\f{d}}})$ where $\?{\f{d}}=e_{\Pi 1}^{\perp} = \bb{R}e_{\Pi 1'}^*$, and $f_{\f{d}}=\Psi(z^{v_{\?{\f{d}}}})^3$ for $v_{\?{\f{d}}}=e_{\Pi 1'}^*$. So for $\?{m}=-3e_{\Pi 1}^*$, $\f{D}_{\?{{\sd}^{\prin}}}$ has the theta function $\vartheta_{\?{m}}=z^{\?{m}}(1+z^{v_{\?{\f{d}}}})^3$.
			On the other hand, the unique wall $(\?{\f{d}}',f_{\?{\f{d}'}})$ of  $\?{\f{D}^{\sd^{\prin}}}$ has $\?{\f{d}}'=\?{\f{d}}$ but a different scattering function $f_{\?{\f{d}}'}$, see \eqref{fprod}. In particular, for $\?{m}=-3 e_{\Pi 1}^*$, $\?{\f{D}^{\sd^{\prin}}}$ has the  theta function
			\begin{align*}
				\?{\vartheta}_{\?{m}}=z^{\?{m}}(1+z^{v_{\?{\f{d}}}}+z^{2v_{\?{\f{d}}}})^3.
			\end{align*}

		\end{example}

        \begin{rem}
        In Example \ref{eg:folding_cyclic_A3}, the group $\Pi$  could be constructed from the cyclic permutation $i\mapsto i+1$ and $i'\mapsto (i+1)'$. It would be interesting to know if one can generalize Example \ref{eg:folding_cyclic_A3}, and Theorem \ref{thm:folding-D} more broadly, to drop the $\Pi$-symmetry assumption (which was only needed for the proof of Lemma \ref{lem:theta-Pi-inv}).
        \end{rem}

		\subsection{Folding}\label{sec:fold}
		
		Suppose that $\sd$ is a covering of $\?{\sd}$ (not necessarily satisfying the $\Pi$-symmetry assumption).  We would like for a mutation $\mu_{\Pi j}$ of $\?{\sd}$ to naturally correspond to a mutation $\hat{\mu}_{\Pi j}:=\prod_{j'\in \Pi j} \mu_{j'}$ of $\sd$, but in general, either of the following might fail to be true:
		\begin{enumerate}
			\item the mutations $\mu_{j'}$, $j'\in \Pi j$, should commute with each other; i.e., we require $\omega(e_{j_1},e_{j_2})=0$ for all $j_1,j_2\in \Pi j$.  This allows for the definition of composite mutations $\hat{\mu}_{\Pi j}:=\prod_{j'\in \Pi j} \mu_{j'}$. 
			\item The composite mutation $\hat{\mu}_{\Pi j}(\sd)$ should be a covering of $\mu_{\Pi j}(\?{\sd})$.
		\end{enumerate}
		We say the covering $\sd$ is $\Pi$-mutable if conditions (1) and (2) above hold for all $\Pi j\in \?{I}$.  
		Given a sequence $\Pi \jj=(\Pi j_1, \ldots, \Pi j_k)$ of elements of $\?{I}$, define $\hat{\mu}_{\Pi \jj}=\hat{\mu}_{\Pi j_k} \cdots  \hat{\mu}_{\Pi j_2} \hat{\mu}_{\Pi j_1} $.  If $\sd$ is totally $\Pi$-mutable---i.e., for any sequence $\Pi \jj$, $\hat{\mu}_{\Pi \jj}(\sd)$ is a $\Pi$-mutable covering of $\mu_{\Pi \jj}(\sd)$---then we say that the covering $\sd$ is an \textbf{unfolding} of $\?{\sd}$, or that $\?{\sd}$ is a \textbf{folding} of $\sd$.

		As in \cite[\S 4]{felikson2012cluster} and \cite{huang2018unfolding}, Conditions (1) and (2) above hold whenever the following is true:
		\begin{align}\label{eq:sign-condition}
			\text{For all $i\in I$, $j\in I\setminus F$, and $j'\in \Pi j$, if $\omega(e_i,e_j)>0$, then $\omega(e_i,e_{j'})\geq 0$.}
		\end{align}
		The condition (1) above---i.e., that $\omega(e_i,e_{i'})=0$ whenever $i'\in \Pi i \in \?{I}\setminus \?{F}$---follows easily from \eqref{eq:sign-condition} combined with \eqref{FoldEqn}.  Using \eqref{eprime}, additional computations reveal that \eqref{eq:sign-condition} implies Condition (2) above as well (cf. \cite[Lem. 2.5]{huang2018unfolding}; note that our conditions $\omega(e_i,e_{i'})=0$ when $i'\in \Pi i$ and our \eqref{eq:sign-condition} correspond to the conditions of having no $\Gamma$-loops nor $\Gamma$-$2$-cycles in loc. cit., respectively). So the covering $\sd$ will be an unfolding whenever \eqref{eq:sign-condition} holds for all seeds $\hat{\mu}_{\Pi \jj}(\sd)$ as above (defined recursively after checking \eqref{eq:sign-condition} for the shorter sequences of composite mutations).
		
		\begin{prop}\label{prop:fold-upper}
			Suppose $\?{\sd}$ is a folding of $\sd$.  Then the projection $\iota^*$ maps elements of $\s{A}_{\sd}^{\up}$ to $\s{A}_{\?{\sd}}^{\up}$.
		\end{prop}
		\begin{proof}
			By construction, every cluster $\s{A}^{\mu_{\Pi \jj}(\?{\sd})}$ of $\s{A}_{\?{\sd}}^{\up}$ is $\iota^*$  of a corresponding cluster $\s{A}^{\hat{\mu}_{\Pi \jj}(\sd)}$ of $\s{A}_{\sd}^{\up}$.  Furthermore, one checks that $\iota^*\circ \hat{\mu}_{\Pi j}^{\s{A}} = \mu_{\Pi j}^{\s{A}}\circ \iota^*$---indeed, this is clear from $$\iota^*\circ \hat{\mu}_{\Pi j}(z^p)=\iota^*\left(\prod_{j'\in \Pi j} \Ad^{-1}_{ \Psi(z^{\omega_1(e_i)})^{1/d_i'}}(z^p)\right)$$
			(cf. \eqref{eq:mut-Ad}) and
			$$\mu_{\Pi j}\circ \iota^*(z^p)=\prod_{j'\in \Pi j} \Ad^{-1}_{ \Psi(z^{\iota^*(\omega_1(e_j))})^{1/d'_j}}(z^{\iota^*(p)})
			$$
			(cf. the factorization of $\f{D}_{\In}^{\?{\sd}}$ in \eqref{D-bar-in}).  Thus, $f$ being a universal Laurent polynomial for the clusters of $\sd$ implies $\iota^*(f)$ is a universal Laurent polynomial for the clusters of $\?{\sd}$. 
		\end{proof}

		We assume from now on that $\sd$ has $\Pi$-symmetry.
		
		\begin{lem}\label{lem:restriction_cluster_complex}
		Suppose $\?{\sd}$ is a folding of $\sd$, and assume that $\sd$ has $\Pi$-symmetry. Then $\f{D}^{\?{\sd}}$ and $\?{\f{D}^{\sd}}$ agree up to equivalence on the cluster complex $\s{C}_{\?{\sd}}$ and on the opposite complex $-\s{C}_{\?{\sd}}$.
		\end{lem}
		\begin{proof}
			Recall that the complexes $\pm\s{C}_{\?{\sd}}$ are generated by the chambers by $C^{\pm}_{\?{\sd}}$ by mutations.  It suffices to show that, for any non-initial wall $\f{d}\subset v^{\Lambda\bot}$ of $\f{D}^{\sd}$, the associated cone $\kappa^{-1}(\f{d})$ does not intersect the chambers $C^{\pm}_{\?{\sd}}$ at codimension $1$. By Lemma \ref{lem:compatible-Lambda-perp}, $\kappa^{-1}(\f{d})$ is contained in $(\iota^* v)^{\?{\Lambda}\bot}$. Notice that the hyperplane $(\iota^* v)^{\?{\Lambda}\bot}$ intersects the chamber $C^{\pm}_{\?{\sd}}$ at codimension $1$ if and only if $\iota^* v$ is parallel to $\?{\omega}_1(e_{\Pi k})$ for some $k\in I\setminus F$. So there must exists some $\Pi k$ such that $v$ is a linear combination of $e_{k'}$ for $k'\in \Pi k$. But such a non-initial wall $\f{d}$ does not exist since $\omega(e_k,e_{k'})=0$ for all $k'\in \Pi k$; see the order-by-order inductive construction of $\f{D}^{\sd}$ in \cite[\S C]{gross2018canonical}.
		\end{proof}

		\begin{cor}\label{cor:gtr}
			 Suppose $\?{\sd}$ is a folding of $\sd$, and assume that $\sd$ has $\Pi$-symmetry. Then $\?{\f{D}^{\sd}}=\f{D}^{\?{\sd}}$ (up to equivalence) when $\?{\sd}$ admits a green-to-red sequence (meaning that $C_{\?{\sd}}^+$ and $C_{\?{\sd}}^-$ are both contained in the cluster complex of $\?{\sd}$, cf. \cite[\S 4.4]{MuGreen}; or, equivalently, $\sd$ is injective-reachable \cite{qin2017triangular}).
		\end{cor}
		\begin{proof}
			A scattering diagram $\f{D}$ is uniquely determined by the path-ordered product $\theta^{\f{D}}_{-,+}$ from the positive chamber to the negative chamber; cf. \cite[Thm. 2.1.6]{WCS}, or see \cite[Thm. 1.17]{gross2018canonical} for this result in the present context.  When $\?{\sd}$ admits a green-to-red sequence, there exists a path from $C^+_{\?{\sd}}$ to $C^-_{\?{\sd}}$ contained entirely within the cluster complex of $\?{\sd}$.  Since $\?{\f{D}^{\sd}}$ and $\f{D}^{\?{\sd}}$ agree on the cluster complex by Lemma \ref{lem:restriction_cluster_complex}, it follows that they agree everywhere.
		\end{proof}

            Let $\?{\s{C}}^{\circ}$ denote the interior of the closure of the cluster complex $\s{C}_{\?{\sd}}$.  Similarly, let $-\?{\s{C}}^{\circ}$ denote the interior of the closure of the opposite complex $-\s{C}_{\?{\sd}}$.

            \begin{lem}\label{lem:restriction-strong}
            Let $\?{\sd}$ be a folding of $\sd$,  and assume that $\sd$ has $\Pi$-symmetry. Then $\f{D}^{\?{\sd}}$ and $\?{\f{D}^{\sd}}$ agree up to equivalence on $\?{\s{C}}^{\circ}$ and on $-\?{\s{C}}^{\circ}$.
            \end{lem}
            \begin{proof}
            It suffices to work up to arbitrary finite order $k$.  As in Theorem \ref{thm:folding-D}, we may assume by taking equivalent scattering diagrams that $\f{D}^{\?{\sd}}\subset \?{\f{D}^{\sd}}$ and that both scattering diagrams have only positive walls.  Suppose there is some non-trivial wall $\f{d}\in \?{\f{D}_k^{\sd}}\setminus \f{D}_k^{\?{\sd}}$ which intersects $\?{\s{C}}^{\circ}$.  Choose generic points $x,x'\in \s{C}_{\?{\sd}}$ on either side of $\f{d}$ so that a line segment $\gamma_0$ from $x$ to $x'$ crosses no walls of $\?{\f{D}_k^{\sd}}$ other than $\f{d}$ (and possibly walls parallel to $\f{d}$).  Then $\theta_{\gamma_0}^{\?{\f{D}_k^{\sd}}}\neq \theta_{\gamma_0}^{\f{D}_k^{\?{\sd}}}$.  On the other hand, since $\s{C}_{\?{\sd}}$ is path-connected, we can find a path $\gamma$ from $x$ to $x'$ contained entirely within $\s{C}_{\?{\sd}}$. Then $\theta_{\gamma_0}^{\?{\f{D}_k^{\sd}}}= \theta_{\gamma_0}^{\f{D}_k^{\?{\sd}}}$ by Lemma \ref{lem:restriction_cluster_complex}.  This contradicts the consistency of $\f{D}^{\?{\sd}}$ and $\?{\f{D}^{\sd}}$, so the result for  $\?{\s{C}}^{\circ}$ follows.  A similar argument yields the result for $-\?{\s{C}}^{\circ}$.
            \end{proof}

		\begin{rem}
			In general, $\?{\f{D}^{\sd}}$ and $\f{D}^{\?{\sd}}$ may differ.  We will see that this is indeed the case for the once-punctured torus.  A natural question is whether this folding construction correctly relates the associated stability scattering diagrams of \cite{bridgeland2017scattering}.  We leave this question to future work \cite{chen2023comparison}.
		\end{rem}

		\begin{rem}
			We note that $\?{\f{D}^{\sd}}$ can equivalently be constructed as the unique-up-to-equivalence consistent scattering in $\?{M}_{\bb{R}}$ with $\theta^{\?{\f{D}^{\sd}}}_{-,+}=\iota^*(\theta^{\f{D}^{\sd}}_{-,+})$.  This is the approach used in \cite[\S 2]{zhou2020cluster}.
		\end{rem}

			\begin{lem}\label{lem:prin-unfold}
				If $\sd$ is an unfolding of $\?{\sd}$, then $\sd^{\prin}$ is an unfolding of $\?{\sd}^{\prin}$.  
			\end{lem}
			\begin{proof}
				Let us denote $I^{\prin}=I_1\sqcup I_2$ with $F\subset I_1$, $F^{\prin}=F\sqcup I_2$ as in \S \ref{prinsub}.  We already noted in Example \ref{ex:Pi-prin} that $\sd^{\prin}$ is a covering of $\?{\sd}^{\prin}$.  By the assumption that $\sd$ is an unfolding of $\?{\sd}$, we see that Condition \eqref{eq:sign-condition} holds for $i\in I_1,j\in I_1\setminus F$, and all $\hat{\mu}_{\Pi \jj}(\sd^{\prin})$ (the extension to principal coefficients does not affect the restriction of $\omega$ to the original lattice $N$, even under sequences of mutations).
    
            To extend \eqref{eq:sign-condition} to $i\in I_2$, first recall that the sign coherence of $c$-vectors \cite[Cor. 5.5]{gross2018canonical} says that for $j\in I_1\setminus F$, $\omega^{\prin}(e_i,e_j)$ (hence also $-\omega^{\prin}(e_j,e_i)$) is either non-negative for all $i\in I_2$ or non-positive for all $i\in I_2$.  Then \eqref{FoldEqn} implies that $\omega^{\prin}(e_j,e_i)$ and $\omega^{\prin}(e_{j'},e_i)$ have the same sign for all $j'\in \Pi j$ and all $i\in I_2$.  In particular, for each $i\in I_2$ and $j'\in \Pi j$, we have that $\omega^{\prin}(e_i,e_j)$ and $\omega^{\prin}(e_i,e_{j'})$ have the same sign; i.e., \eqref{eq:sign-condition} holds, as desired.
			\end{proof}

  \section{Once-punctured closed surfaces}\label{sec:closed_surface}
		
		\subsection{Covering spaces and scattering diagrams}\label{sec:finite_covering}
		
		Fix $\Sigma=(\SSS,\MM)$ an arbitrary triangulable marked surface. Let $\pi:\wt{\SSS}\rightarrow \SSS$ be a finite $s:1$ oriented cover\footnote{Note that we do not consider the more general covers of $\SSS\setminus \MM$ as in \S \ref{Sk-relations} (i.e., covers ramified at punctures).  Such coverings spaces do yield seed coverings in the sense of \S \ref{sec:folding_seeds}, but if $\Sigma$ contains once-punctured digons or monogons, then the folding condition \eqref{eq:sign-condition} may fail in the covering space of the punctured digon or monogon; cf. Footnote \ref{foot:folding-by-sign}.  We note that this is not a problem for once-punctured closed surfaces since they do not contain punctured monogons or digons---i.e., covering spaces of $\SSS\setminus \MM$ (as opposed to just of $\SSS$) do yield unfoldings for once-punctured closed surfaces.} of $\SSS$ for $s\in \bb{Z}_{\geq 1}$. We consider the covering marked surface $\wt{\Sigma}=(\wt{S},\wt{M})$ with $\wt{M}:=\pi^{-1}M$.  Note that any tagged triangulation $\Delta$ of $\SSS$ lifts to a tagged triangulation $\wt{\Delta}$ of $\wt{\SSS}$ whose arcs are the components $\pi^{-1}(\gamma)$ for $\gamma\in \Delta$.    Let $\Pi$ be the group of deck transformations of the covering $\pi$.  Note that $\Pi$ acts on $\wt{\Delta}$, thus determining a partition $\Pi$ of $\wt{\Delta}$.  
		
		\begin{lem}\label{lem:cover-unfold}
			The partition $\Pi$ realizes $\sd_{\wt{\Delta}}$ as an unfolding of $\sd_{\Delta}$. 
			The map $\iota^*:\Sk^{\Box}(\wt{\Sigma})\rar \Sk^{\Box}(\Sigma)$ induced by $\iota^*$ as in \eqref{eq:iota-star} is $[C]\mapsto [\pi(C)]$ for any tagged simple multicurve $C$.  That is, $\iota^*$ coincides with the map $\pi_*$ of \eqref{eq:pi-star}.
		\end{lem}
		\begin{proof}
			Note that each arc in $\Delta$ lifts to $s$ arcs in $\wt{\Delta}$ related by the action of $\Pi$.  Furthermore, triangles in $\Delta$ lift to triangles in $\wt{\Delta}$ with the same orientations.  The conditions for being a covering now follow easily from the definition of the signed adjacency matrix $B=\omega^T$ in \S \ref{sub:ClSk} and \S \ref{S:tagged_sk}.  Flipping an arc of $\Delta$ will correspond to flipping all lifts of that arc in $\wt{\Delta}$, and since the tagged triangulation for $\Sigma$ was arbitrary, the folding conditions of \S \ref{sec:fold} are satisfied.\footnote{Alternatively, the folding condition \eqref{eq:sign-condition} is satisfied because if $\omega(e_i,e_j)>0$ and $\omega(e_i,e_{j'})<0$ for some $j'\in \Pi j$, then the two contributing triangles would necessarily map to a once-punctured digon in $\Sigma$. But since digons are simply connected, the two triangles in $\wt{\Sigma}$ should have also formed a digon, contradicting the supposed values of $\omega$.  Here we use that $\wt{\SSS}$ is a covering space of the \textit{unpunctured} surface $\SSS$.\label{foot:folding-by-sign}}
			
			The claim regarding $\iota^*$ is clear for arcs in $\wt{\Delta}$.  Since the map $[C]\mapsto [\pi(C)]$ is an algebra homomorphism and since every $[C]$ is a Laurent polynomial in the arcs of $\wt{\Delta}$, the claim for general skeins $[C]$ follows.
		\end{proof}

		Recall from \S \ref{FailInj} that all walls of principal coefficients scattering diagrams are closed under addition by $(0,N_{\bb{R}})$, so the scattering diagrams can be completely understood in terms of their projections via $\rho:(m,n)\mapsto m$ (keeping the scattering functions the same).  We work with these projections via $\rho$ throughout \S \ref{sec:closed_surface}.
		
		\subsection{Unfolding and scattering diagrams for once-punctured closed surfaces}\label{sec:unfolding-closed-surfaces}
		
		By Lemma \ref{lem:prin-unfold}, $\sd_{\wt{\Delta}}$ being an unfolding of $\sd_{\Delta}$ implies that $\sd_{\wt{\Delta}}^{\prin}$ is an unfolding of $\sd_{\Delta}^{\prin}$.  When no component of $\SSS$ is a once-punctured closed surface, $\sd_{\Delta}^{\prin}$ admits a green-to-red sequence by Proposition \ref{gdense}, and so $\f{D}^{\sd^{\prin}_{\Delta}}=\?{\f{D}^{\sd^{\prin}_{\wt{\Delta}}}}$ by Corollary \ref{cor:gtr}.  On the other hand, if $\Sigma$ is a once-punctured closed surface,  then Proposition \ref{gdense} and Lemma \ref{lem:restriction-strong} imply that $\f{D}^{\sd^{\prin}_{\Delta}}$ and $\?{\f{D}^{{\sd^{\prin}_{\wt{\Delta}}}}}$ agree except possibly on the half-space $H$ which forms the boundary of $\?{\s{C}}$ and $-\?{\s{C}}$.

		\begin{lem}\label{lem:n0}
			Let $\Sigma$ be a marked surface with tagged triangulation $\Delta$, and let $n_0=\sum_{i\in \Delta} e_{i}$ (i.e., $n_0=(1,1,\ldots,1)$).  Then $\omega_1(N)\subset n_0^{\perp}$.  
			
			If $\Sigma$ is a once-punctured closed surface, then $n_0^{\perp}$ is the hyperplane $H$ forming the boundary of $\?{\s{C}}$ and $-\?{\s{C}}$.  Thus, for any scattering diagram in $M_{\bb{R}}$ over the associated $\f{g}$ and any wall $(\f{d},f)$ with $\f{d}\subset H$, we have that $\f{g}_{\omega_1^{\prin}(n_0)}^{\parallel}\ni f$ is central in $\wh{\f{g}}$.
		\end{lem}
		Here, the notation $\f{g}_{m}^{\parallel}$ is as in the classical limit of \eqref{g-parallel}.
		\begin{proof}
			For each $i\in \Delta$, each triangle of $\Delta$ containing $i$ contributes $+1$ to $\omega(e_i,e_j)$ for one $j$ and $-1$ to $\omega(e_i,e_{j'})$ for some other $j'$ (possible $j=j'$).  Thus, $\sum_{j\in \Delta} \omega(e_i,e_j)=0$ for all $i\in \Delta$; i.e., $\omega_1(e_i)\in n_0^{\perp}$ for all $i\in I$, so the claim $\omega_1(N)\subset n_0^{\perp}$ follows.
			
			Now suppose $\Sigma$ is a once-punctured closed surface.  The hyperplane $H$ is generated by the $g$-vectors of loops, and by \eqref{eq:loop_deg_monoid}, these lie in $\omega_1(N_{\bb{Q}})$, hence in $n_0^{\perp}$.  The claim $H= n_0^{\perp}$ follows.
			
			For the centrality claim, note that $$\Lambda^{\prin}(\omega_1^{\prin}((n,0)),\omega_1^{\prin}(n_0))= \langle n_0,\omega_1^{\prin}((n,0))\rangle =\omega(n,n_0)=0.$$  The centrality now follows from the definition of the Lie bracket in \eqref{classical-bracket}.
		\end{proof}
		
		Note that $n_0\in \ker \omega_1$ implies $\omega_1^{\prin}(n_0)=(0,n_0)$.  Let us denote $m_0:=(0,n_0)\in M^{\prin}$.
		
		\begin{lem}\label{lem:H}
			If $\Sigma$ is a once-punctured closed surface, then, up to equivalence, $\?{\f{D}^{{\sd^{\prin}_{\wt{\Delta}}}}}=\f{D}^{\sd^{\prin}_{\Delta}} \sqcup \{(H,f_H)\}$ for $H$ the bounding hyperplane as above and $f_H\in \f{g}_{m_0}^{\parallel}$.
		\end{lem}
		\begin{proof}
            The only part of this claim which we did not see above is that, up to equivalence, there is only one wall in $H$, and the support is all of $H$.  The fact that $\?{\f{D}^{{\sd^{\prin}_{\wt{\Delta}}}}}$ and $\f{D}^{\sd^{\prin}_{\Delta}}$ are both consistent and $\f{g}_{m_0}^{\parallel}$ is central implies that $\?{\f{D}^{{\sd^{\prin}_{\wt{\Delta}}}}}\setminus \f{D}^{\sd^{\prin}_{\Delta}}$ must be consistent.  Since we know that any walls of $\?{\f{D}^{{\sd^{\prin}_{\wt{\Delta}}}}}\setminus \f{D}^{\sd^{\prin}_{\Delta}}$ must have support in $H$, the only possibility (up to equivalence) which is consistent is indeed where there is a single wall whose support is all of $H$.
		\end{proof}
		
		\subsection{Bracelets and theta functions for once-punctured closed surfaces}\label{sub:bracelets-1p}
		
		For $f_H$ as in Lemma \ref{lem:H} and any $m\in M$, the wall-crossing automorphism associated to crossing $H$ from the side containing $m$ to the other side will act on $z^m$ via 
		\begin{align}\label{eq:adfh}
			\Ad_{f_H}^{-\langle n_0,m\rangle}(z^m)=z^m(P_H)^{|\langle n_0,m\rangle|}
		\end{align}
		for some $P_H\in 1+z^{m_0}\bb{Z}_{\geq 0}\llb z^{m_0}\rrb$.  Fix $m\in -\s{C}\cap M$.  Let $\vartheta_{m,\sQ}$ denote the theta function constructed with respect to $\f{D}^{\sd^{\prin}_{\Delta}}$, and let $\?{\vartheta}_{m,\sQ}$ be the theta function constructed using using $\?{\f{D}^{{\sd^{\prin}_{\wt{\Delta}}}}}=\f{D}^{\sd^{\prin}_{\Delta}} \sqcup \{(H,f_H)\}$.  Then for any generic $\sQ\in \s{C}$, we have
		\begin{align}\label{eq:project_theta}
			\?{\vartheta}_{m,\sQ} = \vartheta_{m,\sQ}P_H^{-\langle n_0,m\rangle}.
		\end{align}

		On the other hand, consider $\sQ_0$ contained in a chamber of $-\s{C}$ which contains $m$, say a chamber corresponding to a triangulation $\Delta^{\diamond}$ consisting of doubly-notched arcs.  Let $\wt{m}\in (\iota^*)^{-1}(m)$ and let $\wt{\sQ}_0$ be in the chamber of $\wt{\s{C}}$ (the cluster complex for $\wt{\Sigma}$) associated to $\pi^{-1}(\Delta^{\diamond})$---i.e., the chamber containing $\kappa(\sQ_0)$ for $\kappa$ as in \S \ref{sec:folding_seeds}.  Then $\?{\vartheta}_{m,\sQ_0}=z^m$ and $\vartheta_{\wt{m},\wt{\sQ}_0}=z^{\wt{m}}$ (to see this, it suffices to work with the initial seed $\sd_\Delta$ in the coefficient-free setting and apply the tag-change automorphism $\iii_P$). So $\?{\vartheta}_{m,\sQ_0}=\iota^*(\vartheta_{\wt{m},\wt{\sQ}_0})$.  We thus know from Lemma \ref{lem:iota-star-theta} that $\?{\vartheta}_{m,\sQ}=\iota^*(\vartheta_{\wt{m},\wt{\sQ}})$.  Since $\vartheta_{\wt{m}}$ is given by the appropriate (laminated) lambda length in $\wt{\Sigma}$ for principal coefficients \cite{fomin2018cluster}, Lemma \ref{lem:cover-unfold} (the agreement of $\iota^*$ and $\pi_*$) implies that $\?{\vartheta}_{m,\sQ}$ is given by the desired (laminated) lambda length;\footnote{For computing laminated Lambda lengths in \cite{fomin2018cluster}, we need to consider opened surfaces $\Sigma'$, $\wt{\Sigma}'$ associated to $\Sigma$ and $\wt{\Sigma}$, which are constructed by replacing the punctures by holes. We then naturally extend $\pi$ to a covering map between the opened surfaces. Choose a multi-lamination $L$ on $\Sigma$ corresponding to the principal coefficients and denote $\wt{L}:=\pi^{-1}(L)$. By using the covering map $\pi$, a point in the laminated Teichmüller space $\?{\mathcal{T}}(\Sigma,L)$ is naturally sent to a point in $\?{\mathcal{T}}(\wt{\Sigma},\wt{L})$ (see \cite[Definition 10.8, Remark 14.8, Definition 15.1]{fomin2018cluster}). Fix a lift $L'$ of $L$ on $\Sigma'$ and lift $\wt{L}$ to $\pi^{-1}(L')$ on $\wt{\Sigma}'$. Then the laminated lambda length of a tagged arc $\wt{\gamma}$ on $\wt{\Sigma}$ is the same as that of $\pi(\wt{\gamma})$ on $\Sigma$ (see \cite[Definition 15.3]{fomin2018cluster}). \label{footnote:lambda_length}} i.e., by the bracelet element $\BracE{C_m}$ with $g$-vector $m$ (see \S \ref{subsub:brace-arcs}). We have thus found that
		\begin{align}\label{eq:1p_tag_arc_theta}
			\BracE{C_m} = \vartheta_{m}P_H^{-\langle n_0,m\rangle}.
		\end{align}
		
		It remains to understand this factor $P_H$.  Let us set our principal coefficients equal to $1$ (i.e., apply the map $\rho:(m,n)\mapsto m$), so $P_H$ becomes a constant $c_H\in \bb{Z}_{\geq 1}$ equal to the sum of the coefficients of $P_H$.  To compute $c_H$, fix a doubly-plain arc $\gamma$ in $\Sigma$ with corresponding doubly-nothced arc $\gamma^{\diamond}$, and let $m$ be the $g$-vector of $\gamma^{\diamond}$.  By Example \ref{ex:doubly-notched-arcloop}, $\gamma^{\diamond}*\gamma=(L_1+L_2)^2$ where $L_1$ and $L_2$ are the two loops forming the boundary for a tubular neighborhood of $\gamma$.
		
		For genus $g\geq 2$, $L_1$ and $L_2$ are distinct, and we compute
		\begin{align}\label{eq:gamma-diamond}
			\?{\vartheta}_m\vartheta_{-m}=\gamma^{\diamond}*\gamma=(L_1+L_2)^2 = 4+\BracE{2L_1}+2\BracE{L_1\sqcup L_2}+\BracE{2 L_2}.
		\end{align}
		Since $\?{\vartheta}_m\vartheta_{-m}=c_H\vartheta_m\vartheta_{-m}$, and since the theta function expansion of $\vartheta_m\vartheta_{-m}$ has non-negative integer coefficients, all coefficients on the right-hand side of \eqref{eq:gamma-diamond} should be divisible by $c_H\in \bb{Z}_{\geq 1}$.  Since two of the coefficients are equal to $1$, it must be the case that $c_H=1$.  Hence, $P_H=1$ as well, so $\BracE{C_m}=\vartheta_m$ even in the principal coefficient setup.
		
		On the other hand, for $g=1$, $L_1$ and $L_2$ are isotopic, so we have $[L_1]=[L_2]=:[L]$.  So in this case we have
		\begin{align}\label{eq:g1-gamma-diamond}
			\?{\vartheta}_m\vartheta_{-m}=\gamma^{\diamond}*\gamma=(2[L])^2 = 8+2\BracE{2L}.
		\end{align}
		
		We compare this with the description of the theta function for this case in \cite[\S 5.3]{zhou2020cluster}.  Fix a triangulation $\{\gamma_1,\gamma_2,\gamma_3\}$ of $\Sigma$ with $\gamma=\gamma_1$, and let $A_i$ denote the cluster variable associated to $\gamma_i$; i.e., $\vartheta_{g(\gamma_i)}=A_i$..  Then $g(\gamma^{\diamond})=-g(\gamma)$, and Zhou finds that
		\begin{align*}
			\vartheta_{-g(\gamma)} = \frac{(A_1^2+A_2^2+A_3^2)^2}{A_1A_2^2A_3^2}.
		\end{align*}
				Thus,
		\begin{align}\label{eqn:theta-gam-gam}
			\vartheta_{g(\gamma)}\vartheta_{-g(\gamma)} = \frac{1}{A_2^2A_3^2}\left(A_1^4+A_2^4+A_3^4+2A_1^2A_2^2+2A_2^2A_3^2+2A_1^2A_3^2\right).
		\end{align}
		Note that the constant coefficient of the above expression is $2$.  On the other hand, the constant coefficient for $\gamma*\gamma^{\diamond}=8+2\BracE{2L}$ as in \eqref{eq:g1-gamma-diamond} is $8$.  So, given that $\gamma=\vartheta_{g(\gamma)}$ and $\gamma^{\diamond}=c_H\vartheta_{-g(\gamma)}$, it follows that we must have $c_H=4$.  That is,
		\begin{align}\label{eq:torus_tag_arc}
			\gamma^{\diamond} = 4\vartheta_{-g(\gamma)}.
		\end{align}
		As a further check, we note that $\vartheta_{g(L)}$ (denoted $\vartheta_{\?{1}-\?{2}}$ in \cite[\S 5.3]{zhou2020cluster}) is shown in loc. cit. to equal \begin{align}\label{eq:theta12}
			\frac{A_1^2+A_2^2+A_3^2}{A_2A_3}.
		\end{align}  So by \eqref{eqn:theta-gam-gam}, we have
		\begin{align}\label{eq:theta-gL2}
			\vartheta_{g(L)}^2 = \vartheta_{g(\gamma)}\vartheta_{-g(\gamma)}.
		\end{align}
		We know from \eqref{eq:g1-gamma-diamond} that $\vartheta_{g(L)}^2=[L]^2=\frac{1}{4}\?\gamma^{\diamond}*\vartheta_{-g(\gamma)}$, and comparing this with \eqref{eq:theta-gL2} recovers \eqref{eq:torus_tag_arc} again.

		To summarize, we have found the following:
		\begin{thm}\label{thm:1p-bracelet}
			Let $\BracE{C}$ be any bracelet consisting of weighted doubly-notched arcs in the once-punctured closed surface $\Sigma$ of genus $g$.  Let $m$ be the $g$-vector of $\BracE{C}$ and $n_0:=(1,1,\ldots,1)$, $m_{0}=(0,n_0)\in M\oplus N$.  For $g\geq 2$,
			\begin{align*}
				\BracE{C}=\vartheta_m
			\end{align*}
			in either the principal coefficient or coefficient-free setting.  For $g=1$,
			\begin{align*}
				\BracE{C}= 4^{\langle n_0,-m\rangle}\vartheta_m
			\end{align*}
			in the coefficient-free setting.  In the principal coefficients setting with $g=1$, 
			\begin{align*}
				\BracE{C}= P_H^{\langle n_0,-m\rangle}\vartheta_m
			\end{align*}
			for some polynomial $P_H\in 1+z^{m_0}\bb{Z}_{\geq 0}[z^{m_0}]$ whose coefficients add up to $4$.
		\end{thm}

		\begin{rem}
		Working out $\gamma^{\diamond}*\gamma$ as in \eqref{eq:g1-gamma-diamond} in the principal coefficients setting should suffice to determine $P_H$.  In fact, Min Huang has used the perfect matching approach of \cite{MusikerSchifflerWilliams09} to compute the $F$-polynomial of $\gamma^{\diamond}$ in comparison to that of \cite[\S 5.3]{zhou2020cluster} for us, and his computation indicates that $P_H$ should equal $(1+z^{m_0})^2$; i.e., the scattering function $f_H$ should equal $\Psi(z^{m_0})^2$.
		\end{rem}

			We denote $[a]_+:=\max\{0,a\}$. The following result allows us to generalize Theorem \ref{thm:1p-bracelet} to general tagged bracelets. 

			\begin{lem}\label{lem:1p_general_bracelet}
				 Let $C$ denote a collection of tagged arcs and $L$ a collection of loops, such that $C\cup L$ is a tagged simple multicurve. Then at $t=1$ we have
				\begin{enumerate}[label=(\roman*),ref=\roman*]
                    \item If $\Sigma$ is not a once-punctured torus, then $\BracE{C\cup L}=\vartheta_{g(C)+g(L)}$ (including in the principal coefficients setting).
				\item If $\Sigma$ is a once-punctured torus, then $\BracE{C\cup L}=\lambda(g(C)+g(L))\vartheta_{g(C)+g(L)}$, where $\lambda(m):=4^{[-\langle m,n_0 \rangle]_+}$. 
				\end{enumerate}
			\end{lem}
			\begin{proof}
			  For convenience, in Case (i) above, let $\lambda(m):=1$ for all $m\in M$.  Now recall in general that $\BracE{C\cup L}=\BracE{C}\BracE{L}$, see \S \ref{sec:classical_general_bracelets}. In addition, we have $\BracE{L}=\vartheta_{g(L)}$ by Lemma \ref{lem:loop_puncture_bracelet} and $\BracE{C}=\lambda(g(C))\vartheta_{g(C)}$ by Theorem \ref{thm:1p-bracelet}.

				Choose an $s:1$ covering $\pi:\wt{\Sigma}\rar \Sigma$ ($s\geq 2$) such that each component of $L$ lifts to $s$ disjoint loops in $\wt{\Sigma}$, and fix a lift $\wt{C}\cup \wt{L}$ of $C\cup L$. This way, $\pi_*(\BracE{\wt{C}\cup \wt{L}}) = \BracE{C\cup L}$. Now, the bracelet $\BracE{\wt{C}\cup \wt{L}}$ for $\wt{\Sigma}$ is a theta function (Lemma \ref{lem:cluster_tagged_bracelet_theta}). Applying $\iota^*$ to the factorization of theta functions $\BracE{\wt{C}\cup \wt{L}}= \BracE{\wt{C}} \BracE{\wt{L}}$ and using Theorem \ref{thm:1p-bracelet}, we obtain \begin{align*}
				\iota^*{\vartheta_{g(\wt{C}\cup \wt{L})}}&=\iota^*(\BracE{\wt{C}\cup\wt{L}})\\
				&=\iota^*(\BracE{\wt{C}})\cdot\iota^*(\BracE{\wt{L}})\\
				&=\lambda(g(C))\vartheta_{g(C)}\cdot \iota^*\BracE{\wt{L}}\\
				&=\lambda(g(C))\vartheta_{g(C)}\cdot \BracE{L}\\
				&=\lambda(g(C))\vartheta_{g(C)}\cdot \vartheta_{g(L)}.
				\end{align*}
				 Applying \eqref{eq:project_theta} to the left hand side, we obtain 
				\begin{align*}
				\lambda(g(C)+g(L))\vartheta_{g(C)+g(L)}=\lambda(g(C))\vartheta_{g(C)}\cdot \vartheta_{g(L)}.
				\end{align*}
				The desired claims follows from $\lambda(g(C)+g(L))=\lambda(g(C))$ and $\iota^*\BracE{\wt{C}\cup \wt{L}}=\BracE{C\cup L}$ (by Lemma \ref{lem:cover-unfold}).
			\end{proof}

\section{Donaldson-Thomas transformation}\label{sec:DT}

Consider the consistent (quantum) scattering diagram $\f{D}=\f{D}^{\sd}$ associated to any initial (quantum) seed $\sd$.  Let $\Delta^+$ denote the collection of seeds $\sd_{\jj}$ mutation-equivalent to $\sd$.  Recall that the corresponding cones $C_{\jj}^+$ as in \eqref{Cjplus} are chambers of $\f{D}$ and form a fan $\s{C}^+\coloneqq \s{C}$, see  Proposition \ref{Chambers}.

Now define the seed $\sd^-$ by replacing each basis element $e_{i}$ of $E_{\sd}$ with $-e_i$.  Similarly construct $\sd_{\jj}^-$ for each $\jj$.  Let $\Delta^-$ denote the collection of these seeds.  The corresponding cones $C_{\jj}^-=-C_{\jj}^+$ again are chambers of $\f{D}$ with wall-crossings corresponding to mutations, and these chambers form a fan $\s{C}^-$, see \cite[Construction 1.3]{gross2018canonical}.  When $C_{\sd}^-=C^+_{\sd^-}$ coincides with some $C_{\sd_{\jj}}^+$, $\sd$ is said to be injective-reachable \cite{qin2017triangular} or, equivalently, one says there exists a green to red sequence \cite{MuGreen}. 
In the following, we do NOT assume this condition holds.

For any two seeds $\sd_1,\sd_2$ in $\Delta^+\cup \Delta^-$, we let $\cross_{\sd_2,\sd_1}=\cross_{\gamma}$ denote the action $\Ad_{\theta_{\gamma,\f{D}^\sd}}$ such that $\gamma$ is a generic path from a generic point in $C_{\sd_1}$ to a generic point in $C_{\sd_2}$. 

Fix a generic basepoint $\sQ\in C^+$.  We assume that all theta functions are Laurent polynomials with respect to $\sQ$. By \eqref{alpha} and the positivity of broken lines, this condition is equivalently to $\vartheta_{-f_i,\sQ}$ being Laurent polynomials (recall $f_i\coloneqq e_i^*$). Since $\cross_{\sd,\sd^-} (z^{-f_i})=\vartheta_{-f_i,\sQ}$, $\cross_{\sd,\sd^-}$ acts on the skew-field of fractions $\mr{\s{A}}_{t}^{\sd}$ ($t=1$ for the classical case). Let $\nu=\nu^\sd$ denote the automorphism of $M_{\sd}$ sending $m$ to $-m$, and also the induced action on $\mr{\s{A}}_{t}^{\sd}$ sending $z^{m}$ to $z^{-m}$. Building on \cite{GS-DT} \cite{kimura2022twist}, one defines the Donaldson-Thomas transformation associated to $\sd$ as the composition
\begin{align}\label{eq:DT-def}
\DT^{\sd}=\cross_{\sd,\sd^-}\circ\nu^\sd.
\end{align}
This construction is a natural generalization of the twist endomorphism of DT-type for injective-reachable seeds in \cite{qin2020dual} \cite[\S 5.1]{kimura2022twist}.

\begin{rem}

The automorphism $\nu$ can be understood as follows. Let $\sd^\op$ denote the opposite seed, whose basis vectors are the same as that of $\sd$ but it has the inverse bilinear form $-\omega$. The opposite scattering diagram $\f{D}^{\sd^\op}$ was considered in \cite[\S A.2]{qin2019bases} (and the quantum version in \cite[\S 2.3.3]{davison2019strong}). Then, for any broken line $\gamma$ with ends $(m,\sQ)$ and attaching monomials $c_i z^{v_i}$ in $\f{D}^\sd$, there is a broken line $-\gamma$ with ends $(-m,-\sQ)$ and attaching monomials $c_i z^{-v_i}$ in $\f{D}^{\sd^\op}$. It follows that, for any generic base point $\sQ$ and any $m\in M$, we have
\begin{align}\label{eq:nu}
\nu (\vartheta_{m,\sQ})=\vartheta^{\sd^\op}_{-m,-\sQ}.
\end{align}
\end{rem}

\begin{rem}
The inverse of $\DT^\sd$ can be constructed as follows, assuming that all theta functions are Laurent polynomials with respect to generic base points $-\sQ$ in $C^-$. By \eqref{alpha}, the positivity of broken lines, and Lemma \ref{lem:negative_cluster_mono}, this condition is equivalently to $\vartheta_{f_i,-\sQ}$ being Laurent polynomials. Since $\cross_{\sd^-,\sd} (z^{f_i})=\vartheta_{f_i,-\sQ}$, $\cross_{\sd^-,\sd}$ acts on the skew-field of fractions $\mr{\s{A}}_{t}^{\sd}$. Then we obtain the inverse of $\DT^\sd$:
\begin{align}\label{eq:DT-1}
(\DT^{\sd})^{-1}=(\nu^\sd)^{-1} \circ \cross_{\sd^-,\sd}=\nu^\sd \circ \cross_{\sd^-,\sd}.
\end{align}

\end{rem}

The tropical mutation $T_k$ in \S \ref{sec:seed-mut} provides an isomorphism between scattering diagrams $\f{D}^\sd$ and $\f{D}^{\sd_k}$, see \cite[Construction 1.3]{gross2018canonical} and \cite[\S A]{davison2019strong}. Moreover, its action on the exponents of Laurent monomials gives us $T_k(\vartheta^\sd_{m,\sQ})=\vartheta^{\sd_k}_{T_k(m),T_k(\sQ)}$ \cite[Prop. A.3]{davison2019strong}. Recall that we have the mutation birational map $\mu_k: \mr{\s{A}}_{t}^{\sd}\rightarrow \mr{\s{A}}_{t}^{\sd_k}$ as in \eqref{mujAInverse}. We have 
$\mu_k (\vartheta^\sd_{m,\sQ})=\vartheta^{\sd_k}_{T_k(m),\sQ'}$ for $\sQ$ and $\sQ'$ generic base points in the positive chambers of $\f{D}$ and $\f{D}^{\sd_k}$ respectively (\cite{gross2018canonical}\cite[Thm. A.1.4]{qin2019bases}).

Denote $\sd_k=\sd'$. We will often omit the superscript $\sd$, and we replace the superscript $\sd'$ by $'$.

Recall that we have $\vartheta_{m,\sQ}=z^m$ if both $m$ and the generic point $\sQ$ are contained in some $C_{\sd_{\jj}}$. 
\begin{lem}\label{lem:negative_cluster_mono}
We have $\vartheta_{m,\sQ}=z^m$ if both $m$ and the generic point $\sQ$ are contained in some $C_{\sd_{\jj}^-}$. 
\end{lem}
\begin{proof}
By applying the tropical mutations $T_k$, it suffices to verify the claim for $C^-$.   The statement may be proved similar to that for $C^+$. Alternatively, using \eqref{eq:nu}, we have $\vartheta_{m,\sQ}=\nu(\vartheta^{\sd^\op}_{-m,-\sQ})=\nu(z^{-m})=z^m$. 
\end{proof}

In particular, for $m\in C^+\cap M$ and generic $\sQ\in C^+$, we have
\begin{align}\label{DT-theta-m}
\DT^{\sd}(\vartheta_{m,\sQ})=\DT^{\sd}(z^m)=\f{p}_{\sd,\sd^-}(z^{-m})=\f{p}_{\sd,\sd^-}(\vartheta_{-m,-\sQ})=\vartheta_{-m,\sQ}.
\end{align}

\begin{lem}\label{lem:DT_change_initial_seed}
	For any $k\in I\setminus F$, we have $\mu_{k}\circ \DT^{\sd}=\DT^{\sd_k}\circ \mu_k$.
\end{lem}
\begin{proof}
The mutation $\mu_k$ sends the initial cluster variables $\vartheta_{f_i}$ to
\begin{align*}
\mu_k (\vartheta_{f_i})&=\vartheta'_{f'_i},\ i\neq k\\
\mu_k(\vartheta_{f_k})&=(t^{\alpha} \vartheta'_{\sum_i [b_{ik}]_+ 
f'_i}+t^\beta \vartheta'_{\sum_j [-b_{jk}]_+ f'_j})(\vartheta_{f'_k})^{-1},
\end{align*}
where $\alpha:=\Lambda(f_k,\sum_i [b_{ik}]_+ f_i)$ and $\beta:=\Lambda(f_k,\sum_j [-b_{jk}]_+ f_j)$

Notice that $T_k^{-1}(-f'_k)=-f_k+\sum_i [b_{ik}]_+ f_i$ is contained in the chamber $C^{(\sd_k)^-}$. We have $\vartheta_{T_k^{-1}(-f'_k),-\sQ'}=z^{T_k^{-1}(-f'_k)}$ for a generic point $-\sQ'\in C^{(\sd_k)^-}$ by Lemma \ref{lem:negative_cluster_mono}. A generic path $\gamma$ from $-\sQ'$ to a generic base point $-\sQ$ in $C^-$ only crosses the $k$-th chamber. Applying $\cross_\gamma$, we get 
\begin{align*}
\vartheta_{T_k^{-1}(-f'_k),-\sQ}&=\cross_\gamma (\vartheta_{T_k^{-1}(-f'_k),-\sQ'})\\
&=z^{f_k-\sum_i [b_{ik}]_+ f_i}+z^{f_k-\sum_j [-b_{jk}]_+ f_j}\\
&=(z^{-f_k} )^{-1}(t^\alpha z^{-\sum_i [b_{ik}]_+ f_i}+t^\beta z^{-\sum_j [-b_{jk}]_+ f_j}).
\end{align*}
In fact, this is the $k$-th mutation at the opposite seed $\sd^\op$. Using Lemma \ref{lem:negative_cluster_mono}, we rewrite it as a relation between theta functions:
$$
\vartheta_{-f_k} \vartheta_{T_k^{-1}(-f'_k)}=t^\alpha \vartheta_{-\sum_i [b_{ik}]_+ f_i}+t^\beta \vartheta_{-\sum_j [-b_{jk}]_+ f_j},
$$

By a similar computation (or by applying $\mu_k$), we have the following relation in $\f{D}'$:
$$
\vartheta'_{T_k(-f_k)} \vartheta'_{-f'_k} =t^\alpha \vartheta'_{-\sum_i [b_{ik}]_+ f'_i}+t^\beta \vartheta'_{-\sum_j [-b_{jk}]_+ f'_j}.
$$

Notice that $\mu_k(\vartheta_{-\sum m_i f_i})=\vartheta'_{-\sum m_i f'_i}$ for $m_i\geq 0$, $m_k=0$. In addition, $\mu_k(\vartheta_{T_k^{-1}(-f'_k)})=\vartheta'_{-f'_k}$. Combining with $\DT(\vartheta_{f_i})=\vartheta_{-f_i}$ and $\DT'(\vartheta'_{f'_i})=\vartheta'_{-f'_i}$, $\forall i$, it is straightforward to check the following relation
\begin{align*}
\mu_k \circ \DT(\vartheta_{f_i})&= \DT' \circ\mu_k (\vartheta_{f_i}),\ \forall i.
\end{align*}
Notice that $\vartheta_{f_i,\sQ}=z^{f_i}$, $\forall i,$ generate the fraction field $\mr{\s{A}}_{t}^{\sd}$. The desired claim follows.
\end{proof}
We often identify the fraction field $\mr{\s{A}}_{t}^{\sd}$ with $\mr{\s{A}}_{t}^{\mu_k \sd}$ by the mutation $\mu_k$. Then, by Lemma \ref{lem:DT_change_initial_seed}, the DT-transformation $\DT^{\sd}$ is identified with $\DT^{\mu_k \sd}$, so we simply denote these by $\DT$.  From this point of view, $\DT$ is independent of the choice of the initial seed.

By Lemma \ref{lem:DT_change_initial_seed}, $\DT$ sends cluster monomials for $\sd$ to cluster monomials for $\sd^-$.  By applying sequences of mutations, we extend this to $\sd_{\jj}$ and $\sd_{\jj}^-$.In combination with \eqref{DT-theta-m}, we have the following:

\begin{prop}\label{DT-monomial}
$\DT$ maps the cluster monomials of $\sd_{\jj}$ to the cluster monomials of $\sd_{\jj}^-$.  In particular, for $m\in C_{\sd}$, $\DT^{\sd}(\vartheta_m)=\vartheta_{-m}$.
\end{prop}

More generally, one expects the following:

\begin{conj}\label{conj:DT-theta}
$\DT$ sends theta functions to theta functions.
\end{conj}
Conjecture \ref{conj:DT-theta} is known to be true for injective-reachable seeds, see \cite{kimura2022twist}.

Let us now consider a triangulable connected marked surface $\Sigma$.  Following \cite{brustle2015tagged}, given a tagged arc $\gamma$ in $\Sigma$, define the \textbf{tagged rotation} $\gamma^{\diamond}$ of $\gamma$ to be the tagged arc obtained as follows:
\begin{itemize}
\item Change the taggings of any ends of $\gamma$ at punctures;
\item For any end of $\gamma$ not at a puncture, slide the end along $\partial \SSS$, against the orientation (i.e. counterclockwise in Definition \ref{def:el-lam}), until reaching the next marked point.
\end{itemize}

\begin{prop}\label{prop:DT_tagged_rotation}
Assume that $\gamma$ is not a doubly-notched arc in a once-punctured closed surface. If $\Sigma$ is not a once-punctured torus, then $\DT(\BracE{\gamma})$ agrees with $\BracE{\gamma^{\diamond}}$ up to adding some fixed element of $M_{\sd,F}$ to each exponent (i.e, up to a frozen monomial factor times  a power of $t$).  
If $\Sigma$ is a once-punctured torus, then in the coefficient-free classical setting, $\BracE{\gamma^{\diamond}}=4\DT(\BracE{\gamma})$ (and $\BracE{\gamma^{\diamond}}=(P_H)\DT(\BracE{\gamma})$ in the classical principal coefficients setting for $P_H$ as in Theorem \ref{thm:1p-bracelet}).
\end{prop}
\begin{proof}
 There exists a tagged triangulation $\Delta'$ which can be obtained from $\Delta$ by a sequence of flips $\jj$ and which contains $\gamma$.  By applying mutations along the sequence $\jj$, it suffices to verify the claim for the triangulation $\Delta'$ and the initial seed $\sd_{\jj}$. Then $g(\gamma)$ becomes a basis vector $f'_i$ and $-b^{\Delta'} (\gamma^\diamond)$ becomes $-f'_i$, see Definition \ref{def:shear}.  So using \eqref{eq:shear_g}, the first claim follows from the equality $\DT(\vartheta_{f'_i})=\vartheta_{-f'_i}$.  The claim for the once-punctured torus follows similarly using Theorem \ref{thm:1p-bracelet}.  
\end{proof}

As a consequence, we have the following result.
\begin{cor}\label{prop:DT}
Suppose we are in the coefficient-free classical setting.  If $\Sigma$ is a closed punctured surface but not a once-punctured torus, then $\DT$ coincides with $\prod_{P\in \MM} \iii_P$.  If $\Sigma$ is a once-punctured torus, then $\DT$ coincides with $\iii_P\circ \lambda$ where $\lambda$ is defined by $\lambda:z^m\mapsto 4^{-\langle m,n_0 \rangle}z^m$ for $n_0$ as in Lemma \ref{lem:n0}. In either case, $\DT$ is an involution.
\end{cor}

\begin{prop}\label{prop:surface-DT}
When $\Sigma$ is not a once-punctured closed surface, $\DT$ is a permutation on the theta functions. In addition, we have $\DT \E{L}=\E{L}$ for any simple loop $L$. Moreover, if $\Sigma$ is closed and has at least two punctures, then $\DT$ is involutive up to frozen variables.
\end{prop}
\begin{proof}
We've seen that $\DT$ is invertible \eqref{eq:DT-1}. The first claim holds by \cite[Thm. 6.1.1]{kimura2022twist} since Proposition \ref{gdense} implies these cases are injective-reachable. 

For the claim regarding $\E{L}$, note that it suffices to work in the setting of principal coefficients.  It is known that $\E{L}$ is bi-pointed at $g(L)$ and $-g(L)$, meaning that 
\begin{align}\label{eq:bipointed}
\E{L}=z^{g(L)} + \left(\sum_{m\in [g(L)+M^+]\cap [-g(L)-M^+]} c_mz^m \right)+z^{-g(L)}
\end{align}
for some coefficients $c_m$; cf. \cite[Thm. 12.2.1]{FockGoncharov06a} (a consequence of \eqref{eq:classical_trace}).  It follows that the theta function $\DT \E{L}=\cross_{\sd,\sd^-} (\nu^{\sd}\E{L})$ is also pointed with lowest degree $g(L)$, so it must coincide with $\E{L}$.

Now suppose $\Sigma$ is a closed surface.  We have shown that for any $m\in M$, we have $(\DT)^2\vartheta_m=\vartheta_{m'}$ for some $m'$. By Proposition \ref{prop:DT}, $\vartheta_{m'}$ coincides with $\vartheta_m$ when we evaluate $t$ and the frozen variables to $1$, since $\DT=\prod_{P\in \MM} \iii_P$ is involutive in this setting. It follows that $m'\in m+M_F$.
\end{proof}

\begin{lem}\label{lem:1p-classical-DT}
If $\Sigma$ is a once-punctured closed surface, then at the classical limit $t=1$, $\DT$ is involutive up to frozen variables. In addition, we have $\DT \E{L}=\E{L}$ for any simple loop $L$. Moreover, $\DT$ is a permutation on the theta functions.
\end{lem}
We conjecture that the claims hold at the quantum level as well.
\begin{proof}
It suffices to work with principal coefficients. Choose any $1<d\in \NN$. As in \S \ref{sec:closed_surface}, let $\pi:\wt{\Sigma}\rar \Sigma$ denote a $d:1$ covering space of $\Sigma$. Choose ideal triangulations $\Delta$, $\wt{\Delta}$ as before and let $\wt{\sd}$, $\sd$ denote the associated seeds of principal coefficients. 

For any $i\in \wt{\Delta}$, denote $m=d e_{\Pi i}^*$ and $\wt{m}=\kappa (m)$. Then $\vartheta_m$ and $\vartheta_{\wt{m}}$ are bracelets associated to plain arcs and $\lambda(-m)\vartheta_{-m}$ and $\vartheta_{-\wt{m}}$ are associated to the corresponding doubly notched arcs, where the constant $ \lambda(-m)$ equals $P_H^d$ (for $P_H$ as in Theorem \ref{thm:1p-bracelet}) when $\Sigma$ is a once-punctured torus and it equals to $1$ otherwise. Since $m$ lies in a chamber of $\s{C}$, we have $\iota^* (\vartheta_{\wt{m}})= \vartheta_{m}$ by Lemma \ref{lem:restriction_cluster_complex}.  Moreover, our findings in \S \ref{sub:bracelets-1p} imply $\iota^* (\vartheta_{-\wt{m}})= \lambda(-m)\vartheta_{-m}$. Combining with $\DT^{\sd}(\vartheta_m)=\vartheta_{-m}$ and $\DT^{\wt{\sd}}(\vartheta_{\wt{m}})=\vartheta_{-\wt{m}}$ (Proposition \ref{DT-monomial}), we deduce that 
\begin{align*}\iota^* \circ \DT^{\wt{\sd}}=\lambda\circ \DT^{\sd}\circ \iota^*
\end{align*}
where $\lambda(z^m)=\lambda(-m)^{-\langle m,n_0\rangle}z^m$.  In particular, $\lambda(\vartheta_{\mp m})=\lambda(-m)^{\pm 1} \vartheta_{\mp m}$. Using this and the fact that $\DT^{\wt{\sd}}$ is involutive up to frozen variables (Proposition \ref{prop:surface-DT}), we compute
\begin{align*}
(\DT^{\sd})^2(\vartheta_{m})=\DT^{\sd}(\vartheta_{-m})&=\frac{1}{\lambda(-m)}\DT^{\sd}(\iota^*(\vartheta_{-\wt{m}})) \\ &= \frac{1}{\lambda(-m)}\lambda^{-1}\circ  \iota^*(\DT^{\wt{\sd}}(\vartheta_{-\wt{m}})) \\ &= \frac{1}{\lambda(-m)}\lambda^{-1}\circ\iota^*(\vartheta_{\wt{m}+f}) \qquad \text{for some $f\in M_F$}\\
&=\frac{1}{\lambda(-m)}\lambda^{-1}\vartheta_{m+f}=\vartheta_{m+f}.
\end{align*}
So $\DT^{\sd}$ is indeed involutive up to frozen variables

For any simple loop $L$, denote $\pi^{-1}(L)=\cup_{1\leq j\leq d} L_j$, where $L_j$ are connected components of $\pi^{-1}(L)$. Then $\iota^* \E{L_j}=\E{L}$ by the construction of bracelets and Lemma \ref{lem:cover-unfold}. Using Proposition \ref{prop:surface-DT}, we have $$\DT^{\sd}\E{L}^d=\DT^{\sd}\iota^*(\prod_j \E{L_j})=\lambda^{-1}(\iota^*\DT^{\wt{\sd}}(\prod_j \E{L_j}))=\lambda^{-1}\iota^* (\prod_j \E{L_j})=\lambda^{-1}(\E{L}^d)=\E{L}^d.$$ It follows that $\DT$ sends $\E{L}$ to $\E{L}$.

Combining the above results with the relationship between the classical tagged bracelets and theta functions given in Lemma \ref{lem:1p_general_bracelet}, we deduce that $\DT$ sends theta functions to theta functions.
\end{proof}

The following observation is used in the proof of Lemma \ref{lem:commute_arc_loop}.  
\begin{lem}\label{lem:support_DT_L}
Assume that $\Sigma$ is a once-punctured closed surface. Take any initial triangulation $\Delta$ and let $\sd$ denote a seed similar to $\sd_{\Delta}$ and satisfying the Injectivity Assumption.   
 Then $\DT^{-1}\E{L}$ is bi-pointed with highest Laurent degree $g(L)$ and the lowest Laurent degree $-g(L)$ (in the ordering of Definition \ref{def:dom}); i.e., it has the form of the right-hand side of \eqref{eq:bipointed}. 
\end{lem}
\begin{proof}
Let $\sQ$ denote any generic point in $C^+$. We have $\DT^{-1}\E{L}=\nu^{\sd} (\cross_{\sd^-,\sd}(\vartheta_{g(L),\sQ}))=\nu^{\sd} (\vartheta_{g(L),-\sQ})$. Now since $\vartheta_{g(L),-\sQ}$ is a bar-invariant positive $g(L)$-pointed Laurent polynomial, $\DT^{-1}\E{L}$ must be a bar-invariant positive Laurent polynomial with lowest Laurent degree $-g(L)$ with coefficient $1$.

Let us decompose $\DT^{-1}\E{L}$ into a convergent sum $\DT^{-1}\E{L}=\sum_{m} c_m \vartheta_m$.  By Lemma \ref{lem:1p-classical-DT}, its classical limit at $t=1$ is $\DT^{-1}\E{L}=\E{L}$.  If $m_0$ is maximal out of the terms with $c_m\neq 0$, then the coefficient of $z^{m_0}$ in $\DT^{-1}\E{L}$ is $c_{m_0}$.  Since $\DT^{-1}\E{L}$ is positive, we must have $c_{m_0}|_{t=1}\in \bb{Z}_{\geq 1}$, hence $m_0=g(L)$ and $c_{m_0}|_{t=1}=1$.  Bar-invariance then implies $c_{m_0}=1$. Therefore, the decomposition must take the form $\DT^{-1}\E{L}=\E{L}+\sum_{m\prec_{\sd} g(L)} c_m \vartheta_m$. In particular, $\DT^{-1}\E{L}$ has the highest Laurent degree $g(L)$ with coefficient $1$.
\end{proof}

		\bibliographystyle{amsalphaURL2}		
		\bibliography{biblio,referenceEprint}        

		\index{Bibliography@\emph{Bibliography}}

	\end{document}

%% file: macroEnv.tex
\ifdefined\thm
\else
	\theoremstyle{plain}
 	\newtheorem{thm}{Theorem}[section]
\fi

\ifdefined\prop
\else
	\theoremstyle{plain}
 	\newtheorem{prop}[thm]{Proposition}
\fi\ifdefined\cor
\else
	\theoremstyle{plain}
 	\newtheorem{cor}[thm]{Corollary}
\fi

\ifdefined\lem
\else
	\theoremstyle{plain}
 	\newtheorem{lem}[thm]{Lemma}
\fi

\ifdefined\def
\else
	\theoremstyle{definition}
	\newtheorem{def}[thm]{Definition}
\fi

\ifdefined\defn
\else
	\theoremstyle{definition}
	\newtheorem{defn}[thm]{Definition}
\fi

\ifdefined\example
\else
	\theoremstyle{definition}
    \newtheorem{example}[thm]{Example}
\fi

\ifdefined\eg
\else
		\theoremstyle{definition}
        \newtheorem{eg}[thm]{Example}
\fi

\ifdefined\rem
\else
		\theoremstyle{remark}
        \newtheorem{rem}[thm]{Remark}
\fi

%% file: macro.tex
\usepackage{tikz-cd}

\newcommand{\cross}{\mathfrak{p}}

\newcommand{\NN}{\mathbb{N}}
\newcommand{\ZZ}{\mathbb{Z}}

\newcommand{\op}{{\mathrm{op}}}

\newcommand{\hf}{\frac{1}{2}}

\newcommand{\SMultiTag}{\SMulti^{\Box}}
\newcommand{\wSMultiTag}{\wSMulti^{\Box}}

\newcommand{\Teich}{\operatorname{Teich}}

\newcommand{\ufv}{{\operatorname{uf}}}

\newcommand{\cX}{{\mathcal X}}

\newcommand{\uC}{\underline{C}}

\renewcommand{\tw}{\operatorname{tw}}

\ifdefined\supp
\else
\newcommand{\supp}{\operatorname{supp}}
\fi

\ifdefined\tw
\else
	\newcommand{\tw}{{\opname tw}}
\fi

\ifdefined\G
    \renewcommand{\G}{{\mathbb G}}
\else
   \newcommand{\G}{{\mathbb G}}
\fi

\ifdefined\C
    \renewcommand{\C}{{\mathbb C}}
\else
   \newcommand{\C}{{\mathbb C}}
\fi

\usepackage{comment}


%% file: Annulus_quiver.pdf_tex
\begingroup%
  \makeatletter%
  \providecommand\color[2][]{%
    \errmessage{(Inkscape) Color is used for the text in Inkscape, but the package 'color.sty' is not loaded}%
    \renewcommand\color[2][]{}%
  }%
  \providecommand\transparent[1]{%
    \errmessage{(Inkscape) Transparency is used (non-zero) for the text in Inkscape, but the package 'transparent.sty' is not loaded}%
    \renewcommand\transparent[1]{}%
  }%
  \providecommand\rotatebox[2]{#2}%
  \newcommand*\fsize{\dimexpr\f@size pt\relax}%
  \newcommand*\lineheight[1]{\fontsize{\fsize}{#1\fsize}\selectfont}%
  \ifx\svgwidth\undefined%
    \setlength{\unitlength}{68.22899951bp}%
    \ifx\svgscale\undefined%
      \relax%
    \else%
      \setlength{\unitlength}{\unitlength * \real{\svgscale}}%
    \fi%
  \else%
    \setlength{\unitlength}{\svgwidth}%
  \fi%
  \global\let\svgwidth\undefined%
  \global\let\svgscale\undefined%
  \makeatother%
  \begin{picture}(1,0.54346638)%
    \lineheight{1}%
    \setlength\tabcolsep{0pt}%
    \put(0,0){\includegraphics[width=\unitlength,page=1]{Annulus_quiver.pdf}}%
    \put(0.03004147,0.2417988){\color[rgb]{0,0,0}\makebox(0,0)[lt]{\lineheight{1.25}\smash{\begin{tabular}[t]{l}$2$\end{tabular}}}}%
    \put(0.83936507,0.24959202){\color[rgb]{0,0,0}\makebox(0,0)[lt]{\lineheight{1.25}\smash{\begin{tabular}[t]{l}$1$\end{tabular}}}}%
    \put(0.42949717,0.02685807){\color[rgb]{0,0,0}\makebox(0,0)[lt]{\lineheight{1.25}\smash{\begin{tabular}[t]{l}$4$\end{tabular}}}}%
    \put(0.42934661,0.46709055){\color[rgb]{0,0,0}\makebox(0,0)[lt]{\lineheight{1.25}\smash{\begin{tabular}[t]{l}$3$\end{tabular}}}}%
  \end{picture}%
\endgroup%

%% file: p-contractible.pdf_tex
\begingroup%
  \makeatletter%
  \providecommand\color[2][]{%
    \errmessage{(Inkscape) Color is used for the text in Inkscape, but the package 'color.sty' is not loaded}%
    \renewcommand\color[2][]{}%
  }%
  \providecommand\transparent[1]{%
    \errmessage{(Inkscape) Transparency is used (non-zero) for the text in Inkscape, but the package 'transparent.sty' is not loaded}%
    \renewcommand\transparent[1]{}%
  }%
  \providecommand\rotatebox[2]{#2}%
  \newcommand*\fsize{\dimexpr\f@size pt\relax}%
  \newcommand*\lineheight[1]{\fontsize{\fsize}{#1\fsize}\selectfont}%
  \ifx\svgwidth\undefined%
    \setlength{\unitlength}{158.46346964bp}%
    \ifx\svgscale\undefined%
      \relax%
    \else%
      \setlength{\unitlength}{\unitlength * \real{\svgscale}}%
    \fi%
  \else%
    \setlength{\unitlength}{\svgwidth}%
  \fi%
  \global\let\svgwidth\undefined%
  \global\let\svgscale\undefined%
  \makeatother%
  \begin{picture}(1,0.42423836)%
    \lineheight{1}%
    \setlength\tabcolsep{0pt}%
    \put(0,0){\includegraphics[width=\unitlength,page=1]{p-contractible.pdf}}%
    \put(0.49555397,0.14930836){\color[rgb]{0,0,0}\makebox(0,0)[lt]{\lineheight{1.25}\smash{\begin{tabular}[t]{l}$=0$\end{tabular}}}}%
    \put(0,0){\includegraphics[width=\unitlength,page=2]{p-contractible.pdf}}%
  \end{picture}%
\endgroup%

%% file: bdy-contractible.pdf_tex
\begingroup%
  \makeatletter%
  \providecommand\color[2][]{%
    \errmessage{(Inkscape) Color is used for the text in Inkscape, but the package 'color.sty' is not loaded}%
    \renewcommand\color[2][]{}%
  }%
  \providecommand\transparent[1]{%
    \errmessage{(Inkscape) Transparency is used (non-zero) for the text in Inkscape, but the package 'transparent.sty' is not loaded}%
    \renewcommand\transparent[1]{}%
  }%
  \providecommand\rotatebox[2]{#2}%
  \newcommand*\fsize{\dimexpr\f@size pt\relax}%
  \newcommand*\lineheight[1]{\fontsize{\fsize}{#1\fsize}\selectfont}%
  \ifx\svgwidth\undefined%
    \setlength{\unitlength}{158.46346964bp}%
    \ifx\svgscale\undefined%
      \relax%
    \else%
      \setlength{\unitlength}{\unitlength * \real{\svgscale}}%
    \fi%
  \else%
    \setlength{\unitlength}{\svgwidth}%
  \fi%
  \global\let\svgwidth\undefined%
  \global\let\svgscale\undefined%
  \makeatother%
  \begin{picture}(1,0.42926957)%
    \lineheight{1}%
    \setlength\tabcolsep{0pt}%
    \put(0,0){\includegraphics[width=\unitlength,page=1]{bdy-contractible.pdf}}%
    \put(0.49555397,0.15433958){\color[rgb]{0,0,0}\makebox(0,0)[lt]{\lineheight{1.25}\smash{\begin{tabular}[t]{l}$=0$\end{tabular}}}}%
    \put(0,0){\includegraphics[width=\unitlength,page=2]{bdy-contractible.pdf}}%
  \end{picture}%
\endgroup%

%% file: q_unknot.pdf_tex
\begingroup%
  \makeatletter%
  \providecommand\color[2][]{%
    \errmessage{(Inkscape) Color is used for the text in Inkscape, but the package 'color.sty' is not loaded}%
    \renewcommand\color[2][]{}%
  }%
  \providecommand\transparent[1]{%
    \errmessage{(Inkscape) Transparency is used (non-zero) for the text in Inkscape, but the package 'transparent.sty' is not loaded}%
    \renewcommand\transparent[1]{}%
  }%
  \providecommand\rotatebox[2]{#2}%
  \newcommand*\fsize{\dimexpr\f@size pt\relax}%
  \newcommand*\lineheight[1]{\fontsize{\fsize}{#1\fsize}\selectfont}%
  \ifx\svgwidth\undefined%
    \setlength{\unitlength}{370.94810633bp}%
    \ifx\svgscale\undefined%
      \relax%
    \else%
      \setlength{\unitlength}{\unitlength * \real{\svgscale}}%
    \fi%
  \else%
    \setlength{\unitlength}{\svgwidth}%
  \fi%
  \global\let\svgwidth\undefined%
  \global\let\svgscale\undefined%
  \makeatother%
  \begin{picture}(1,0.18122827)%
    \lineheight{1}%
    \setlength\tabcolsep{0pt}%
    \put(0,0){\includegraphics[width=\unitlength,page=1]{q_unknot.pdf}}%
    \put(0.20243955,0.06892327){\color[rgb]{0,0,0}\makebox(0,0)[lt]{\lineheight{1.25}\smash{\begin{tabular}[t]{l}$=-(q^2+q^{-2})$\end{tabular}}}}%
    \put(0,0){\includegraphics[width=\unitlength,page=2]{q_unknot.pdf}}%
  \end{picture}%
\endgroup%

%% file: Skein.pdf_tex
\begingroup%
  \makeatletter%
  \providecommand\color[2][]{%
    \errmessage{(Inkscape) Color is used for the text in Inkscape, but the package 'color.sty' is not loaded}%
    \renewcommand\color[2][]{}%
  }%
  \providecommand\transparent[1]{%
    \errmessage{(Inkscape) Transparency is used (non-zero) for the text in Inkscape, but the package 'transparent.sty' is not loaded}%
    \renewcommand\transparent[1]{}%
  }%
  \providecommand\rotatebox[2]{#2}%
  \newcommand*\fsize{\dimexpr\f@size pt\relax}%
  \newcommand*\lineheight[1]{\fontsize{\fsize}{#1\fsize}\selectfont}%
  \ifx\svgwidth\undefined%
    \setlength{\unitlength}{394.18421609bp}%
    \ifx\svgscale\undefined%
      \relax%
    \else%
      \setlength{\unitlength}{\unitlength * \real{\svgscale}}%
    \fi%
  \else%
    \setlength{\unitlength}{\svgwidth}%
  \fi%
  \global\let\svgwidth\undefined%
  \global\let\svgscale\undefined%
  \makeatother%
  \begin{picture}(1,0.17213835)%
    \lineheight{1}%
    \setlength\tabcolsep{0pt}%
    \put(0,0){\includegraphics[width=\unitlength,page=1]{Skein.pdf}}%
    \put(0.19921501,0.06122817){\color[rgb]{0,0,0}\makebox(0,0)[lt]{\lineheight{1.25}\smash{\begin{tabular}[t]{l}$=q$\end{tabular}}}}%
    \put(0,0){\includegraphics[width=\unitlength,page=2]{Skein.pdf}}%
    \put(0.55673884,0.06126882){\color[rgb]{0,0,0}\makebox(0,0)[lt]{\lineheight{1.25}\smash{\begin{tabular}[t]{l}$+q^{-1}$\end{tabular}}}}%
  \end{picture}%
\endgroup%

%% file: c_loop_p.pdf_tex
\begingroup%
  \makeatletter%
  \providecommand\color[2][]{%
    \errmessage{(Inkscape) Color is used for the text in Inkscape, but the package 'color.sty' is not loaded}%
    \renewcommand\color[2][]{}%
  }%
  \providecommand\transparent[1]{%
    \errmessage{(Inkscape) Transparency is used (non-zero) for the text in Inkscape, but the package 'transparent.sty' is not loaded}%
    \renewcommand\transparent[1]{}%
  }%
  \providecommand\rotatebox[2]{#2}%
  \newcommand*\fsize{\dimexpr\f@size pt\relax}%
  \newcommand*\lineheight[1]{\fontsize{\fsize}{#1\fsize}\selectfont}%
  \ifx\svgwidth\undefined%
    \setlength{\unitlength}{153.34212142bp}%
    \ifx\svgscale\undefined%
      \relax%
    \else%
      \setlength{\unitlength}{\unitlength * \real{\svgscale}}%
    \fi%
  \else%
    \setlength{\unitlength}{\svgwidth}%
  \fi%
  \global\let\svgwidth\undefined%
  \global\let\svgscale\undefined%
  \makeatother%
  \begin{picture}(1,0.43840715)%
    \lineheight{1}%
    \setlength\tabcolsep{0pt}%
    \put(0.47870639,0.17002802){\color[rgb]{0,0,0}\makebox(0,0)[lt]{\lineheight{1.25}\smash{\begin{tabular}[t]{l}$=2$\end{tabular}}}}%
    \put(0,0){\includegraphics[width=\unitlength,page=1]{c_loop_p.pdf}}%
  \end{picture}%
\endgroup%

%% file: compatible_pair_proof_2.pdf_tex
\begingroup%
  \makeatletter%
  \providecommand\color[2][]{%
    \errmessage{(Inkscape) Color is used for the text in Inkscape, but the package 'color.sty' is not loaded}%
    \renewcommand\color[2][]{}%
  }%
  \providecommand\transparent[1]{%
    \errmessage{(Inkscape) Transparency is used (non-zero) for the text in Inkscape, but the package 'transparent.sty' is not loaded}%
    \renewcommand\transparent[1]{}%
  }%
  \providecommand\rotatebox[2]{#2}%
  \newcommand*\fsize{\dimexpr\f@size pt\relax}%
  \newcommand*\lineheight[1]{\fontsize{\fsize}{#1\fsize}\selectfont}%
  \ifx\svgwidth\undefined%
    \setlength{\unitlength}{271.13396974bp}%
    \ifx\svgscale\undefined%
      \relax%
    \else%
      \setlength{\unitlength}{\unitlength * \real{\svgscale}}%
    \fi%
  \else%
    \setlength{\unitlength}{\svgwidth}%
  \fi%
  \global\let\svgwidth\undefined%
  \global\let\svgscale\undefined%
  \makeatother%
  \begin{picture}(1,0.99983488)%
    \lineheight{1}%
    \setlength\tabcolsep{0pt}%
    \put(0,0){\includegraphics[width=\unitlength,page=1]{compatible_pair_proof_2.pdf}}%
    \put(0.44722771,0.55976899){\makebox(0,0)[lt]{\lineheight{1.25}\smash{\begin{tabular}[t]{l}$x_j$\end{tabular}}}}%
    \put(0,0){\includegraphics[width=\unitlength,page=2]{compatible_pair_proof_2.pdf}}%
    \put(0.66145584,0.73145271){\makebox(0,0)[lt]{\lineheight{1.25}\smash{\begin{tabular}[t]{l}$x_{k_1}$\end{tabular}}}}%
    \put(0.67531357,0.28489211){\makebox(0,0)[lt]{\lineheight{1.25}\smash{\begin{tabular}[t]{l}$x_{k_2}$\end{tabular}}}}%
    \put(0.19791503,0.28431351){\makebox(0,0)[lt]{\lineheight{1.25}\smash{\begin{tabular}[t]{l}$x_{k_3}$\end{tabular}}}}%
    \put(0.1921691,0.74774061){\makebox(0,0)[lt]{\lineheight{1.25}\smash{\begin{tabular}[t]{l}$x_{k_4}$\end{tabular}}}}%
    \put(0,0){\includegraphics[width=\unitlength,page=3]{compatible_pair_proof_2.pdf}}%
  \end{picture}%
\endgroup%

%% file: Ptolemy.pdf_tex
\begingroup%
  \makeatletter%
  \providecommand\color[2][]{%
    \errmessage{(Inkscape) Color is used for the text in Inkscape, but the package 'color.sty' is not loaded}%
    \renewcommand\color[2][]{}%
  }%
  \providecommand\transparent[1]{%
    \errmessage{(Inkscape) Transparency is used (non-zero) for the text in Inkscape, but the package 'transparent.sty' is not loaded}%
    \renewcommand\transparent[1]{}%
  }%
  \providecommand\rotatebox[2]{#2}%
  \newcommand*\fsize{\dimexpr\f@size pt\relax}%
  \newcommand*\lineheight[1]{\fontsize{\fsize}{#1\fsize}\selectfont}%
  \ifx\svgwidth\undefined%
    \setlength{\unitlength}{504.10388577bp}%
    \ifx\svgscale\undefined%
      \relax%
    \else%
      \setlength{\unitlength}{\unitlength * \real{\svgscale}}%
    \fi%
  \else%
    \setlength{\unitlength}{\svgwidth}%
  \fi%
  \global\let\svgwidth\undefined%
  \global\let\svgscale\undefined%
  \makeatother%
  \begin{picture}(1,0.25956738)%
    \lineheight{1}%
    \setlength\tabcolsep{0pt}%
    \put(0,0){\includegraphics[width=\unitlength,page=1]{Ptolemy.pdf}}%
    \put(0.20600564,0.14698929){\color[rgb]{0,0,0}\makebox(0,0)[lt]{\lineheight{1.25}\smash{\begin{tabular}[t]{l}$\gamma$\end{tabular}}}}%
    \put(0.78325306,0.09960202){\color[rgb]{0,0,0}\makebox(0,0)[lt]{\lineheight{1.25}\smash{\begin{tabular}[t]{l}$\gamma'$\end{tabular}}}}%
  \end{picture}%
\endgroup%

%% file: Annulus.pdf_tex
\begingroup%
  \makeatletter%
  \providecommand\color[2][]{%
    \errmessage{(Inkscape) Color is used for the text in Inkscape, but the package 'color.sty' is not loaded}%
    \renewcommand\color[2][]{}%
  }%
  \providecommand\transparent[1]{%
    \errmessage{(Inkscape) Transparency is used (non-zero) for the text in Inkscape, but the package 'transparent.sty' is not loaded}%
    \renewcommand\transparent[1]{}%
  }%
  \providecommand\rotatebox[2]{#2}%
  \newcommand*\fsize{\dimexpr\f@size pt\relax}%
  \newcommand*\lineheight[1]{\fontsize{\fsize}{#1\fsize}\selectfont}%
  \ifx\svgwidth\undefined%
    \setlength{\unitlength}{112.14528195bp}%
    \ifx\svgscale\undefined%
      \relax%
    \else%
      \setlength{\unitlength}{\unitlength * \real{\svgscale}}%
    \fi%
  \else%
    \setlength{\unitlength}{\svgwidth}%
  \fi%
  \global\let\svgwidth\undefined%
  \global\let\svgscale\undefined%
  \makeatother%
  \begin{picture}(1,0.76708861)%
    \lineheight{1}%
    \setlength\tabcolsep{0pt}%
    \put(0,0){\includegraphics[width=\unitlength,page=1]{Annulus.pdf}}%
    \put(0.73592953,0.33059038){\color[rgb]{0,0,0}\makebox(0,0)[lt]{\lineheight{1.25}\smash{\begin{tabular}[t]{l}$b_1$\end{tabular}}}}%
    \put(0.31211488,0.2356366){\color[rgb]{0,0,0}\makebox(0,0)[lt]{\lineheight{1.25}\smash{\begin{tabular}[t]{l}$b_2$\end{tabular}}}}%
    \put(0.12792085,0.36085495){\color[rgb]{0,0,0}\makebox(0,0)[lt]{\lineheight{1.25}\smash{\begin{tabular}[t]{l}$\gamma_1$\end{tabular}}}}%
    \put(0.37703549,0.57041603){\color[rgb]{0,0,0}\makebox(0,0)[lt]{\lineheight{1.25}\smash{\begin{tabular}[t]{l}$\gamma_2$\end{tabular}}}}%
  \end{picture}%
\endgroup%

%% file: flipped_annulus.pdf_tex
\begingroup%
  \makeatletter%
  \providecommand\color[2][]{%
    \errmessage{(Inkscape) Color is used for the text in Inkscape, but the package 'color.sty' is not loaded}%
    \renewcommand\color[2][]{}%
  }%
  \providecommand\transparent[1]{%
    \errmessage{(Inkscape) Transparency is used (non-zero) for the text in Inkscape, but the package 'transparent.sty' is not loaded}%
    \renewcommand\transparent[1]{}%
  }%
  \providecommand\rotatebox[2]{#2}%
  \newcommand*\fsize{\dimexpr\f@size pt\relax}%
  \newcommand*\lineheight[1]{\fontsize{\fsize}{#1\fsize}\selectfont}%
  \ifx\svgwidth\undefined%
    \setlength{\unitlength}{128.26971894bp}%
    \ifx\svgscale\undefined%
      \relax%
    \else%
      \setlength{\unitlength}{\unitlength * \real{\svgscale}}%
    \fi%
  \else%
    \setlength{\unitlength}{\svgwidth}%
  \fi%
  \global\let\svgwidth\undefined%
  \global\let\svgscale\undefined%
  \makeatother%
  \begin{picture}(1,0.67065999)%
    \lineheight{1}%
    \setlength\tabcolsep{0pt}%
    \put(0,0){\includegraphics[width=\unitlength,page=1]{flipped_annulus.pdf}}%
    \put(0.64341783,0.28903276){\color[rgb]{0,0,0}\makebox(0,0)[lt]{\lineheight{1.25}\smash{\begin{tabular}[t]{l}$b_1$\end{tabular}}}}%
    \put(0.27287977,0.20601537){\color[rgb]{0,0,0}\makebox(0,0)[lt]{\lineheight{1.25}\smash{\begin{tabular}[t]{l}$b_2$\end{tabular}}}}%
    \put(0.47372707,0.36422079){\color[rgb]{0,0,0}\makebox(0,0)[lt]{\lineheight{1.25}\smash{\begin{tabular}[t]{l}$\gamma'_1$\end{tabular}}}}%
    \put(0.32963939,0.49871058){\color[rgb]{0,0,0}\makebox(0,0)[lt]{\lineheight{1.25}\smash{\begin{tabular}[t]{l}$\gamma_2$\end{tabular}}}}%
  \end{picture}%
\endgroup%

%% file: iota_of_noose.pdf_tex
\begingroup%
  \makeatletter%
  \providecommand\color[2][]{%
    \errmessage{(Inkscape) Color is used for the text in Inkscape, but the package 'color.sty' is not loaded}%
    \renewcommand\color[2][]{}%
  }%
  \providecommand\transparent[1]{%
    \errmessage{(Inkscape) Transparency is used (non-zero) for the text in Inkscape, but the package 'transparent.sty' is not loaded}%
    \renewcommand\transparent[1]{}%
  }%
  \providecommand\rotatebox[2]{#2}%
  \newcommand*\fsize{\dimexpr\f@size pt\relax}%
  \newcommand*\lineheight[1]{\fontsize{\fsize}{#1\fsize}\selectfont}%
  \ifx\svgwidth\undefined%
    \setlength{\unitlength}{158.39000039bp}%
    \ifx\svgscale\undefined%
      \relax%
    \else%
      \setlength{\unitlength}{\unitlength * \real{\svgscale}}%
    \fi%
  \else%
    \setlength{\unitlength}{\svgwidth}%
  \fi%
  \global\let\svgwidth\undefined%
  \global\let\svgscale\undefined%
  \makeatother%
  \begin{picture}(1,0.41038509)%
    \lineheight{1}%
    \setlength\tabcolsep{0pt}%
    \put(0,0){\includegraphics[width=\unitlength,page=1]{iota_of_noose.pdf}}%
    \put(0.17952932,0.00927917){\color[rgb]{0,0,0}\makebox(0,0)[lt]{\lineheight{1.25}\smash{\begin{tabular}[t]{l}$m$\end{tabular}}}}%
    \put(0.82204942,0.01218625){\color[rgb]{0,0,0}\makebox(0,0)[lt]{\lineheight{1.25}\smash{\begin{tabular}[t]{l}$m$\end{tabular}}}}%
    \put(0.81834999,0.26259413){\color[rgb]{0,0,0}\makebox(0,0)[lt]{\lineheight{1.25}\smash{\begin{tabular}[t]{l}$p$\end{tabular}}}}%
    \put(0.18246084,0.27214151){\color[rgb]{0,0,0}\makebox(0,0)[lt]{\lineheight{1.25}\smash{\begin{tabular}[t]{l}$p$\end{tabular}}}}%
    \put(0.27695252,0.34357568){\color[rgb]{0,0,0}\makebox(0,0)[lt]{\lineheight{1.25}\smash{\begin{tabular}[t]{l}$\alpha$\end{tabular}}}}%
    \put(0.81526773,0.12694267){\color[rgb]{0,0,0}\makebox(0,0)[lt]{\lineheight{1.25}\smash{\begin{tabular}[t]{l}$\iota(\alpha)$\end{tabular}}}}%
  \end{picture}%
\endgroup%

%% file: c1c2.pdf_tex
\begingroup%
  \makeatletter%
  \providecommand\color[2][]{%
    \errmessage{(Inkscape) Color is used for the text in Inkscape, but the package 'color.sty' is not loaded}%
    \renewcommand\color[2][]{}%
  }%
  \providecommand\transparent[1]{%
    \errmessage{(Inkscape) Transparency is used (non-zero) for the text in Inkscape, but the package 'transparent.sty' is not loaded}%
    \renewcommand\transparent[1]{}%
  }%
  \providecommand\rotatebox[2]{#2}%
  \newcommand*\fsize{\dimexpr\f@size pt\relax}%
  \newcommand*\lineheight[1]{\fontsize{\fsize}{#1\fsize}\selectfont}%
  \ifx\svgwidth\undefined%
    \setlength{\unitlength}{62.52399711bp}%
    \ifx\svgscale\undefined%
      \relax%
    \else%
      \setlength{\unitlength}{\unitlength * \real{\svgscale}}%
    \fi%
  \else%
    \setlength{\unitlength}{\svgwidth}%
  \fi%
  \global\let\svgwidth\undefined%
  \global\let\svgscale\undefined%
  \makeatother%
  \begin{picture}(1,0.47938094)%
    \lineheight{1}%
    \setlength\tabcolsep{0pt}%
    \put(0,0){\includegraphics[width=\unitlength,page=1]{c1c2.pdf}}%
  \end{picture}%
\endgroup%

%% file: digon.pdf_tex
\begingroup%
  \makeatletter%
  \providecommand\color[2][]{%
    \errmessage{(Inkscape) Color is used for the text in Inkscape, but the package 'color.sty' is not loaded}%
    \renewcommand\color[2][]{}%
  }%
  \providecommand\transparent[1]{%
    \errmessage{(Inkscape) Transparency is used (non-zero) for the text in Inkscape, but the package 'transparent.sty' is not loaded}%
    \renewcommand\transparent[1]{}%
  }%
  \providecommand\rotatebox[2]{#2}%
  \newcommand*\fsize{\dimexpr\f@size pt\relax}%
  \newcommand*\lineheight[1]{\fontsize{\fsize}{#1\fsize}\selectfont}%
  \ifx\svgwidth\undefined%
    \setlength{\unitlength}{199.98159897bp}%
    \ifx\svgscale\undefined%
      \relax%
    \else%
      \setlength{\unitlength}{\unitlength * \real{\svgscale}}%
    \fi%
  \else%
    \setlength{\unitlength}{\svgwidth}%
  \fi%
  \global\let\svgwidth\undefined%
  \global\let\svgscale\undefined%
  \makeatother%
  \begin{picture}(1,0.4281786)%
    \lineheight{1}%
    \setlength\tabcolsep{0pt}%
    \put(0,0){\includegraphics[width=\unitlength,page=1]{digon.pdf}}%
    \put(0.21957613,0.23309394){\makebox(0,0)[lt]{\lineheight{1.25}\smash{\begin{tabular}[t]{l}$\gamma$\end{tabular}}}}%
    \put(0.66314311,0.23358761){\makebox(0,0)[lt]{\lineheight{1.25}\smash{\begin{tabular}[t]{l}$\gamma'$\end{tabular}}}}%
    \put(0.45934593,0.37042421){\makebox(0,0)[lt]{\lineheight{1.25}\smash{\begin{tabular}[t]{l}$\alpha$\end{tabular}}}}%
    \put(0.45271621,0.01999526){\makebox(0,0)[lt]{\lineheight{1.25}\smash{\begin{tabular}[t]{l}$\beta$\end{tabular}}}}%
  \end{picture}%
\endgroup%

%% file: ggp.pdf_tex
\begingroup%
  \makeatletter%
  \providecommand\color[2][]{%
    \errmessage{(Inkscape) Color is used for the text in Inkscape, but the package 'color.sty' is not loaded}%
    \renewcommand\color[2][]{}%
  }%
  \providecommand\transparent[1]{%
    \errmessage{(Inkscape) Transparency is used (non-zero) for the text in Inkscape, but the package 'transparent.sty' is not loaded}%
    \renewcommand\transparent[1]{}%
  }%
  \providecommand\rotatebox[2]{#2}%
  \newcommand*\fsize{\dimexpr\f@size pt\relax}%
  \newcommand*\lineheight[1]{\fontsize{\fsize}{#1\fsize}\selectfont}%
  \ifx\svgwidth\undefined%
    \setlength{\unitlength}{326.97533279bp}%
    \ifx\svgscale\undefined%
      \relax%
    \else%
      \setlength{\unitlength}{\unitlength * \real{\svgscale}}%
    \fi%
  \else%
    \setlength{\unitlength}{\svgwidth}%
  \fi%
  \global\let\svgwidth\undefined%
  \global\let\svgscale\undefined%
  \makeatother%
  \begin{picture}(1,0.28448468)%
    \lineheight{1}%
    \setlength\tabcolsep{0pt}%
    \put(0,0){\includegraphics[width=\unitlength,page=1]{ggp.pdf}}%
    \put(0.34181124,0.13842553){\makebox(0,0)[lt]{\lineheight{1.25}\smash{\begin{tabular}[t]{l}$\?{\gamma}$\end{tabular}}}}%
    \put(0,0){\includegraphics[width=\unitlength,page=2]{ggp.pdf}}%
    \put(0.85312848,0.13833407){\makebox(0,0)[lt]{\lineheight{1.25}\smash{\begin{tabular}[t]{l}$\?{\gamma}'$\end{tabular}}}}%
    \put(0.38557429,0.2302441){\makebox(0,0)[lt]{\lineheight{1.25}\smash{\begin{tabular}[t]{l}$L_2$\end{tabular}}}}%
    \put(0.22228832,0.14764332){\makebox(0,0)[lt]{\lineheight{1.25}\smash{\begin{tabular}[t]{l}$L_1$\end{tabular}}}}%
    \put(0.89418291,0.22989818){\makebox(0,0)[lt]{\lineheight{1.25}\smash{\begin{tabular}[t]{l}$L_2$\end{tabular}}}}%
    \put(0.73319502,0.14764332){\makebox(0,0)[lt]{\lineheight{1.25}\smash{\begin{tabular}[t]{l}$L_1$\end{tabular}}}}%
  \end{picture}%
\endgroup%

%% file: abggp.pdf_tex
\begingroup%
  \makeatletter%
  \providecommand\color[2][]{%
    \errmessage{(Inkscape) Color is used for the text in Inkscape, but the package 'color.sty' is not loaded}%
    \renewcommand\color[2][]{}%
  }%
  \providecommand\transparent[1]{%
    \errmessage{(Inkscape) Transparency is used (non-zero) for the text in Inkscape, but the package 'transparent.sty' is not loaded}%
    \renewcommand\transparent[1]{}%
  }%
  \providecommand\rotatebox[2]{#2}%
  \newcommand*\fsize{\dimexpr\f@size pt\relax}%
  \newcommand*\lineheight[1]{\fontsize{\fsize}{#1\fsize}\selectfont}%
  \ifx\svgwidth\undefined%
    \setlength{\unitlength}{228.36401838bp}%
    \ifx\svgscale\undefined%
      \relax%
    \else%
      \setlength{\unitlength}{\unitlength * \real{\svgscale}}%
    \fi%
  \else%
    \setlength{\unitlength}{\svgwidth}%
  \fi%
  \global\let\svgwidth\undefined%
  \global\let\svgscale\undefined%
  \makeatother%
  \begin{picture}(1,0.60572335)%
    \lineheight{1}%
    \setlength\tabcolsep{0pt}%
    \put(0,0){\includegraphics[width=\unitlength,page=1]{abggp.pdf}}%
    \put(0.38977024,0.10162361){\makebox(0,0)[lt]{\lineheight{1.25}\smash{\begin{tabular}[t]{l}$\gamma$\end{tabular}}}}%
    \put(0,0){\includegraphics[width=\unitlength,page=2]{abggp.pdf}}%
    \put(0.77642717,0.31178086){\makebox(0,0)[lt]{\lineheight{1.25}\smash{\begin{tabular}[t]{l}$\gamma'$\end{tabular}}}}%
    \put(0,0){\includegraphics[width=\unitlength,page=3]{abggp.pdf}}%
    \put(0.66159611,0.33158777){\makebox(0,0)[lt]{\lineheight{1.25}\smash{\begin{tabular}[t]{l}$\alpha$\end{tabular}}}}%
    \put(0,0){\includegraphics[width=\unitlength,page=4]{abggp.pdf}}%
    \put(0.80117596,0.20960212){\makebox(0,0)[lt]{\lineheight{1.25}\smash{\begin{tabular}[t]{l}$\beta$\end{tabular}}}}%
  \end{picture}%
\endgroup%

%% file: ab.pdf_tex
\begingroup%
  \makeatletter%
  \providecommand\color[2][]{%
    \errmessage{(Inkscape) Color is used for the text in Inkscape, but the package 'color.sty' is not loaded}%
    \renewcommand\color[2][]{}%
  }%
  \providecommand\transparent[1]{%
    \errmessage{(Inkscape) Transparency is used (non-zero) for the text in Inkscape, but the package 'transparent.sty' is not loaded}%
    \renewcommand\transparent[1]{}%
  }%
  \providecommand\rotatebox[2]{#2}%
  \newcommand*\fsize{\dimexpr\f@size pt\relax}%
  \newcommand*\lineheight[1]{\fontsize{\fsize}{#1\fsize}\selectfont}%
  \ifx\svgwidth\undefined%
    \setlength{\unitlength}{329.68682714bp}%
    \ifx\svgscale\undefined%
      \relax%
    \else%
      \setlength{\unitlength}{\unitlength * \real{\svgscale}}%
    \fi%
  \else%
    \setlength{\unitlength}{\svgwidth}%
  \fi%
  \global\let\svgwidth\undefined%
  \global\let\svgscale\undefined%
  \makeatother%
  \begin{picture}(1,0.28386466)%
    \lineheight{1}%
    \setlength\tabcolsep{0pt}%
    \put(0,0){\includegraphics[width=\unitlength,page=1]{ab.pdf}}%
    \put(0.86190837,0.16792999){\makebox(0,0)[lt]{\lineheight{1.25}\smash{\begin{tabular}[t]{l}$\?{\beta}$\end{tabular}}}}%
    \put(0,0){\includegraphics[width=\unitlength,page=2]{ab.pdf}}%
    \put(0.32973116,0.16606119){\makebox(0,0)[lt]{\lineheight{1.25}\smash{\begin{tabular}[t]{l}$\?{\alpha}$\end{tabular}}}}%
    \put(0,0){\includegraphics[width=\unitlength,page=3]{ab.pdf}}%
    \put(0.38375227,0.21742265){\makebox(0,0)[lt]{\lineheight{1.25}\smash{\begin{tabular}[t]{l}$L_2$\end{tabular}}}}%
    \put(0.89505319,0.21627371){\makebox(0,0)[lt]{\lineheight{1.25}\smash{\begin{tabular}[t]{l}$L_2$\end{tabular}}}}%
    \put(0.22183692,0.14618521){\makebox(0,0)[lt]{\lineheight{1.25}\smash{\begin{tabular}[t]{l}$L_1$\end{tabular}}}}%
    \put(0.72854195,0.14618521){\makebox(0,0)[lt]{\lineheight{1.25}\smash{\begin{tabular}[t]{l}$L_1$\end{tabular}}}}%
  \end{picture}%
\endgroup%

%% file: local_digon.pdf_tex
\begingroup%
  \makeatletter%
  \providecommand\color[2][]{%
    \errmessage{(Inkscape) Color is used for the text in Inkscape, but the package 'color.sty' is not loaded}%
    \renewcommand\color[2][]{}%
  }%
  \providecommand\transparent[1]{%
    \errmessage{(Inkscape) Transparency is used (non-zero) for the text in Inkscape, but the package 'transparent.sty' is not loaded}%
    \renewcommand\transparent[1]{}%
  }%
  \providecommand\rotatebox[2]{#2}%
  \newcommand*\fsize{\dimexpr\f@size pt\relax}%
  \newcommand*\lineheight[1]{\fontsize{\fsize}{#1\fsize}\selectfont}%
  \ifx\svgwidth\undefined%
    \setlength{\unitlength}{295.47631849bp}%
    \ifx\svgscale\undefined%
      \relax%
    \else%
      \setlength{\unitlength}{\unitlength * \real{\svgscale}}%
    \fi%
  \else%
    \setlength{\unitlength}{\svgwidth}%
  \fi%
  \global\let\svgwidth\undefined%
  \global\let\svgscale\undefined%
  \makeatother%
  \begin{picture}(1,0.22815951)%
    \lineheight{1}%
    \setlength\tabcolsep{0pt}%
    \put(0,0){\includegraphics[width=\unitlength,page=1]{local_digon.pdf}}%
    \put(0.27084134,0.09532479){\color[rgb]{0,0,0}\makebox(0,0)[lt]{\lineheight{1.25}\smash{\begin{tabular}[t]{l}$=$\end{tabular}}}}%
    \put(0.66149976,0.09030262){\color[rgb]{0,0,0}\makebox(0,0)[lt]{\lineheight{1.25}\smash{\begin{tabular}[t]{l}$+$\end{tabular}}}}%
    \put(0,0){\includegraphics[width=\unitlength,page=2]{local_digon.pdf}}%
  \end{picture}%
\endgroup%

%% file: shear.pdf_tex
\begingroup%
  \makeatletter%
  \providecommand\color[2][]{%
    \errmessage{(Inkscape) Color is used for the text in Inkscape, but the package 'color.sty' is not loaded}%
    \renewcommand\color[2][]{}%
  }%
  \providecommand\transparent[1]{%
    \errmessage{(Inkscape) Transparency is used (non-zero) for the text in Inkscape, but the package 'transparent.sty' is not loaded}%
    \renewcommand\transparent[1]{}%
  }%
  \providecommand\rotatebox[2]{#2}%
  \newcommand*\fsize{\dimexpr\f@size pt\relax}%
  \newcommand*\lineheight[1]{\fontsize{\fsize}{#1\fsize}\selectfont}%
  \ifx\svgwidth\undefined%
    \setlength{\unitlength}{273.16238818bp}%
    \ifx\svgscale\undefined%
      \relax%
    \else%
      \setlength{\unitlength}{\unitlength * \real{\svgscale}}%
    \fi%
  \else%
    \setlength{\unitlength}{\svgwidth}%
  \fi%
  \global\let\svgwidth\undefined%
  \global\let\svgscale\undefined%
  \makeatother%
  \begin{picture}(1,0.3512645)%
    \lineheight{1}%
    \setlength\tabcolsep{0pt}%
    \put(0,0){\includegraphics[width=\unitlength,page=1]{shear.pdf}}%
    \put(0.13913272,0.17741899){\makebox(0,0)[lt]{\lineheight{1.25}\smash{\begin{tabular}[t]{l}$\gamma$\end{tabular}}}}%
    \put(0.76548567,0.20415935){\makebox(0,0)[lt]{\lineheight{1.25}\smash{\begin{tabular}[t]{l}$\gamma$\end{tabular}}}}%
    \put(-0.00364652,0.16915414){\makebox(0,0)[lt]{\lineheight{1.25}\smash{\begin{tabular}[t]{l}$+1$\end{tabular}}}}%
    \put(0.54116748,0.16399945){\makebox(0,0)[lt]{\lineheight{1.25}\smash{\begin{tabular}[t]{l}$-1$\end{tabular}}}}%
    \put(0.22301637,0.31788809){\makebox(0,0)[lt]{\lineheight{1.25}\smash{\begin{tabular}[t]{l}$l$\end{tabular}}}}%
    \put(0.89074968,0.24403858){\makebox(0,0)[lt]{\lineheight{1.25}\smash{\begin{tabular}[t]{l}$l$\end{tabular}}}}%
  \end{picture}%
\endgroup%

%% file: elementary_lamination.pdf_tex
\begingroup%
  \makeatletter%
  \providecommand\color[2][]{%
    \errmessage{(Inkscape) Color is used for the text in Inkscape, but the package 'color.sty' is not loaded}%
    \renewcommand\color[2][]{}%
  }%
  \providecommand\transparent[1]{%
    \errmessage{(Inkscape) Transparency is used (non-zero) for the text in Inkscape, but the package 'transparent.sty' is not loaded}%
    \renewcommand\transparent[1]{}%
  }%
  \providecommand\rotatebox[2]{#2}%
  \newcommand*\fsize{\dimexpr\f@size pt\relax}%
  \newcommand*\lineheight[1]{\fontsize{\fsize}{#1\fsize}\selectfont}%
  \ifx\svgwidth\undefined%
    \setlength{\unitlength}{204.24305004bp}%
    \ifx\svgscale\undefined%
      \relax%
    \else%
      \setlength{\unitlength}{\unitlength * \real{\svgscale}}%
    \fi%
  \else%
    \setlength{\unitlength}{\svgwidth}%
  \fi%
  \global\let\svgwidth\undefined%
  \global\let\svgscale\undefined%
  \makeatother%
  \begin{picture}(1,0.14178105)%
    \lineheight{1}%
    \setlength\tabcolsep{0pt}%
    \put(0,0){\includegraphics[width=\unitlength,page=1]{elementary_lamination.pdf}}%
    \put(0.24189753,0.02732581){\makebox(0,0)[lt]{\lineheight{1.25}\smash{\begin{tabular}[t]{l}$\gamma$\end{tabular}}}}%
    \put(0.18729407,0.09714218){\makebox(0,0)[lt]{\lineheight{1.25}\smash{\begin{tabular}[t]{l}$e(\gamma)$\end{tabular}}}}%
    \put(0,0){\includegraphics[width=\unitlength,page=2]{elementary_lamination.pdf}}%
    \put(0.80309491,0.0254711){\makebox(0,0)[lt]{\lineheight{1.25}\smash{\begin{tabular}[t]{l}$\gamma$\end{tabular}}}}%
    \put(0.76996941,0.10370816){\makebox(0,0)[lt]{\lineheight{1.25}\smash{\begin{tabular}[t]{l}$e(\gamma)$\end{tabular}}}}%
    \put(0,0){\includegraphics[width=\unitlength,page=3]{elementary_lamination.pdf}}%
  \end{picture}%
\endgroup%

%% file: shear2.pdf_tex
\begingroup%
  \makeatletter%
  \providecommand\color[2][]{%
    \errmessage{(Inkscape) Color is used for the text in Inkscape, but the package 'color.sty' is not loaded}%
    \renewcommand\color[2][]{}%
  }%
  \providecommand\transparent[1]{%
    \errmessage{(Inkscape) Transparency is used (non-zero) for the text in Inkscape, but the package 'transparent.sty' is not loaded}%
    \renewcommand\transparent[1]{}%
  }%
  \providecommand\rotatebox[2]{#2}%
  \newcommand*\fsize{\dimexpr\f@size pt\relax}%
  \newcommand*\lineheight[1]{\fontsize{\fsize}{#1\fsize}\selectfont}%
  \ifx\svgwidth\undefined%
    \setlength{\unitlength}{262.66247541bp}%
    \ifx\svgscale\undefined%
      \relax%
    \else%
      \setlength{\unitlength}{\unitlength * \real{\svgscale}}%
    \fi%
  \else%
    \setlength{\unitlength}{\svgwidth}%
  \fi%
  \global\let\svgwidth\undefined%
  \global\let\svgscale\undefined%
  \makeatother%
  \begin{picture}(1,0.26363581)%
    \lineheight{1}%
    \setlength\tabcolsep{0pt}%
    \put(0,0){\includegraphics[width=\unitlength,page=1]{shear2.pdf}}%
    \put(0.13898378,0.07876165){\makebox(0,0)[lt]{\lineheight{1.25}\smash{\begin{tabular}[t]{l}$\gamma$\end{tabular}}}}%
    \put(0.75611098,0.07230699){\makebox(0,0)[lt]{\lineheight{1.25}\smash{\begin{tabular}[t]{l}$\gamma$\end{tabular}}}}%
    \put(-0.00379229,0.09300911){\makebox(0,0)[lt]{\lineheight{1.25}\smash{\begin{tabular}[t]{l}$+\frac{1}{2}$\end{tabular}}}}%
    \put(0.60277594,0.08764836){\makebox(0,0)[lt]{\lineheight{1.25}\smash{\begin{tabular}[t]{l}$-\frac{1}{2}$\end{tabular}}}}%
    \put(0.21194383,0.22892518){\makebox(0,0)[lt]{\lineheight{1.25}\smash{\begin{tabular}[t]{l}$l$\end{tabular}}}}%
    \put(0.90922519,0.13662292){\makebox(0,0)[lt]{\lineheight{1.25}\smash{\begin{tabular}[t]{l}$l$\end{tabular}}}}%
  \end{picture}%
\endgroup%

%% file: Dehn.pdf_tex
\begingroup%
  \makeatletter%
  \providecommand\color[2][]{%
    \errmessage{(Inkscape) Color is used for the text in Inkscape, but the package 'color.sty' is not loaded}%
    \renewcommand\color[2][]{}%
  }%
  \providecommand\transparent[1]{%
    \errmessage{(Inkscape) Transparency is used (non-zero) for the text in Inkscape, but the package 'transparent.sty' is not loaded}%
    \renewcommand\transparent[1]{}%
  }%
  \providecommand\rotatebox[2]{#2}%
  \newcommand*\fsize{\dimexpr\f@size pt\relax}%
  \newcommand*\lineheight[1]{\fontsize{\fsize}{#1\fsize}\selectfont}%
  \ifx\svgwidth\undefined%
    \setlength{\unitlength}{556.23611006bp}%
    \ifx\svgscale\undefined%
      \relax%
    \else%
      \setlength{\unitlength}{\unitlength * \real{\svgscale}}%
    \fi%
  \else%
    \setlength{\unitlength}{\svgwidth}%
  \fi%
  \global\let\svgwidth\undefined%
  \global\let\svgscale\undefined%
  \makeatother%
  \begin{picture}(1,0.21814278)%
    \lineheight{1}%
    \setlength\tabcolsep{0pt}%
    \put(0,0){\includegraphics[width=\unitlength,page=1]{Dehn.pdf}}%
    \put(0.45962588,0.1255904){\makebox(0,0)[lt]{\lineheight{1.25}\smash{\begin{tabular}[t]{l}$\tw_{l_c}$\end{tabular}}}}%
    \put(0.24915059,0.17852543){\makebox(0,0)[lt]{\lineheight{1.25}\smash{\begin{tabular}[t]{l}$l_c$\end{tabular}}}}%
    \put(0,0){\includegraphics[width=\unitlength,page=2]{Dehn.pdf}}%
  \end{picture}%
\endgroup%

%% file: Bangles.pdf_tex
\begingroup%
  \makeatletter%
  \providecommand\color[2][]{%
    \errmessage{(Inkscape) Color is used for the text in Inkscape, but the package 'color.sty' is not loaded}%
    \renewcommand\color[2][]{}%
  }%
  \providecommand\transparent[1]{%
    \errmessage{(Inkscape) Transparency is used (non-zero) for the text in Inkscape, but the package 'transparent.sty' is not loaded}%
    \renewcommand\transparent[1]{}%
  }%
  \providecommand\rotatebox[2]{#2}%
  \newcommand*\fsize{\dimexpr\f@size pt\relax}%
  \newcommand*\lineheight[1]{\fontsize{\fsize}{#1\fsize}\selectfont}%
  \ifx\svgwidth\undefined%
    \setlength{\unitlength}{91.91200446bp}%
    \ifx\svgscale\undefined%
      \relax%
    \else%
      \setlength{\unitlength}{\unitlength * \real{\svgscale}}%
    \fi%
  \else%
    \setlength{\unitlength}{\svgwidth}%
  \fi%
  \global\let\svgwidth\undefined%
  \global\let\svgscale\undefined%
  \makeatother%
  \begin{picture}(1,0.38569714)%
    \lineheight{1}%
    \setlength\tabcolsep{0pt}%
    \put(0,0){\includegraphics[width=\unitlength,page=1]{Bangles.pdf}}%
  \end{picture}%
\endgroup%

%% file: Bands.pdf_tex
\begingroup%
  \makeatletter%
  \providecommand\color[2][]{%
    \errmessage{(Inkscape) Color is used for the text in Inkscape, but the package 'color.sty' is not loaded}%
    \renewcommand\color[2][]{}%
  }%
  \providecommand\transparent[1]{%
    \errmessage{(Inkscape) Transparency is used (non-zero) for the text in Inkscape, but the package 'transparent.sty' is not loaded}%
    \renewcommand\transparent[1]{}%
  }%
  \providecommand\rotatebox[2]{#2}%
  \newcommand*\fsize{\dimexpr\f@size pt\relax}%
  \newcommand*\lineheight[1]{\fontsize{\fsize}{#1\fsize}\selectfont}%
  \ifx\svgwidth\undefined%
    \setlength{\unitlength}{91.91201287bp}%
    \ifx\svgscale\undefined%
      \relax%
    \else%
      \setlength{\unitlength}{\unitlength * \real{\svgscale}}%
    \fi%
  \else%
    \setlength{\unitlength}{\svgwidth}%
  \fi%
  \global\let\svgwidth\undefined%
  \global\let\svgscale\undefined%
  \makeatother%
  \begin{picture}(1,0.38569714)%
    \lineheight{1}%
    \setlength\tabcolsep{0pt}%
    \put(0,0){\includegraphics[width=\unitlength,page=1]{Bands.pdf}}%
    \put(0.34433093,0.16334658){\makebox(0,0)[lt]{\lineheight{1.25}\smash{\begin{tabular}[t]{l}$\frac{1}{5!} \sum_{\sigma \in \mathfrak{S}_5} \sigma$\end{tabular}}}}%
  \end{picture}%
\endgroup%

%% file: Bracelets.pdf_tex
\begingroup%
  \makeatletter%
  \providecommand\color[2][]{%
    \errmessage{(Inkscape) Color is used for the text in Inkscape, but the package 'color.sty' is not loaded}%
    \renewcommand\color[2][]{}%
  }%
  \providecommand\transparent[1]{%
    \errmessage{(Inkscape) Transparency is used (non-zero) for the text in Inkscape, but the package 'transparent.sty' is not loaded}%
    \renewcommand\transparent[1]{}%
  }%
  \providecommand\rotatebox[2]{#2}%
  \newcommand*\fsize{\dimexpr\f@size pt\relax}%
  \newcommand*\lineheight[1]{\fontsize{\fsize}{#1\fsize}\selectfont}%
  \ifx\svgwidth\undefined%
    \setlength{\unitlength}{91.93329601bp}%
    \ifx\svgscale\undefined%
      \relax%
    \else%
      \setlength{\unitlength}{\unitlength * \real{\svgscale}}%
    \fi%
  \else%
    \setlength{\unitlength}{\svgwidth}%
  \fi%
  \global\let\svgwidth\undefined%
  \global\let\svgscale\undefined%
  \makeatother%
  \begin{picture}(1,0.3864362)%
    \lineheight{1}%
    \setlength\tabcolsep{0pt}%
    \put(0,0){\includegraphics[width=\unitlength,page=1]{Bracelets.pdf}}%
  \end{picture}%
\endgroup%

%% file: annuli.pdf_tex
\begingroup%
  \makeatletter%
  \providecommand\color[2][]{%
    \errmessage{(Inkscape) Color is used for the text in Inkscape, but the package 'color.sty' is not loaded}%
    \renewcommand\color[2][]{}%
  }%
  \providecommand\transparent[1]{%
    \errmessage{(Inkscape) Transparency is used (non-zero) for the text in Inkscape, but the package 'transparent.sty' is not loaded}%
    \renewcommand\transparent[1]{}%
  }%
  \providecommand\rotatebox[2]{#2}%
  \ifx\svgwidth\undefined%
    \setlength{\unitlength}{318.77622783bp}%
    \ifx\svgscale\undefined%
      \relax%
    \else%
      \setlength{\unitlength}{\unitlength * \real{\svgscale}}%
    \fi%
  \else%
    \setlength{\unitlength}{\svgwidth}%
  \fi%
  \global\let\svgwidth\undefined%
  \global\let\svgscale\undefined%
  \makeatother%
  \begin{picture}(1,0.26986134)%
    \put(0,0){\includegraphics[width=\unitlength,page=1]{annuli.pdf}}%
    \put(0.38749569,0.12106674){\color[rgb]{0,0,0}\makebox(0,0)[lb]{\smash{$\alpha_1$}}}%
    \put(0.48436505,0.19477823){\color[rgb]{0,0,0}\makebox(0,0)[lb]{\smash{$\alpha_2$}}}%
    \put(0,0){\includegraphics[width=\unitlength,page=2]{annuli.pdf}}%
    \put(0.75875446,0.04919103){\color[rgb]{0,0,0}\makebox(0,0)[lb]{\smash{$\alpha_{3}$}}}%
    \put(0.82941049,0.19729534){\color[rgb]{0,0,0}\makebox(0,0)[lb]{\smash{$\alpha_2$}}}%
    \put(0,0){\includegraphics[width=\unitlength,page=3]{annuli.pdf}}%
    \put(0.14284051,0.06046302){\color[rgb]{0,0,0}\makebox(0,0)[lb]{\smash{$L$}}}%
    \put(0.13734634,0.18641226){\color[rgb]{0,0,0}\makebox(0,0)[lb]{\smash{$\alpha_2$}}}%
  \end{picture}%
\endgroup%

%% file: Annulus_Cutting.pdf_tex
\begingroup%
  \makeatletter%
  \providecommand\color[2][]{%
    \errmessage{(Inkscape) Color is used for the text in Inkscape, but the package 'color.sty' is not loaded}%
    \renewcommand\color[2][]{}%
  }%
  \providecommand\transparent[1]{%
    \errmessage{(Inkscape) Transparency is used (non-zero) for the text in Inkscape, but the package 'transparent.sty' is not loaded}%
    \renewcommand\transparent[1]{}%
  }%
  \providecommand\rotatebox[2]{#2}%
  \ifx\svgwidth\undefined%
    \setlength{\unitlength}{334.54779128bp}%
    \ifx\svgscale\undefined%
      \relax%
    \else%
      \setlength{\unitlength}{\unitlength * \real{\svgscale}}%
    \fi%
  \else%
    \setlength{\unitlength}{\svgwidth}%
  \fi%
  \global\let\svgwidth\undefined%
  \global\let\svgscale\undefined%
  \makeatother%
  \begin{picture}(1,0.11861552)%
    \put(0,0){\includegraphics[width=\unitlength,page=1]{Annulus_Cutting.pdf}}%
    \put(0.33930872,0.0508999){\color[rgb]{0,0,0}\makebox(0,0)[lb]{\smash{$=$}}}%
    \put(0.65942239,0.05179878){\color[rgb]{0,0,0}\makebox(0,0)[lb]{\smash{$\cup$}}}%
  \end{picture}%
\endgroup%

%% file: Labelled_Disconnecting_Loop.pdf_tex
\begingroup%
  \makeatletter%
  \providecommand\color[2][]{%
    \errmessage{(Inkscape) Color is used for the text in Inkscape, but the package 'color.sty' is not loaded}%
    \renewcommand\color[2][]{}%
  }%
  \providecommand\transparent[1]{%
    \errmessage{(Inkscape) Transparency is used (non-zero) for the text in Inkscape, but the package 'transparent.sty' is not loaded}%
    \renewcommand\transparent[1]{}%
  }%
  \providecommand\rotatebox[2]{#2}%
  \newcommand*\fsize{\dimexpr\f@size pt\relax}%
  \newcommand*\lineheight[1]{\fontsize{\fsize}{#1\fsize}\selectfont}%
  \ifx\svgwidth\undefined%
    \setlength{\unitlength}{580.76444408bp}%
    \ifx\svgscale\undefined%
      \relax%
    \else%
      \setlength{\unitlength}{\unitlength * \real{\svgscale}}%
    \fi%
  \else%
    \setlength{\unitlength}{\svgwidth}%
  \fi%
  \global\let\svgwidth\undefined%
  \global\let\svgscale\undefined%
  \makeatother%
  \begin{picture}(1,0.1606265)%
    \lineheight{1}%
    \setlength\tabcolsep{0pt}%
    \put(0,0){\includegraphics[width=\unitlength,page=1]{Labelled_Disconnecting_Loop.pdf}}%
    \put(0.54783482,0.06514988){\color[rgb]{0,0,0}\makebox(0,0)[lt]{\lineheight{1.25}\smash{\begin{tabular}[t]{l}{\Huge $\rightsquigarrow$}\end{tabular}}}}%
    \put(0.23914196,0.06318772){\color[rgb]{0,0,0}\makebox(0,0)[lt]{\lineheight{1.25}\smash{\begin{tabular}[t]{l}{\small $L$}\end{tabular}}}}%
    \put(0.1539058,0.05376664){\color[rgb]{0,0,0}\makebox(0,0)[lt]{\lineheight{1.25}\smash{\begin{tabular}[t]{l}{\small $\alpha_L$}\end{tabular}}}}%
    \put(0.72439843,0.06115969){\color[rgb]{0,0,0}\makebox(0,0)[lt]{\lineheight{1.25}\smash{\begin{tabular}[t]{l}{\small $L$}\end{tabular}}}}%
    \put(0.63916224,0.05173861){\color[rgb]{0,0,0}\makebox(0,0)[lt]{\lineheight{1.25}\smash{\begin{tabular}[t]{l}{\small $\alpha_L$}\end{tabular}}}}%
  \end{picture}%
\endgroup%

%% file: Arc_times_loop_2.pdf_tex
\begingroup%
  \makeatletter%
  \providecommand\color[2][]{%
    \errmessage{(Inkscape) Color is used for the text in Inkscape, but the package 'color.sty' is not loaded}%
    \renewcommand\color[2][]{}%
  }%
  \providecommand\transparent[1]{%
    \errmessage{(Inkscape) Transparency is used (non-zero) for the text in Inkscape, but the package 'transparent.sty' is not loaded}%
    \renewcommand\transparent[1]{}%
  }%
  \providecommand\rotatebox[2]{#2}%
  \newcommand*\fsize{\dimexpr\f@size pt\relax}%
  \newcommand*\lineheight[1]{\fontsize{\fsize}{#1\fsize}\selectfont}%
  \ifx\svgwidth\undefined%
    \setlength{\unitlength}{328.67264104bp}%
    \ifx\svgscale\undefined%
      \relax%
    \else%
      \setlength{\unitlength}{\unitlength * \real{\svgscale}}%
    \fi%
  \else%
    \setlength{\unitlength}{\svgwidth}%
  \fi%
  \global\let\svgwidth\undefined%
  \global\let\svgscale\undefined%
  \makeatother%
  \begin{picture}(1,0.08739651)%
    \lineheight{1}%
    \setlength\tabcolsep{0pt}%
    \put(0,0){\includegraphics[width=\unitlength,page=1]{Arc_times_loop_2.pdf}}%
    \put(0.22695868,0.04872567){\color[rgb]{0,0,0}\makebox(0,0)[lt]{\begin{minipage}{0.09920897\unitlength}\raggedright =\end{minipage}}}%
    \put(0.49144362,0.03553479){\color[rgb]{0,0,0}\makebox(0,0)[lt]{\lineheight{1.25}\smash{\begin{tabular}[t]{l}+\end{tabular}}}}%
    \put(0.75201442,0.03584666){\color[rgb]{0,0,0}\makebox(0,0)[lt]{\lineheight{1.25}\smash{\begin{tabular}[t]{l}+\end{tabular}}}}%
  \end{picture}%
\endgroup%

%% file: main.bbl
\def\cprime{$'$}
\providecommand{\bysame}{\leavevmode\hbox to3em{\hrulefill}\thinspace}
\providecommand{\MR}{\relax\ifhmode\unskip\space\fi MR }
\providecommand{\MRhref}[2]{%
  \href{http://www.ams.org/mathscinet-getitem?mr=#1}{#2}
}
\providecommand{\href}[2]{#2}
\begin{thebibliography}{CIKLFP13}

\bibitem[AB20]{allegretti2020monodromy}
Dylan Allegretti and Tom Bridgeland, \emph{The monodromy of meromorphic
  projective structures}, Transactions of the American Mathematical Society
  \textbf{373} (2020), no.~9, 6321--6367, \href
  {http://arxiv.org/abs/1802.02505} {{arXiv:1802.02505}}.

\bibitem[AK17]{Allegretti2015duality}
Dylan G.~L. Allegretti and Hyun~Kyu Kim, \emph{A duality map for quantum
  cluster varieties from surfaces}, Advances in Mathematics \textbf{306}
  (2017), 1164--1208, \href {http://arxiv.org/abs/1509.01567}
  {{arXiv:1509.01567}}.

\bibitem[All16]{allegretti2016geometry}
Dylan Gregory~Lucasi Allegretti, \emph{The geometry of cluster varieties from
  surfaces}, Ph.D. thesis, Yale University, 2016, \href
  {http://arxiv.org/abs/1606.07788} {{arXiv:1606.07788}}.

\bibitem[{Bou}20a]{Bou3}
Pierrick {Bousseau}, \emph{Quantum mirrors of log {C}alabi-{Y}au surfaces and
  higher-genus curve counting}, Compos. Math. \textbf{156} (2020), no.~2,
  360--411, \href {http://arxiv.org/abs/1808.07336} {{arXiv:1808.07336}}.

\bibitem[Bou20b]{bousseau2020strong}
Pierrick Bousseau, \emph{Strong positivity for the skein algebras of the $4
  $-punctured sphere and of the $1 $-punctured torus}, \href
  {http://arxiv.org/abs/2009.02266} {{arXiv:2009.02266}}.

\bibitem[BQ15]{brustle2015tagged}
Thomas Br{\"u}stle and Yu~Qiu, \emph{{Tagged mapping class groups:
  Auslander--Reiten translation}}, Mathematische Zeitschrift \textbf{279}
  (2015), no.~3, 1103--1120, \href {http://arxiv.org/abs/1212.0007}
  {{arXiv:1212.0007}}.

\bibitem[Bri17]{bridgeland2017scattering}
Tom Bridgeland, \emph{{Scattering diagrams, {H}all algebras and stability
  conditions}}, Algebraic Geometry \textbf{4} (2017), no.~5, 523--561, \href
  {http://arxiv.org/abs/1603.00416} {{arXiv:1603.00416}}.

\bibitem[BS15]{BS}
Tom Bridgeland and Ivan Smith, \emph{Quadratic differentials as stability
  conditions}, Publ. Math. Inst. Hautes \'{E}tudes Sci. \textbf{121} (2015),
  155--278, \href {http://arxiv.org/abs/1302.7030} {{arXiv:1302.7030}}.

\bibitem[BW11]{BW}
Francis Bonahon and Helen Wong, \emph{Quantum traces for representations of
  surface groups in {${\rm SL}_2(\Bbb C)$}}, Geom. Topol. \textbf{15} (2011),
  no.~3, 1569--1615, \href {http://arxiv.org/abs/1003.5250}
  {{arXiv:1003.5250}}.

\bibitem[BZ05]{BerensteinZelevinsky05}
Arkady Berenstein and Andrei Zelevinsky, \emph{Quantum cluster algebras}, Adv.
  Math. \textbf{195} (2005), no.~2, 405--455, \href
  {http://arxiv.org/abs/math/0404446} {{arXiv:math/0404446}}.

\bibitem[BZ14]{BerensteinZelevinsky2012}
Arkady Berenstein and Andrei Zelevinsky, \emph{Triangular bases in quantum
  cluster algebras}, International Mathematics Research Notices \textbf{2014}
  (2014), no.~6, 1651--1688, \href {http://arxiv.org/abs/1206.3586}
  {{arXiv:1206.3586}}.

\bibitem[CF99]{CF}
L.~O. {Chekhov} and V.~V. {Fock}, \emph{Quantum {T}eichm\"{u}ller spaces},
  Teoret. Mat. Fiz. \textbf{120} (1999), no.~3, 511--528, \href
  {http://arxiv.org/abs/math/9908165} {{arXiv:math/9908165}}.

\bibitem[CFMM20]{cheung2020quantization}
Man-Wai~Mandy Cheung, Juan~Bosco Fr{\'\i}as-Medina, and Timothy Magee,
  \emph{{Quantization of deformed cluster Poisson varieties}}, \href
  {http://arxiv.org/abs/2007.02479} {{arXiv:2007.02479}}.

\bibitem[CIKLFP13]{CKLP}
Giovanni Cerulli~Irelli, Bernhard Keller, Daniel Labardini-Fragoso, and
  Pierre-Guy Plamondon, \emph{Linear independence of cluster monomials for
  skew-symmetric cluster algebras}, Compos. Math. \textbf{149} (2013), no.~10,
  1753--1764, \href {http://arxiv.org/abs/1203.1307} {{arXiv:1203.1307}}.

\bibitem[CKKO20]{cho2020laurent}
So~Young Cho, Hyuna Kim, Hyun~Kyu Kim, and Doeun Oh, \emph{Laurent positivity
  of quantized canonical bases for quantum cluster varieties from surfaces},
  Communications in Mathematical Physics \textbf{373} (2020), no.~2, 655--705,
  \href {http://arxiv.org/abs/1710.06217} {{arXiv:1710.06217}}.

\bibitem[CLS15]{CLS}
Ilke Canakci, Kyungyong Lee, and Ralf Schiffler, \emph{On cluster algebras from
  unpunctured surfaces with one marked point}, Proc. Amer. Math. Soc. Ser. B
  \textbf{2} (2015), 35--49, \href {http://arxiv.org/abs/1407.5060}
  {{arXiv:1407.5060}}.

\bibitem[CMMM]{CMMM}
Man-Wai {Cheung}, Timothy {Magee}, Travis {Mandel}, and Greg {Muller},
  \emph{{Tropical theta functions and cluster varieties}}, (in preparation).

\bibitem[CMQ23]{chen2023comparison}
Qiyue Chen, Travis Mandel, and Fan Qin, \emph{A comparison between scattering
  diagrams}, in preparation.

\bibitem[CPS22]{carl2022tropical}
Michael Carl, Max Pumperla, and Bernd Siebert, \emph{{A tropical view on
  Landau-Ginzburg models}}, \href {http://arxiv.org/abs/2205.07753}
  {{arXiv:2205.07753}}.

\bibitem[{\c{C}}T19]{CT}
\.{I}lke {\c{C}}anak{\c{c}}{\i} and Pavel Tumarkin, \emph{Bases for cluster
  algebras from orbifolds with one marked point}, Algebr. Comb. \textbf{2}
  (2019), no.~3, 355--365, \href {http://arxiv.org/abs/1711.00446}
  {{arXiv:1711.00446}}.

\bibitem[Dav18]{DavPos}
Ben Davison, \emph{Positivity for quantum cluster algebras}, Ann. of Math. (2)
  \textbf{187} (2018), no.~1, 157--219, \href {http://arxiv.org/abs/1601.07918}
  {{arXiv:1601.07918}}.

\bibitem[DM21]{davison2019strong}
Ben Davison and Travis Mandel, \emph{Strong positivity for quantum theta bases
  of quantum cluster algebras}, Inventiones mathematicae (2021), 1--119, \href
  {http://arxiv.org/abs/1910.12915} {{arXiv:1910.12915}}.

\bibitem[Dup11]{dupont2011generic}
Gr{\'e}goire Dupont, \emph{Generic variables in acyclic cluster algebras},
  Journal of Pure and Applied Algebra \textbf{215} (2011), no.~4, 628--641,
  \href {http://arxiv.org/abs/0811.2909} {{arXiv:0811.2909}}.

\bibitem[DWZ10]{DerksenWeymanZelevinsky09}
Harm Derksen, Jerzy Weyman, and Andrei Zelevinsky, \emph{Quivers with
  potentials and their representations {II}: {Applications to cluster
  algebras}}, J. Amer. Math. Soc. \textbf{23} (2010), no.~3, 749--790, \href
  {http://arxiv.org/abs/0904.0676} {{arXiv:0904.0676}}.

\bibitem[FG00]{frohman2000skein}
Charles Frohman and R{\u{a}}zvan Gelca, \emph{Skein modules and the
  noncommutative torus}, Transactions of the American Mathematical Society
  \textbf{352} (2000), no.~10, 4877--4888, \href {http://arxiv.org/abs/9806107}
  {{arXiv:9806107}}.

\bibitem[FG06]{FockGoncharov06a}
Vladimir~V. Fock and Alexander~B. Goncharov, \emph{Moduli spaces of local
  systems and higher {T}eichm\"uller theory}, Publ. Math. Inst. Hautes \'Etudes
  Sci. \textbf{103} (2006), no.~1, 1--211, \href
  {http://arxiv.org/abs/math/0311149} {{arXiv:math/0311149}}.

\bibitem[FG07]{FG4}
V.~V. {Fock} and A.~{Goncharov}, \emph{Dual {T}eichm\"{u}ller and lamination
  spaces}, Handbook of {T}eichm\"{u}ller theory. {V}ol. {I}, IRMA Lect. Math.
  Theor. Phys., vol.~11, Eur. Math. Soc., Z\"{u}rich, 2007, pp.~647--684, \href
  {http://arxiv.org/abs/math/0510312} {{arXiv:math/0510312}}.

\bibitem[FG09]{FG1}
V.~{Fock} and A.~{Goncharov}, \emph{{Cluster ensembles, quantization and the
  dilogarithm}}, Ann. Sci.\'Ec. Norm. Sup. (4) \textbf{42} (2009), no.~6,
  865--930, \href {http://arxiv.org/abs/math/0311245} {{arXiv:math/0311245}}.

\bibitem[FST08]{FominShapiroThurston08}
Sergey Fomin, Michael Shapiro, and Dylan Thurston, \emph{Cluster algebras and
  triangulated surfaces. {I}. {C}luster complexes}, Acta Math. \textbf{201}
  (2008), no.~1, 83--146, \href {http://arxiv.org/abs/math/0608367}
  {{arXiv:math/0608367}}.

\bibitem[FST12]{felikson2012cluster}
Anna Felikson, Michael Shapiro, and Pavel Tumarkin, \emph{Cluster algebras of
  finite mutation type via unfoldings}, International Mathematics Research
  Notices \textbf{2012} (2012), no.~8, 1768--1804, \href
  {http://arxiv.org/abs/1006.4276} {{arXiv:1006.4276}}.

\bibitem[FT17]{FeT}
Anna Felikson and Pavel Tumarkin, \emph{Bases for cluster algebras from
  orbifolds}, Adv. Math. \textbf{318} (2017), 191--232, \href
  {http://arxiv.org/abs/1511.08023} {{arXiv:1511.08023}}.

\bibitem[FT18a]{FT}
S.~{Fomin} and D.~{Thurston}, \emph{Cluster algebras and triangulated surfaces
  {P}art {II}: {L}ambda lengths}, Mem. Amer. Math. Soc. \textbf{255} (2018),
  no.~1223, v+97, Available from: \url{https://doi.org/10.1090/memo/1223},
  \MR{3852257}

\bibitem[FT18b]{fomin2018cluster}
Sergey Fomin and Dylan Thurston, \emph{{Cluster algebras and triangulated
  surfaces Part II: Lambda lengths}}, vol. 255, American Mathematical Society,
  2018, \href {http://arxiv.org/abs/1210.5569} {{arXiv:1210.5569}}.

\bibitem[FZ02]{FominZelevinsky02}
Sergey Fomin and Andrei Zelevinsky, \emph{Cluster algebras. {I}.
  {F}oundations}, J. Amer. Math. Soc. \textbf{15} (2002), no.~2, 497--529
  (electronic), \href {http://arxiv.org/abs/math/0104151}
  {{arXiv:math/0104151}}.

\bibitem[FZ07]{FominZelevinsky07}
Sergey Fomin and Andrei Zelevinsky, \emph{Cluster algebras {IV}: Coefficients},
  Compositio Mathematica \textbf{143} (2007), 112--164, \href
  {http://arxiv.org/abs/math/0602259} {{arXiv:math/0602259}}.

\bibitem[GHK15]{gross2013birational}
Mark Gross, Paul Hacking, and Sean Keel, \emph{Birational geometry of cluster
  algebras}, Algebraic Geometry \textbf{2} (2015), no.~2, 137--175, \href
  {http://arxiv.org/abs/1309.2573} {{arXiv:1309.2573}}.

\bibitem[GHKK18]{gross2018canonical}
Mark Gross, Paul Hacking, Sean Keel, and Maxim Kontsevich, \emph{Canonical
  bases for cluster algebras}, Journal of the American Mathematical Society
  \textbf{31} (2018), no.~2, 497--608, \href {http://arxiv.org/abs/1411.1394}
  {{arXiv:1411.1394}}.

\bibitem[GLFS20]{geiss2020generic}
Christof Gei{\ss}, Daniel Labardini-Fragoso, and Jan Schr{\"o}er,
  \emph{{Generic Caldero-Chapoton functions with coefficients and applications
  to surface cluster algebras}}, \href {http://arxiv.org/abs/2007.05483}
  {{arXiv:2007.05483}}.

\bibitem[GLS12]{GeissLeclercSchroeer10b}
Christof Gei\ss, Bernard Leclerc, and Jan Schr{\"o}er, \emph{Generic bases for
  cluster algebras and the {C}hamber {A}nsatz}, J. Amer. Math. Soc. \textbf{25}
  (2012), no.~1, 21--76, \href {http://arxiv.org/abs/1004.2781v3}
  {{arXiv:1004.2781v3}}.

\bibitem[GS11]{GS11}
Mark {Gross} and Bernd {Siebert}, \emph{From real affine geometry to complex
  geometry}, Ann. of Math. (2) \textbf{174} (2011), no.~3, 1301--1428, \href
  {http://arxiv.org/abs/math/0703822} {{arXiv:math/0703822}}.

\bibitem[GS15]{GonSh}
Alexander Goncharov and Linhui Shen, \emph{Geometry of canonical bases and
  mirror symmetry}, Invent. Math. \textbf{202} (2015), no.~2, 487--633, \href
  {http://arxiv.org/abs/1309.5922} {{arXiv:1309.5922}}.

\bibitem[GS18]{GS-DT}
Alexander Goncharov and Linhui Shen, \emph{Donaldson-{T}homas transformations
  of moduli spaces of {G}-local systems}, Adv. Math. \textbf{327} (2018),
  225--348, \href {http://arxiv.org/abs/1602.06479} {{arXiv:1602.06479}}.

\bibitem[GS19]{goncharov2019quantum}
Alexander Goncharov and Linhui Shen, \emph{Quantum geometry of moduli spaces of
  local systems and representation theory}, \href
  {http://arxiv.org/abs/1904.10491} {{arXiv:1904.10491}}.

\bibitem[GSV03]{GekhtmanShapiroVainshtein03}
Michael Gekhtman, Michael Shapiro, and Alek Vainshtein, \emph{Cluster algebras
  and {P}oisson geometry}, Mosc. Math. J. \textbf{3} (2003), no.~3, 899--934,
  1199, \{Dedicated to Vladimir Igorevich Arnold on the occasion of his 65th
  birthday\}, \href {http://arxiv.org/abs/math/0208033} {{arXiv:math/0208033}}.

\bibitem[GSV05]{GekhtmanShapiroVainshtein05}
Michael Gekhtman, Michael Shapiro, and Alek Vainshtein, \emph{Cluster algebras
  and {W}eil-{P}etersson forms}, Duke Math. J. \textbf{127} (2005), no.~2,
  291--311, \href {http://arxiv.org/abs/math/0309138} {{arXiv:math/0309138}}.

\bibitem[HL10]{HernandezLeclerc09}
David Hernandez and Bernard Leclerc, \emph{Cluster algebras and quantum affine
  algebras}, Duke Math. J. \textbf{154} (2010), no.~2, 265--341, \href
  {http://arxiv.org/abs/0903.1452} {{arXiv:0903.1452}}.

\bibitem[HL18]{huang2018unfolding}
Min Huang and Fang Li, \emph{Unfolding of sign-skew-symmetric cluster algebras
  and its applications to positivity and f-polynomials}, Advances in
  Mathematics \textbf{340} (2018), 221--283, \href
  {http://arxiv.org/abs/1609.05981} {{arXiv:1609.05981}}.

\bibitem[Hua22]{huang2022expansion}
Min Huang, \emph{An expansion formula for quantum cluster algebras from
  unpunctured triangulated surfaces}, Selecta Mathematica \textbf{28} (2022),
  no.~2, 1--58, Available from:
  \url{https://doi.org/10.1007/s00029-021-00750-2}.

\bibitem[Kas91]{Kas:crystal}
M.~Kashiwara, \emph{On crystal bases of the {$Q$}-analogue of universal
  enveloping algebras}, Duke Math. J. \textbf{63} (1991), no.~2, 465--516,
  Available from: \url{http://dx.doi.org/10.1215/S0012-7094-91-06321-0},

\bibitem[Kir95]{Kir}
Anatol~N. Kirillov, \emph{Dilogarithm identities}, no. 118, 1995, Quantum field
  theory, integrable models and beyond (Kyoto, 1994), pp.~61--142, \href
  {http://arxiv.org/abs/hep-th/9408113} {{arXiv:hep-th/9408113}}.

\bibitem[KQW22]{kimura2022twist}
Yoshiyuki Kimura, Fan Qin, and Qiaoling Wei, \emph{{Twist automorphisms and
  Poisson structures}}, \href {http://arxiv.org/abs/2201.10284}
  {{arXiv:2201.10284}}.

\bibitem[KS06]{KS}
Maxim Kontsevich and Yan Soibelman, \emph{Affine structures and
  non-{A}rchimedean analytic spaces}, The unity of mathematics, Progr. Math.,
  vol. 244, Birkh\"auser Boston, Boston, MA, 2006, pp.~321--385, \href
  {http://arxiv.org/abs/math/0406564} {{arXiv:math/0406564}}.

\bibitem[KS14]{WCS}
Maxim Kontsevich and Yan Soibelman, \emph{Wall-crossing structures in
  {D}onaldson-{T}homas invariants, integrable systems and mirror symmetry},
  Homological mirror symmetry and tropical geometry, Lect. Notes Unione Mat.
  Ital., vol.~15, Springer, Cham, 2014, pp.~197--308, \href
  {http://arxiv.org/abs/1303.3253} {{arXiv:1303.3253}}.

\bibitem[L\^19]{Le-qtrace}
Thang T.~Q. L\^{e}, \emph{{Quantum Teichmüller spaces and quantum trace map}},
  J. Inst. Math. Jussieu \textbf{18} (2019), no.~2, 249--291, \href
  {http://arxiv.org/abs/1511.06054} {{arXiv:1511.06054}}.

\bibitem[Le19]{LeIan}
Ian Le, \emph{{Cluster structures on higher Teichmüller spaces for classical
  groups}}, Forum Math. Sigma \textbf{7} (2019), Paper No. e13, 165, \href
  {http://arxiv.org/abs/1603.03523} {{arXiv:1603.03523}}.

\bibitem[Liu09]{liu2009quantum}
Xiaobo Liu, \emph{{The quantum Teichm{\"u}ller space as a noncommutative
  algebraic object}}, Journal of Knot Theory and its Ramifications \textbf{18}
  (2009), no.~05, 705--726, \href {http://arxiv.org/abs/0408361}
  {{arXiv:0408361}}.

\bibitem[LLRZ14]{LLRZpnas}
Kyungyong Lee, Li~Li, Dylan Rupel, and Andrei Zelevinsky, \emph{Greedy bases in
  rank 2 quantum cluster algebras}, Proc. Natl. Acad. Sci. USA \textbf{111}
  (2014), no.~27, 9712--9716, \href {http://arxiv.org/abs/1405.2311}
  {{arXiv:1405.2311}}.

\bibitem[LLZ14]{LLZ}
Kyungyong Lee, Li~Li, and Andrei Zelevinsky, \emph{Positivity and tameness in
  rank 2 cluster algebras}, J. Algebraic Combin. \textbf{40} (2014), no.~3,
  823--840, \href {http://arxiv.org/abs/1303.5806} {{arXiv:1303.5806}}.

\bibitem[Lus90]{Lusztig90}
G.~Lusztig, \emph{Canonical bases arising from quantized enveloping algebras},
  J. Amer. Math. Soc. \textbf{3} (1990), no.~2, 447--498, Available from:
  \url{https://doi.org/10.2307/1990961}.

\bibitem[Lus91]{Lusztig91}
G.~Lusztig, \emph{Quivers, perverse sheaves, and quantized enveloping
  algebras}, J. Amer. Math. Soc. \textbf{4} (1991), no.~2, 365--421, Available
  from: \url{https://doi.org/10.2307/2939279}.

\bibitem[Lus94]{Lusztig96}
G.~Lusztig, \emph{Total positivity in reductive groups}, Lie theory and
  geometry, Progr. Math., vol. 123, Birkh\"auser Boston, Boston, MA, 1994,
  pp.~531--568, Available from:
  \url{https://doi.org/10.1007/978-1-4612-0261-5_20}.

\bibitem[Lus00]{Lusztig00}
G.~Lusztig, \emph{Semicanonical bases arising from enveloping algebras}, Adv.
  Math. \textbf{151} (2000), no.~2, 129--139, Available from:
  \url{https://doi.org/10.1006/aima.1999.1873}.

\bibitem[{Man}17]{ManAtomic}
Travis {Mandel}, \emph{Theta bases are atomic}, Compos. Math. \textbf{153}
  (2017), no.~6, 1217--1219, \href {http://arxiv.org/abs/1605.03202}
  {{arXiv:1605.03202}}.

\bibitem[Man21]{mandel2021scattering}
Travis Mandel, \emph{Scattering diagrams, theta functions, and refined tropical
  curve counts}, Journal of the London Mathematical Society \textbf{104}
  (2021), no.~5, 2299--2334, \href {http://arxiv.org/abs/1503.06183}
  {{arXiv:1503.06183}}.

\bibitem[Mou19]{mou2019scattering}
Lang Mou, \emph{Scattering diagrams of quivers with potentials and mutations},
  \href {http://arxiv.org/abs/:1910.13714} {{arXiv::1910.13714}}.

\bibitem[MSW11]{MusikerSchifflerWilliams09}
Gregg Musiker, Ralf Schiffler, and Lauren Williams, \emph{Positivity for
  cluster algebras from surfaces}, Adv. Math. \textbf{227} (2011), no.~6,
  2241--2308, \href {http://arxiv.org/abs/0906.0748} {{arXiv:0906.0748}}.

\bibitem[MSW13]{musiker2013bases}
Gregg Musiker, Ralf Schiffler, and Lauren Williams, \emph{Bases for cluster
  algebras from surfaces}, Compositio Mathematica \textbf{149} (2013), no.~02,
  217--263, \href {http://arxiv.org/abs/arXiv:1110.4364}
  {{arXiv:arXiv:1110.4364}}.

\bibitem[{Mul}13]{MuLaca}
Greg {Muller}, \emph{Locally acyclic cluster algebras}, Adv. Math. \textbf{233}
  (2013), 207--247, \href {http://arxiv.org/abs/1111.4468} {{arXiv:1111.4468}}.

\bibitem[{Mul}16a]{MuGreen}
Greg {Muller}, \emph{The existence of a maximal green sequence is not invariant
  under quiver mutation}, Electron. J. Combin. \textbf{23} (2016), no.~2, Paper
  2.47, 23, \href {http://arxiv.org/abs/1503.04675} {{arXiv:1503.04675}}.

\bibitem[Mul16b]{muller2016skein}
Greg Muller, \emph{Skein and cluster algebras of marked surfaces}, Quantum
  topology \textbf{7} (2016), no.~3, 435--503, \href
  {http://arxiv.org/abs/1204.0020} {{arXiv:1204.0020}}.

\bibitem[MW13]{musiker2013matrix}
Gregg Musiker and Lauren Williams, \emph{Matrix formulae and skein relations
  for cluster algebras from surfaces}, International Mathematics Research
  Notices \textbf{2013} (2013), no.~13, 2891--2944, \href
  {http://arxiv.org/abs/1108.3382} {{arXiv:1108.3382}}.

\bibitem[Nak22]{nakanishi2022pentagon}
Tomoki Nakanishi, \emph{Pentagon relation in quantum cluster scattering
  diagrams}, \href {http://arxiv.org/abs/2202.01588} {{arXiv:2202.01588}}.

\bibitem[{Pen}87]{Pen87}
R.~C. {Penner}, \emph{The decorated {T}eichm\"{u}ller space of punctured
  surfaces}, Comm. Math. Phys. \textbf{113} (1987), no.~2, 299--339, Available
  from: \url{http://projecteuclid.org/euclid.cmp/1104160216}.

\bibitem[{Pen}04]{Pen04}
R.~C. {Penner}, \emph{Decorated {T}eichm\"{u}ller theory of bordered surfaces},
  Comm. Anal. Geom. \textbf{12} (2004), no.~4, 793--820, \href
  {http://arxiv.org/abs/math/0210326} {{arXiv:math/0210326}}.

\bibitem[Pen12]{Pen12}
Robert~C. Penner, \emph{Decorated {T}eichm\"{u}ller theory}, QGM Master Class
  Series, European Mathematical Society (EMS), Z\"{u}rich, 2012, With a
  foreword by Yuri I. Manin, Available from: \url{https://doi.org/10.4171/075}.

\bibitem[Pla13]{plamondon2013generic}
Pierre-Guy Plamondon, \emph{Generic bases for cluster algebras from the cluster
  category}, International Mathematics Research Notices \textbf{2013} (2013),
  no.~10, 2368--2420, \href {http://arxiv.org/abs/1111.4431}
  {{arXiv:1111.4431}}.

\bibitem[Qin14]{Qin12}
Fan Qin, \emph{t-analog of q-characters, bases of quantum cluster algebras, and
  a correction technique}, International Mathematics Research Notices
  \textbf{2014} (2014), no.~22, 6175--6232, \href
  {http://arxiv.org/abs/1207.6604} {{arXiv:1207.6604}},

\bibitem[Qin17]{qin2017triangular}
Fan Qin, \emph{Triangular bases in quantum cluster algebras and monoidal
  categorification conjectures}, Duke Mathematical Journal \textbf{166} (2017),
  no.~12, 2337--2442, \href {http://arxiv.org/abs/1501.04085}
  {{arXiv:1501.04085}}.

\bibitem[Qin20]{qin2020dual}
Fan Qin, \emph{Dual canonical bases and quantum cluster algebras}, \href
  {http://arxiv.org/abs/2003.13674} {{arXiv:2003.13674}}.

\bibitem[Qin22]{qin2019bases}
Fan Qin, \emph{Bases for upper cluster algebras and tropical points}, Journal
  of the European Mathematical Society (2022), \href
  {http://arxiv.org/abs/1902.09507} {{arXiv:1902.09507}}.

\bibitem[Qin23]{qin2022freezing}
Fan Qin, \emph{{Bases for strata of algebraic groups via operations in cluster
  theory}}, in preparation (2023).

\bibitem[Que22]{queffelec2022gl2}
Hoel Queffelec, \emph{Gl2 foam functoriality and skein positivity}, \href
  {http://arxiv.org/abs/2209.08794} {{arXiv:2209.08794}}.

\bibitem[QW18]{QW-ewp}
Hoel Queffelec and Paul Wedrich, \emph{Extremal weight projectors}, Math. Res.
  Lett. \textbf{25} (2018), no.~6, 1911--1936, \href
  {http://arxiv.org/abs/1701.02316} {{arXiv:1701.02316}}.

\bibitem[QW21]{QW-Khovanov}
Hoel Queffelec and Paul Wedrich, \emph{Khovanov homology and categorification
  of skein modules}, Quantum Topol. \textbf{12} (2021), no.~1, 129--209, \href
  {http://arxiv.org/abs/1806.03416} {{arXiv:1806.03416}}.

\bibitem[Rea14]{reading2014universal}
Nathan Reading, \emph{Universal geometric cluster algebras from surfaces},
  Transactions of the American Mathematical Society \textbf{366} (2014),
  no.~12, 6647--6685, \href {http://arxiv.org/abs/1209.4095}
  {{arXiv:1209.4095}}.

\bibitem[Rea20]{reading2020scattering}
Nathan Reading, \emph{Scattering fans}, International Mathematics Research
  Notices \textbf{2020} (2020), no.~23, 9640--9673, \href
  {http://arxiv.org/abs/1712.06968} {{arXiv:1712.06968}}.

\bibitem[RY14]{roger2014skein}
Julien Roger and Tian Yang, \emph{{The skein algebra of arcs and links and the
  decorated Teichm{\"u}ller space}}, Journal of Differential Geometry
  \textbf{96} (2014), no.~1, 95--140, \href {http://arxiv.org/abs/1110.2748}
  {{arXiv:1110.2748}}.

\bibitem[SSW23]{Shen2023skein}
Linhui Shen, Zhe Sun, and Daping Weng, \emph{{The punctured $SL_n$ skein
  algebra and quantization of $A_{SL_n,\hat{S}}$ moduli space}}, in preparation
  (2023).

\bibitem[SZ04]{SZ}
Paul Sherman and Andrei Zelevinsky, \emph{Positivity and canonical bases in
  rank 2 cluster algebras of finite and affine types}, Mosc. Math. J.
  \textbf{4} (2004), no.~4, 947--974, 982, \href
  {http://arxiv.org/abs/math/0307082} {{arXiv:math/0307082}}.

\bibitem[Thu14]{Thurst}
Dylan~Paul Thurston, \emph{Positive basis for surface skein algebras}, Proc.
  Natl. Acad. Sci. USA \textbf{111} (2014), no.~27, 9725--9732, \href
  {http://arxiv.org/abs/1310.1959} {{arXiv:1310.1959}}.

\bibitem[Tra11]{Tran09}
Thao Tran, \emph{F-polynomials in quantum cluster algebras}, Algebr. Represent.
  Theory \textbf{14} (2011), no.~6, 1025--1061, \href
  {http://arxiv.org/abs/0904.3291v1} {{arXiv:0904.3291v1}}.

\bibitem[Wil20]{wilson2020surface}
Jon Wilson, \emph{Surface cluster algebra expansion formulae via loop graphs},
  \href {http://arxiv.org/abs/2006.13218} {{arXiv:2006.13218}}.

\bibitem[Yur20]{yurikusa2020density}
Toshiya Yurikusa, \emph{Density of {$g$}-vector cones from triangulated
  surfaces}, International Mathematics Research Notices \textbf{2020} (2020),
  no.~21, 8081--8119, \href {http://arxiv.org/abs/1904.12479}
  {{arXiv:1904.12479}}.

\bibitem[Zho20]{zhou2020cluster}
Yan Zhou, \emph{Cluster structures and subfans in scattering diagrams}, SIGMA.
  Symmetry, Integrability and Geometry: Methods and Applications \textbf{16}
  (2020), 013, \href {http://arxiv.org/abs/1901.04166} {{arXiv:1901.04166}}.

\end{thebibliography}
